\documentclass[a4paper,12pt,reqno]{amsart}
\usepackage[english]{babel}
\usepackage{amsmath}
\usepackage{amscd}
\usepackage{amsgen}
\usepackage{amsfonts}
\usepackage{amssymb}
\usepackage{amsthm}
\usepackage{latexsym}
\usepackage{url}
\usepackage{slashed}
\usepackage{mathrsfs}
\usepackage{tocvsec2}
\usepackage{comment}

\newtheorem{theorem}{Theorem}

\newtheorem{prop}[theorem]{Proposition}
\newtheorem{lem}[theorem]{Lemma}
\newtheorem{cor}[theorem]{Corollary}

\newtheorem{defn}[theorem]{Definition}
\newtheorem{bem}[theorem]{Remark}
\newtheorem{example}[theorem]{Example}
\newtheorem{thm}{Theorem}

\newcommand{\R}{\mathbb{R}}

\newcommand{\C}{\mathbb{C}}
\newcommand{\N}{\mathbb{N}}

\newcommand{\Ha}{\mathcal H}
\newcommand{\Lie}{\mathcal L}
\def\H{{\mathcal H}}

\newcommand{\id}{\operatorname{Id}}

\newcommand{\tr}{\operatorname{tr}}
\newcommand{\pr}{pr}
\newcommand{\abs}[1]{\lvert#1\rvert}
\newcommand{\grad}{\text{grad}}

\newcommand{\gog}{{\mathfrak g}}
\newcommand{\gol}{{\mathfrak l}}
\newcommand{\gom}{{\mathfrak m}}
\newcommand{\goa}{{\mathfrak a}}
\newcommand{\gon}{{\mathfrak n}}
\newcommand{\gop}{{\mathfrak p}}

\newcommand{\goso}{{\mathfrak so}}
\newcommand{\Pol}{\operatorname{Pol}}
\newcommand{\Sol}{\operatorname{Sol}}

\newcommand{\Diff}{\operatorname{Diff}}
\newcommand{\Ad}{\operatorname{Ad}}
\newcommand{\GL}{\operatorname{GL}}

\newcommand{\Ind}{\operatorname{Ind}}
\newcommand{\Hom}{\operatorname{Hom}}
\newcommand{\ch}{{\,\vee}}
\newcommand{\Ch}{{\bf ch}}

\newcommand{\dm}{d}

\def\st{\stackrel{\text{def}}{=}}
\def\X{{\mathfrak X}}

\def\dim{\operatorname{dim}}

\numberwithin{theorem}{subsection}
\numberwithin{equation}{section}

\allowdisplaybreaks

\begin{document}

\title[Symmetry breaking operators on differential forms]
{Conformal symmetry breaking differential operators on differential forms}

\author[Matthias Fischmann, Andreas Juhl and Petr Somberg]
{Matthias Fischmann$^1$, Andreas Juhl$^2$ and Petr Somberg$^1$}

\address{$^1$ E. \v{C}ech Institute, Mathematical Institute of Charles University,
Sokolovsk\'a 83, Praha 8 - Karl\'{\i}n, Czech Republic}

\email{fischmann@karlin.mff.cuni.cz, somberg@karlin.mff.cuni.cz}

\address{$^2$ Humboldt-Universit\"at, Institut f\"ur Mathematik, Unter den Linden
6, 10099 Berlin, Germany}

\email{ajuhl@math.hu-berlin.de}

\keywords{Symmetry breaking operators, Branson-Gover operators, $Q$-curvature
operators, residue families, generalized Verma modules, branching problems,
singular vectors, Jacobi and Gegenbauer polynomials}

\subjclass[2010]{Primary 22E46, 35J30, 53A30; Secondary 22E47, 33C45}

\thanks{$^{1}$ Research supported by grant GA~CR~P201/12/G028.
\newline \indent
$^2$ Research supported by SFB $647$ "Space-Time-Matter" of DFG at
Humboldt-University Berlin.}

\begin{abstract}
We study conformal symmetry breaking differential operators which map
differential forms on $\R^n$ to differential forms on a codimension one
subspace $\R^{n-1}$. These operators are equivariant with respect to the
conformal Lie algebra of the subspace $\R^{n-1}$. They correspond to
homomorphisms of generalized Verma modules for $\goso(n,1)$ into generalized
Verma modules for $\goso(n+1,1)$ both being induced from fundamental form
representations of a parabolic subalgebra. We apply the $F$-method to derive
explicit formulas for such homomorphisms. In particular, we find explicit
formulas for the generators of the intertwining operators of the related
branching problems restricting generalized Verma modules for $\goso(n+1,1)$ to
$\goso(n,1)$. As consequences, we find closed formulas for all conformal
symmetry breaking differential operators in terms of the first-order operators
$\dm$, $\delta$, $\bar{\dm}$ and $\bar{\delta}$ and certain hypergeometric
polynomials. A dominant role in these studies will be played by two infinite
sequences of symmetry breaking differential operators which depend on a complex
parameter $\lambda$. These will be termed the conformal first and second type
symmetry breaking operators. Their values at special values of $\lambda$ appear
as factors in two systems of factorization identities which involve the
Branson-Gover operators of the Euclidean metrics on $\R^n$ and $\R^{n-1}$ and
the operators $\dm$, $\delta$, $\bar{\dm}$ and $\bar{\delta}$ as factors,
respectively. Moreover, they are shown to naturally recover the gauge companion
and $Q$-curvature operators of the Euclidean metric on the subspace $\R^{n-1}$,
respectively.
\end{abstract}

\maketitle

\centerline \today

\tableofcontents

\section{Introduction}\label{intro}

The present paper is motivated by recent developments in conformal differential
geometry. One of the central notions in this area is Branson's critical
$Q$-curvature $Q_n(g) \in C^\infty(M)$ of a Riemannian manifold $(M,g)$ of even
dimension $n$. The curvature invariant $Q_n(g)$ is the last element in the
sequence $Q_2(g), Q_4(g), \dots, Q_n(g)$ of curvature invariants of respective
orders $2,4,\dots,n$ \cite{Branson}. If $(M,g)$ is regarded as a hypersurface
in the Riemannian manifold $(X,g)$, the $Q$-curvatures of $(M,g)$ naturally
contribute to the structure of one-parameter families of conformally covariant
differential operators
\begin{equation}\label{curved-D}
   D_{2N}(g;\lambda): C^\infty(X) \to C^\infty(M), \; \lambda \in \C
\end{equation}
of order $2N$ mapping functions on the ambient manifold $X$ to functions on $M$
\cite{Juhl}. Here conformal covariance means that
$$
   e^{(\lambda+N)\iota^*(\varphi)} D_{2N}(e^{2\varphi}g;\lambda)(u) = D_{2N}(g;\lambda)
   \left(e^{\lambda \varphi} u\right)
$$
for all $u, \varphi \in C^\infty(X)$, and $\iota^*$ denotes the pull-back
defined by the embedding $\iota: M \hookrightarrow X$. Such families generalize
the even-order families
\begin{equation}\label{intertwine}
   D_{2N}(\lambda): C^\infty(S^{n+1}) \to C^\infty(S^n), \; \lambda \in \C
\end{equation}
of differential operators which are associated to the equatorial embedding $S^n
\hookrightarrow S^{n+1}$ of spheres, and which intertwine spherical principal
series representations of the conformal group of the embedded round sphere
$S^n$, but not of the conformal group of the ambient round sphere $S^{n+1}$.
The latter lack of invariance motivates to refer to them as {\it symmetry
breaking operators} \cite{kobayashi-speh, Kobayashi-Pevzner}.

The closely related residue families \cite{Juhl} are associated to embeddings
$\iota: M \to M \times [0,\varepsilon)$, $\iota(m) = (m,0)$ of $M$ into small
neighborhoods. In this case, the metrics on the neighborhoods are determined by
the Poincar\'e-Einstein metric of the metric $g$ on $M$ in the sense of
Fefferman and Graham \cite{FG3}. The connection of residue families to
$Q$-curvatures reveals astonishing recursive structures in the sequence
$Q_{2N}$ of $Q$-curvatures \cite{Juhl1}.

The equivariant flat models $C^\infty(\R^{n+1}) \to C^\infty(\R^n)$ of
\eqref{intertwine} coincide with the residue families for the flat metric on
$\R^n$. In turn, these operators correspond to homomorphisms of generalized
Verma modules. Their classification is an algebraic problem \cite{Juhl},
\cite{koss}.

In the present paper, we study conformal symmetry breaking differential
operators
$$
   \Omega^p(\R^n) \to \Omega^q(\R^{n-1})
$$
on differential forms which are associated to the embedding $\iota: \R^{n-1}
\to \R^n$. We shall use Euclidean coordinates $x_1,\dots,x_n$ on $\R^n$ and
assume that $\R^{n-1} \subset \R^n$ is the hyperplane $x_n = 0$. The operators
of interest are equivariant for the conformal Lie algebra
$\mathfrak{so}(n,1,\R)$ of $\R^{n-1}$ with respect to principal series
representations of $\mathfrak{so}(n,1,\R)$ and $\mathfrak{so}(n+1,1,\R)$
realized on forms on $\R^{n-1}$ and $\R^n$, respectively. The operators
correspond to symmetry breaking differential operators $\Omega^p (S^n) \to
\Omega^q(S^{n-1})$. We emphasize that the study of symmetry breaking operators
of that type which are {\em not} differential operators remains outside the
scope of this paper. However, for the sake of simplicity we shall often refer
to conformal symmetry breaking differential operators just as conformal
symmetry breaking operators.

A fundamental method for the systematic study of symmetry breaking differential
operators was introduced in \cite{koss,kob,Kobayashi-Pevzner} and is termed the
$F$-method. It provides a systematic method to study the corresponding
homomorphism of generalized Verma modules. In our situation, it yields a
description of the corresponding homomorphisms of generalized Verma modules for
the two orthogonal Lie algebras $\mathfrak{so}(n,1,\R)$ and
$\mathfrak{so}(n+1,1,\R)$. Applying the $F$-method, actually leads to explicit
descriptions of all conformal symmetry breaking differential operators on
differential forms. It turns out that the operators of interest are given by
two infinite sequences
$$
   D_N^{(p \to p)}(\lambda): \Omega^p(\R^n) \to \Omega^p(\R^{n-1}) \quad
   \mbox{and} \quad D_N^{(p \to p-1)}(\lambda): \Omega^p(\R^n) \to
   \Omega^{p-1}(\R^{n-1}), \; \N \in \N_0
$$
of one-parameter families, some additional operators, and compositions of these
with the Hodge star operator of the Euclidean metric on $\R^{n-1}$. These
operators will be referred to as the first and second type families and the
third and fourth type operators, respectively. We shall display explicit
formulas for all operators and establish their basic mapping properties.

The following theorem is our main result on even-order families of the first
type (Theorem \ref{coeffeven}).

\begin{thm}\label{main-even} Assume that $N \in \N$ and $0 \le p \le n-1$. The one-parameter family
\begin{equation*}
   D_{2N}^{(p \to p)}(\lambda): \Omega^p(\R^n) \to \Omega^p(\R^{n-1}), \; \lambda \in \C,
\end{equation*}
of differential operators, which is defined by the formula
\begin{align}\label{SBO-even-0}
   & D_{2N}^{(p \to p)}(\lambda) = \sum_{i=1}^{N-1} (\lambda\!+\!p\!-\!2i) \alpha_i^{(N)}(\lambda) (\dm \delta)^{N-i}
   \iota^* (\bar{\delta} \bar{\dm})^i \notag \\
   & + (\lambda\!+\!p) \sum_{i=0}^N \alpha_i^{(N)}(\lambda) (\dm\delta)^{N-i} \iota^* (\bar{\dm} \bar{\delta})^i
   + (\lambda\!+\!p\!-\!2N) \sum_{i=0}^N \alpha_i^{(N)}(\lambda) (\delta \dm)^{N-i}
   \iota^*(\bar{\delta} \bar{\dm})^i
\end{align}
with the coefficients
\begin{equation}\label{SBO-coeff-a}
   \alpha_i^{(N)}(\lambda) = (-1)^{i} 2^N \frac{N!}{(2N)!} \binom{N}{i}
   \prod_{k=i+1}^N (2\lambda\!+\!n\!-\!2k) \prod_{k=1}^i (2\lambda\!+\!n\!-\!2k\!-\!2N\!+\!1),
\end{equation}
satisfies the intertwining relation
\begin{equation}\label{inter-even}
   \dm\pi^{\prime (p)}_{-\lambda+2N-p}(X) D_{2N}^{(p \to p)}(\lambda) =
   D_{2N}^{(p \to p)}(\lambda) \dm\pi_{-\lambda-p}^{(p)}(X) \quad
   \mbox{for all $X \in \mathfrak{so}(n,1,\R)$}.
\end{equation}
\end{thm}

Here we used the following notation. In \eqref{SBO-even-0}, the operators $\dm$
and $\delta$ are the differential and co-differential on forms on $\R^{n-1}$.
Their counterparts on $\R^n$ are denoted by $\bar{\dm}$ and $\bar{\delta}$.
$\iota^*$ denotes the pull-back induced by the embedding $\iota: \R^{n-1}
\hookrightarrow \R^n$. Finally, the representations $\pi_{\lambda}^{(p)}$ on
$\Omega^p(\R^{n})$ are defined by
$$
   \pi_{\lambda}^{(p)}(\gamma) = e^{\lambda \Phi_\gamma} \gamma_*
$$
for all conformal diffeomeorphisms $\gamma$ of the Euclidean metric $g_0$ on
$\R^n$ so that $\gamma_*(g_0) = e^{2\Phi_\gamma} g_0$ for some $\Phi_\gamma \in
C^\infty(\R^{n})$. The analogous representations on $\Omega^p(\R^{n-1})$ are
denoted by $\pi^{\prime (p)}_\lambda$. The coefficients
$\alpha_i^{(N)}(\lambda)$ are related to Jacobi polynomials (Remark
\ref{coeff-Jacobi}).

Note that one may also include (a constant multiple of) the operator $\iota^*$
as the case $N=0$ in the sequence $D_{2N}^{(p \to p)}(\lambda)$.

By formula \eqref{SBO-even-0}, the families $D_{2N}^{(p \to p)}(\lambda)$ are
{\em linear} in the degree $p$ of the forms. The family $D_{2N}^{(0 \to
0)}(\lambda)$ equals the product of $(\lambda-2N)$ and the equivariant
even-order family $D_{2N}(\lambda)$ studied in \cite{Juhl} and
\cite{koss}.\footnote{The families $D_{2N}^{(p\to p)}(\lambda)$ raise conformal
weights. In contrast, all families in \cite{Juhl} lower conformal weights
because of opposite conventions for representations. The conventions here
correspond to those in \cite{koss}.} We recall that the families
$D_{2N}(\lambda): C^\infty(S^n) \to C^\infty(S^{n-1})$ interpolate between the
GJMS-operators of order $2N$ on $S^n$ and $S^{n-1}$. Similarly, the conformal
symmetry breaking operators $D_{2N}^{(p \to p)}(\lambda)$ interpolate between
\begin{equation}\label{interpol-1}
   D_{2N}^{(p \to p)}(N\!-\!\tfrac{n-1}{2}) = -L_{2N}^{(p)} \iota^*
\end{equation}
and
\begin{equation}\label{interpol-2}
   D_{2N}^{(p \to p)}(N\!-\!\tfrac{n}{2}) = -\iota^* \bar{L}_{2N},
\end{equation}
where $L_{2N}^{(p)}$ and $\bar{L}_{2N}^{(p)}$ are the respective Branson-Gover
operators of order $2N$ on the Euclidean spaces $\R^{n-1}$ and $\R^n$
\cite{BransonGover}. These identities are special cases of the factorizations
of the families $D_{2N}^{(p \to p)}(\lambda)$ (for specific choices of
$\lambda$) into products of lower-order families and Branson-Gover operators
(Theorem \ref{MainFactEven1}).

The following result provides the analogous formula for odd-order families
$\Omega^p(\R^n) \to \Omega^p(\R^{n-1})$ of the first type (see Theorem
\ref{coeffodd2} and Remark \ref{gamma-alternative}).

\begin{thm}\label{main-odd} Assume that $N \in \N_0$ and $0 \le p \le n-1$. The
one-parameter family
\begin{equation*}
   D_{2N+1}^{(p \to p)}(\lambda): \Omega^p(\R^n) \to \Omega^p(\R^{n-1}), \; \lambda \in \C,
\end{equation*}
of differential operators, which is defined by the formula
\begin{align}\label{SBO-odd-0}
   D_{2N+1}^{(p \to p)}(\lambda) & = \sum_{i=1}^N  \gamma_i^{(N)}(\lambda;p)(\dm \delta)^{N-i}\dm \iota^* i_{\partial_n}
   (\bar{\delta}\bar{\dm})^i \notag \\
   & + (\lambda\!+\!p) \sum_{i=0}^N  \beta_i^{(N)}(\lambda)
   (\dm\delta)^{N-i}\dm \iota^* i_{\partial_n}(\bar{\dm}\bar{\delta})^i \notag \\
   & + (\lambda\!+\!p\!-\!2N\!-\!1) \sum_{i=0}^{N}  \beta_i^{(N)}(\lambda) (\delta\dm)^{N-i}
   \iota^* i_{\partial_n} \bar{\dm}(\bar{\delta}\bar{\dm})^i
\end{align}
with the coefficients
\begin{equation}
   \beta_i^{(N)}(\lambda) \st (-1)^i 2^N \frac{N!}{(2N\!+\!1)!} \binom{N}{i}
   \prod_{k=i+1}^N (2\lambda\!+\!n\!-\!2k) \prod_{k=1}^i (2\lambda\!+\!n\!-\!2k\!-\!2N\!-\!1)
\end{equation}
and
\begin{equation}
   \gamma_i^{(N)}(\lambda;p) = (\lambda\!+\!p\!-\!2i) \beta_i^{(N)}(\lambda) -(\lambda\!+\!p\!-\!2i\!+\!1)
   \beta_{i-1}^{(N)}(\lambda),
\end{equation}
satisfies the intertwining relation
\begin{equation}\label{inter-odd}
   \dm\pi^{\prime (p)}_{-\lambda+2N+1-p}(X) D_{2N+1}^{(p \to p)}(\lambda) =
   D_{2N+1}^{(p \to p)}(\lambda) \dm\pi_{-\lambda-p}^{(p)}(X) \quad
   \mbox{for all $X \in \mathfrak{so}(n,1,\R)$}.
\end{equation}
\end{thm}

Again, the coefficients $\beta_i^{(N)}(\lambda)$ are related to Jacobi
polynomials (Remark \ref{coeff-Jacobi}). Moreover, $\partial_n$ denotes the
normal vector field of the hyperplane $\R^{n-1}$ and $i_{\partial_n}$ inserts
that field into the forms. For $p=0$, the family $D_{2N+1}^{(p \to
p)}(\lambda)$ reduces to the product of $(\lambda-2N-1)$ and the equivariant
odd-order family $D_{2N+1}(\lambda)$ studied in \cite{Juhl}

The analogous formulas for the conformal symmetry breaking families of the
second type follow from the above results for the families of the first type by
{\em conjugation} with the Hodge star operators on $\R^n$ and $\R^{n-1}$
(Theorem \ref{Hodge-c} and Theorems \ref{coeffeven2}, \ref{coeffodd}). These
operators generalize the operator $\iota^* i_{\partial_n}: \Omega^p(\R^n) \to
\Omega^{p-1}(\R^{n-1})$.

For general form degrees $p$, the families in Theorem \ref{main-even} and
Theorem \ref{main-odd} are uniquely determined by their equivariance, up to
scalar multiples. However, on middle degree forms $p=\frac{n}{2}$ if $n$ is
even and $p=\frac{n-1}{2}$ if $n$ is odd, there are additional equivariant
families. In fact, if $n$ is even, additional conformal symmetry breaking
operators
$$
   \Omega^{\frac{n}{2}}(\R^n) \to \Omega^{\frac{n}{2}}(\R^{n-1}) \quad \mbox{and}
   \quad \Omega^{\frac{n}{2}}(\R^n) \to \Omega^{\frac{n}{2}-1}(\R^{n-1})
$$
are given by the respective compositions
\begin{equation}\label{add-1}
   D_N^{(\frac{n}{2} \to \frac{n}{2})}(\lambda) \, \bar{\star} \quad \mbox{and} \quad
   D_N^{(\frac{n}{2} \to \frac{n}{2}-1)}(\lambda) \, \bar{\star},
\end{equation}
where $\bar{\star}$ denotes the Hodge star operator of the Euclidean metric on
$\R^n$. Similarly, for odd $n$, additional conformal symmetry breaking
operators
$$
   \Omega^{\frac{n-1}{2}}(\R^n) \to \Omega^{\frac{n-1}{2}}(\R^{n-1}) \quad
   \mbox{and} \quad \Omega^{\frac{n+1}{2}}(\R^n) \to \Omega^{\frac{n-1}{2}}(\R^{n-1})
$$
are given by the respective compositions
\begin{equation}\label{add-2}
   \star \, D_N^{(\frac{n-1}{2} \to \frac{n-1}{2})}(\lambda) \quad \mbox{and} \quad
   \star D_N^{(\frac{n+1}{2} \to \frac{n-1}{2})}(\lambda),
\end{equation}
where $\star$ is the Hodge star operator of the Euclidean metric on $\R^{n-1}$.

For general form degrees, the analogous {\em compositions} of families of the
first and the second type with Hodge star operators yield further equivariant
differential operators on differential forms. The following classification
result describes all such symmetry breaking differential operators.

\begin{thm}\label{classification} The linear differential operators
$$
   D: \Omega^p(\R^n) \to \Omega^q(\R^{n-1})
$$
which satisfy the intertwining relation
\begin{equation}\label{inter}
   \dm\pi^{\prime (q)}_\eta (X) D = D \dm\pi_\mu^{(p)}(X)
\end{equation}
for all $X \in \mathfrak{so}(n,1,\R)$ and some $\mu, \eta \in \C$ are generated
by the following operators.
\begin{itemize}
\item [(1)] {The case $q=p$}. The families $D_N^{(p \to p)}(\lambda)$ of the first type. Here $N \in
\N_0$, $\lambda \in \C$ and $\mu=-\lambda-p$, $\eta = -\lambda-p+N$.
\item [(2)] {The case $q=p-1$}. The families $D_N^{(p \to p-1)}(\lambda)$ of the second type. Here $N \in
\N_0$, $\lambda \in \C$ and $\mu=-\lambda-p$, $\eta = -\lambda-(p-1)+N$.
\item [(3)] {The case $q=1$ and $p=0$}. The operators $d \dot{D}_N^{(0 \to
0)}(N)$ with $N \in \N_0$. Here $\mu=-N$ and $\eta=0$.
\item [(4)] {The case $q=p+1$}. The operator $d \iota^*$. Here $\mu=\eta=0$.
\item [(5)] {The case $q=n-2$ and $p=n$}. The operators $\delta \dot{D}_N^{(n \to
n-1)}(N)$ with $N \in \N_0$. Here $\mu=-n-N$ and $\eta=-n+3$.
\item [(6)] {The case $q=p-2$}. The operator $\delta \iota^* i_{\partial_n}$. Here $\mu=n-2p$ and $\eta=n-2p+3$.
\item [(7)] The compositions of the above operators with $\star$.
\end{itemize}
In (3) and (5), the dot denotes derivatives of the families with respect to the
parameter $\lambda$.
\end{thm}

Some comments concerning this classification are in order.

Firstly, we observe that $d D_N^{(p \to p)}(N-p) = 0$ and that the derivative
of the composition $d D_N^{(p \to p)}(\lambda)$ (with respect to the parameter
$\lambda$) at $\lambda = N-p$ is equivariant only if $p=0$.\footnote{Only for
the degree $p=0$ the first type families $D_N^{(p \to p)}(\lambda)$ are given
by one sum. They have the property $D_N^{(0 \to 0)}(N)=0$.} This yields the
operators in Theorem \ref{classification}/(3). Similarly, $\delta D_N^{(p \to
p-1)}(N-n+p) = 0$ and the derivative of the composition $\delta D_N^{(p \to
p-1)}(\lambda)$ at $\lambda = N-n+p$ is equivariant only if
$p=n$.\footnote{Only for the degree $p=n$, the second type families $D_N^{(p
\to p-1)}(\lambda)$ are given by one sum. They have the property $D_N^{(n \to
n-1)}(N) = 0$.} This yields the operators in Theorem \ref{classification}/(5).

There are some overlaps in the classification: Indeed, for $N=0$, (3) is a
special case of (4), and (5) is a special case of (6).

Theorem \ref{Hodge-c} implies that the operators in Theorem
\ref{classification}/(1),(2), Theorem \ref{classification}/(3),(5) and Theorem
\ref{classification}/(4),(6) are Hodge-conjugate to each other.

The classification implies that all conformal symmetry breaking operators of
the type
$$
   \Omega^p(\R^n) \to \Omega^p(\R^{n-1})
$$
are given by the families
\begin{equation*}
   D_N^{(p \to p)}(\lambda): \Omega^p(\R^n) \to \Omega^p(\R^{n-1})
\end{equation*}
and
\begin{equation*}
   \star \, D_N^{(\frac{n}{2} \to \frac{n}{2}-1)}(\lambda):
   \Omega^{\frac{n}{2}}(\R^n) \to \Omega^{\frac{n}{2}}(\R^{n-1})
\end{equation*}
for even $n$ and
\begin{equation*}
   \star \, D_N^{(\frac{n-1}{2} \to \frac{n-1}{2})}(\lambda):
   \Omega^{\frac{n-1}{2}}(\R^n) \to \Omega^{\frac{n-1}{2}}(\R^{n-1})
\end{equation*}
for odd $n$ as well as the additional operators
\begin{align*}
  \star \, d \iota^*: \Omega^{\frac{n}{2}-1} (\R^n) & \to \Omega^{\frac{n}{2}-1}(\R^{n-1}), \\
  \star \, \delta \iota^*i_{\partial_n}: \Omega^{\frac{n+1}{2}}(\R^n) & \to
  \Omega^{\frac{n+1}{2}}(\R^{n-1}).
\end{align*}
Since the second family is proportional to $D_N^{(\frac{n}{2} \to
\frac{n}{2})}(\lambda) \, \bar{\star}$, it decomposes as the direct sum of two
families which are given by $D_N^{(\frac{n}{2} \to \frac{n}{2})}(\lambda)$. An
analogous comment concerns the third family. In addition, for $n=2$, we have
the operators
$$
   \star \, d \dot{D}_N^{(0 \to 0)}(N): \Omega^0(\R^2) \to \Omega^0(\R^1).
$$

Similarly, it follows that all conformal symmetry breaking operators of the
type
$$
\Omega^p(\R^n) \to \Omega^{p+1}(\R^{n-1})
$$
are given by the families
$$
   \star \, D_N^{(\frac{n}{2}-1 \to \frac{n}{2}-1)}(\lambda):
   \Omega^{\frac{n}{2}-1}(\R^n) \to \Omega^{\frac{n}{2}}(\R^{n-1})
$$
for even $n$ and
$$
    \star \, D_N^{(\frac{n-1}{2} \to \frac{n-3}{2})}(\lambda):
    \Omega^{\frac{n-1}{2}}(\R^n) \to \Omega^{\frac{n+1}{2}}(\R^{n-1})
$$
for odd $n$ as well as the additional operators
\begin{align*}
   d \iota^*: \Omega^p(\R^n) & \to \Omega^{p+1}(\R^{n-1}), \\
   d \dot{D}_N^{(0 \to 0)}(N): \Omega^0(\R^n) & \to \Omega^1(\R^{n-1})
\end{align*}
and
$$
   \star \, d \dot{D}_N^{(0 \to 0)}(N): \Omega^0(\R^3) \to \Omega^1(\R^2).
$$

In dimension $n=2$, the conformal symmetry breaking operators
$D_N^{(1,0)}(\lambda)$ and $D_N^{(1,0)}(\lambda) \, \bar{\star}$ were found in
\cite{KKP} (see Remark \ref{special-2}).

We continue with the description of two basic properties of the families of the
first and second type.

Firstly, we note that both types of conformal symmetry breaking operators are
connected through systems of additional factorization identities which involve
the operators $d$, $\delta$, $\bar{d}$ and $\bar{\delta}$ as factors (Theorems
\ref{SuppFact} and \ref{SuppFact-2}). The simplest of these relations are
\begin{equation}\label{B-fact-1}
   -(2N) \dm D_{2N-1}^{(p \to p-1)}(-p\!+\!2N) = D_{2N}^{(p \to p)}(-p\!+\!2N)
\end{equation}
and
\begin{equation}\label{B-fact-2}
   (2N) D_{2N-1}^{(p+1 \to p)}(-p\!-\!1) \bar{\dm} = D_{2N}^{(p \to p)}(-p).
\end{equation}
Note that the latter identities imply that $d D_{2N}^{(p \to p)}(-p\!+\!2N) =
0$ and $D_{2N}^{(p \to p)}(-p) \bar{d} = 0$.

The following result states respective connections of both types of conformal
symmetry breaking operators to the critical $Q$-curvature operators
$Q_{n-1-2p}^{(p)}$ and the gauge companion operators $G^{(p)}_{n-2p}$ of
$(\R^{n-1},g_0)$ (we refer to Section \ref{Branson-Gover} for the definition of
these concepts). Let dot denote the derivative with respect to $\lambda$.

\begin{thm}\label{G-Q-holo} If $n-1$ is even and $n-2p \ge 3$, we have
\begin{equation}\label{Q-holo}
    \dot{D}^{(p \to p)}_{n-1-2p}(-p)|_{\ker(\bar{\dm})} = Q_{n-1-2p}^{(p)} \iota^*.
\end{equation}
Similarly, if $n-1$ is even and $n-2p \ge 1$, then
\begin{equation}\label{G-holo}
    D^{(p\to p-1)}_{n-2p}(-p)|_{\ker(\bar{\dm})} = - G^{(p)}_{n-2p} \iota^*.
\end{equation}
\end{thm}

Although the $Q$-curvature $Q_n = Q_n^{(0)}$ of the Euclidean metric on $\R^n$
vanishes, the formula \eqref{Q-holo} is non-trivial. However, it is an easy
consequence of formula \eqref{SBO-even-0} for the even-order families of the
first type. An important aspect of formula \eqref{Q-holo} is the fact that it
resembles its analog
$$
   \dot{D}_n^{res}(0;g)(1) = Q_n(g)
$$
for general metrics $g$ on a manifold of even dimension $n$ \cite{Juhl},
\cite{JG-holo}. For more details see Theorem \ref{TheGOperator} and Theorem
\ref{holo-formula}.

\vspace{3pt}

We finish this section with an outline of the content of the paper.

In Section \ref{FMethod}, we first review an application of the Fourier
transform to generalized Verma modules which is known as the $F$-method. We
explain the role of singular vectors in this connection and describe
qualitative results on some branching laws.

Section \ref{SingularVectors} is devoted to the construction of singular
vectors which correspond to conformal symmetry breaking operators
$\Omega^p(\R^n) \to \Omega^q(\R^{n-1})$. We find two one-parameter families of
singular vectors which describe the embedding of the $\goso(n,1,\R)$-submodules
in the branching laws of the generalized Verma modules which are induced from
twisted fundamental representations on exterior forms. They are termed the
singular vectors of first and second type. Their construction rests on solving
systems of ordinary differential equations produced by the $F$-method. It turns
out that the solution spaces are determined by Gegenbauer polynomials. The main
results are Theorems \ref{OddFromPtoP}, \ref{EvenFromPtoP}, \ref{OddFromPtoP-1}
and \ref{EvenFromPtoP-1}. Finally, reducible summands in the branching laws
lead to additional singular vectors of the third and the fourth type (Theorem
\ref{sv-third}, Theorem \ref{sv-fourth}). There are some subtleties for middle
degree forms.

Then, in Section \ref{DiffOp}, we apply the results on singular vectors to
define conformal symmetry breaking operators. The singular vectors of the first
type induce the first type families $D_N^{(p \to p)}(\lambda)$ of conformal
symmetry breaking operators of order $N \in \N$ acting on differential
$p$-forms on $\R^n$ and taking values in $p$-forms on $\R^{n-1}$. By
conjugation with the Hodge star operators on $\R^n$ and $\R^{n-1}$, we obtain
the second type families $D_N^{(p \to p-1)}(\lambda)$ of conformal symmetry
breaking operators. They correspond to the singular vectors of the second type
found in Section \ref{SingularVectors} and act on differential $p$-forms on
$\R^n$ and take values in $(p\!-\!1)$-forms on $\R^{n-1}$. The main results
here are Theorems \ref{EvenDiffOp-type1} and \ref{OddDiffOp-type1} (first type
families) and Theorems \ref{EvenDiffOp-type2} and \ref{OddDiffOp-type2} (second
type families). Next, singular vectors of type three and four induce conformal
symmetry breaking operators of type three and four, respectively. Again, these
two types are Hodge conjugate to each other. The main results here are Theorems
\ref{DO3-Even} and \ref{DO3-Odd} (type three) and Theorems \ref{DO4-Even} and
\ref{DO4-Odd} (type four). Finally, in Section \ref{proof-classify}, we combine
the previous results in a proof of the classification result Theorem
\ref{classification}. In the last two subsections of Section \ref{DiffOp}, we
display explicit formulas for low-order special cases, confirm the equivariance
of the first-order families both by direct calculations and as consequences of
the conformal covariance of their curved analogs and prove that the results of
\cite{KKP} in dimension $n=2$ are special cases of ours.

In Section \ref{geometric}, we derive alternative closed formulas for all
conformal symmetry breaking operators in terms of the geometric operators $d$,
$\delta$ and their bar-versions. That process amounts to changing to another
basis of invariants. Although the resulting formulas allow quick proofs of some
of the basic properties of the families in later sections, it should be
stressed that neither of the two ways to write them is canonical. The
discussion of the first and second type families is divided into two
subsections according to the parity of their orders. The main results are
contained in Theorems \ref{coeffeven} and \ref{coeffeven2} (even order) and
Theorems \ref{coeffodd2} and \ref{coeffodd} (odd order). These results cover
Theorem \ref{main-even} and Theorem \ref{main-odd}. An analogous discussion
yields geometric formulas for operators of type three and four.

In Section \ref{Properties}, we establish basic properties of the first and
second type families of conformal symmetry breaking operators. We first recall
some results on Branson-Gover, gauge companion and $Q$-curvature operators
\cite{BransonGover} on forms. Then we derive systems of identities which show
that at special arguments any family of the first type factors into the product
of a lower-order family of the first type and a Branson-Gover operator for the
Euclidean metrics on $\R^n$ or $\R^{n-1}$ (Theorems \ref{MainFactEven1} and
\ref{MainFactEven2}). The special arguments for which such {\em main
factorizations} take place are naturally determined by the conformal weights of
the Branson-Gover operators involved. As very special cases, we obtain the
relations \eqref{interpol-1} and \eqref{interpol-2}. Furthermore, we show that
both types of conformal symmetry breaking operators are linked by systems of
{\em supplementary factorizations} (Theorems \ref{SuppFact} and
\ref{SuppFact-2}) which involve the four operators $d$, $\delta$, $\bar{\dm}$
and $\bar{\delta}$ as factors. The overall interest in both types of
factorization identities comes from the conjecture that they literally extend
to curved analogs of the families. Theorem \ref{TheGOperator} (see also
\eqref{G-holo}) gives a description of the gauge companion operators for the
Euclidean metric in terms of odd-order conformal symmetry breaking operators of
the second type. Finally, for the Euclidean metric, we introduce the notion of
$Q$-curvature polynomials (Definition \ref{QCurvPoly}) on differential forms
(generalizing the case $p=0$ analyzed in \cite{Juhl1}) and show how these cover
the $Q$-curvature operator on closed differential forms (see Theorem
\ref{TangQPoly}). \vspace{3pt}

The results of the present paper solve a basic problem of conformal geometry in
a very special but important case. Indeed, it is natural to ask for a
description of all {\em conformally covariant} differential operators mapping
differential forms on a Riemannian manifold $(X,g)$ to differential forms on a
submanifold $(M,g)$ (with the induced metric).\footnote{A version of the
problem appears as Problem 2.2 in \cite{KKP}. It replaces conformal covariance
by equivariance with respect to an appropriate group of conformal
diffeomorphisms.} In the codimension one case and for differential operators
acting on functions, that problem has been analyzed in \cite{Juhl}. The case of
differential operators on forms on $(X,M) = (\R^n,\R^{n-1})$ with the Euclidean
metric on the background manifold $\R^n$ is studied here. The significance of
the present results for the general curved case comes from the fact that the
formulas in the curved case arise from the geometric formulas derived here by
adding lower-order {\em curvature correcting terms} which involve curvature of
the submanifold and curvature being associated to the embedding. However,
closed formulas for these correction terms seem to be out of reach presently.
There are at least two ways to attempt to understand their structure. In the
codimension one case, a purely algebraic approach could be based on the
construction of lifts of our families of homomorphisms to families of
homomorphisms of semi-holonomic Verma modules \cite{eastwoodslovak}. A
different approach will be developed in the forthcoming paper \cite{fjs}, where
we identify the families constructed here with two types of so-called residue
family operators on differential forms. For general metrics, these residue
families are defined by the asymptotic expansions of eigenforms of the
Laplacian of a Poincar\'e-Einstein metric of a conformal manifold. Their
construction extends the construction of residue families in \cite{Juhl}.

\section{Preliminaries}\label{FMethod}

In the present section, we briefly review basic notation and results initiated
and developed in \cite{ko, Kobayashi-Pevzner, koss, kob}. We apply these
results to describe the situation of our interest: the branching problem for
generalized Verma modules of real orthogonal simple Lie algebras, conformal
parabolic subalgebras and fundamental representations on forms twisted by
characters as inducing data.

\subsection{The $F$-method}\label{F-method}

We first introduce our setting. Let $G$ be a connected real reductive Lie group
with Lie algebra $\gog(\R)$, $P\subset G$ a parabolic subgroup with Lie algebra
$\gop(\R)$, $\gop(\R) = \gol(\R) \oplus \gon_+(\R)$ its Levi decomposition and
$\gon_-(\R)$ the opposite (negative) nilradical,
$\gog(\R)=\gon_-(\R)\oplus\gop(\R)$. The complexifications of Lie algebras
$\gog(\R)$, $\gop(\R)$, $\gol(\R)$, $\gon_+(\R)$, etc., will be denoted $\gog$,
$\gop$, $\gol$, $\gon_+$, etc.

Given a complex finite dimensional $P$-module $(\rho,V)$, we consider the
induced representation $\left(\pi,\Ind_{P}^{G}(V)\right)$ of $G$ on smooth
sections of the homogeneous vector bundle $G \times_{P}V \to G/P$, i.e., $\pi$
acts on the space
\begin{equation*}
   \Ind_{P}^{G}(V) \st C^\infty(G,V)^{P}=\left\{f\in C^\infty(G,V) \,|\, f(g p) =
   \rho(p^{-1}) f(g),\, g\in  G,p\in  P\right\}
\end{equation*}
by the left regular representation. Given another data $(G^\prime,
P^\prime,\rho^\prime:P^\prime\to \GL(V^\prime))$ such that $G^\prime\subset G$
is reductive, $P^\prime=P\cap G^\prime$ and $\gop$ is $\gog^\prime$-compatible,
cf. \cite[Definition $4.5$]{Kobayashi-Pevzner}, the space of continuous
$G^\prime$-equivariant homomorphisms
\begin{equation}\label{eq:SymmBreakingOp}
   \Hom_{G^\prime}\left(\Ind^{G}_{P}(V),\Ind^{G^\prime}_{P^\prime}(V^\prime)\right)
\end{equation}
is termed the space of {\it symmetry breaking operators} \cite{kobayashi-speh}.
Here the $G^\prime$-equivariance is defined with respect to the actions
$\pi(g^\prime)$ and $\pi^\prime(g^\prime)$ for all $g^\prime\in G^\prime$,
where $\pi(g^\prime)$ is realized by embedding $G^\prime\subset G$. The space
of $G^\prime$-equivariant differential operators $D:\Ind^{G}_{P}(V)\to
\Ind^{G\prime}_{P^\prime}(V^\prime)$, cf. \cite[Section 2]{Kobayashi-Pevzner},
denoted by
\begin{equation}\label{eq:DiffSymmBreakingOp}
   \Diff_{G^\prime} \left(\Ind^{G}_{P}(V),\Ind^{G\prime}_{P^\prime}(V^\prime)\right),
\end{equation}
is a subspace of \eqref{eq:SymmBreakingOp}.

Let ${\fam2 U}(\gog)$ be the universal enveloping algebra of $\gog$. Denoting
$V^\ch$ the dual (or, contragredient) representation to $V$, the generalized
Verma module ${\fam2 M}^{\gog}_{\gop}(V^\ch)$ is defined by
\begin{equation*}
   {\fam2 M}^{\gog}_{\gop}(V^\ch) \st {\fam2 U}(\gog)\otimes_{{\fam2 U}(\gop)} V^\ch.
\end{equation*}
The induced action will be denoted by $\pi^\ch$.

The space \eqref{eq:DiffSymmBreakingOp} of $G^\prime$-equivariant differential
operators is isomorphic to the space of $\gog^\prime$-homomorphisms of
generalized Verma modules,
\begin{equation}\label{eq:one-to-one}
   \Diff_{G^\prime}\left(\Ind^{G}_{P}(V),\Ind^{G\prime}_{P^\prime}(V^\prime)\right)\simeq
   \Hom_{\gog^\prime} \left({\fam2 M}^{\gog^\prime}_{\gop^\prime}( (V^\prime)^\ch),
   {\fam2 M}^{\gog}_{\gop}(V^\ch) \right),
\end{equation}
where in addition
\begin{align}\label{eq:isomorphism1}
   \Hom_{\gog^\prime} \left({\fam2 M}^{\gog^\prime}_{\gop^\prime}(
   (V^\prime)^\ch), {\fam2 M}^{\gog}_{\gop}(V^\ch) \right) \simeq
   \Hom_{\gop^\prime}\left( (V^\prime)^\ch,{\fam2M}^{\gog}_{\gop}(V^\ch) \right),
\end{align}
cf. \cite[Theorem $2.7$]{Kobayashi-Pevzner}, \cite{koss}. In
\cite{Kobayashi-Pevzner}, the elements of $\Hom_{\gop^\prime} \left(
(V^\prime)^\ch,{\fam2M}^{\gog}_{\gop}(V^\ch) \right)$ are referred to as {\it
singular vectors}. Alternatively, we may first define the $L^\prime$-submodule
\begin{equation}\label{eq:SingularVectors}
   {\fam2 M}^{\gog}_{\gop}(V^\ch)^{\gon^\prime_+}
   \st \{v\in {\fam2 M}^{\gog}_{\gop}(V^\ch) \,|\, \dm\pi^\ch(Z)v=0, \; Z \in \gon^\prime_+\}
\end{equation}
of ${\fam2 M}^{\gog}_{\gop}(V^\ch)$ as the {\em space of singular vectors} and
then study the $\gol^\prime$-homomorphisms of $(V^\prime)^\ch$ into this space.
Then any irreducible $\gol^\prime$-submodule $(W^\prime)^\ch$ of ${\fam2
M}^{\gog}_{\gop}(V^\ch)^{\gon^\prime_+}$ yields a $\gog^\prime$-homomorphism
${\fam2 M}^{\gog^\prime}_{\gop^\prime}((W^\prime)^\ch)$ to ${\fam2
M}^{\gog}_{\gop}(V^\ch)$. The latter method has been used in \cite{koss}.

The $F$-method, cf. \cite{koss}, \cite[Section$4$]{Kobayashi-Pevzner}, is a
constructive method to determine singular vectors. The idea is the following.
We first note that a generalized Verma module ${\fam2 M}^{\gog}_{\gop}(V^\ch)$
can be realized as a space of $V^\ch$-valued distributions on $N_-$ supported
at $o \st eP\in G/P$. In particular, there is a $\mathcal{U}(\gog)$-module
isomorphism
\begin{equation}\label{eq:Identifications}
   \phi: {\fam2 M}^{\gog}_{\gop}(V^\ch) \to \mathcal{D}^\prime_{[o]}(N_-,V^\ch)
\end{equation}
(\cite{koss}, \cite{Kobayashi-Pevzner}). When the nilpotent group $N_-$ is
commutative, it may be identified with $\gon_-(\R)$ via the exponential map.
Hence
$$
   \mathcal{D}^\prime_{[o]}(N_-,V^\ch) \simeq \mathcal{D}^\prime_{[0]}(\gon_-(\R)) \otimes V^\ch.
$$
Let $\Pol(\gon_-^*(\R))$ be the space of polynomials on $\gon_-^*(\R)$. The
algebraic {\it Fourier transform}\footnote{By an artificial use of the pairing
of $\gon_+$ and $\gon_-$ defined by the Killing form, one may regard the range
of the Fourier transform as consisting of polynomials on $\gon_+$. Although
this has been the perspective in \cite{koss}, we prefer not to follow it here.
It would also complicate formulas by unpleasant constants.}
\begin{equation}\label{eq:FourierTrafo}
   \mathcal{F}: \mathcal{D}^\prime_{[0]}(\gon_-(\R)) \to \Pol(\gon_-^*(\R)), \quad
   \mathcal{F}(f)(\xi) \st \langle f(\cdot), e^{i\langle \cdot,\xi\rangle} \rangle
   = \int_{\gon_-(\R)}e^{i\langle x,\xi\rangle} f(x) dx,
\end{equation}
is an algebra isomorphism mapping convolutions into products. $\mathcal{F}$
extends to an isomorphism of vector spaces
\begin{equation}\label{eq:FourierIso}
   \mathcal{F}\otimes \id_{V^\ch}: \mathcal{D}^\prime_{[0]}(\gon_-(\R)) \otimes
   V^\ch \to \Pol(\gon_-^*(\R)) \otimes V^\ch .
\end{equation}
We use this isomorphism to transport the $\gog$-module structure. Let
$\dm\tilde{\pi}$ denote the induced action of $\gog$ on $\Pol(\gon_-^*(\R))
\otimes V^\ch$. Thus the space of singular vectors is identified by Fourier
transform with the space
\begin{equation}\label{sol-space}
   \Sol(\gog,\gog^\prime;V^\ch) \st \{f\in\Pol(\gon_-^*(\R))\otimes V^\ch \,|\,
   \dm\tilde{\pi}(Z)f=0, \; Z \in \gon_+^\prime\}.
\end{equation}
In summary, searching for singular vectors, which is a combinatorial problem,
is equivalent to finding elements in $\Sol(\gog,\gog^\prime;V^\ch)$, a problem
in algebraic analysis. Finally, we note that the operators $\dm\tilde{\pi}(Z)$
in \eqref{sol-space} form a system of {\em second-order} partial differential
operators on $\Pol(\gon_-^*(\R)) \otimes V^\ch$.

\subsection{Notation and induced representations}\label{notation}

In the remainder of the paper, we assume that $\N \ni n \ge 2$. Let
$G=SO_0(n+1,1,\R)$ be the connected component of the identity of the group
preserving the quadratic form
$$
   2 x_0 x_{n+1} + x_1^2 + \dots + x_n^2
$$
on $\R^{n+2}$. Let $\{e_j, j=0,\dots,n+1\}$ be the standard basis of
$\R^{n+2}$. The group $G$ preserves the cone
$$
   \mathcal{C} = \{ x = (x_0,\dots,x_{n+1}) \in \R^{n+1} \,|\, 2 x_0 x_{n+1} + x_1^2 + \dots + x_n^2 = 0 \}.
$$
Let $P = P_+ \subset G $ and $P_- \subset G$ be the isotropy subgroups of the
lines generated by $e_0 = (1,0,\dots,0)$ and $e_{n+1} = (0,0,\dots,1)$,
respectively. Then $P_\pm$ are parabolic subgroups with Langlands
decompositions $P_\pm = LN_\pm = MAN_\pm$, where $M \simeq SO(n,\R)$, $A\simeq
\R^+$ and $N_\pm \simeq \R^n$. The elements of $G$ can be written in the block
form
$$
\begin{pmatrix}
         1\times 1 & 1\times n & 1\times 1\\
         n\times 1 & n\times n & n\times 1\\
         1\times 1 & 1\times n & 1\times 1\\
\end{pmatrix}
$$
with respect to the decomposition
$$
\R^{n+2} = \R e_0 \oplus \bigoplus_{j=1}^{n} \R e_j \oplus \R e_{n+1}.
$$
The real Lie algebras corresponding to the above Lie groups are
\begin{equation}\label{eq:LieAlgebras}
   \gog(\R) = \mathfrak{so}(n+1,1,\R), \quad \gop_\pm(\R) = \gol(\R) \oplus \gon_\pm(\R) =
   \gom(\R) \oplus \goa(\R) \oplus \gon_\pm(\R)
\end{equation}
with
\begin{equation*}
   \gom(\R) \simeq \mathfrak{so}(n,\R), \quad \goa(\R) \simeq \R \quad \mbox{and} \quad \gon_\pm(\R) \simeq \R^n.
\end{equation*}
The basis elements
\begin{align*}
   E = \begin{pmatrix}
           1 & 0 & 0\\
           0 & 0 & 0\\
           0 & 0 & -1\\
          \end{pmatrix},\quad
   E_j^+ = \begin{pmatrix}
              0 & e_j  & 0\\
              0 & 0 & -e_j^t\\
              0 & 0 & 0\\
             \end{pmatrix},\quad
   E_j^- = \begin{pmatrix}
                        0 & 0  & 0\\
                        e_j^t & 0 & 0\\
                        0 & -e_j & 0\\
        \end{pmatrix}, \quad 1 \le j \le n \nonumber
\end{align*}
and
$$
   M_{ij} = \begin{pmatrix}
                    0 & 0  & 0\\
                    0 & {(M_{ij})_{rs}}=\delta_{ir}\delta_{js}-\delta_{is}\delta_{jr}& 0\\
                    0 & 0 & 0\\
         \end{pmatrix}, \quad i,j,r,s=1, \dots ,n,\, i<j
$$
satisfy the commutation relations
\begin{equation*}
   [E_i^+,E_j^-] = \delta_{ij} E + M_{ij}
\end{equation*}
and
\begin{equation*}
   [M_{ij},E_r^+] = \delta_{jr} E_r^+ - \delta_{ir} E_j^+,\; [E,E_j^+] = E_j^+, \; [E,M_{ij}] = 0
\end{equation*}
and analogously for $E^-_j$. We shall also use the notation $E_j = E_j^-$.

The elements $E_j^\pm$, $1 \le j \le n$, form respective bases of
$\gon_\pm(\R)$. We use these bases to identify both spaces with $\R^n$. The
Euclidean metric $g_0$ on $\R^n$ induces scalar products on $\gon_\pm(\R)$ for
which both bases are orthonormal. The adjoint action of $M$ on $\gon_\pm(\R)$
preserves these scalar products. Hence $M$ can be identified with $SO(n,\R)$.
Sometimes it will be useful to identify elements of $\gon_-(\R)$ with column
vectors vectors and elements of $\gon_+(\R)$ with row vectors. Let
$(E_j^\pm)^*$ for $1 \le j \le n$ be the dual basis elements of
$\gon_\pm(\R)^*$. They are orthonormal with respect to corresponding dual
scalar products.

Furthermore, we consider the subgroup $G^\prime \simeq SO_0(n,1,\R)$ of $G$
which preserves the line in the cone $\mathcal{C}$ generated by $e_n$. The
embedding of $G^\prime$ is realized as the matrix block inclusion
\begin{align*}
\begin{small}
\begin{pmatrix}
         1\times 1 & 1\times n-1 & 1\times 1\\
         n-1\times 1 & n-1\times n-1 & n-1\times 1\\
         1\times 1 & 1\times n-1 & 1\times 1\\
\end{pmatrix} \hookrightarrow
\begin{pmatrix}
          1\times 1 & 1\times n-1 & 0 & 1 \times 1\\
          n-1\times 1 & n-1\times n-1 & 0 & n-1\times 1 \\
          0 & 0 & 0 & 0         \\
          1\times 1 & 1\times n-1 & 0 & 1\times 1 \\
\end{pmatrix}
\end{small}.
\end{align*}
The group $P^\prime = P \cap G^\prime$ is a parabolic subgroup of $G^\prime $
with Langlands decomposition $P^\prime = L^\prime N^\prime_+ = M^\prime A
N^\prime_+$, where $M^\prime\simeq SO(n-1,\R)$, $A \simeq \R^+$ and $N_+^\prime
\simeq \R^{n-1}$. The real Lie algebras of these groups are
\begin{equation}\label{eq:LieAlgebras'}
   \gog^\prime(\R) = \mathfrak{so}(n,1,\R), \quad \gop^\prime(\R) = \gol^\prime(\R) \oplus \gon^\prime_+(\R)
   = \gom^\prime(\R) \oplus \goa(\R) \oplus \gon^\prime_+(\R)
\end{equation}
with
\begin{equation*}
   \gom^\prime(\R) \simeq \mathfrak{so}(n-1,\R), \quad \goa(\R) \simeq \R \quad \mbox{and} \quad
   \gon^\prime_+(\R) \simeq \R^{n-1}.
\end{equation*}
The elements $E_j^\pm$ for $1 \le j \le n-1$ form respective bases of
$\gon^\prime_\pm(\R)$.

The natural action of $M \simeq SO(n,\R)$ on $\R^n$ induces representations
$\sigma_p$ on the spaces $\Lambda^p(\R^{n})$ and $\Lambda^p(\R^{n})^*$ of
multilinear forms.

\begin{bem}\label{AlgebraicAction} For $X, Y \in \R^n$, the tensor
$X \otimes Y-Y \otimes X$ acts on $\R^n$ with the Euclidean metric $g_0$ by
$$
   T: v \mapsto X g_0(Y,v) - Y g_0(X,v).
$$
In view of $g_0(T(v),w) + g_0(v,T(w))=0$, we have $T \in \mathfrak{so}(n,\R)$.
The action $T$ naturally extends to $\Lambda^p(\R^n)$ by
$$
   (X \otimes Y-Y \otimes X)(v_1 \wedge \dots \wedge v_p)
   = \sum_{l=1}^p v_1\wedge \dots \wedge v_{l-1} \wedge (X g_0(Y,v_l) - Y g_0(X,v_l)) \wedge
   v_{l+1}\wedge \dots \wedge v_p
$$
for $v_1\wedge \dots \wedge v_p \in \Lambda^p(\R^n)$. It coincides with the
infinitesimal representation $\dm\sigma_p$ on $\Lambda^p(\R^n)$.
\end{bem}

We shall use the notation $\sigma_p$ also for the corresponding representations
of $M \simeq SO(n,\R)$ on the complex vector spaces $\Lambda^p(\R^{n}) \otimes
\C$.

Next, we fix some notation. We define the fundamental weights
\begin{align}\label{fund-w-even}
   \Lambda_j & = e_1 + \dots + e_j \hspace{3cm} \mbox{for $j=1,\dots,\tfrac{n}{2}-2$}, \notag \\
   \Lambda_{\frac{n}{2}-1} & = \tfrac{1}{2} (e_1 + \dots + e_{\frac{n}{2}-1} - e_{\frac{n}{2}}) \notag \\
   \Lambda_{\frac{n}{2}} & = \tfrac{1}{2} (e_1 + \dots + e_{\frac{n}{2}-1} + e_{\frac{n}{2}})
\end{align}
for even $n$ and
\begin{align}\label{fund-w-odd}
   \Lambda_j & = e_1 + \dots + e_j \hspace{3cm} \mbox{for $j=1,\dots,\tfrac{n-3}{2}$}, \notag\\
   \Lambda_{\frac{n-1}{2}} & = \tfrac{1}{2} (e_1 + \dots + e_{\frac{n-1}{2}})
\end{align}
for odd $n$. Here the vectors $e_j$ denote respective standard basis vectors in
$\R^{\frac{n-1}{2}}$ and $\R^{\frac{n}{2}}$. Let $\Lambda_0 = 0$. We also set
\begin{equation*}
   1_p \st (\underbrace{1,\dots,1}_{p \; \text{entries}},0,\dots,0) \quad
   \mbox{and} \quad 1_{\frac{n}{2}}^\pm \st (\underbrace{1,\dots,1,\pm 1}_{\frac{n}{2} \; \text{entries}})
\end{equation*}
for even $n$ and
\begin{equation*}
   1_p \st (\underbrace{1,\dots,1}_{p \; \text{entries}},0,\dots,0) \quad
   \mbox{and} \quad
   1_{\frac{n-1}{2}} = (\underbrace{1,\dots,1}_{\frac{n-1}{2} \; \text{entries}}).
\end{equation*}
for odd $n$.


The following lemma collects basic information on the representations
$\sigma_p$.

\begin{lem}\label{form-rep} (i) If $n$ is odd, the representation $\sigma_p$ of $SO(n,\R)$
on $\Lambda^p(\R^n) \otimes \C$ is irreducible for all $p=0,1,\dots,n$.
$\sigma_p$ is equivalent to $\sigma_{n-p}$. For $p \le (n-1)/2$,
its highest weight equals $1_p$. \\
(ii) If $n$ is even and $p \ne \frac{n}{2}$, the representation $\sigma_p$ of
$SO(n,\R)$ on $\Lambda^p(\R^n) \otimes \C$ is irreducible. $\sigma_p$ is
equivalent to $\sigma_{n-p}$. For $p \le n/2-1$, its highest weight equals
$1_p$. \\
(iii) If $n$ is even, we have the decomposition $\sigma_{\frac{n}{2}} =
\sigma_{\frac{n}{2}}^+ \oplus \sigma_{\frac{n}{2}}^-$, where the irreducible
representations $\sigma_\frac{n}{2}^\pm$ act on the eigenspaces
$$
   \Lambda_\pm^{\frac{n}{2}}(\R^n) \otimes \C \st \left\{ \omega \in
   \Lambda^{\frac{n}{2}}(\R^n) \otimes \C \,|\, \bar{\star} \, \omega = \mu_\pm \omega \right\}
$$
of the Hodge star operator $\bar{\star}$ of the Euclidean metric $g_0$ on
$\R^n$. Here $\mu_\pm = \pm 1$ if $n \equiv 0 \! \pmod{4}$ and $\mu_\pm = \pm
i$ if $n \equiv 2 \! \pmod{4}$. The highest weights of the representations
$\sigma^+_{\frac{n}{2}}$ and  $\sigma^-_{\frac{n}{2}}$
are $1_{\frac{n}{2}}^+$ and $1_{\frac{n}{2}}^-$, respectively. \\
(iv) For all $p=0,1,\dots,n$, there is an isomorphism of $SO(n-1,\R)$-modules
\begin{equation*}
   \Lambda^p(\R^n) \simeq \Lambda^p(\R^{n-1}) \oplus \Lambda^{p-1}(\R^{n-1});
\end{equation*}
here we set $\Lambda^{-1}(\R^{n-1})=0$. \\
(v) If $n$ is even, there are isomorphisms
$$
   \Lambda_+^{\frac{n}{2}}(\R^n) \otimes \C \simeq \Lambda^{\frac{n}{2}-1}(\R^{n-1}) \otimes \C
   \quad \mbox{and} \quad
   \Lambda_-^{\frac{n}{2}}(\R^n) \otimes \C \simeq \Lambda^{\frac{n}{2}-1}(\R^{n-1}) \otimes \C,
$$
which are given by
$$
   \Lambda^{\frac{n}{2}-1}(\R^{n-1}) \otimes \C \ni \omega \mapsto
   \begin{cases} \star \, \omega + \omega \wedge e_n & \mbox{if  } n \equiv 0 \!\! \pmod{4} \\
   \star \, \omega - i \omega \wedge e_n & \mbox{if  } n \equiv  2 \!\! \pmod{4}
   \end{cases}
$$
and
$$
   \Lambda^{\frac{n}{2}-1}(\R^{n-1}) \otimes \C \ni \omega \mapsto
   \begin{cases} \star \, \omega - \omega \wedge e_n & \mbox{if  } n \equiv 0 \!\! \pmod{4} \\
   \star \, \omega + i \omega \wedge e_n & \mbox{if  } n \equiv  2 \!\! \pmod{4}.
   \end{cases}
$$
Here $\star$ denotes the Hodge star operator (of the Euclidean metric) on
$\R^{n-1}$. In particular, the highest weights of the $SO(n\!-\!1,\R)$-modules
$\Lambda_\pm^{\frac{n}{2}}(\R^n)$ are $1_{\frac{n}{2}-1}$.
\end{lem}

\begin{proof} The claims in (i), (ii) and (iii) are well-known; see also Proposition \ref{IT}.
Note that the $SO(n,\R)$-invariance of the Hodge star operator implies the
equivalencies in (i) and (ii), and shows that the eigenspaces in (iii) are
well-defined. (iv) follows from the $SO(n\!-\!1,\R)$-equivariance of the
decomposition
$$
   \omega = \omega' + \omega'' \wedge e_n \quad \mbox{with $\omega' \in
   \Lambda^p(\R^{n-1})$ and $\omega'' \in \Lambda^{p-1}(\R^{n-1})$}
$$
of any $\omega \in \Lambda^p(\R^n)$. In order to prove (v), we first observe
that
$$
   \bar{\star} \, (\star \,\omega) = \omega \wedge e_n \quad \mbox{and} \quad  \bar{\star} \, (\omega \wedge e_n)
   = (-1)^{\frac{n}{2}} \star \omega
$$
for $\omega \in \Lambda^{\frac{n}{2}-1}(\R^{n-1})$. Then we obtain
$$
   \bar{\star} \, (\star \, \omega \pm \omega \wedge e_n) = \pm (\star \,\omega \pm \omega \wedge e_n)
$$
if $n \equiv 0 \pmod{4}$ and
$$
   \bar{\star} \, (\star \, \omega \pm i \omega \wedge e_n) = \mp i (\star \,\omega \pm i \omega \wedge e_n)
$$
if $n \equiv 2 \pmod{4}$. The proof is complete.
\end{proof}

Next, we introduce some notation for induced representations. For $\lambda\in
\C$, we denote by $(\xi_\lambda,\C_\lambda)$ the $1$-dimensional complex
representation $\xi_\lambda(a) = \exp (\lambda \log a)$ of $A$. We trivially
extend $\xi_\lambda$ to a character of $P$. Its dual representation is
$\C^\ch_\lambda\simeq \C_{-\lambda}$. For $p=0,1,\dots,n$ and $\lambda \in \C$,
we define the representation
$$
   \rho_{\lambda,p} \st \sigma_p \otimes \xi_\lambda
$$
of $L=MA$ on
\begin{equation}\label{defV}
   V_{\lambda,p} \st \Lambda^p (\gon_-(\R)) \otimes \C_\lambda \simeq \Lambda^p(\R^{n}) \otimes \C_\lambda.
\end{equation}
Here we regard the spaces $\Lambda^p(\R^n)$ as trivial $A$-modules. We also
recall that for even $n$ the representation $\sigma_{\frac{n}{2}}$ is {\em not}
irreducible. The dual representation $\rho_{\lambda,p}^\ch$ of $L=MA$ acts on
$$
   V_{\lambda,p}^\ch \simeq \Lambda^p (\gon_-(\R))^* \otimes \C_{-\lambda} \simeq
   \Lambda^p(\R^{n})^* \otimes \C_{-\lambda}.
$$
We trivially extend $\rho_{\lambda,p}$ and $\rho_{\lambda,p}^\ch$ to
representations of $P$, i.e.,
\begin{equation}\label{eq:IrreRep}
   \rho_{\lambda,p} (man)(v\otimes 1) = \sigma_p(m) (v) \otimes a^{\lambda}
\end{equation}
and
\begin{equation}\label{eq:IrreRep-dual}
   \rho_{\lambda,p}^\ch (man)(v \otimes 1) = \sigma_p^\ch (m) (v) \otimes a^{-\lambda}
\end{equation}
for $man \in P = MAN_+$. Let
$$
   (\pi_{\lambda,p},\Ind_P^G(\rho_{\lambda,p})) \quad \mbox{and} \quad
   (\pi_{\lambda,p}^\ch,\Ind_P^G( \rho_{\lambda,p}^\ch))
$$
be the resulting induced representations of $G$.\footnote{In \cite{Juhl},
induced representations are defined by using the opposite sign of $\lambda$.}
The analogous induced representations for $G^\prime$ are denotes by
$\pi^\prime_{\lambda,p}$ and $\pi_{\lambda,p}^{\prime \ch}$.

\begin{bem}\label{rep-geo} The representation $\pi_{\lambda,p}^\ch$ can be naturally
identified with the left regular representation of $G$ on the space
$\Omega^p(G/P,\C_{-\lambda-p})$ of weighted $p$-forms on $G/P$ (with weight
$-\lambda-p$). Alternatively, we may naturally identify the induced
representations $\pi_{\lambda,p}^\ch$ with the geometrically defined
representations
\begin{equation}\label{rep-geom}
   \pi_\lambda^{(p)}(\gamma) \st e^{\lambda \Phi_\gamma} \gamma_*:
   \Omega^p (S^n) \to \Omega^p(S^n), \; \gamma \in G
\end{equation}
on complex-valued $p$-forms on the round sphere $S^n$ through the identity
\begin{equation}\label{rep-relation}
   \pi_{-\lambda-p}^{(p)} = \pi_{\lambda,p}^\ch.
\end{equation}
The definition \eqref{rep-geom} rests on the fact that $G$ acts on $S^n = G/P$
by conformal diffeomorphisms of the round metric $g_0$, i.e., $\gamma_*(g_0) =
e^{2\Phi_\gamma} g_0$ for some $\Phi_\gamma \in C^\infty(S^n)$. For more
details on these identifications in the case $p=0$ see Section 2.3 in
\cite{Juhl}.
\end{bem}

\subsection{A branching problem}\label{charident}

In the present section, we use a result of \cite{ko} to describe the
decompositions of the generalized Verma modules
$$
   {\fam2 M}^\gog_\gop\left(\Lambda^p(\R^n) \otimes \C_\lambda \right), \; \lambda \in \C
$$
for $\gog$ under restriction to the subalgebra $\gog^\prime$ on the level of
characters.

Let $S(V) = \oplus_{N =0}^\infty S^N(V)$ be the symmetric tensor algebra over
the vector space $V$. We extend the adjoint action of $\gol^\prime$ to
$S(\gon_-/\gon^\prime_-)$. Then the $\gol^\prime$-module
$S(\gon_-/\gon^\prime_-)$ is the free commutative ring generated by the
$1$-dimensional $\gol^\prime$-module $\gon_-/\gon_-^\prime$ isomorphic to
$\C_{-1}$, i.e., $S(\gon_-/\gon^\prime_-) \simeq \oplus_{N \ge 0} \C_{-N}$. For
any finite dimensional irreducible $\gol$-module $W$ and any finite dimensional
irreducible $\gol^\prime$-module $V^\prime$ we define
\begin{equation}\label{multiplicity-gen}
   m(W,V^\prime) \st \dim_\C \mathrm{Hom}_{\gol^\prime}
   \left(V^\prime,W|_{\gol^\prime} \otimes S(\gon_-/\gon^\prime_-)\right).
\end{equation}
Then, by \cite[Theorem $3.10$]{ko}, it holds
\begin{equation}\label{chid-general}
   {\fam2 M}^\gog_\gop (W) |_{\gog^\prime} \simeq \bigoplus_{V^\prime} m(W,V^\prime)
   {\fam2M}^{\gog^\prime}_{\gop^\prime} (V^\prime)
\end{equation}
in the Grothendieck group $K({\fam2 O}^{\gop^\prime})$ of the
Bernstein-Gelfand-Gelfand parabolic category ${\fam2 O}^{\gop^\prime}$. The
isomorphism \eqref{chid-general} is equivalent to the equality of formal
characters of both sides.

For the convenience of the reader, we recall the main arguments of the proof of
\eqref{chid-general}. First, the identifications ${\fam2 M}^\gog_\gop (W)
\simeq U(\gon_-) \otimes W \simeq S(\gon_-) \otimes W$ imply that on $MA^+$ the
formal character of this module is given by
\begin{align*}
   \Ch({\fam2 M}^\gog_\gop (W))(ma) & = \sum_{N \ge 0} \tr(\Ad(ma)|_{S^N(\gon_-)}) \Ch(W)(ma)\\
   & = \det(1-\Ad(ma)|_{\gon_-})^{-1} \Ch(W)(ma), \; ma \in MA^+ \subset L.
\end{align*}
We restrict this formula to $M^\prime A^+ \subset L^\prime$. The relation
\begin{align}\label{det-formula}
   \det(1-\Ad(m'a)|_{\gon_-})^{-1} & = \det(1-\Ad(m'a)|_{\gon_-^\prime})^{-1}
   \det(1-\Ad(a)|_{\gon_-/\gon^\prime_-})^{-1} \notag \\
   & = \det(1-\Ad(m'a)|_{\gon_-^\prime})^{-1}
   \sum_{N \ge 0} \tr (\Ad(a)|_{S^N(\gon_-/\gon^\prime_-)})
\end{align}
for $m^\prime a \in M^\prime A^+$ yields
\begin{align*}
   \Ch({\fam2 M}^\gog_\gop (W))(m^\prime a) & = \det(1-\Ad(m'a)|_{\gon_-^\prime})^{-1}
   \Ch(W \otimes S(\gon_-/\gon^\prime_-))(m^\prime a) \\
   & = \det(1-\Ad(m'a)|_{\gon_-^\prime})^{-1} \sum_{V^\prime} m(W,V^\prime)
   \Ch(V^\prime)(m^\prime a) \\
   & = \sum_{V^\prime} m(W,V^{\prime}) \Ch({\fam2 M}^{\gog^\prime}_{\gop^\prime}
   (V^\prime))(m^\prime a).
\end{align*}
This proves the assertion.\footnote{The identity \eqref{det-formula} also
implies an identity for Selberg zeta functions which in turn suggests part of
the theory in \cite{Juhl} (as explained in Section 1.2 of this reference).}

The isomorphism \eqref{chid-general} may be regarded as the main step in the
determination of a branching rule of the restriction to $\gog^\prime$ of the
generalized Verma modules ${\fam2 M}^\gog_\gop (W)$. We also stress that the
modules on the right-hand side of \eqref{chid-general} may be reducible: this
effect will actually play a role later.

Now let $W = V_{p,\lambda}$. In order to determine the multiplicities
\begin{equation}\label{multiplicity}
   m_{V^\prime}(p,\lambda) \st m(V_{p,\lambda},V^\prime),
\end{equation}
we distinguish several cases.

{\bf The generic case.} We assume that $n$ is odd and $p \ne \frac{n\pm 1}{2}$
or $n$ is even and $p \ne \frac{n}{2}$. Then Lemma \ref{form-rep} implies that
the $SO(n,\R)$-module $\Lambda^{p}(\R^{n}) \otimes \C$ is irreducible and the
multiplicity \eqref{multiplicity} equals one iff
$$
   V' \simeq \Lambda^p(\R^{n-1}) \otimes \C_{\lambda-N} \quad \mbox{or} \quad
   V' \simeq  \Lambda^{p-1}(\R^{n-1}) \otimes \C_{\lambda-N}
$$
for some $N \in \N_0$.

{\bf Middle degree cases ($n$ odd).} If $p=\frac{n-1}{2}$, the multiplicity
\eqref{multiplicity} equals one iff
$$
   V' \simeq \Lambda_+^\frac{n-1}{2}(\R^{n-1}) \otimes \C_{\lambda-N} \;\; \mbox{or} \;\;
   V' \simeq  \Lambda_-^\frac{n-1}{2}(\R^{n-1}) \otimes \C_{\lambda-N} \;\; \mbox{or} \;\;
   V' \simeq  \Lambda^\frac{n-3}{2}(\R^{n-1}) \otimes \C_{\lambda-N}
$$
for some $N \in \N_0$. Similarly, if $p=\frac{n+1}{2}$, the multiplicity
\eqref{multiplicity} equals one iff
$$
   V' \simeq \Lambda^\frac{n+1}{2}(\R^{n-1}) \otimes \C_{\lambda-N} \;\; \mbox{or} \;\;
   V' \simeq  \Lambda_+^\frac{n-1}{2}(\R^{n-1}) \otimes \C_{\lambda-N} \;\; \mbox{or} \;\;
   V' \simeq  \Lambda_-^\frac{n-1}{2}(\R^{n-1}) \otimes \C_{\lambda-N}
$$
for some $N \in \N_0$.

{\bf Middle degree cases ($n$ even).} Assume that $p = \frac{n}{2}$. Then Lemma
\ref{form-rep} shows that the multiplicity \eqref{multiplicity} is one iff
$$
   V' \simeq \Lambda^{\frac{n}{2}-1}(\R^{n-1}) \otimes \C_{\lambda-N} \simeq
   \Lambda^{\frac{n}{2}}(\R^{n-1}) \otimes \C_{\lambda-N}
$$
for some $N \in \N_0$.

Using these observations, the following results are special cases of
\eqref{chid-general}.

\begin{prop}[\bf Character identities. Generic cases]\label{charind} We consider the
compatible pair of simple Lie algebras
\begin{equation*}
   \gog(\R) = \mathfrak{so}(n\!+\!1,1,\R), \quad \gog^\prime(\R) = \mathfrak{so}(n,1,\R)
\end{equation*}
and their respective conformal parabolic subalgebras $\gop(\R)$ and
$\gop^\prime(\R)$. Assume that $p \ne \frac{n \pm 1}{2}$ if $n$ is odd and $p
\ne \frac{n}{2}$ if $n$ is even. Then, in the Grothendieck group $K({\fam2
O}^\gop)$ of the Bernstein-Gelfand-Gelfand parabolic category ${\fam2
O}^{\gop}$, it holds
\begin{multline}\label{diag-branch-generic}
   {\fam2 M}^\gog_\gop\left(\Lambda^p(\R^n) \otimes \C_\lambda\right)|_{\gog^\prime} \\
   \simeq \bigoplus_{N \in \N_0} {\fam2M}^{\gog^\prime}_{\gop^\prime}
   \left(\Lambda^{p-1}(\R^{n-1}) \otimes\C_{\lambda-N}\right) \oplus
   \bigoplus_{N \in \N_0} {\fam2M}^{\gog^\prime}_{\gop^\prime}
   \left(\Lambda^{p}(\R^{n-1}) \otimes\C_{\lambda-N}\right).
\end{multline}
\end{prop}

For generic $\lambda$, the identity \eqref{diag-branch-generic} refines to an
actual branching law (with direct sums of irreducible modules on the right
hand-side).

\begin{prop}[\bf Character identities. Middle degree cases]\label{charind2} With the same notation as in
Proposition \ref{charind}, it holds
\begin{align}\label{diag-branch-md-odd+}
   {\fam2 M}^\gog_\gop \big(\Lambda^{\frac{n-1}{2}} & (\R^n) \otimes \C_\lambda \big)|_{\gog^\prime} \notag \\
   & \simeq \bigoplus_{N \in \N_0} {\fam2M}^{\gog^\prime}_{\gop^\prime}
   \big(\Lambda_+^{\frac{n-1}{2}}(\R^{n-1}) \otimes \C_{\lambda-N} \big)
   \oplus \bigoplus_{N \in \N_0} {\fam2M}^{\gog^\prime}_{\gop^\prime} \big(\Lambda_-^{\frac{n-1}{2}}(\R^{n-1})
   \otimes\C_{\lambda-N}\big) \notag \\
   & \oplus \bigoplus_{N \in \N_0} {\fam2M}^{\gog^\prime}_{\gop^\prime}
   \big(\Lambda^{\frac{n-3}{2}}(\R^{n-1}) \otimes \C_{\lambda-N} \big)
\end{align}
and
\begin{align}\label{diag-branch-md-odd-}
   {\fam2 M}^\gog_\gop \big(\Lambda^{\frac{n+1}{2}}& (\R^n) \otimes \C_\lambda
   \big)|_{\gog^\prime} \notag \\
   & \simeq  \bigoplus_{N \in \N_0} {\fam2M}^{\gog^\prime}_{\gop^\prime}
   \big(\Lambda^{\frac{n+1}{2}}(\R^{n-1}) \otimes \C_{\lambda-N} \big) \notag \\
   & \oplus \bigoplus_{N \in \N_0} {\fam2M}^{\gog^\prime}_{\gop^\prime}
   \big(\Lambda_+^{\frac{n-1}{2}}(\R^{n-1}) \otimes \C_{\lambda-N} \big)
   \oplus \bigoplus_{N \in \N_0} {\fam2M}^{\gog^\prime}_{\gop^\prime}
   \big(\Lambda_-^{\frac{n-1}{2}}(\R^{n-1}) \otimes\C_{\lambda-N}\big)
\end{align}
for odd $n$ and
\begin{equation}\label{diag-branch-md-even}
  {\fam2 M}^\gog_\gop\left(\Lambda_\pm^{\frac{n}{2}}(\R^n) \otimes \C_\lambda\right)|_{\gog^\prime}
  \simeq \bigoplus_{N \in \N_0} {\fam2M}^{\gog^\prime}_{\gop^\prime}
  \left(\Lambda^{\frac{n}{2}-1}(\R^{n-1}) \otimes \C_{\lambda-N}\right)
\end{equation}
for even $n$.
\end{prop}

The $F$-method is a procedure to construct the emdeddings of the ${\fam2
U}(\gog^\prime)$-submodules in the decompositions in Proposition \ref{charind}
and Proposition \ref{charind2}. It has been explained in Section
\ref{F-method}. Its realization requires to work with the induced
representations in the non-compact model. The following lemma provides the
necessary details. In the non-compact model of the infinitesimal induced
representation $\dm\pi_{\lambda,p}$, the $\gog(\R)$-module
$\mathrm{Ind}^G_P(V_{\lambda,p})$ is given by the space $C^\infty(\gon_-(\R))
\otimes V_{\lambda,p} \simeq C^\infty(\R^n) \otimes V_{\lambda,p}$ of smooth
$V_{\lambda,p}$-valued functions on $\gon_-(\R) \simeq \R^n$. We shall write
$\gon_-(\R) \ni X = \sum_{k=1}^n x_k E_j^-$ and $\gon_-^*(\R) \ni \xi =
\sum_{k=1}^n \xi_k (E_k^-)^*$.

\begin{lem}\label{fundamental} (1) The operator $\dm\pi_{\lambda,p}(E_j^+)$, $j=1,\dots, n-1$,
acts on $C^\infty(\R^n) \otimes V_{\lambda,p}$ by
\begin{align}\label{dualverma}
   \dm\pi_{\lambda,p}(E_j^+)(u \otimes \omega)(x) & =
   \Big(-\tfrac{1}{2} \sum_{k=1}^n x^2_k \partial_{x_j}
   + x_j(\lambda + \sum_{k=1}^{n} x_k\partial_{x_k})\Big)(u) \otimes \omega \notag \\
   & - \sum_{k=1}^n x_k u \otimes (E_k^- \otimes E_j^+ - E_j^- \otimes E_k^+)(\omega),
\end{align}
where $u \in C^\infty(\R^n)$ and $\omega \in \Lambda^p(\R^n)$. \\
(2) The operator $d\tilde{\pi}_{\lambda,p}(E_j^+)$, $j=1,\dots,n-1$, acts on
$\Pol(\R^n) \otimes V_{\lambda,p}$ by
\begin{align}\label{fourdualverma}
   \dm \tilde{\pi}_{\lambda,p}(E_j^+) (p \otimes \omega)
   & = -i\left(\tfrac{1}{2} \xi_j \Delta_\xi + (\lambda-E_\xi)\partial_{\xi_j} \right)(p) \otimes \omega \notag \\
   & +i \sum_{k=1}^{n} \partial_{\xi_k}(p) \otimes (E_k^- \otimes E_j^+ - E_j^- \otimes E_k^+)(\omega),
\end{align}
where $p \in \Pol(\R^n)$ and $\omega \in \Lambda^p(\R^n)$. Here $i\in\C$ is the
complex unit, $\Delta_\xi=\sum_{k=1}^n\partial_{\xi_k}^2$ is the Laplace
operator in the variables $\xi_1,\dots,\xi_n$ of $\gon_-^*(\R)$ and
$E_\xi = \sum_{k=1}^n\xi_k\partial_{\xi_k}$ is the Euler operator. \\
(3) The operator $(\dm\pi_{\lambda,p})^\ch(E_j^+)$, $j=1,\dots,n-1$, acts on
$C^\infty(\R^n) \otimes V^\ch_{\lambda,p}$ by
\begin{align}\label{adjoint}
   (\dm\pi_{\lambda,p})^\ch(E_j^+) (u \otimes \omega)(x) & = \Big(-\tfrac{1}{2}\sum_{k=1}^n x^2_k \partial_{x_j}
   + x_j(-\lambda + \sum_{k=1}^{n} x_k\partial_{x_k})\Big)(u) \otimes \omega \notag \\
   & + \sum_{k=1}^n x_k u \otimes ((E_j^+)^* \otimes (E_k^-)^* - (E_k^+)^* \otimes (E_j^-)^*)(\omega),
\end{align}
where $u \in C^\infty(\R^n)$ and $\omega \in \Lambda^p(\R^n)^*$.
\end{lem}

We recall that through the identifications $\gon_\pm(\R) \simeq \R^n$ the
elements $E^-_k \otimes E_j^+ - E_j^- \otimes E_k^+$ are regarded as elements
of $\mathfrak{so}(n,\R)$. They act on $\Lambda^p(\R^n)$ as stated in Remark
\ref{AlgebraicAction}.

We also note that the formulation of Lemma \ref{fundamental} has been
simplified by suppressing the tensor products with $\C_\lambda$.

\begin{proof} (i) The proof is a minor extension of the proof of \cite[Lemma 4.1]{koss}.
It basically suffices to determine the additional contribution by the action on
$\Lambda^p(\R^{n})$. We calculate in the matrix realization. Let $\cdot$ denote
the matrix product. The result of the left action of $n=\exp(Z) \in N_+$ on
$x=\exp(X)\in N_-$ is
\begin{align*}
   n^{-1} \cdot x =
   \begin{pmatrix}
   a & -Z + \tfrac 12 \abs{Z}^2 X^t & -\tfrac 12 \abs{Z}^2 \\
   X - \tfrac 12 \abs{X}^2 Z^t & \id - Z^t \otimes X^t & Z^t \\
   -\tfrac 12 \abs{X}^2 & -X^t & 1
   \end{pmatrix},
\end{align*}
where $a \st 1 - Z \cdot X + \tfrac{\abs{Z}^2\abs{X}^2}{4}$. Here we identify
$Z$ with an arrow vector and $X$ with a column vector. If $Z$ and $X$ are
sufficiently small, we have $a \ne 0$, and this element decomposes as a product
$\tilde{x} \cdot p$, where $\tilde{x} \in N_-$ and
\begin{equation*}
   p = \begin{pmatrix} a & * & * \\ 0 & m &*\\ 0 & 0 & a^{-1} \end{pmatrix}\in P.
\end{equation*}
The elements $a$, $m$ and $\tilde{x}$ are given up to the first-order terms in
$\abs{Z}$ by
\begin{equation}\label{FirstOrderContr1}
   a \sim 1 - Z \cdot X \quad \mbox{and} \quad m \sim \id - Z^t \otimes X^t + X \otimes Z
\end{equation}
and
\begin{equation}\label{FirstOrderContr2}
   \tilde{x} = \exp (\tilde{X}), \; \tilde{X} = (1+ Z \cdot X) (X - \tfrac{1}{2} \abs{X}^2 Z^t).
\end{equation}
Now let $X=\sum_{k=1}^n x_k E^-_k \in \gon_-$, $Z = tE_j^+ \in \gon_+$, and let
$a_j(t)$, $m_j(t)$ and $\tilde{x}_j(t)$ be the corresponding $1$-parameter
subgroups of elements in \eqref{FirstOrderContr1} and \eqref{FirstOrderContr2}.
Then using $(E_j^+)^t = E_j^-$ we find
\begin{equation*}
   \dm \xi_\lambda \left(\tfrac{d}{dt}|_{t=0}(a_j(t))\right) = - \lambda x_j
\end{equation*}
and
\begin{equation}\label{eq:algebraicAction}
   \tfrac{d}{dt}|_{t=0}(m_j(t)) = \sum_{k=1}^n x_k (E^-_k \otimes E_j^+ - E_j^-\otimes E_k^+)
   \in \mathfrak{so}(n,\R).
\end{equation}
Now, for $u \in \Ind_P^G(V_{\lambda,p})$, we obtain
\begin{equation*}
   \pi_{\lambda,p}(\exp(tE_j^+))(u)(x) = u(\exp(-tE_j^+) \cdot x)
   = \xi_{\lambda}(a_j(t)^{-1}) \sigma_p(m_j(t)^{-1}) u(\tilde{x}_j(t)).
\end{equation*}
The assertion (i) follows by differentiation.

(ii) The representation $\dm \tilde{\pi}_{\lambda,p}$ is defined by Fourier
transform of the non-compact model of the induced representation with the
inducing module $V_{\lambda,p}$. Therefore, the formula for the action of $\dm
\tilde{\pi}_{\lambda,p}(E_j^+)$ follows from the formula in (i) in two steps:
first reverse the order of the composition and then apply the Fourier transform
using $x_j \mapsto -i\partial_{\xi_j}$ and $\partial_{x_j} \mapsto -i\xi_j$
preserving the order of the composition.

(iii) We recall that $\C_{-\lambda}$ is the dual of $\C_\lambda$. Moreover, the
dual of the action of $E_k^- \otimes E_j^+$ on $\Lambda^p(\R^n)$ is given by
the negative of the action of $(E_j^+)^* \otimes(E_k^-)^*$ on
$\Lambda^p(\R^n)^*$. Hence the assertion follows from (i).
\end{proof}

\section{Singular vectors}\label{SingularVectors}

In this section, we determine the singular vectors
$$
   \Hom_{\gop^\prime}(\Lambda^q(\gon_-^\prime(\R)) \otimes \C_{\lambda-N},
   \Pol_N(\gon_-^*(\R)) \otimes \Lambda^p(\gon_-(\R)) \otimes \C_\lambda)
$$
which correspond to the homomorphisms
\begin{equation*}
   {\fam2U}(\gog^\prime) \otimes_{{\fam2U}(\gop^\prime)}(\Lambda^{q}(\gon_-^\prime(\R))
   \otimes \C_{\lambda-N}) \to {\fam2U}(\gog)\otimes_{{\fam2U}(\gop)}(\Lambda^p(\gon_-(\R)) \otimes \C_\lambda)
\end{equation*}
of generalized Verma modules. We apply the $F$-method and proceed as follows.
We first determine the consequences of the $\gol^\prime$-equivariance of (the
Fourier transform of) singular vectors. This leads to a natural distinction
between four types of singular vectors together with a natural ansatz in each
case. The remaining analysis of their $\gon_+^\prime$-equivariance will be
carried out in four separate sections according to the type. These results are
the key technical results of the paper.

\subsection{The $\gol^\prime$-equivariance}\label{first-test}

We study homomorphisms of the form
\begin{equation}\label{space-3}
   \Hom_{SO(n-1)}(\Lambda^q(\R^{n-1})^* \otimes \C, \Pol_N(\R^n) \otimes \Lambda^p(\R^n)^*)
\end{equation}
for arbitrary $p$ and $q$ and all $N \in \N_0$. Here $\Pol_N(\R^n)$ denotes the
vector space of complex-valued polynomials on $\R^n$ which are homogenous of
degree $N \in \N_0$. The results will be applied to evaluate the $\gol^
\prime$-equivariance of singular vectors and provide a natural ansatz for the
analysis of their $\gon_+^\prime$-invariance.

Among the homomorphisms of the form \eqref{space-3}, those in the spaces
\begin{equation}\label{space-1}
   \Hom_{SO(n-1)}(\Lambda^p(\R^{n-1})^* \otimes \C, \Pol_N(\R^n) \otimes \Lambda^p(\R^n)^*)
\end{equation}
and
\begin{equation}\label{space-2}
   \Hom_{SO(n-1)}(\Lambda^{p-1}(\R^{n-1})^* \otimes \C,\Pol_N(\R^n)\otimes \Lambda^p(\R^n)^*)
\end{equation}
will play a dominant role in what follows.

We identify the spaces
$$
   \Pol_N^p(\R^n) \st \Pol_N(\R^n) \otimes \Lambda^p(\R^n)^*
$$
with the respective subspaces of $\Omega^p(\R^n)$ consisting
of differential $p$-forms on $\R^n$ with complex-valued homogeneous polynomial
coefficients of degree $N$. In particular, we apply to the elements of
$\Pol_N^p(\R^n)$ the usual rules of the calculus of differential forms.

The spaces $\Lambda^p(\R^n)$, $\Lambda^p(\R^n)$ and $\Pol_N(\R^n)$ will be
regarded as $SO(n,\R)$-modules with the usual induced and push-forward actions,
respectively. The fact that some of the modules $\Lambda^p(\R^n) \otimes \C$
are reducible (Lemma \ref{form-rep}) will play only a minor role in the
following and henceforth will be suppressed for the sake of uniform statements.

We use Euclidean coordinates $x_1,\dots,x_n$ on $\R^n$, basis tangential
vectors $\partial_1,\dots,\partial_n$ and dual covectors $dx_1,\dots,dx_n$. Let
$\dm$ and $\delta$ denote the exterior differential and co-differential on
differential forms on $\R^n$, respectively. Let $\Delta = \delta \dm + \dm
\delta$ be the corresponding Laplacian on forms. Let $x_n$ be the normal
variable for the subspace $\R^{n-1}$.\footnote{In later sections, we shall use
the notation $\bar{d}$, $\bar{\delta}$ and $\bar{\Delta}$ for these operators.}

Any $p$-form $\omega \in \Pol^p_N(\R^n)$, $p \ge 1$, admits a normal Taylor
series
\begin{equation}\label{Taylor}
   \omega = \sum_{j=0}^N x_n^{N-j} (\omega_j^+ + dx_n \wedge \omega_j^-).
\end{equation}
Hence any
$$
   \omega \in \Hom_{SO(n-1)}(\Lambda^p (\R^{n-1})^* \otimes \C,\Pol^p_N(\R^n)), \; p\ge 1
$$
has Taylor coefficients
$$
   \omega_j^+ \in \Hom_{SO(n-1)}(\Lambda^p (\R^{n-1})^* \otimes \C,\Pol^p_j(\R^{n-1}))
$$
and
$$
   \omega_j^- \in \Hom_{SO(n-1)}(\Lambda^p (\R^{n-1})^* \otimes \C,\Pol^{p-1}_j(\R^{n-1}));
$$
for $p=0$, the components $\omega_j^-$ vanish, of course.

Next, we recall the decomposition
$$
   \Pol_N(\R^n) \simeq \bigoplus_{k=0}^{\left[\frac{N}{2}\right]} r^{2k} \Ha_{N-2k}(\R^n), \quad r^2 = |x|^2,
$$
where the space $\Ha_N(\R^n) = \ker \Delta |_{\Pol_N(\R^n)}$ of homogeneous
harmonic polynomials of degree $N$ is an irreducible $SO(n,\R)$-module of
highest weight $N \Lambda_1$.

The following result extends this fact to polynomial forms. Let
$$
   \Ha^p_N(\R^n) = \left\{ \omega \in \Pol^p_N(\R^n) \; |\; \Delta \omega = 0,
   \delta \omega = 0 \right\} \subset \Omega^p(\R^n)
$$
be the subspace of co-closed {\em harmonic} homogeneous polynomial $p$-forms
$\omega$ of degree $N$. The element $\alpha \st \sum_{j=1}^n x_j \otimes dx_j
\in \Pol_1^1(\R^n)$ defines a map
$$
   \alpha \wedge: \Pol^p_N(\R^n) \to \Pol^{p+1}_{N+1}(\R^n)
$$
by $ \alpha \wedge (p \otimes dx_{\alpha_1} \wedge \dots \wedge dx_{\alpha_p})
= \sum_{j=1}^n x_j p \otimes dx_j \wedge dx_{\alpha_1} \wedge \dots \wedge
dx_{\alpha_p}$. Furthermore, the element $E \st \sum_{j=1}^n x_j \otimes
\partial_j$ defines a map
$$
   i_E: \Pol^p_N(\R^n) \to \Pol^{p-1}_{N+1}(\R^n)
$$
by $i_E(p \otimes \omega) = \sum_{j=1}^n x_j p \otimes i_{\partial_j}(\omega)$.


\begin{prop}[\cite{IT}, Theorem 6.8]\label{IT} (1) For $p=1,\dots,n$ and $N \in \N_0$, we have
the decomposition
$$
   \Pol^p_N(\R^n) = \Ha^p_N(\R^n) \oplus (r^2 \Pol_{N-2}^p(\R^n) + \alpha \wedge
   \Pol_{N-1}^{p-1}(\R^n)), \; r^2=\abs{x}^2.
$$
Moreover, the $SO(n,\R)$-module $\Ha_N^p(\R^n)$ decomposes as
$$
   \Ha_N^p(\R^n) = {^\prime \Ha}_N^p(\R^n) \oplus {''\Ha}_N^p(\R^n),
   \quad p\!+\!N \ne 0, \; n\!-\!p\!+\!N \ne 0,
$$
where
$$
   {'\Ha}^p_N(\R^n) = \Ha_N^p(\R^n) \cap \ker d \quad \mbox{and}
   \quad {''\Ha}^p_N(\R^n) = \Ha_N^p(\R^n) \cap \ker i_E.
$$
(2) For odd $n$, the $SO(n,\R)$-module ${'\Ha}^p_N(\R^n)$ is irreducible and
has highest weights
$$
\begin{cases}
   N \Lambda_1 + \Lambda_p, & 0 < p < \frac{n-1}{2} \\
   N \Lambda_1 + 2 \Lambda_{\frac{n-1}{2}}, & p = \frac{n-1}{2},\frac{n+1}{2} \\
   N \Lambda_1 + \Lambda_{n-p}, & \frac{n+1}{2} < p \le n-1
\end{cases}
$$
or, equivalently,
$$
\begin{cases}
   N 1_1 + 1_p & 0 < p \le \frac{n-1}{2} \\
   N 1_1 + 1_{n-p}, & \frac{n+1}{2} \le p \le n-1.
\end{cases}
$$
(3) For odd $n$, the $SO(n,\R)$-module ${''\Ha}^p_N(\R^n)$ ($N \ge 1$) is
irreducible and has highest weights
$$
\begin{cases}
   (N\!-\!1) \Lambda_1 + \Lambda_{p+1}, & 0 \le p < \frac{n-3}{2} \\
   (N\!-\!1) \Lambda_1 + 2 \Lambda_{\frac{n-1}{2}}, & p = \frac{n-3}{2},\frac{n-1}{2} \\
   (N\!-\!1) \Lambda_1 + \Lambda_{n-p-1}, & \frac{n+1}{2} \le p \le n-1
\end{cases}
$$
or, equivalently,
$$
\begin{cases}
   (N\!-\!1) 1_1 + 1_{p+1}, & 0 \le p \le \frac{n-3}{2} \\
   (N\!-\!1) 1_1 + 1_{n-p-1}, & \frac{n-1}{2} \le p \le n-1.
\end{cases}
$$
(4) For even $n$, the $SO(n,\R)$-module ${'\Ha}^p_N(\R^n)$ decomposes into
irreducible modules with highest weights
$$
\begin{cases}
   N \Lambda_1 + \Lambda_p, & 0 < p < \frac{n}{2}-1 \\
   N \Lambda_1 + \Lambda_{\frac{n}{2}-1}+ \Lambda_\frac{n}{2}, & p = \frac{n}{2}-1,\frac{n}{2}+1 \\
   N \Lambda_1 + 2\Lambda_{\frac{n}{2}-1}, N\Lambda_1 + 2\Lambda_{\frac{n}{2}}, & p = \frac{n}{2} \\
   N\Lambda_1 + \Lambda_{n-p}, & \frac{n}{2}+1 < p \le n-1
\end{cases}
$$
or, equivalently,
$$
\begin{cases}
   N 1_1 + 1_p, & 0 < p \le \frac{n}{2}-1 \\
   N 1_1 + 1_{\frac{n}{2}}^+, N 1_1 + 1_{\frac{n}{2}}^-, & p = \frac{n}{2} \\
   N 1_1 + 1_{n-p}, & \frac{n}{2}+1 \le p \le n-1.
\end{cases}
$$
The decompositions are multiplicity free. \\
(5) For even $n$, the $SO(n,\R)$-module ${''\Ha}^p_N(\R^n)$ ($N \ge 1$)
decomposes into irreducible modules with highest weights
$$
\begin{cases}
   (N\!-\!1) \Lambda_1 + \Lambda_{p+1}, & 0 \le p < \frac{n}{2}-2 \\
   (N\!-\!1) \Lambda_1 + \Lambda_{\frac{n}{2}-1}+ \Lambda_\frac{n}{2}, & p = \frac{n}{2}-2,\frac{n}{2} \\
   (N\!-\!1) \Lambda_1 + 2\Lambda_{\frac{n}{2}-1}, (N-1) \Lambda_1 + 2 \Lambda_{\frac{n}{2}}, & p = \frac{n}{2}-1 \\
   (N\!-\!1) \Lambda_1 + \Lambda_{n-p-1}, & \frac{n}{2} < p \le n-1
\end{cases}
$$
or, equivalently,
$$
\begin{cases}
   (N\!-\!1) 1_1 + 1_{p+1}, & 0 \le p \le \frac{n}{2}-2 \\
   (N\!-\!1) 1_1 + 1_{\frac{n}{2}}^+, (N\!-\!1) 1_1 + 1_{\frac{n}{2}}^-, & p = \frac{n}{2}-1 \\
   (N\!-\!1) 1_1 + 1_{n-p-1}, & \frac{n}{2} \le p \le n-1.
\end{cases}
$$
The decompositions are multiplicity free. \\
(6) We have
\begin{equation}
   \Ha_0^0 (\R^n) = {'\Ha}_0^0 (\R^n) = {''\Ha}_0^0(\R^n) = \C \quad \mbox{and} \quad
   \Ha_0^n (\R^n) = {'\Ha}_0^n (\R^n) = \C.
\end{equation}
Moreover, ${'\Ha}_N^0(\R^n) = 0$ for $N \ge 1$ and ${''\Ha}_0^p(\R^n) = 0$ for
$0 < p \le n$.
\end{prop}

Proposition \ref{IT}/(1) implies that any element of $\Pol_N^p(\R^{n-1})$ can
be decomposed according to
\begin{equation}\label{deco-harmonic}
   \Ha_N^p(\R^{n-1}) \oplus r^2 \Ha_{N-2}^p(\R^{n-1}) \oplus \dots + \alpha \wedge
   (\Ha_{N-1}^{p-1}(\R^{n-1}) \oplus r^2 \Ha_{N-3}^{p-1}(\R^{n-1}) \oplus \cdots),
\end{equation}
where $r^2 = |x'|^2 = \sum_{k=1}^{n-1} x_k^2$ and $\alpha = \sum_{i=1}^{n-1}
x_i \otimes dx_i$, $x' \in \R^{n-1}$.

In the following, spaces of harmonic polynomials on $\R^{n-1}$ will be denoted
by $\Ha^p_N$, ${'\Ha}^p_N$ and ${''\Ha}^p_N$. We shall also assume that $p \ne
\frac{n-1}{2}$ if $n$ is odd. Then the $SO(n-1,\R)$-module
$\Lambda^p(\R^{n-1})^* \otimes \C$ is irreducible and has highest weight $1_p$
or $1_{n-1-p}$ depending on the size of $p$.

We continue with the discussion of the structure of the spaces \eqref{space-1}.
For this purpose, we need to find the contributions of the irreducible
$SO(n-1,\R)$-module of respective highest weights $1_p$ and $1_{n-1-p}$ to
$\Pol^p_N(\R^n)$. For this purpose, we apply the decomposition
\eqref{deco-harmonic} to the Taylor coefficients $\omega_j^\pm$ of $\omega \in
\Pol_N^p(\R^n)$ in \eqref{Taylor}.

We start with the discussion of the terms $\omega_j^+$.

First, by Proposition \ref{IT}/(1), the decomposition \eqref{deco-harmonic} of
the space $\Pol^p_{2j+1}(\R^{n-1})$ of polynomial forms of {\em odd degree} on
$\R^{n-1}$ contains the $SO(n-1,\R)$-modules
\begin{equation}\label{mod-1}
    {'\Ha}_{2k+1}^p, \; {''\Ha}_{2k+1}^p, \; 0 \le k \le j
\end{equation}
and
\begin{equation}\label{mod-2}
   \alpha \wedge {'\Ha}_{2k}^{p-1}, \; \alpha \wedge {''\Ha}_{2k}^{p-1}, \; 0 \le k \le j.
\end{equation}

Now assume that $n$ is even. By Proposition \ref{IT}/(2),(3), these modules are
of respective highest weights
$$
   \begin{cases}
   (2k+1) 1_1 + 1_p, & 0 < p \le \frac{n}{2}-1 \\
   (2k+1) 1_1 + 1_{n-1-p}, & \frac{n}{2} \le p \le n-2
   \end{cases}
   \quad , \quad
   \begin{cases}
   (2k) 1_1 + 1_{p+1}, & 0 \le p \le \frac{n}{2}-2 \\
   (2k) 1_1 + 1_{n-2-p}, & \frac{n}{2}-1 \le p \le n-2
   \end{cases}
$$
and
$$
   \begin{cases}
   (2k) 1_1 + 1_{p-1}, & 1 < p \le \frac{n}{2} \\
   (2k) 1_1 + 1_{n-p}, & \frac{n}{2}+1 \le p \le n-1
   \end{cases}
   \quad , \quad
   \begin{cases}
   (2k-1) 1_1 + 1_p, & 1 \le p \le \frac{n}{2}-1 \\
   (2k-1) 1_1 + 1_{n-1-p}, & \frac{n}{2} \le p \le n-1.
   \end{cases}
$$
Among these contributions, the module of highest weight $1_p$ (for $0 \le p \le
\frac{n}{2}-1$) appears iff $p=\frac{n}{2}-1$ and it is realized as
$$
   r^{2j} {''\Ha}_1^{\frac{n}{2}-1} \subset \Pol_{2j+1}^{\frac{n}{2}-1}(\R^{n-1}).
$$
The module ${''\Ha}_1^{\frac{n}{2}-1}$ is isomorphic to $\Ha_0^{\frac{n}{2}}$
embedded into $\Pol_1^{\frac{n}{2}-1}(\R^{n-1})$ by $\omega \mapsto
i_E(\omega)$. Hence it coincides with the range of the map
$$
   \Lambda^{\frac{n}{2}-1}(\R^{n-1})^* \otimes \C \ni \omega \mapsto i_E( \star \, \omega) \in
   \Pol^{\frac{n}{2}-1}_{1}(\R^{n-1}).
$$
Thus, for even $n$, we find the homomorphisms
\begin{equation}\label{exotic-1}
   \Lambda^{\frac{n}{2}-1}(\R^{n-1})^* \otimes \C \ni \omega \mapsto
   x_n^{N-2j-1} r^{2j} i_E( \star \, \omega) \in \Pol_N^{\frac{n}{2}-1}(\R^{n-1}), \; 0 \le 2j \le N-1
\end{equation}
in the spaces \eqref{space-1}.

Similarly, the module of highest weight $1_{n-1-p}$ (for $\frac{n}{2} \le p \le
n-1$) appears iff $p=\frac{n}{2}$ and it is realized as
$$
   r^{2j} \alpha \wedge {'\Ha}_0^{\frac{n}{2}-1} \subset \Pol_{2j+1}^{\frac{n}{2}}(\R^{n-1}).
$$
Thus, for even $n$, we find the homomorphisms
\begin{equation}\label{exotic-2}
    \Lambda^{\frac{n}{2}}(\R^{n-1})^* \otimes \C \ni \omega \mapsto
    x_n^{N-2j-1} r^{2j} \alpha \wedge \star \, \omega \in \Pol_N^{\frac{n}{2}}(\R^{n-1}), \; 0 \le 2j \le N-1
\end{equation}
in the spaces \eqref{space-1}.

Now assume that $n$ is odd. By Proposition \ref{IT}/(4),(5), the modules
\eqref{mod-1} and \eqref{mod-2} are of respective highest weights
$$
   \begin{cases}
   (2k+1) 1_1 + 1_p, & 0 < p \le \frac{n-3}{2} \\
   (2k+1) 1_1 + 1_{\frac{n-1}{2}}^\pm, & p=\frac{n-1}{2} \\
   (2k+1) 1_1 + 1_{n-1-p}, & \frac{n+1}{2} \le p \le n-2
   \end{cases}
   \quad , \quad
   \begin{cases}
   (2k) 1_1 + 1_{p+1}, & 0 \le p \le \frac{n-5}{2} \\
   (2k) 1_1 + 1_{\frac{n-1}{2}}^\pm, & p= \frac{n-3}{2} \\
   (2k) 1_1 + 1_{n-2-p}, & \frac{n-1}{2} \le p \le n-2
   \end{cases}
$$
and
$$
   \begin{cases}
   (2k) 1_1 + 1_{p-1}, & 1 < p \le \frac{n-1}{2} \\
   (2k) 1_1 + 1_{\frac{n-1}{2}}^\pm, & p= \frac{n+1}{2} \\
   (2k) 1_1 + 1_{n-p}, & \frac{n+3}{2} \le p \le n-1
   \end{cases}
   \quad , \quad
   \begin{cases}
   (2k-1) 1_1 + 1_p, & 1 \le p \le \frac{n-3}{2} \\
   (2k-1) 1_1 + 1_{\frac{n-1}{2}}^\pm, & p= \frac{n-1}{2} \\
   (2k-1) 1_1 + 1_{n-1-p}, & \frac{n+1}{2} \le p \le n-1.
   \end{cases}
$$
The modules of highest weights $1_p$ and $1_{n-1-p}$ do not appear among these
modules.

Second, we consider the space $\Pol^p_{2j}(\R^{n-1})$ of polynomial forms of
{\em even degree} $2j$. The corresponding decomposition \eqref{deco-harmonic}
contains the $SO(n-1,\R)$-modules
\begin{equation}\label{mod-3}
   {'\Ha}_{2k}^p, \; {''\Ha}_{2k}^p, \; 0 \le k \le j
\end{equation}
and \begin{equation}\label{mod-4}
   \alpha \wedge {'\Ha}_{2k-1}^{p-1}, \; \alpha \wedge {''\Ha}_{2k-1}^{p-1}, \; 0 \le k \le j.
\end{equation}

Now assume that $n$ is even. By Proposition \ref{IT}/(2),(3), these modules are
of respective highest weights
$$
   \begin{cases}
   (2k) 1_1 + 1_p, & 0 < p \le \frac{n}{2}-1 \\
   (2k) 1_1 + 1_{n-1-p}, & \frac{n}{2} \le p \le n-2
   \end{cases}
   \quad , \quad
   \begin{cases}
   (2k-1) 1_1 + 1_{p+1}, & 0 \le p \le \frac{n}{2}-2 \\
   (2k-1) 1_1 + 1_{n-2-p}, & \frac{n}{2}-1 \le p \le n-2
   \end{cases}
$$
and
$$
   \begin{cases}
   (2k-1) 1_1 + 1_{p-1}, & 1 < p \le \frac{n}{2} \\
   (2k-1) 1_1 + 1_{n-p}, & \frac{n}{2}+1 \le p \le n-1
   \end{cases}
   \quad , \quad
   \begin{cases}
   (2k-2) 1_1 + 1_p, & 1 \le p \le \frac{n}{2}-1 \\
   (2k-2) 1_1 + 1_{n-1-p}, & \frac{n}{2} \le p \le n-1.
   \end{cases}
$$
It follows that the representations of highest weights $1_p$ and $1_{n-1-p}$
are realized by
$$
   r^{2j} {'\Ha}_0^p \quad \mbox{and} \quad r^{2j-2} \alpha \wedge {''\Ha}_1^{p-1}.
$$
The module ${''\Ha}_1^{p-1}$ is isomorphic to $\Ha^p_0$ embedded into
$\Pol_1^{p-1}(\R^{n-1})$ by $\omega \mapsto i_E (\omega)$. In fact, it is
isomorphic to $\ker i_E = i_E (\Ha_0^p)$. Thus, we find the homomorphisms
\begin{equation}\label{type-1-pp1}
   \Lambda^p(\R^{n-1})^* \otimes \C \simeq \Ha_0^p \ni \omega \mapsto x_n^{N-2j} r^{2j}
   \omega \in \Pol_N^p(\R^{n-1}), \; 0 \le 2j \le N
\end{equation}
and
\begin{equation}\label{type-1-pp2}
   \Lambda^p(\R^{n-1})^* \otimes \C \simeq \Ha_0^p \ni \omega
   \mapsto x_n^{N-2j} r^{2j-2} \alpha \wedge i_E (\omega) \in
   \Pol_N^p(\R^{n-1}), \; 2 \le 2j \le N
\end{equation}
in the spaces \eqref{space-1}.

Now assume that $n$ is odd. By Proposition \ref{IT}/(4),(5), the modules
\eqref{mod-3} and \eqref{mod-4} are of respective highest weights
$$
   \begin{cases}
   (2k) 1_1 + 1_p, & 0 < p \le \frac{n-3}{2} \\
   (2k) 1_1 + 1_{\frac{n-1}{2}}^\pm, & p = \frac{n-1}{2} \\
   (2k) 1_1 + 1_{n-1-p}, & \frac{n+1}{2} \le p \le n-2
   \end{cases}
   \quad , \quad
   \begin{cases}
   (2k-1) 1_1 + 1_{p+1}, & 0 \le p \le \frac{n-5}{2} \\
   (2k-1) 1_1 + 1_{\frac{n-1}{2}}^\pm, & p = \frac{n-3}{2} \\
   (2k-1) 1_1 + 1_{n-2-p}, & \frac{n-1}{2} \le p \le n-2
   \end{cases}
$$
and
$$
   \begin{cases}
   (2k-1) 1_1 + 1_{p-1}, & 1 < p \le \frac{n-1}{2} \\
   (2k-1) 1_1 + 1_{\frac{n-1}{2}}^\pm, & p = \frac{n+1}{2} \\
   (2k-1) 1_1 + 1_{n-p}, & \frac{n+3}{2} \le p \le n-1
   \end{cases}
   \quad , \quad
   \begin{cases}
   (2k-2) 1_1 + 1_p, & 1 \le p \le \frac{n-3}{2} \\
   (2k-2) 1_1 + 1_{\frac{n-1}{2}}^\pm, & p = \frac{n-1}{2} \\
   (2k-2) 1_1 + 1_{n-1-p}, & \frac{n+1}{2} \le p \le n-1.
   \end{cases}
$$
Again, it follows that the representations of highest weights $1_p$ and
$1_{n-1-p}$ are realized by
$$
   r^{2j} {'\Ha}_0^p \quad \mbox{and} \quad r^{2j-2} \alpha \wedge {''\Ha}_1^{p-1}.
$$
This yields the homomorphisms \eqref{type-1-pp1} and \eqref{type-1-pp2}.
Moreover, in this case, we find the additional homomorphisms
\begin{equation}\label{exotic-7}
   \Lambda^{\frac{n-1}{2}}(\R^{n-1})^* \otimes \C \ni \omega \mapsto x_n^{N-2j} r^{2j}
   \star \omega \in \Pol_N^{\frac{n-1}{2}}(\R^{n-1}), \; 0 \le 2j \le N
\end{equation}
and
\begin{equation}\label{exotic-8}
   \Lambda^{\frac{n-1}{2}}(\R^{n-1})^* \otimes \C \ni \omega
   \mapsto x_n^{N-2j} r^{2j-2} \alpha \wedge i_E (\star \, \omega) \in
   \Pol_N^{\frac{n-1}{2}}(\R^{n-1}), \; 2 \le 2j \le N.
\end{equation}

This finishes the discussion of the contributions of the terms $\omega_j^+$.

We continue with a summary of an analogous discussion of the contributions by
the terms $\omega_j^-$. Similar arguments as above imply that the terms
$\omega_{2j}^-$ of {\em even degree} only contribute the homomorphisms
\begin{equation}\label{exotic-3}
   \Lambda^\frac{n}{2}(\R^{n-1})^* \otimes \C \ni \omega \mapsto x_n^{N-2j} r^{2j} dx_n \wedge \star \,
   \omega \in \Pol_N^{\frac{n}{2}}(\R^{n-1}), \;  0 \le 2j \le N
\end{equation}
and
\begin{equation}\label{exotic-4}
   \Lambda^\frac{n}{2}(\R^{n-1})^* \otimes \C \ni \omega \mapsto x_n^{N-2j} r^{2j} dx_n \wedge \alpha \wedge
   i_E (\star \, \omega) \in \Pol_N^{\frac{n}{2}}(\R^{n-1}), \; 0 \le 2j \le N
\end{equation}
if $n$ is even. There are no contributions for odd $n$. Finally, we find that
the {\em odd degree} terms $\omega_{2j-1}^-$ contribute through
$$
   r^{2j-2} {''\Ha}_1^{p-1} \subset \Pol_{2j-1}^{p-1}(\R^{n-1})
$$
and
$$
   r^{2j-2} i_E ({''\Ha}_1^{\frac{n-1}{2}}) \subset
   \Pol_{2j-1}^{\frac{n-3}{2}}(\R^{n-1}) \quad \mbox{and} \quad
   r^{2j-2} \alpha \wedge {'\Ha}_0^{\frac{n-3}{2}} \subset
   \Pol_{2j-1}^{\frac{n-1}{2}}(\R^{n-1})
$$
for odd $n$. Since the module ${''\Ha}_1^{p-1}$ is isomorphic to $i_E
(\Ha_0^p)$, we find the contributions
$$
   r^{2j-2} dx_n \wedge i_E (\Ha_0^p)
$$
and
$$
   r^{2j-2} dx_n \wedge i_E {'\Ha}_0^{\frac{n+1}{2}}
   \quad \mbox{and} \quad
   r^{2j-2} dx_n \wedge \alpha \wedge {'\Ha}_0^{\frac{n-3}{2}}
$$
for odd $n$. Altogether, this leads to the homomorphisms
\begin{equation}\label{type-1-pp3}
    \Lambda^p(\R^{n-1})^* \otimes \C \simeq \Ha_0^p  \ni \omega
    \mapsto x_n^{N-2j+1} r^{2j-2} dx_n \wedge i_E(\omega) \in \Pol_N^p(\R^{n-1})
\end{equation}
and
\begin{align}
    \Lambda^{\frac{n-1}{2}}(\R^{n-1})^* \otimes \C \ni \omega
    & \mapsto x_n^{N-2j+1} r^{2j-2} dx_n \wedge i_E (\star \, \omega) \in
    \Pol^{\frac{n-1}{2}}_N(\R^{n-1}), \label{exotic-5} \\
    \Lambda^{\frac{n+1}{2}}(\R^{n-1})^* \otimes \C \ni \omega
    & \mapsto x_n^{N-2j+1} r^{2j-2} dx_n \wedge \alpha \wedge \star \, \omega \in
    \Pol^{\frac{n+1}{2}}_N(\R^{n-1}) \label{exotic-6}
\end{align}
for odd $n$ and $2 \le 2j \le N+1$ in the space \eqref{space-1}.

We summarize the above discussion as

\begin{prop}\label{structure-type-1} For any $N \in \N_0$, the space
\eqref{space-1} is generated by
\begin{itemize}
\item the homomorphisms \eqref{type-1-pp1}, \eqref{type-1-pp2}, \eqref{type-1-pp3}
(for general form degrees $p$), and
\item the eight exotic homomorphisms \eqref{exotic-1},
\eqref{exotic-2}, \eqref{exotic-7}, \eqref{exotic-8}, \eqref{exotic-3},
\eqref{exotic-4}, \eqref{exotic-5} and \eqref{exotic-6} (for specific form
degrees).
\end{itemize}
All these homomorphisms are generated by the operations
\begin{itemize}
\item multiplication by positive integer powers of $x_n$ and $r^2$,
\item $dx_n \wedge$, $\alpha \wedge$, $i_E$ and
\item $\star$.
\end{itemize}
Any composition of the latter operations yields an element of \eqref{space-1}
for some $N \in \N_0$. The exotic homomorphisms all contain the Hodge star
operator $\star$ as a factor.
\end{prop}

Similar arguments can be used to describe the spaces \eqref{space-2}. Again, we
apply the decomposition \eqref{deco-harmonic} to the Taylor coefficients of
$\omega \in \Pol_N^p(\R^n)$ in \eqref{Taylor}. As a result, we find the
homomorphisms
\begin{equation}\label{type-2-pp-1}
    \Lambda^{p-1}(\R^{n-1})^* \otimes \C \simeq \Ha_0^{p-1} \ni \omega
    \mapsto x_n^{N-1-2j} r^{2j} \alpha \wedge \omega \in \Pol^p_N(\R^{n-1})
\end{equation}
for $0 \le 2j \le N\!-\!1$,
\begin{equation}\label{type-2-pp-2}
   \Lambda^{p-1}(\R^{n-1})^* \otimes \C \simeq \Ha_0^{p-1} \ni \omega \mapsto x_n^{N-2j}
   r^{2j} dx_n \wedge \omega \in \Pol^p_N(\R^{n-1})
\end{equation}
for $0 \le 2j \le N$ and
\begin{equation}\label{type-2-pp-3}
   \Lambda^{p-1}(\R^{n-1})^* \otimes \C \simeq \Ha_0^{p-1} \ni \omega
   \mapsto x_n^{N-2j-2} r^{2j} dx_n \wedge \alpha \wedge i_E(\omega) \in \Pol^p_N(\R^{n-1})
\end{equation}
for $0 \le 2j \le N\!-\!2$ and general $p$. In addition, there are eight exotic
homomorphisms for specific form degrees.

A similar discussion yields the following classification of equivariant
homomorphisms \eqref{space-3} for arbitrary $p$ and $q$.

\begin{prop}\label{hom-structure} For any $N \in \N_0$, the space \eqref{space-3}
is generated by the homomorphisms of the form \eqref{type-1-pp1},
\eqref{type-1-pp2} and \eqref{type-1-pp3} (for $q=p$), the homomorphisms of the
form \eqref{type-2-pp-1}, \eqref{type-2-pp-2} and \eqref{type-2-pp-3} (for
$q=p-1$), the two additional homomorphisms
\begin{equation}\label{add-homo-1}
   \Lambda^{p+1}(\R^{n-1})^* \otimes \C \ni \omega \mapsto x_n^{N-2j-1} r^{2j}
   i_E(\omega) \in \Pol^p_N(\R^{n-1}), \; 0 \le 2j \le N\!-\!1
\end{equation}
and
\begin{equation}\label{add-homo-2}
    \Lambda^{p-2}(\R^{n-1})^* \otimes \C \ni \omega \mapsto x_n^{N-2j-1} r^{2j}
    dx_n \wedge \alpha \wedge \omega \in \Pol^p_N(\R^{n-1}), \; 0 \le 2j \le
    N\!-\!1
\end{equation}
as well as the compositions of these homomorphisms with $\star$. The exotic
homomorphisms in \eqref{space-1} and \eqref{space-2} are restrictions of the
latter compositions with $\star$ to appropriate form degrees.
\end{prop}

Now we apply these results as follows.

We consider forms on $\gon_-^*(\R) \simeq (\R^n)^*$ with polynomial
coefficients. We regard these forms as elements of $\Pol(\gon_-^*(\R)) \otimes
\Lambda^p(\gon_-(\R))$. We use coordinates $\xi_j$ on $\gon_-^*(\R)$ and the
standard bases $\{E_j = E_j^-, j=1,\dots,n\}$ and $\{E_j^*, j=1,\dots,n\}$ of
$\gon_-(\R)$ and $\gon^*_-(\R)$.

Since $\R^n \simeq (\R^n)^*$ as $SO(n,\R)$-modules, \eqref{type-1-pp1},
\eqref{type-1-pp2} and \eqref{type-1-pp3} imply that there are three types of
equivariant homomorphisms in
\begin{equation*}\label{first-type}
   \Hom_{SO(\gon_-^\prime(\R))}(\Lambda^{p}(\gon_-^\prime(\R)) \otimes \C,
   \Pol_N(\gon_-^*(\R)) \otimes \Lambda^p(\gon_-(\R)) \otimes \C).
\end{equation*}
These are given by the maps
\begin{align}\label{type-1-N-gen}
   \omega & \mapsto \xi_n^{N-2j} \abs{\xi'}^{2j} \otimes \omega, \notag \\
   \omega & \mapsto \xi_n^{N-2j-1} \abs{\xi'}^{2j} E_n \wedge i_E (\omega), \notag \\
   \omega & \mapsto \xi_n^{N-2j-2} \abs{\xi'}^{2j} \alpha \wedge i_E (\omega)
\end{align}
for appropriate values of $j$. Here we use the definitions
$$
   \alpha \st \sum_{j=1}^{n-1} \xi_j \otimes E_j \quad \mbox{and} \quad
   i_E \st \sum_{j=1}^{n-1} \xi_j \otimes i_{E_j^*}.
$$
Equivalently, these maps take the form
\begin{align}\label{type-1-N}
   E_{\alpha_1} \wedge \dots \wedge E_{\alpha_p} & \mapsto \xi_n^{N-2j} \abs{\xi'}^{2j}
   \otimes E_{\alpha_1} \wedge \dots \wedge E_{\alpha_p}, \notag \\
   E_{\alpha_1}\wedge \dots \wedge E_{\alpha_p} & \mapsto \xi_n^{N-2j-1} \abs{\xi'}^{2j}
   \xi_{[\alpha_1} \otimes E_n \wedge E_{\alpha_2}\wedge \cdots \wedge E_{\alpha_p]}, \notag \\
   E_{\alpha_1}\wedge \dots \wedge E_{\alpha_p} & \mapsto \xi_n^{N-2j-2} \abs{\xi'}^{2j}
   \sum_{j=1}^{n-1} \xi_j \xi_{[\alpha_1} \otimes E_j \wedge E_{\alpha_2}\wedge \cdots \wedge
   E_{\alpha_p]}
\end{align}
on basis elements of $\Ha_0^p$ given by partitions $1\leq \alpha_1 < \dots
<\alpha_{p} \leq n-1$. Here we use the convention
\begin{equation}\label{anti}
   \xi_{[\alpha_1} \otimes \omega \otimes E_{\alpha_2} \wedge \cdots \wedge
   E_{\alpha_p]} = \sum_{l=1}^{p} (-1)^{l+1} \xi_{\alpha_l} \otimes \omega \otimes E_{\alpha_1}\wedge
   \cdots \wedge \widehat{E_{\alpha_l}} \wedge \cdots \wedge E_{\alpha_p}
\end{equation}
for any $\omega \in \Lambda^*(\gon_-(\R))$, and we suppress obvious tensor
products with copies of $\C$.

Similarly, by \eqref{type-2-pp-1}--\eqref{type-2-pp-3}, there are three types
of equivariant homomorphisms
\begin{equation*}\label{second-type}
   \Hom_{SO(\gon_-^\prime(\R))}(\Lambda^{p-1}(\gon_-^\prime(\R)) \otimes \C,
   \Pol_N(\gon_-^*(\R)) \otimes \Lambda^p (\gon_-(\R)) \otimes \C)
\end{equation*}
which are given by the maps
\begin{align}\label{type-2-N-gen}
   \omega & \mapsto \xi_n^{N-2j} \abs{\xi'}^{2j} \otimes E_n \wedge \omega, \notag \\
   \omega & \mapsto \xi_n^{N-2j-1} \abs{\xi'}^{2j} \alpha \wedge \omega, \notag \\
   \omega & \mapsto \xi_n^{N-2j-2} \abs{\xi'}^{2j} E_n \wedge \alpha \wedge i_E (\omega)
\end{align}
(for appropriate values of $j$) or, equivalently, by the maps
\begin{align}\label{type-2-N}
   E_{\alpha_1}\wedge \cdots \wedge E_{\alpha_{p-1}} &
   \mapsto \xi_n^{N-2j} \abs{\xi'}^{2j} \otimes E_n \wedge E_{\alpha_1} \wedge \cdots \wedge E_{\alpha_{p-1}},
   \notag \\
   E_{\alpha_1}\wedge \cdots \wedge E_{\alpha_{p-1}} &
   \mapsto \xi_n^{N-2j-1} \abs{\xi'}^{2j} \sum_{j=1}^{n-1} \xi_j \otimes E_j
   \wedge E_{\alpha_1}\wedge\cdots\wedge E_{\alpha_{p-1}}, \notag \\
   E_{\alpha_1} \wedge \cdots \wedge E_{\alpha_{p-1}} & \mapsto
   \xi_n^{N-2j-2} \abs{\xi'}^{2j} \sum_{j=1}^{n-1} \xi_j \xi_{[\alpha_1}
   \otimes E_n \wedge E_j \wedge E_{\alpha_2} \wedge \cdots \wedge E_{\alpha_{p-1}]}
\end{align}
on basis elements of $\Ha_0^{p-1}$ given by partitions $1\leq \alpha_1 < \cdots
<\alpha_{p-1} \leq n-1$. Again, we use the convention \eqref{anti} and we
suppress obvious tensor products with copies of $\C$.

Finally, according to Proposition \ref{hom-structure}, there are two additional
types
\begin{equation*}
   \Hom_{SO(\gon_-^\prime(\R))}(\Lambda^{p+1}(\gon_-^\prime(\R)) \otimes \C,
   \Pol_N(\gon_-^*(\R)) \otimes \Lambda^p(\gon_-(\R)) \otimes \C)
\end{equation*}
(for $0 \le p \le n-2$) and
\begin{equation*}
   \Hom_{SO(\gon_-^\prime(\R))}(\Lambda^{p-2}(\gon_-^\prime(\R)) \otimes \C,
   \Pol_N(\gon_-^*(\R)) \otimes \Lambda^p(\gon_-(\R)) \otimes \C)
\end{equation*}
(for $2 \le p \le n$) of homomorphisms which are given by the respective maps
\begin{equation}\label{type-3-N-gen}
   \omega \mapsto \xi_n^{N-2j-1} \abs{\xi^\prime}^{2j} i_E(\omega)
\end{equation}
and
\begin{equation}\label{type-4-N-gen}
   \omega \mapsto \xi_n^{N-2j-1} \abs{\xi^\prime}^{2j} E_n \wedge \alpha \wedge \omega
\end{equation}
(for appropriate values of $j$) or, equivalently, by the maps
\begin{equation}\label{type-3-N}
   E_{\alpha_1}\wedge \cdots \wedge E_{\alpha_{p+1}} \mapsto \xi_n^{N-2j-1} \abs{\xi^\prime}^{2j}
   \xi_{[\alpha_1}\otimes E_{\alpha_2}\wedge\cdots\wedge E_{\alpha_{p+1}]}
\end{equation}
and
\begin{equation}\label{type-4-N}
   E_{\alpha_1}\wedge \cdots \wedge E_{\alpha_{p-2}} \mapsto \xi_n^{N-2j-1} \abs{\xi^\prime}^{2j}
   \sum_{k=1}^{n-1} \xi_k \otimes E_n \wedge E_k \wedge E_{\alpha_1}\wedge\cdots\wedge
   E_{\alpha_{p-2}}.
\end{equation}

Now, by applying the Fourier transform (see Section \ref{F-method}), we regard
singular vectors which correspond to $\gog^\prime$-homomorphisms
\begin{equation*}
   {\fam2U}(\gog^\prime) \otimes_{{\fam2U}(\gop^\prime)}(\Lambda^{q}(\gon_-^\prime(\R))
   \otimes \C_{\eta}) \to {\fam2U}(\gog)\otimes_{{\fam2U}(\gop)}(\Lambda^p(\gon_-(\R)) \otimes \C_\mu),
   \; \eta, \mu \in \C
\end{equation*}
of generalized Verma modules as elements of the spaces
$$
   \Hom_{\gop^\prime}(\Lambda^q(\gon_-^\prime(\R)) \otimes \C_{\eta},
   \Pol_N(\gon_-^*(\R)) \otimes \Lambda^p(\gon_-(\R)) \otimes \C_\mu)
$$
for some $N \in \N_0$; here the modules $\Lambda^q(\gon_-^\prime(\R))$ and
$\Lambda^p(\gon_-(\R))$ are regarded as trivial $A$-modules and $A$ acts on
$\Pol_N(\gon_-^*(\R))$ by the push-forward action. It follows that
$$
    \eta = -N + \mu.
$$
Therefore, it suffices to study the homomorphisms in the spaces
\begin{equation}\label{homo-type}
   \Hom_{\gop^\prime}(\Lambda^q(\gon_-^\prime(\R)) \otimes \C_{\lambda-N},
   \Pol_N(\gon_-^*(\R)) \otimes \Lambda^p(\gon_-(\R)) \otimes \C_\lambda), \; \lambda \in \C.
\end{equation}
By Proposition \ref{hom-structure}, their $\gol'$-invariance implies that they
are linear combinations of the homogeneous homomorphisms of degree $N$ listed
in \eqref{type-1-N}, \eqref{type-2-N}, \eqref{type-3-N} and \eqref{type-4-N} as
well as their compositions with $\star$. Each choice of $q \in
\{p-2,p-1,p,p+1\}$ corresponds to one of these four types of homomorphisms. In
order to find explicit formulas for the corresponding singular vectors, it
remains to determine those respective linear combinations which, in addition,
are annihilated by the operators
\begin{equation}\label{major}
   \dm \tilde{\pi}_{\lambda,p}(Z), \; Z \in \gon_+^\prime(\R)
\end{equation}
(see Lemma \ref{fundamental}).  This will be the subject of the following
subsections. The complete list of singular vectors which correspond to the
homomorphisms of the form \eqref{homo-type} then follows by composing these
results with the operator $\star$. In particular, the discussion will determine
{\em all} singular vectors which describe the embeddings of the
${\fam2U}(\gog^\prime)$-submodules in Propositions \ref{charind} and
\ref{charind2}.

For the following considerations, it will be convenient to introduce some
additional notation. We define the operators
\begin{equation}\label{P-operator}
   P_j(\lambda) \st (\tfrac{1}{2} \xi_j \Delta_\xi + (\lambda - E_\xi)\partial_{\xi_j}) \otimes 1
   - \sum_{k=1}^{n} \partial_{\xi_k} \otimes (E_k^- \otimes E_j^+ - E_j^- \otimes E_k^+)
\end{equation}
for $j=1,\dots,n-1$ on $\Pol_N(\gon_-^*(\R)) \otimes V_{\lambda,p} \simeq
\Pol_N(\R^n) \otimes V_{\lambda,p}$. We recall that
$$
   P_j(\lambda) = i \dm \tilde{\pi}_{\lambda,p}(E_j^+)
$$
(see \eqref{fourdualverma}). We shall also write $\partial_j$ instead of
$\partial_{\xi_j}$.

\subsection{Families of singular vectors of the first type}\label{sv-type1}

Assume that $N \in \N_0$. We first consider singular vectors of {\em odd}
homogeneity $2N+1$ in
$$
   \Hom_{\gop^\prime}(\Lambda^{p}(\gon_-^\prime(\R)) \otimes \C_{\lambda-(2N+1)},
   \Pol_{2N+1}(\gon_-^*(\R)) \otimes \Lambda^p(\gon_-(\R)) \otimes \C_\lambda),
$$
which are linear combinations of the homomorphisms in \eqref{type-1-N-gen}. In
the following, we shall refer to these singular vectors as singular vectors of
the {\em first} type. They correspond to $\gog^\prime$-homomorphisms
\begin{equation*}
   {\fam2 U}(\gog^\prime)\otimes_{{\fam2 U}(\gop^\prime)}(\Lambda^{p}(\gon_-^\prime(\R)) \otimes
   \C_{\lambda-(2N+1)}) \to{\fam2 U}(\gog)\otimes_{{\fam2U}(\gop)}(\Lambda^p(\gon_-(\R)) \otimes\C_\lambda).
\end{equation*}
By \eqref{type-1-N-gen}, such singular vectors have the form
\begin{equation*}
   v_{2N+1}^{(p\to p)}(\lambda) =
   \xi_n^{2N+1} P(t) \otimes \id + \xi_n^{2N} Q(t) E_n \wedge i_E + \xi_n^{2N-1} R(t) \alpha \wedge i_E
\end{equation*}
or, equivalently,
\begin{align}\label{ansatz-odd-2}
   v_{2N+1}^{(p\to p)}(\lambda)(E_{\alpha_1} \wedge \cdots \wedge E_{\alpha_{p}})
   & = \xi_n^{2N+1}P(t) \otimes E_{\alpha_1}\wedge\cdots\wedge E_{\alpha_p} \notag \\
   & + \xi_n^{2N}Q(t)\xi_{[\alpha_1}\otimes E_n\wedge E_{\alpha_2}\wedge\cdots\wedge E_{\alpha_p]} \notag \\
   & + \xi_n^{2N-1}R(t)\sum_{i=1}^{n-1}\xi_i\xi_{[\alpha_1}
   \otimes E_i\wedge E_{\alpha_2}\wedge\cdots\wedge E_{\alpha_p]}
\end{align}
for any partition $1\leq \alpha_1 < \cdots <\alpha_{p} \leq n-1$. Here
$$
   t \st \frac{\sum_{i=1}^{n-1}\xi_i^2}{\xi_n^2} = \frac{\abs{\xi'}^2}{\xi_n^2}
$$
and $P(t)$, $Q(t)$, $R(t)$ are polynomials in $t$ of respective degrees $N$,
$N$ and $N-1$ to be determined; the contributions by $R(t)$ do not appear for
$N=0$. In order to simplify the notation, we have suppressed obvious tensor
products with copies of $\C$.

We treat $\xi_n$ and $t$ as independent variables. Now using the rules
\begin{equation*}
    \partial_n=-\frac{2}{\xi_n}t \frac{\partial}{\partial t} \quad
    \mbox{and} \quad
    \partial_j=\frac{2\xi_j}{\xi_n^2} \frac{\partial}{\partial t}, \quad
    j=1,\ldots,n-1,
\end{equation*}
we find
\begin{align*}
   & \frac{1}{2} \partial_n^2 (\xi_n^{2N+1} P) = \xi_n^{2N-1}(2t^2 {P}'' + (1-4N) t {P}' + N(2N+1) P) \\
   & \frac{1}{2} \Delta' (\xi_n^{2N+1} P) = \xi_n^{2N-1} (2t {P}'' + (n-1) {P}') \\
   & \big(\lambda - \sum_{i=1}^n \xi_i \partial_i\Big) \partial_j (\xi_n^{2N+1} P)
   = \xi_n^{2N-1} \xi_j (2\lambda - 4N) {P}'.
\end{align*}
These results allow to express the action of $P_j(\lambda)$ on
$$
   v_1 \st \xi_n^{2N+1}P(t) \otimes E_{\alpha_1}\wedge\cdots\wedge E_{\alpha_p}
$$
as the sum
\begin{align*}
    P_j(\lambda) v_1 & = \big[
    2t(t\!+\!1) P^{\prime\prime} + (-4N\!+\!1)t P^\prime + (2\lambda\!+\!n\!-\!4N\!-\!1) P^\prime \\
    & + N(2N\!+\!1)P \big]
    \xi_n^{2N-1} \xi_j\otimes E_{\alpha_1} \wedge \cdots \wedge E_{\alpha_{p}}\\
    & + 2P^\prime\xi_n^{2N-1} \xi_{[\alpha_1}\otimes E_j \wedge E_{\alpha_2} \wedge\cdots
    \wedge E_{\alpha_{p}]} \\
    & - 2P^\prime\xi_n^{2N-1}\sum_{i=1}^{n-1}\xi_i\delta_{j[\alpha_1}
    \otimes E_i\wedge E_{\alpha_2}\wedge\cdots\wedge E_{\alpha_{p}]}\\
    & + \left[2tP^\prime - (2N\!+\!1)P\right] \xi_n^{2N}\delta_{j[\alpha_1}
    \otimes E_n\wedge E_{\alpha_2}\wedge\cdots\wedge E_{\alpha_{p}]}.
\end{align*}
Similarly, the action of $P_j(\lambda)$ on
$$
   v_2 \st \xi_n^{2N}Q(t)\xi_{[\alpha_1}\otimes E_n\wedge E_{\alpha_2}\wedge\cdots\wedge E_{\alpha_p]}
$$
is given by
\begin{align*}
    P_j(\lambda)v_2 & = \big[2t(t\!+\!1) Q^{\prime\prime} + (-4N\!+\!3)tQ^\prime +
    (2\lambda\!+\!n\!-\!4N\!+\!1) Q^\prime \\
    & + N(2N\!-\!1)Q \big]\xi_n^{2N-2}\xi_j
    \xi_{[\alpha_1} \otimes E_n \wedge E_{\alpha_2} \cdots \wedge E_{\alpha_{p]}}\\
    &+ (\lambda\!+\!p\!-\!2N\!-\!1) Q \xi_n^{2N}
    \delta_{j[\alpha_1}\otimes E_n\wedge E_{\alpha_2}\wedge\cdots\wedge E_{\alpha_p]}\\
    &+ \big[2NQ^\prime - 2t Q^\prime\big]\xi_n^{2N-1}
    \xi_{[\alpha_1}\otimes E_j\wedge E_{\alpha_2}\wedge\cdots\wedge E_{\alpha_p]}.
\end{align*}
Finally, the action of $P_j(\lambda)$ on
$$
   v_3 \st \xi_n^{2N-1}R(t)\sum_{i=1}^{n-1}\xi_i\xi_{[\alpha_1}
   \otimes E_i\wedge E_{\alpha_2}\wedge\cdots\wedge E_{\alpha_p]}
$$
takes the form
\begin{align*}
    P_j(\lambda)v_3 & = \big[2t(t\!+\!1) R^{\prime\prime} + (-4N\!+\!5)tR^\prime +
    (2\lambda\!+\!n\!-\!4N\!+\!1)R^\prime\\
    & \quad + (N\!-\!1)(2N\!-\!1) R\big] \xi_n^{2N-3}\xi_j\sum_{i=1}^{n-1}\xi_i\xi_{[\alpha_1}
      \otimes E_i\wedge E_{\alpha_2}\wedge\cdots\wedge E_{\alpha_{p}]}\\
    & \quad + \left[2tR^\prime + (\lambda\!+\!n\!-\!p\!-\!2N) R\right] \xi_n^{2N-1}\xi_{[\alpha_1}
      \otimes E_j\wedge E_{\alpha_2}\wedge\cdots\wedge E_{\alpha_{p}]}\\
    & \quad + (\lambda\!+\!p\!-\!2N\!-\!1) R\xi_n^{2N-1}\sum_{i=1}^{n-1}\xi_i\delta_{j[\alpha_1}
      \otimes E_i\wedge E_{\alpha_2}\wedge\cdots\wedge E_{\alpha_{p}]}\\
    & \quad + \left[2tR^\prime - (2N\!-\!1)R\right] \xi_n^{2N-2}\xi_j \xi_{[\alpha_1}
      \otimes E_n\wedge E_{\alpha_2}\wedge\cdots\wedge E_{\alpha_{p}]}.
\end{align*}

We summarize these results and find that the condition
\begin{equation*}
   P_j(\lambda) v^{(p\to p)}_{2N+1}(\lambda) = 0 \quad \mbox{for $j=1,\dots,n-1$}
\end{equation*}
is equivalent to the system
\begin{align}\label{eq:OddSystemPtoP}
    0 & = 2t(t\!+\!1)P^{\prime\prime} + (-4N\!+\!1) tP^\prime
    + (2\lambda\!+\!n\!-\!4N\!-\!1)P^\prime + N(2N\!+\!1)P, \notag\\
    0 & = 2P^\prime + 2NQ - 2tQ^\prime + 2tR^\prime + (\lambda\!+\!n\!-\!p\!-\!2N)R, \notag\\
    0 & = -2P^\prime + (\lambda\!+\!p\!-\!2N\!-\!1) R,\notag\\
    0 & = 2tP^\prime - (2N\!+\!1)P + (\lambda\!+\!p\!-\!2N\!-\!1) Q,\notag\\
    0 & = 2t(t\!+\!1) Q^{\prime\prime} + (-4N\!+\!3)t Q^\prime
    + (2\lambda\!+\!n\!-\!4N\!+\!1) Q^\prime + N(2N\!-\!1)Q + 2 tR^\prime -(2N\!-\!1)R,\notag\\
    0 & = 2t(t\!+\!1)R^{\prime\prime}+ (-4N\!+\!5)t R^\prime + (2\lambda\!+\!n\!-\!4N\!+\!1) R^\prime
    +(N\!-\!1)(2N\!-\!1)R
\end{align}
of ordinary differential equations for $P$, $Q$ and $R$.

Now assume that $N \ge 1$ and that the polynomials
$$
   P(t) = \sum_{j=0}^{N} p_j t^j, \quad Q(t) = \sum_{j=0}^{N} q_j t^j \quad
   \mbox{and} \quad R(t) = \sum_{j=0}^{N-1} r_j t^j
$$
with unknown coefficients solve the system \eqref{eq:OddSystemPtoP}. We divide
the analysis of this assumption into a series of steps. The following arguments
will make use of the recursive relations \eqref{Gegen-rec1}, \eqref{Gegen-rec2}
for (unnormalized) Gegenbauer coefficients.

{\em Step 1}. The first equation gives the recurrence relation
\begin{align*}
   (N\!-\!j\!+\!1)(2N\!-\!2j\!+\!3) p_{j-1} + j(2\lambda\!+\!n\!-\!4N\!+\!2j\!-\!3) p_j = 0
\end{align*}
for $j=1,\ldots,N$. By \eqref{Gegen-rec2}, it implies that $p_j =
p^{(N)}_j(\lambda)$ are odd Gegenbauer coefficients with $p_N^{(N)}(\lambda)$
being undetermined.

{\em Step 2}. The sixth equation gives the recurrence relation
\begin{equation*}
   ((N\!-\!1)\!-\!j\!+\!1)(2(N\!-\!1)\!-\!2j\!+\!3) r_{j-1}
   + j(2(\lambda\!-\!1)\!+\!n\!-\!4(N\!-\!1)\!+\!2j\!-\!3 ) r_j = 0
\end{equation*}
for $j=1,\ldots, N\!-\!1$. By \eqref{Gegen-rec2}, it implies that $r_j =
r^{(N-1)}_j(\lambda\!-\!1)$ are odd Gegenbauer coefficients with
$r_{N-1}^{(N-1)}(\lambda\!-\!1)$ being undetermined.

{\em Step 3}. Now assume that
$$
   \lambda\!+\!p\!-\!2N\!-\!1 \ne 0.
$$
Then the fourth equation relates $P(t)$ and $Q(t)$ through
\begin{equation}\label{eq:Relation1}
   q_j = \frac{(2N\!-\!2j\!+\!1)}{(\lambda\!+\!p\!-\!2N\!-\!1)}
   p_j^{(N)}(\lambda), \quad j=0,\dots,N.
\end{equation}
Combining this relation with the explicit formula for $p_j^{(N)}(\lambda)$
implies that $q_j = q_j^{(N)}(\lambda\!-\!1)$ are even Gegenbauer coefficients
with the normalization
\begin{equation}\label{eq:Relation1d}
   q_N^{(N)}(\lambda\!-\!1) = \frac{1}{(\lambda\!+\!p\!-\!2N\!-\!1)} p_N^{(N)}(\lambda).
\end{equation}

{\em Step 4}. Similarly, the third equation relates $P(t)$ and $R(t)$ through
\begin{equation}\label{eq:Relation1a}
   r_{j-1} = \frac{2j}{(\lambda\!+\!p\!-\!2N\!-\!1)} p_{j}^{(N)}(\lambda), \quad j=1,\dots,N.
\end{equation}
By combining this relation with the  explicit formula for $p_j^{(N)}(\lambda)$,
we again conclude that $r_j^{(N-1)}(\lambda\!-\!1)$ are odd Gegenbauer
coefficients with the normalization
\begin{equation}\label{eq:Relation1b}
   r_{N-1}^{(N-1)}(\lambda\!-\!1) = \frac{2N}{(\lambda\!+\!p\!-\!2N\!-\!1)} p_N^{(N)}(\lambda).
\end{equation}

{\em Step 5}. The sum of the second and the third equation gives
\begin{equation}\label{eq:Relation2}
   q_j^{(N)}(\lambda\!-\!1) = -\frac{(2\lambda\!+\!n\!-\!4N\!+\!2j\!-\!1)}{(2N\!-\!2j)}
   r^{(N-1)}_{j}(\lambda\!-\!1), \quad j=0,\dots, N-1.
\end{equation}
This relation between $R(t)$ and $Q(t)$ already follows by comparing
\eqref{eq:Relation1} and \eqref{eq:Relation1a} using the recurrence relation
\eqref{Gegen-rec2} for $p_j^{(N)}(\lambda)$. In particular, we find the
relation
$$
   2N q_N^{(N)}(\lambda\!-\!1) = r_{N-1}^{(N-1)}(\lambda\!-\!1).
$$

{\em Step 6}. The fifth equation is satisfied by the polynomials determined in
the previous steps. In fact, the fifth equation is equivalent to the recurrence
relation
\begin{equation*}
   (N\!-\!j)(2N\!-\!2j\!-\!1) q_j - (2N\!-\!2j\!-\!1)r_j +
   (j\!+\!1)(2\lambda\!+\!n\!-\!4N\!+\!2j\!+\!1) q_{j+1} = 0.
\end{equation*}
By $q_j = q_j^{(N)}(\lambda\!-\!1)$, $r_j = r^{(N-1)}_j(\lambda\!-\!1)$ and the
identity
$$
  (N\!-\!j)(2N\!-\!2j\!-\!1) q_j^{(N)}(\lambda\!-\!1) = -(j\!+\!1)(2(\lambda\!-\!1)\!+\!n\!-\!4N\!+\!2j\!+\!1)
  q_{j+1}^{(N)}(\lambda\!-\!1)
$$
(see \eqref{Gegen-rec1}), this relation is equivalent to
$$
   (2N\!-\!2j\!-\!1) r_j^{(N-1)}(\lambda\!-\!1) = 2(j\!+\!1) q_{j+1}^{(N)}(\lambda\!-\!1).
$$
Again, using \eqref{Gegen-rec1}, the latter relation is equivalent to
\eqref{eq:Relation2}.

{\em Step 7}. If $\lambda\!+\!p\!-\!2N\!-\!1=0$, then the third and the fourth
equation are equivalent to $P=0$. Hence the first equation is trivially
satisfied. The sixth equations implies $r_j=r_j^{(N-1)}(\lambda\!-\!1)$ with an
undetermined coefficient $r_{N-1}^{(N-1)}(\lambda\!-\!1)$ (as in Step 2). The
second equation yields the relation
$$
   q_j = -\frac{(2\lambda\!+\!n\!-\!4N\!+\!2j\!-\!1)}{(2N\!-\!2j)}
   r^{(N-1)}_{j}(\lambda\!-\!1), \quad j=0,\dots,N-1.
$$
It follows that the coefficients $q_j$ are even Gegenbauer coefficients: $q_j =
q_j^{(N)}(\lambda\!-\!1)$. Then the fifth equation is satisfied by the
arguments in Step 6.

For $N=0$, the only solutions are $P=p_0$, $Q=q_0$ with the relation $p_0 =
(\lambda+p-1) q_0$.

This completes the analysis of the system \eqref{eq:OddSystemPtoP}.

Finally, we choose the normalization $p_N^{(N)}(\lambda) =
(\lambda\!+\!p\!-\!2N\!-\!1)$ and summarize the above results.

\begin{theorem}\label{OddFromPtoP} Assume that $N \in \N_0$. The first type
homomorphisms
\begin{equation*}
   \xi_n^{2N+1} P(t) \otimes \id + \xi_n^{2N} Q(t) E_n \wedge i_E + \xi_n^{2N-1} R(t) \alpha \wedge i_E
\end{equation*}
with the polynomial coefficients
$$
   P(t) = \sum_{j=0}^N p^{(N)}_j(\lambda;p) t^j, \quad
   Q(t) = \sum_{j=0}^N q^{(N)}_j(\lambda\!-\!1) t^j \quad \mbox{and} \quad
   R(t) = \sum_{j=0}^{N-1} r^{(N-1)}_j(\lambda\!-\!1)t^j
$$
in the variable $t = \abs{\xi'}^2/\xi_n^2$ are singular vectors of odd
homogeneity $2N+1$ in the space
$$
   \Hom_{\gop^\prime}(\Lambda^p (\gon_-^\prime(\R)) \otimes \C_{\lambda-(2N+1)},
   \Pol_{2N+1}(\gon_-^*(\R)) \otimes \Lambda^p (\gon_-(\R)) \otimes \C_\lambda)
$$
iff
\begin{align*}
   p_j^{(N)}(\lambda;p) & = (\lambda\!+\!p\!-\!2N\!-\!1) b_j^{(N)}(\lambda), \notag \\
   q_j^{(N)}(\lambda) & = a_j^{(N)}(\lambda), \notag \\
   r_j^{(N-1)}(\lambda) & = 2N b_j^{(N-1)}(\lambda),
\end{align*}
up to a constant multiple. These singular vectors will be denoted by
$v_{2N+1}^{(p\to p)}(\lambda)$. They describe the embedding of the module
$$
   {\fam2 M}^{\gog^\prime}_{\gop^\prime} \left(\Lambda^{p}(\R^{n-1}) \otimes \C_{\lambda-(2N+1)}\right)
$$
in Proposition \ref{charind}.
\end{theorem}

We continue with the discussion of singular vectors of {\em even} homogeneity
$2N$, $N \in \N_0$, in
$$
   \Hom_{\gop^\prime}(\Lambda^{p}(\gon_-^\prime(\R)) \otimes \C_{\lambda-2N},
   \Pol_{2N}(\gon_-^*(\R)) \otimes \Lambda^p(\gon_-(\R)) \otimes \C_\lambda),
$$
which are linear combinations of the homomorphisms in \eqref{type-1-N-gen}.
Again, these will be referred to as singular vectors of the {\em first type}.
They correspond to $\gog^\prime$-homomorphisms
\begin{equation*}
   {\fam2 U}(\gog^\prime)\otimes_{{\fam2 U}(\gop^\prime)}(\Lambda^p(\gon_-(\R))
   \otimes\C_{\lambda-2N}) \to{\fam2 U}(\gog)\otimes_{{\fam2 U}(\gop)}(\Lambda^p(\gon_-(\R)) \otimes \C_\lambda).
\end{equation*}
Similarly, as in the case of odd homogeneity, we use the ansatz
\begin{equation*}
   v_{2N}^{(p\to p)}(\lambda) =
   \xi_n^{2N} P(t) \otimes \id + \xi_n^{2N-1} Q(t) E_n \wedge i_E + \xi_n^{2N-2} R(t) \alpha \wedge i_E
\end{equation*}
or, equivalently,
\begin{align}\label{ansatz-even2}
    v_{2N}^{(p\to p)}(\lambda)(E_{\alpha_1}\wedge\cdots\wedge E_{\alpha_p})
    & = \xi_n^{2N} P(t)\otimes E_{\alpha_1}\wedge\cdots\wedge E_{\alpha_p}  \notag \\
    & + \xi_n^{2N-1} Q(t)\xi_{[\alpha_1}\otimes E_n\wedge E_{\alpha_2}\wedge\cdots\wedge E_{\alpha_p]} \notag \\
    & + \xi_n^{2N-2} R(t)\sum_{i=1}^{n-1}\xi_i\xi_{[\alpha_1}
    \otimes E_i\wedge E_{\alpha_2}\wedge\cdots\wedge E_{\alpha_p]}
\end{align}
for any partition $1\leq \alpha_1 < \cdots <\alpha_{p} \leq n-1$, where $P(t),
Q(t)$ and $R(t)$ are polynomials of respective degrees $N$, $N-1$ and $N-1$ to
be determined.

Analogous calculations as in the case of odd homogeneity show that the condition
\begin{equation*}
    P_j(\lambda) v^{(p\to p)}_{2N}(\lambda) = 0 \quad \mbox{for $j=1,\dots,n-1$}
\end{equation*}
is equivalent to the system
\begin{align}\label{eq:EvenSystemPtoP}
    0 & = 2t(t\!+\!1)P^{\prime\prime} + (-4N\!+\!3)t P^\prime + (2\lambda\!+\!n\!-\!4N\!+\!1) P^\prime
    + N(2N\!-\!1)P, \notag\\
    0 & = 2P^\prime + (2N\!-\!1)Q - 2tQ^\prime + 2tR^\prime + (\lambda\!+\!n\!-\!p\!-\!2N\!+\!1)R, \notag\\
    0 & = -2P^\prime + (\lambda\!+\!p\!-\!2N) R,\notag\\
    0 & = 2tP^\prime -2NP + (\lambda\!+\!p\!-\!2N) Q,\notag\\
    0 & = 2t(t\!+\!1)Q^{\prime\prime} + (-4N\!+\!5)t Q^\prime + (2\lambda\!+\!n\!-\!4N\!+\!3) Q^\prime
    + (N\!-\!1)(2N\!-\!1)Q \notag \\
    & + 2tR^\prime -(2N\!-\!2)R,\notag\\
    0 & = 2t(t\!+\!1)R^{\prime\prime} + (-4N\!+\!7)tR^\prime + (2\lambda\!+\!n\!-\!4N\!+\!3) R^\prime +
    (N\!-\!1)(2N\!-\!3)R
\end{align}
of ordinary differential equations for $P$, $Q$ and $R$. We omit the details of
the calculation. However, we note that although the system
\eqref{eq:EvenSystemPtoP} arises from the system \eqref{eq:OddSystemPtoP} by
the substitution $N \mapsto N-\tfrac{1}{2}$ the degrees of the involved
polynomials do not coincide.

Now assume that $N \ge 1$ and that the polynomials
$$
   P(t) = \sum_{j=0}^{N} p_j t^j, \quad Q(t) = \sum_{j=0}^{N-1} q_j t^j \quad
   \mbox{and} \quad R(t) = \sum_{j=0}^{N-1} r_j t^j
$$
with unknown coefficients solve the system \eqref{eq:EvenSystemPtoP}. We divide
the analysis of this condition into a series of steps. The following arguments
are similar to those in the discussion of the system \eqref{eq:OddSystemPtoP}.

{\em Step 1}. By \eqref{Gegen-rec1}, the first equation implies that $p_j =
p^{(N)}_j(\lambda)$  are even Gegenbauer coefficients with $p_N^{(N)}(\lambda)$
being undetermined.

{\em Step 2}. By \eqref{Gegen-rec1}, the sixth equation implies that
$r_j=r^{(N-1)}_j(\lambda-1)$ are even Gegenbauer coefficients with
$r_{N-1}^{(N-1)}(\lambda\!-\!1)$ being undetermined.

{\em Step 3}. Now assume that
$$
   \lambda\!+\!p\!-\!2N \ne 0.
$$
Then the fourth equation relates $P(t)$ and $Q(t)$ through
\begin{equation}\label{eq:Relation3}
   q_j = \frac{(2N\!-\!2j)}{(\lambda\!+\!p\!-\!2N)} p_j^{(N)}(\lambda), \quad j=0,\dots,N-1.
\end{equation}
By combining this relation with the explicit formula for $p_j^{(N)}(\lambda)$,
we conclude that $q_j=q_j^{(N-1)}(\lambda\!-\!1)$ are odd Gegenbauer
coefficients which are normalized by
$$
   q_{N-1}^{(N-1)}(\lambda\!-\!1)
   = -\frac{2N(2\lambda\!-\!2N\!+\!n\!-\!1)}{(\lambda\!+\!p\!-\!2N)} p_{N}^{(N)}(\lambda).
$$

{\em Step 4}. The third equation yields the relation
\begin{align}\label{eq:Relation3a}
   r_j^{(N-1)}(\lambda\!-\!1) = \frac{2(j\!+\!1)}{(\lambda\!+\!p\!-\!2N)}
   p_{j+1}^{(N)}(\lambda), \quad j=0,\dots,N-1.
\end{align}
In particular, this relates the normalizations of $P(t)$ and $R(t)$ through
\begin{equation*}
   r^{(N-1)}_{N-1}(\lambda\!-\!1) = \frac{2N}{(\lambda\!+\!p\!-\!2N)} p_N^{(N)}(\lambda).
\end{equation*}

{\em Step 5}. The sum of the second and the third equation gives
\begin{equation}\label{eq:Relation4}
   q_j^{(N-1)}(\lambda\!-\!1) = -\frac{(2\lambda\!+\!n\!-\!4N\!+\!2j\!+\!1)}{(2N\!-\!2j\!-\!1)}
   r^{(N-1)}_{j}(\lambda\!-\!1), \quad j=0,\dots, N-1.
\end{equation}
This relation between $R(t)$ and $Q(t)$ already follows by comparing
\eqref{eq:Relation3} and \eqref{eq:Relation3a} using the recurrence relation
\eqref{Gegen-rec1} for $p_j^{(N)}(\lambda)$.

{\em Step 6}. The fifth equation is satisfied by the polynomials $Q$ and $R$
determined in the previous steps. In fact, the fifth equation is equivalent to
the recurrence relation
$$
   (N\!-\!j\!-\!1)(2N\!-\!2j\!-\!1) q_j + (j\!+\!1)(2\lambda\!+\!n\!-\!4N\!+\!2j\!+\!3) q_{j+1}
   + (2j\!-\!2N\!+\!2) r_j = 0.
$$
By $q_j = q_j^{(N-1)}(\lambda\!-\!1)$, $r_j=r^{(N-1)}_j(\lambda\!-\!1)$ and the
identity
$$
   (N\!-\!j\!-\!1)(2N\!-\!2j\!-\!1) q_j^{(N-1)}(\lambda)
   + (j\!+\!1) (2\lambda\!+\!n\!-\!4N\!+\!2j\!+\!3) q_{j+1}^{(N-1)}(\lambda) = 0
$$
(see \eqref{Gegen-rec2}), this relation is equivalent to
$$
   (2N\!-\!2j\!-\!2) r_j^{(N-1)}(\lambda\!-\!1) = 2 (j\!+\!1) q_{j+1}^{(N-1)}(\lambda\!-\!1).
$$
Again, using \eqref{Gegen-rec2}, the latter relation is equivalent to
\eqref{eq:Relation4}.

{\em Step 7}. If $\lambda+p-2N=0$, then the third and the fourth equation are
equivalent to $P=0$. Hence the first equation is trivially satisfied. The sixth
equation implies $r_j=r_j^{(N-1)}(\lambda-1)$ with an undetermined coefficient
$r_{N-1}^{(N-1)}(\lambda-1)$ (as in Step 2). The second equation yields the
relation
$$
   q_j^{(N-1)}(\lambda\!-\!1) = -\frac{(2\lambda\!+\!n\!-\!4N\!+\!2j\!+\!1)}{(2N\!-\!2j\!-\!1)}
   r^{(N-1)}_{j}(\lambda\!-\!1), \quad j=0,\dots, N-1.
$$
It follows that the coefficients are odd Gegenbauer coefficients: $q_j =
q_j^{(N-1)}(\lambda\!-\!1)$. The fifth equation is satisfied by the arguments
in Step 5.

For $N=0$, the only solution is $P=p_0$.

This completes the analysis of the system \eqref{eq:EvenSystemPtoP}.

Finally, we choose the normalization $p_N^{(N)}(\lambda) =
(\lambda\!+\!p\!-\!2N)$ and summarize the above results.

\begin{theorem}\label{EvenFromPtoP} Assume that $N \in \N_0$. The first type
homomorphisms
\begin{equation*}
   \xi_n^{2N} P(t) \otimes \id + \xi_n^{2N-1} Q(t) E_n \wedge i_E + \xi_n^{2N-2} R(t) \alpha \wedge i_E
\end{equation*}
with the polynomial coefficients
$$
   P(t) = \sum_{j=0}^N p^{(N)}_j(\lambda;p) t^j, \quad
   Q(t) = \sum_{j=0}^{N-1} q^{(N-1)}_j(\lambda\!-\!1) t^j \quad \mbox{and} \quad
   R(t) = \sum_{j=0}^{N-1} r^{(N-1)}_j(\lambda\!-\!1) t^j
$$
in the variable $t = \abs{\xi'}^2/\xi_n^2$ are singular vectors of even
homogeneity $2N$ in the space
$$
   \Hom_{\gop^\prime}(\Lambda^{p}(\gon_-^\prime(\R)) \otimes \C_{\lambda-2N},
   \Pol_{2N}(\gon_-^*(\R)) \otimes \Lambda^p(\gon_-(\R)) \otimes \C_\lambda)
$$
iff
\begin{align*}
   p_j^{(N)}(\lambda;p) & = (\lambda\!+\!p\!-\!2N) a_j^{(N)}(\lambda), \\
   q_j^{(N-1)}(\lambda) & = -2N(2\lambda\!+\!n\!-\!2N\!+\!1) b_j^{(N-1)}(\lambda), \\
   r_j^{(N-1)}(\lambda) & = 2N a_j^{(N-1)}(\lambda),
\end{align*}
up to a constant multiple. These singular vectors are denoted by $v^{(p\to
p)}_{2N}(\lambda)$. They describe the embedding of the module
$$
   {\fam2 M}^{\gog^\prime}_{\gop^\prime} \left(\Lambda^{p}(\R^{n-1}) \otimes \C_{\lambda-2N}\right)
$$
in Proposition \ref{charind}.
\end{theorem}

We finish this section with explicit examples of singular vectors of the first
type.

\begin{example}\label{ExamplesPtoP} We display the low-homogeneity singular vectors
$v_k^{(p \to p)}(\lambda)$ for $k \le 3$. The singular vectors of homogeneity
$0$ is given by
\begin{equation*}
   v_0^{(p\to p)} = (\lambda\!+\!p) \otimes \id.
\end{equation*}
The first type singular vector of homogeneity $1$ is given by
\begin{equation*}
   v_1^{(p\to p)}(\lambda)  = (\lambda\!+\!p\!-\!1) \xi_n \otimes + E_n \wedge i_E.
\end{equation*}
The first type singular vector of homogeneity $2$ is given by
\begin{align*}
   v_2^{(p\to p)}(\lambda) & = (\lambda\!+\!p\!-\!2) \sum_{i=1}^{n-1}\xi_i^2 \otimes
   \id -(2\lambda\!+\!n\!-\!3)(\lambda\!+\!p\!-\!2) \xi_n^2 \otimes \id \\
   & -2(2\lambda\!+\!n\!-\!3)\xi_n E_n \wedge i_E + 2 \alpha \wedge i_E.
\end{align*}
Finally, the first type singular vector of homogeneity $3$ is given by the sum
\begin{align*}
   v_3^{(p\to p)}(\lambda) & = (\lambda\!+\!p\!-\!3)\xi_n\sum_{i=1}^{n-1}\xi_i^2 \otimes
   \id -\tfrac 13 (2\lambda\!+\!n\!-\!5)(\lambda\!+\!p\!-\!3) \xi_n^3 \otimes \id \\
   & + \sum_{i=1}^{n-1}\xi_i^2 E_n \wedge i_E - (2\lambda\!+\!n\!-\!5) \xi_n^2  E_n \wedge
   i_E + 2\xi_n \alpha \wedge i_E.
\end{align*}
\end{example}

\subsection{Families of singular vectors of the second type}\label{sv-type2}

Let $N \in \N$. We consider singular vectors of the form
$$
   v^{(p-1 \to p)}_N(\lambda) \in \Hom_{\gop^\prime}(\Lambda^{p-1}(\gon_-^\prime(\R)) \otimes \C_{\lambda-N},
   \Pol_N (\gon_-^*(\R)) \otimes \Lambda^p(\gon_-(\R)) \otimes \C_\lambda),
$$
which are linear combinations of homomorphisms listed in \eqref{type-2-N-gen}.
In the following, we shall refer to these as singular vectors of the second
type. Such vectors are related to singular vectors $v_N^{(p \to p)}(\lambda)$
of the first type through conjugation with the Hodge star operators
$$
   \bar{\star}: \Pol_{N}(\gon_-^*(\R)) \otimes \Lambda^p(\gon_-(\R)) \to
   \Pol_{N}(\gon_-^*(\R)) \otimes \Lambda^{n-p}(\gon_-(\R)))
$$
and
$$
   \star: \Lambda^{p}(\gon_-^\prime(\R)) \to \Lambda^{n-1-p}(\gon_-^\prime(\R)),
$$
i.e.,
$$
   v_N^{(p-1 \to p)}(\lambda) = \bar{\star} \, v_N^{(n-p \to n-p)}(\lambda) \star.
$$

The following two results explicate the structure of singular vectors of the
second type. We first describe singular vectors of odd homogeneity.

\begin{theorem}\label{OddFromPtoP-1} Assume that $N \in \N_0$. The second type
homomorphism
\begin{equation}
   \xi_n^{2N+1} P(t) \otimes E_n + \xi_n^{2N} Q(t) \alpha + \xi_n^{2N-1} R(t) E_n \wedge \alpha \wedge i_E
\end{equation}
with the polynomial coefficients
$$
   P(t) = \sum_{j=0}^N p_j(\lambda;N,p) t^j, \;\;
   Q(t) = \sum_{j=0}^N q^{(N)}_j(\lambda\!-\!1) t^j  \;\; \mbox{and}  \;\;
   R(t) = \sum_{j=0}^{N-1} r^{(N-1)}_j(\lambda\!-\!1) t^j
$$
in the variable $t = \abs{\xi'}^2/\xi_n^2$ are singular vectors of odd
homogeneity $2N+1$ in the space
$$
   \Hom_{\gop^\prime}(\Lambda^{p-1}(\gon_-^\prime(\R)) \otimes \C_{\lambda-(2N+1)},
   \Pol_{2N+1}(\gon_-^*(\R)) \otimes \Lambda^p(\gon_-(\R)) \otimes \C_\lambda)
$$
iff
\begin{align*}
   p_j(\lambda;N,p) & = -(\lambda\!+\!n\!-\!p\!-\!2N\!+\!2j\!-\!1) b_j^{(N)}(\lambda), \\
   q_j^{(N)}(\lambda) & = a_j^{(N)}(\lambda), \\
   r_j^{(N-1)}(\lambda) & = 2N b_j^{(N)}(\lambda).
\end{align*}
These singular vectors are denoted by of $v^{(p-1 \to p)}_{2N+1}(\lambda)$.
They describes the embedding of the submodule
$$
   {\fam2 M}^{\gog^\prime}_{\gop^\prime} \left(\Lambda^{p-1}(\R^{n-1}) \otimes \C_{\lambda-(2N+1)}\right)
$$
in Proposition \ref{charind}.
\end{theorem}

For singular vectors of even homogeneity we have the following analogous
result.

\begin{theorem}\label{EvenFromPtoP-1} Assume that $N \in \N_0$. The second type
homomorphism
\begin{equation}
   \xi_n^{2N} P(t) \otimes E_n + \xi_n^{2N-1} Q(t) \alpha + \xi_n^{2N-2} R(t) E_n \wedge \alpha \wedge i_E
\end{equation}
with the polynomial coefficients
$$
   P(t) = \sum_{j=0}^N p_j(\lambda;N,p) t^j, \;\;
   Q(t) = \sum_{j=0}^{N-1} q^{(N-1)}_j(\lambda\!-\!1)t^j \;\; \mbox{and} \;\;
   R(t) = \sum_{j=0}^{N-1} r^{(N-1)}_j(\lambda\!-\!1)t^j
$$
in the variable $t = \abs{\xi'}^2/\xi_n^2$ are singular vectors of even
homogeneity $2N$ in the space
$$
   \Hom_{\gop^\prime}(\Lambda^{p-1}(\gon_-^\prime(\R)) \otimes \C_{\lambda-2N},
   \Pol_{2N}(\gon_-^*(\R)) \otimes \Lambda^p(\gon_-(\R)) \otimes \C_\lambda)
$$
iff
\begin{align*}
   p_j(\lambda;N,p) & = -(\lambda\!+\!n\!-\!p\!-\!2N\!+\!2j) a_j^{(N)}(\lambda;p), \\
   q_j^{(N-1)}(\lambda) & = -2N(2\lambda\!+\!n\!-\!2N\!+\!1) b_j^{(N-1)}(\lambda), \\
   r_j^{(N-1)}(\lambda) & = 2N a_j^{(N-1)}(\lambda).
\end{align*}
These singular vectors are denoted by $v^{(p-1 \to p)}_{2N}(\lambda)$. They
describe the embedding of the submodule
$$
   {\fam2 M}^{\gog^\prime}_{\gop^\prime} \left(\Lambda^{p-1}(\R^{n-1})\otimes\C_{\lambda-2N}\right)
$$
in Proposition \ref{charind}.
\end{theorem}

There are direct proofs of the latter two results using similar arguments as in
Section \ref{sv-type1} (and resting on the formulas in \eqref{type-2-N}). We
omit these details. Both theorems also follow from the results on families of
homomorphisms of the first type by conjugation with the Hodge star operators on
$\R^n$ and $\R^{n-1}$. However, we shall not apply this method either. Instead,
we shall indirectly verify both results by using the following method. In
Section \ref{DiffOp}, we shall describe the conformal symmetry breaking
operators of the first type which are induced by the singular vectors of the
first type. By conjugation of these operators with Hodge star operators, we
obtain conformal symmetry breaking operators of the second type. These
operators correspond to the above singular vectors of the second type, of
course.

We finish this section with explicit examples of singular vectors of the second
type.

\begin{example}\label{ExamplesPtoP-1} We display the low-homogeneity singular vectors
$v_k^{(p-1 \to p)}(\lambda)$ for $k \le 3$. The second type singular vector of
homogeneity $0$ is given by
$$
   v_0^{(p-1 \to p)}(\lambda) = -(\lambda\!+\!n\!-\!p) \otimes E_n.
$$
The second type singular vector of homogeneity $1$ reads
\begin{equation*}
   v_1^{(p-1\to p)}(\lambda) = -(\lambda\!+\!n\!-\!p\!-\!1) \xi_n \otimes E_n + \alpha.
\end{equation*}
The second type singular vector of homogeneity $2$ is given by
\begin{align*}
   v_2^{(p-1\to p)}(\lambda) & = -(\lambda\!+\!n\!-\!p)\sum_{i=1}^{n-1}\xi_i^2 \otimes E_n
   + (2\lambda\!+\!n\!-\!3)(\lambda\!+\!n\!-\!p\!-\!2)\xi_n^2 \otimes E_n \\
   & -2(2\lambda\!+\!n\!-\!3) \xi_n \alpha + 2 E_n \wedge \alpha \wedge i_E.
\end{align*}
Finally, the second type singular vector of homogeneity $3$ is given by the sum
\begin{align*}
   v_3^{(p-1\to p)}(\lambda) & = - (\lambda\!+\!n\!-\!p\!-\!1) \xi_n\sum_{i=1}^{n-1}\xi_i^2 \otimes E_n
   + \tfrac{1}{3} (2\lambda\!+\!n\!-\!5)(\lambda\!+\!n\!-\!p\!-\!3) \xi_n^3
   \otimes E_n \\
   & + \sum_{i=1}^{n-1}\xi_i^2 \alpha - (2\lambda\!+\!n\!-\!5) \xi_n^2 \alpha +2 \xi_n E_n \wedge \alpha \wedge i_E.
\end{align*}
\end{example}

\subsection{Singular vectors of the third type}\label{sv-type-3}

Let $N \in \N$ and $1 \leq p \leq n-1$. The singular vectors of the third type
in
$$
   \Hom_{\gop^\prime}(\Lambda^p(\gon_-^\prime(\R)) \otimes \C_{\lambda-N},
   \Pol_N(\gon_-^*(\R)) \otimes \Lambda^{p-1}(\gon_-(\R)) \otimes \C_\lambda)
$$
have the form
\begin{equation}\label{eq:HOM1}
   v^{(p\to p-1)}_N = \xi_n^{N-1} P(t) i_E,
\end{equation}
where
\begin{align*}
   P(t) = \sum_{j=0}^{\left[\frac{N-1}{2}\right]} p_j t^j \quad \mbox{and} \quad
   t = \frac{|\xi'|^2}{\xi_n^2}
\end{align*}
(see \eqref{type-3-N-gen}), i.e.,
\begin{equation*}
   v^{(p\to p-1)}_N (E_{\alpha_1} \wedge \cdots \wedge E_{\alpha_p})
   = \xi_n^{N-1} P(t) \xi_{[\alpha_1} \otimes E_{\alpha_2} \wedge \cdots \wedge
   E_{\alpha_{p}]}.
\end{equation*}

Now assume that the homogeneity is even. In order to compute the actions of the
operators $P_j(\lambda)$ (defined in \eqref{P-operator}) on $v_{2N}^{(p\to
p-1)}$, we calculate
\begin{multline*}
   \tfrac{1}{2} \xi_j \Delta_\xi (v_{2N}^{(p\to p-1)}(E_{\alpha_1} \wedge \cdots \wedge E_{\alpha_p})) \\
   = \left[2t(t\!+\!1)P'' + (-4N\!+\!5) tP^\prime + (n\!+\!1) P^\prime + (N\!-\!1)(2N\!-\!1) P\right] \\
   \times \xi_j \xi_n^{2N-3} \xi_{[\alpha_1}\otimes E_{\alpha_2}\wedge\cdots\wedge E_{\alpha_p]},
\end{multline*}
\begin{multline*}
   (\lambda-E_\xi)\partial_j(v_{2N}^{(p\to  p-1)}(E_{\alpha_1} \wedge \cdots \wedge E_{\alpha_p})) \\
   = (\lambda\!-\!2N\!+\!1) \left[\xi_n^{2N-1} P\delta_{j[\alpha_1}\otimes E_{\alpha_2}
   \wedge \cdots \wedge E_{\alpha_p]}
   + 2\xi_j \xi_n^{2N-3} P^{\prime} \xi_{[\alpha_1} \otimes E_{\alpha_2} \wedge \cdots \wedge E_{\alpha_p]}\right]
\end{multline*}
and
\begin{align*}
    \sum_{k=1}^{n} \partial_k \otimes & (E^-_k\otimes E^+_j- E_j^- \otimes E_k^+)
    (v_{2N}^{(p\to p-1)}(E_{\alpha_1} \wedge \cdots \wedge E_{\alpha_p})) \\
    & = -(p\!-\!1)\xi_n^{2N-1}P  \delta_{j[\alpha_1}\otimes E_{\alpha_2}\wedge\cdots\wedge E_{\alpha_p]}\\
    & + \left[(2N\!-\!1)P - 2t P^\prime\right] \xi_n^{2N-2}\xi_{[\alpha_1}\delta_{j[\alpha_2} \otimes
    E_n \wedge E_{\alpha_3}\wedge\cdots\wedge E_{\alpha_p]]}\\
    & + 2 \xi_n^{2N-3} P^\prime \sum_{k=1}^{n-1} \xi_k \xi_{[\alpha_1}\delta_{j[\alpha_2}
    \otimes E_k \wedge E_{\alpha_3} \wedge \cdots\wedge E_{\alpha_p]]}.
\end{align*}
Hence we conclude
\begin{align*}
    & P_j(\lambda)(v_{2N}^{(p\to  p-1)}(E_{\alpha_1} \wedge \cdots \wedge E_{\alpha_p})) \\
    & = \big [2t(t\!+\!1)P^{\prime\prime} + (-4N\!+\!5) t P^\prime + (2\lambda\!+\!n\!-\!4N\!+\!3)P^\prime
    + (N\!-\!1)(2N\!-\!1)P\big] \\
    & \times \xi_j\xi_n^{2N-3}\xi_{[\alpha_1}\otimes E_{\alpha_2}\wedge\cdots\wedge E_{\alpha_p]} \\
    & + (\lambda\!+\!p\!-\!2N)P \xi_n^{2N-1}\delta_{j[\alpha_1}\otimes E_{\alpha_2}\wedge\cdots\wedge E_{\alpha_p]}\\
    & - \big[(2N\!-\!1)P - 2t P^\prime\big] \xi_n^{2N-2}\xi_{[\alpha_1}\delta_{j[\alpha_2}
       \otimes E_n\wedge E_{\alpha_3}\wedge\cdots\wedge E_{\alpha_p]]}\\
    & - 2P^\prime \xi_n^{2N-3} \sum_{k=1}^{n-1}\xi_k \xi_{[\alpha_1}\delta_{j[\alpha_2}
       \otimes E_k\wedge E_{\alpha_3}\wedge\cdots\wedge E_{\alpha_p]]}.
\end{align*}
A similar result holds for $P_j(\lambda)(v_{2N-1}^{(p\to  p-1)})$ (shifting $N
\mapsto N-\frac 12$).

\begin{theorem}\label{sv-third} Let $v^{(p\to p-1)}_{N}$ be given by \eqref{eq:HOM1}. \\
(i) For $p=1$, we find the non-trivial singular vector
\begin{equation}\label{sv-type-3-even}
   v_{2N}^{(1\to 0)} = \xi_n^{2N-1} \left(\sum_{j=0}^{N-1} b_j^{(N-1)}(2N\!-\!1) t^j \right)
   i_E,\; N \in \N.
\end{equation}
Here $b_j^{(N)}(\lambda)$ are odd Gegenbauer coefficients. \\
(ii) For $p=1$, we find the non-trivial singular vector
\begin{equation}\label{sv-type-3-odd}
   v_{2N+1}^{(1\to 0)} = \xi_n^{2N} \left(\sum_{j=0}^N a_j^{(N)}(2N) t^j\right)
   i_E, \; N \in \N_0
\end{equation}
Here $a_j^{(N)}(\lambda)$ are even Gegenbauer coefficients. \\
(iii) For $p=1,\dots,n-1$, we find the singular vector
\begin{equation}\label{sv-type-3-1}
   v_{1}^{(p\to p-1)} = i_E.
\end{equation}
These are all singular vectors of the third type, up to constant multiples.
\end{theorem}

\begin{proof} We first prove (i). For $p=1$, the condition $P_j(\lambda)(v_{2N}^{(1\to  0)})=0$
for $j=1,\ldots,n-1$ is equivalent to the system of ordinary differential
equation
\begin{align*}
      2t(t\!+\!1)P^{\prime\prime} + (-4N\!+\!5) t P^\prime + (2\lambda\!+\!n\!-\!4N\!+\!3) P^\prime
      +(N\!-\!1)(2N\!-\!1)P & = 0,\\
      (\lambda\!-\!2N\!+\!1)P & = 0
\end{align*}
for the polynomial $P(t)$. The first equation is satisfied by the polynomial
\begin{align*}
      P(t) = \sum_{j=0}^{N-1} b_j^{(N-1)}(\lambda) t^j,
\end{align*}
where $b_j^{(N)}(\lambda)$ are odd Gegenbauer coefficients. The second equation
yields $\lambda=2N-1$.

We continue with the analogous proof of (ii). For $p=1$, the condition
$P_j(\lambda)(v_{2N+1}^{(1\to 0)})=0$ for $j=1,\ldots,n-1$ is equivalent to the
system of ordinary differential equation
\begin{align*}
   2t(t\!+\!1)P^{\prime\prime} + (-4N\!+\!3) tP^\prime + (2\lambda\!+\!n\!-\!4N\!+\!1)P^\prime + N(2N\!-\!1)P & = 0,\\
   (\lambda\!-\!2N)P & = 0
    \end{align*}
for the polynomial $P(t)$. Again, the first equation is satisfied by the
polynomial
\begin{align*}
      P(t)=\sum_{j=0}^{N}a_j^{(N)}(\lambda) t^j,
\end{align*}
where $a_j^{(N)}(\lambda)$ are even Gegenbauer coefficients. The second
equation implies $\lambda=2N$.

Next, we prove (iii). Let $p=1,\ldots,n-1$. By obvious reasons we find that the
condition $P_j(\lambda)(v_{1}^{(p\to p-1)})=0$ for $j=1,\ldots,n-1$ holds iff
$\lambda=-(p-1)$.

Finally, in all other cases there do not exist non-trivial polynomial solutions
of the equations $P_j(\lambda)(v_N^{(p\to p-1)})=0$, $j=1,\ldots,n-1$, for any
$\lambda$.
\end{proof}

\begin{bem}\label{sv3-der} The singular vectors in Theorem \ref{sv-third} can be described in
terms of the families of the first type. In fact, it holds
\begin{equation}\label{sv-van-3}
   v_{N-1}^{(0 \to 0)}(N\!-\!1) i_E = 0
\end{equation}
and we have the relation
\begin{equation}
    v_N^{(1 \to 0)} = \dot{v}_{N-1}^{(0 \to 0)}(N\!-\!1) i_E
\end{equation}
for all $N \in \N$. Here the singular vectors $v_{N-1}^{(0 \to 0)}(\lambda)$
are defined in Theorem \ref{OddFromPtoP} and Theorem \ref{EvenFromPtoP}.
\end{bem}

\begin{proof} For even homogeneity, the assertions follow from the formulas
$$
   v_{2N}^{(1 \to 0)} = \xi_n^{2N-1} \left(\sum_{j=0}^{N-1} b_j^{(N-1)}(2N\!-\!1) t^j \right) i_E
$$
(see \eqref{sv-type-3-even}) and
$$
   v_{2N-1}^{(0 \to 0)}(\lambda) = (\lambda\!-\!2N\!+\!1) \xi_n^{2N-1}
   \left(\sum_{j=0}^{N-1} b_j^{(N-1)}(\lambda) t^j \right)
$$
(by Theorem \ref{OddFromPtoP}). Similarly, for odd homogeneity, we use the
formulas
$$
   v_{2N+1}^{(1 \to 0)} = \xi_n^{2N} \left(\sum_{j=0}^N a_j^{(N)}(2N) t^j\right) i_E
$$
(see \eqref{sv-type-3-odd}) and
$$
   v_{2N}^{(0 \to 0)}(\lambda) = (\lambda\!-\!2N) \xi_n^{2N}
   \left(\sum_{j=0}^{N} a_j^{(N)}(\lambda) t^j \right)
$$
(by Theorem \ref{EvenFromPtoP}).
\end{proof}

Note that \eqref{sv-van-3} is a special case of the vanishing formula
\begin{equation}\label{sv-van-1-gen}
   v_N^{(p \to p)}(N\!-\!p) i_E  = 0.
\end{equation}

\subsection{Singular vectors of the fourth type}\label{sv-type-4}

Let and $N \in \N$ and $0 \leq p \leq n-2$. The singular vectors of the fourth
type in
$$
   \Hom_{\gop^\prime}(\Lambda^p(\gon_-^\prime(\R)) \otimes \C_{\lambda-N},
   \Pol_N(\gon_-^*(\R)) \otimes \Lambda^{p+2}(\gon_-(\R)) \otimes \C_\lambda)
$$
have the form
\begin{equation}\label{eq:HOM2}
   v^{(p\to p+2)}_N = \xi_n^{N-1} P(t) E_n \wedge \alpha
\end{equation}
(see \eqref{type-4-N-gen}) with
\begin{equation*}
   P(t)=\sum_{j=0}^{[\frac{N-1}{2}]}p_j t^j \quad \mbox{and} \quad t=
   \frac{\abs{\xi^\prime}}{\xi_n^2},
\end{equation*}
i.e.,
\begin{equation*}
   v^{(p\to p+2)}_N (E_{\alpha_1}\wedge\cdots\wedge E_{\alpha_{p}})
   = \xi_n^{N-1} P(t) \sum_{k=1}^{n-1}\xi_{k}
   \otimes E_n \wedge E_k \wedge E_{\alpha_1}\wedge\cdots\wedge E_{\alpha_{p}}
\end{equation*}

Now assume that the homogeneity is even. In order to compute the action of the
operators $P_j(\lambda)$ on $v_{2N}^{(p\to  p+2)}$, we calculate
\begin{multline*}
   \tfrac{1}{2} \xi_j \Delta_\xi(v_{2N}^{(p\to p+2)}(E_{\alpha_1}\wedge \cdots \wedge E_{\alpha_p})) \\
   = \left[2t(t\!+\!1)P^{\prime\prime} + (-4N\!+\!5)tP^\prime + (n\!+\!1)P^\prime + (N\!-\!1)(2N\!-\!1)P \right] \\
   \times \xi_j \xi_n^{2N-3}\sum_{k=1}^{n-1}\xi_{k}
   \otimes E_n \wedge E_k \wedge E_{\alpha_1}\wedge\cdots\wedge E_{\alpha_p},
\end{multline*}
\begin{multline*}
   (\lambda-E_\xi)\partial_j(v_{2N}^{(p\to p+2)}(E_{\alpha_1}\wedge \cdots \wedge E_{\alpha_p})) \\
   = (\lambda\!-\!2N\!+\!1) \Big[2\xi_j \xi_n^{2N-3} P^{\prime}
   \sum_{k=1}^{n-1} \xi_{k} \otimes E_n \wedge E_k \wedge E_{\alpha_1} \wedge \cdots \wedge E_{\alpha_p} \\
   +\xi_n^{2N-1}P \otimes E_n \wedge E_j \wedge E_{\alpha_1}\wedge\cdots\wedge E_{\alpha_p}\Big]
\end{multline*}
and
\begin{align*}
   \sum_{k=1}^n \partial_k \otimes & (E^-_k\otimes E^+_j-E_j^-\otimes E_k^+)(v_{2N}^{(p\to p+2)}
   (E_{\alpha_1}\wedge \cdots \wedge E_{\alpha_p}))\\
   & =- \left[2tP^\prime+(n\!-\!p\!-\!2)P\right] \xi_n^{2N-1} \otimes E_n \wedge E_j
   \wedge E_{i_1} \wedge \cdots \wedge E_{i_p}\\
   & - \left[(2N\!-\!1)P-2tP^\prime\right]\xi_n^{2N-2} \sum_{k=1}^{n-1} \xi_{k}
   \otimes E_j \wedge E_k \wedge E_{\alpha_1} \wedge\cdots\wedge E_{\alpha_p}\\
   & + 2\xi_j \xi_n^{2N-3}P^\prime\sum_{k=1}^{n-1} \xi_{k}
   \otimes E_n \wedge E_k \wedge E_{\alpha_1}\wedge\cdots\wedge E_{\alpha_p}\\
   & - 2\xi_n^{2N-3} P^\prime \sum_{k=1}^{n-1} \xi_{k} \xi_{[\alpha_1}
   \otimes E_n \wedge E_k \wedge E_{\alpha_2}\wedge\cdots\wedge E_{\alpha_p]}
\end{align*}
Hence we conclude
\begin{align*}
   & P_j(\lambda)(v_{2N}^{(p\to  p+2)}(E_{\alpha_1} \wedge \cdots \wedge E_{\alpha_p})) \\
   & = \big[2t(t\!+\!1) P^{\prime\prime} + (-4N\!+\!5)tP^\prime
   + (2\lambda\!+\!n\!-\!4N\!+\!3)P^\prime + (N\!-\!1)(2N\!-\!1)P\big] \\
   & \times \xi_j \xi_n^{2N-3}\sum_{k=1}^{n-1}\xi_{k}
       \otimes E_n\wedge E_k\wedge E_{\alpha_1}\wedge\cdots\wedge E_{\alpha_{p}}\\
   & - 2 P^\prime\xi_j \xi_n^{2N-3}\sum_{k=1}^{n-1}\xi_{k}
       \otimes E_n\wedge E_k\wedge E_{\alpha_1}\wedge\cdots\wedge E_{\alpha_{p}}\\
   & +\big[(\lambda\!+\!n\!-\!p\!-\!2N\!-\!1)P + 2t P^\prime\big]\xi_n^{2N-1}
       \otimes E_n\wedge E_j\wedge E_{\alpha_1}\wedge\cdots\wedge E_{\alpha_p}\\
   & +\big[(2N\!-\!1)P - 2tP^\prime\big] \xi_n^{2N-2}\sum_{k=1}^{n-1}\xi_{k}
       \otimes E_j\wedge E_k\wedge E_{\alpha_1}\wedge\cdots\wedge E_{\alpha_{p}}\\
   & + 2\xi_n^{2N-3} P^\prime\sum_{k=1}^{n-1}\xi_{k}\xi_{[\alpha_1}
       \otimes E_n\wedge E_k\wedge E_j\wedge  E_{\alpha_2}\wedge\cdots\wedge E_{\alpha_{p}]}.
\end{align*}
By reasons which will become clear in the proof of the following result, we do
not join here the first and the second sum on the right-hand side. A similar
computation (shifting $\N\ni N \to N-\frac 12$) gives a formula for the action
of $P_j(\lambda)$ on $v^{(p\to p+2)}_{2N-1}$.

\begin{theorem}\label{sv-fourth} Let $v^{(p\to p+2)}_N$ be given by \eqref{eq:HOM2}. \\
(i) For $p=n-2$, we find the non-trivial singular vector
\begin{equation}\label{sv-4-even}
   v^{(n-2\to n)}_{2N} = \xi_n^{2N-1} \left( \sum_{j=0}^{N-1} b_j^{(N-1)}(2N\!-\!1) t^j \right) E_n \wedge
   \alpha, \; N \in \N.
\end{equation}
Here $b_j^{(N)}(\lambda)$ are odd Gegenbauer coefficients. \\
(ii) For $p=n-2$, we find the non-trivial singular vector
\begin{equation}\label{sv-4-odd}
   v^{(n-2\to n)}_{2N+1} = \xi_n^{2N-1} \left( \sum_{j=0}^N a_j^{(N)}(2N) t^j \right) E_n \wedge
   \alpha, \; N \in \N_0.
\end{equation}
Here $a_j^{(N)}(\lambda)$ are even Gegenbauer coefficients. \\
(iii) For $p=0,\ldots,n-2$, we find the singular vector
\begin{equation}\label{sv-4-1}
   v_1^{(p\to p+2)} = E_n \wedge \alpha.
\end{equation}
These are all singular vectors of the fourth type, up to constant multiples.
\end{theorem}

\begin{proof} We first prove (i). The identity
\begin{align*}
   & \sum_{k=1}^{n-1} \xi_{k}\xi_{[\alpha_1}
   \otimes E_n\wedge E_k \wedge E_j\wedge  E_{\alpha_2}\wedge\cdots\wedge E_{\alpha_{p}]}
   -\xi_j \sum_{k=1}^{n-1}\xi_{k} \otimes E_n\wedge E_k\wedge E_{\alpha_1}\wedge\cdots\wedge E_{\alpha_{p}}\\
   &=-\sum_{k=1}^{n-1}\xi_k^2\otimes E_n\wedge E_j\wedge E_{\alpha_1}\wedge\cdots\wedge E_{\alpha_p},
\end{align*}
which only holds for the form-degree $p=n-2$, leads to a cancelation in the
above formula for the action of $P_j(\lambda)$ on $v_{2N}^{(n-2\to n)}$. It
follows that the condition $P_j(\lambda)(v_{2N}^{(n-2\to  n)})=0$ for
$j=1,\ldots,n-1$ is equivalent to the system of ordinary differential equation
\begin{align*}
      2t(t\!+\!1)P^{\prime\prime} + (-4N\!+\!5)tP^\prime + (2\lambda\!+\!n\!-\!4N\!+\!3) P^\prime +
      (N\!-\!1)(2N\!-\!1)P & = 0,\\
      (\lambda\!-\!2N\!+\!1)P & = 0
\end{align*}
for the polynomial $P(t)$. The first equation is satisfied by the polynomial
\begin{align*}
      P(t)=\sum_{j=0}^{N-1}b_j^{(N-1)}(\lambda) t^j,
\end{align*}
where $b_j^{(N)}(\lambda)$ denote odd Gegenbauer coefficients. The second
equation yields $\lambda=2N-1$.

We continue with the proof of (ii). By similar arguments as in (1), we find
that for $p=n-2$ the condition $P_j(\lambda)(v_{2N+1}^{(n-2\to 2)})=0$,
$j=1,\ldots,n-1$, is equivalent to the system of ordinary differential equation
\begin{align*}
      2t(t\!+\!1)P^{\prime\prime} + (-4N\!+\!3)t P^\prime + (2\lambda\!+\!n\!-\!4N\!+\!1) P^\prime
      + N(2N\!-\!1)P & = 0,\\
      (\lambda\!-\!2N)P & = 0
\end{align*}
for the polynomial $P(t)$. The first equation is satisfied by the polynomial
\begin{align*}
      P(t)=\sum_{j=0}^{N}a_j^{(N)}(\lambda) t^j,
\end{align*}
where $a_j^{(N)}(\lambda)$ are even Gegenbauer coefficients. The second
equation fixes $\lambda=2N$.

Next, we prove (iii). Let $p=0,\ldots,n-2$. By obvious reasons we find that the
condition $P_j(\lambda)(v_{1}^{(p\to p+2)})=0$ for $j=1,\ldots,n-1$ holds iff
$\lambda=p-n+2$.

Finally, we note that in all other cases there do not exist non-trivial
polynomial solutions of the equations $P_j(\lambda)(v_N^{(p\to p+2)})=0$,
$j=1,\ldots,n-1$, for any $\lambda$.
\end{proof}

\begin{bem}\label{sv4-der} The singular vectors in Theorem \ref{sv-fourth} can be described in
terms of the families of the second type. In fact, for all $N \in \N$, it holds
\begin{equation}\label{sv-van-4}
   v_{N-1}^{(n-1 \to n)}(N\!-\!1) \wedge \alpha = 0
\end{equation}
and we have the relation
\begin{equation}
   v_N^{(n-2 \to n)} = -\dot{v}_{N-1}^{(n-1 \to n)}(N\!-\!1) \wedge \alpha
\end{equation}
Here the singular vectors $v_N^{(n-1 \to n)}(\lambda)$ are defined in Theorem
\ref{OddFromPtoP-1} and Theorem \ref{EvenFromPtoP-1}.
\end{bem}

\begin{proof} For even homogeneity, we have
$$
   v^{(n-2\to n)}_{2N} = \xi_n^{2N-1} \left( \sum_{j=0}^{N-1} b_j^{(N-1)}(2N\!-\!1) t^j \right)
   E_n \wedge \alpha
$$
(by \eqref{sv-4-even}) and
\begin{align*}
   v_{2N-1}^{(n-1 \to n)}(\lambda) \wedge \alpha & = \sum_{j=0}^{N-1} p_j(\lambda;N\!-\!1;n)
   |\xi'|^{2j} \xi_n^{2N-2j-1} \otimes E_n \wedge \alpha \\
   & + \sum_{j=1}^{N-1} r_{j-1}^{(N-2)}(\lambda\!-\!1) |\xi'|^{2j} \xi_n^{2N-2j-1} \otimes
   E_n \wedge \alpha
\end{align*}
(by Theorem \ref{OddFromPtoP-1} and $i_E \wedge \alpha = |\xi'|^2 \otimes
\id$). Now the relations
\begin{align*}
   p_j(\lambda;N\!-\!1;n) & = -(\lambda\!-\!2N\!+\!2j\!+\!1) b_j^{(N-1)}(\lambda), \\
   r_{j-1}^{(N-2)}(\lambda\!-\!1) & = (2N\!-\!2) b_{j-1}^{(N-2)}(\lambda\!-\!1) = 2j b_j^{(N-1)}(\lambda)
\end{align*}
yield
$$
   p_j(\lambda;N\!-\!1;n) +  r_{j-1}^{(N-2)}(\lambda\!-\!1) = -(\lambda\!-\!2N\!+\!1) b_j^{(N-1)}(\lambda).
$$
Hence
\begin{align*}
   v_{2N-1}^{(n-1 \to n)}(\lambda) \wedge \alpha & = -(\lambda\!-\!2N\!+\!1) \sum_{j=0}^{N-1} b_j^{(N-1)}(\lambda)
   |\xi'|^{2j} \xi_n^{2N-2j-1} \otimes E_n \wedge \alpha \\
   & = -(\lambda\!-\!2N\!+\!1) \xi_n^{2N-1} \left( \sum_{j=0}^{N-1} b_j^{(N-1)}(\lambda) t^j \right) E_n \wedge \alpha.
\end{align*}
This proves the claims. Analogous arguments apply for odd homogeneity using
Theorem \ref{EvenFromPtoP-1}. We omit the details.   
\end{proof}

Note that \eqref{sv-van-4} is a special case of the vanishing result
\begin{equation}\label{sv-van-2-general}
   v_N^{(p-1 \to p)}(N\!-\!n\!+\!p) \alpha = 0
\end{equation}

Finally, we note that the results of this section also follow from those in
Section \ref{sv-type-3} by a Hodge star conjugation argument. This is analogous
to the relation between the results in Section \ref{sv-type1} and Section
\ref{sv-type2}.

\subsection{Middle degree cases}\label{fam3}

Here we describe the singular vectors which describe the embeddings of the
submodules in Proposition \ref{charind2}.

{\bf Case 1a: Let $n$ be odd and $p=\frac{n-1}{2}$}.

The $SO(\gon_-^\prime(\R))$-module $\Lambda^{\frac{n-1}{2}}(\gon_-^\prime(\R))$
is not irreducible and the explicit formulas for the singular vectors
$$
v_N^{(\frac{n-1}{2} \to \frac{n-1}{2})}(\lambda)
$$
of the first type, which are defined in Theorem \ref{OddFromPtoP} (for odd $N$)
and Theorem \ref{EvenFromPtoP} (for even $N$), show that their restrictions
$$
   v_N^{(\frac{n-1}{2} \to \frac{n-1}{2}),\pm}(\lambda)
$$
to the subspaces
$$
   \Lambda^{\frac{n-1}{2}}_\pm(\gon_-^\prime(\R)) \subset \Lambda^{\frac{n-1}{2}}(\gon_-^\prime(\R))
$$
are non-trivial. Then the space
\begin{align*}
   \Hom_{\gop^\prime}(\Lambda^{\frac{n-1}{2}}_\pm(\gon_-^\prime(\R)) \otimes \C_{\lambda-N},
   \Pol_N(\gon_-^*(\R)) \otimes \Lambda^{\frac{n-1}{2}}(\gon_-(\R)) \otimes \C_\lambda)
\end{align*}
is generated by $v_N^{(\frac{n-1}{2} \to \frac{n-1}{2}),\pm}(\lambda)$. These
singular vectors describe the embeddings of the modules in the first two sums
in the decomposition \eqref{diag-branch-md-odd+}.

The singular vectors which describe the embeddings of the modules in the last
sum in the decomposition \eqref{diag-branch-md-odd+} are given by homomorphisms
of the second type.

{\bf Case 1b: Let $n$ be odd and $p=\frac{n+1}{2}$}.

The space
\begin{align*}
   \Hom_{\gop^\prime}(\Lambda^{\frac{n-1}{2}}_\pm(\gon_-^\prime(\R)) \otimes \C_{\lambda-N},
   \Pol_N(\gon_-^*(\R)) \otimes \Lambda^{\frac{n+1}{2}}(\gon_-(\R)) \otimes \C_\lambda)
\end{align*}
is generated by the homomorphism
\begin{equation}\label{sv-case2}
   \bar{\star} \, v_N^{(\frac{n-1}{2}\to \frac{n-1}{2}),\pm}(\lambda) \sim
   v_N^{(\frac{n-1}{2}\to \frac{n+1}{2}),\pm}(\lambda),
\end{equation}
where $\bar{\star}$ denotes the Hodge star operator on $\gon_-(\R) \simeq \R^n$
with the Euclidean metric. These singular vectors describe the embeddings of
the modules in the last two sums in the decomposition
\eqref{diag-branch-md-odd-}.

The singular vectors which describe the embeddings of the modules in the first
sum of the decomposition \eqref{diag-branch-md-odd-} are given by Theorem
\ref{OddFromPtoP} (for odd $N$) and Theorem \ref{EvenFromPtoP} (for even $N$).

{\bf Case 2: Let $n$ be even and $p=\frac{n}{2}$}.

The $SO(\gon_-(\R))$-module $\Lambda^{\frac{n}{2}}(\gon_-(\R))$ is not
irreducible and the spaces
\begin{equation*}
   \Hom_{\gop^\prime}(\Lambda^{\frac{n}{2}}(\gon_-^\prime(\R)) \otimes \C_{\lambda-N},
   \Pol_N(\gon_-^*(\R)) \otimes \Lambda_\pm^{\frac{n}{2}}(\gon_-(\R)) \otimes \C_\lambda)
\end{equation*}
are generated by the projections
\begin{equation}
   v_N^{(\frac{n}{2}\to \frac{n}{2}),\pm}(\lambda)
   \st pr_{\pm}\big( v_N^{(\frac{n}{2}\to \frac{n}{2})}(\lambda)\big)
\end{equation}
of the singular vectors
$$
   v_N^{(\frac{n}{2} \to \frac{n}{2})}(\lambda)
$$
which are defined in Theorem \ref{OddFromPtoP} (for odd $N$) and Theorem
\ref{EvenFromPtoP} (for even $N$) to the subspaces
$$
   \Lambda^{\frac{n}{2}}_\pm(\gon_-(\R)) \subset \Lambda^{\frac{n}{2}}(\gon_-(\R))
$$
Here $pr_\pm: \Lambda^{\frac{n}{2}}(\gon_-(\R))\to
\Lambda_\pm^{\frac{n}{2}}(\gon_-(\R))$ denotes the projections onto the
eigenspaces of the Hodge star operator. These singular vectors describe the
embeddings of the modules in the decomposition \eqref{diag-branch-md-even}.
\smallskip

\section{Conformal symmetry breaking operators on differential forms}\label{DiffOp}

In this section, we translate the four types of homomorphisms of generalized
Verma modules constructed in Section \ref{SingularVectors} into four types of
conformal symmetry breaking operators
\begin{equation*}
   D^{(p\to q)}_N: \Omega^p(\R^n) \to \Omega^q(\R^{n-1})
\end{equation*}
of order $N\in\N$ on differential forms. These operators will be referred to as
operators of the first, second, third and fourth type, respectively.

We first fix some conventions. As before, we consider $\R^{n-1}$ as a
codimension subspace of $\R^n$ by
\begin{equation*}
   \iota: \R^{n-1} \hookrightarrow \R^n,\quad (x_1,\dots,x_{n-1})\mapsto (x_1,\dots ,x_{n-1},0).
\end{equation*}
Then $\dm$, $\delta$ and $\bar{\dm}$, $\bar{\delta}$ denote the respective
differentials and co-differentials on differential forms on $\R^{n-1}$ and
$\R^n$, respectively. Let $\{e_i\}_{i=1}^{n-1}$ denote the standard orthonormal
basis on the Euclidean space $\R^{n-1}$ and let $\partial_i$ denote the partial
derivative in the $i^{th}$ coordinate. We have
\begin{equation*}
   \dm \omega(X_0,\dots,X_{p}) =
   \sum_{i=0}^p(-1)^i X_i\left(\omega(X_0,\dots,\widehat{X}_i,\dots,X_p)\right)
\end{equation*}
and
\begin{equation*}
   \delta \omega(X_1,\dots,X_{p-1}) = -\sum_{i=1}^{n-1} \partial_i\omega(e_i,X_1,\dots,X_{p-1})
\end{equation*}
for $\omega\in\Omega^p(\R^{n-1})$ and smooth vector fields
$X_j\in\X(\R^{n-1})$.

The Laplacian on forms on $\R^{n-1}$ is defined by $\Delta = \delta d +
d\delta$. Similarly, we set
$$
   \bar{\Delta} = \bar{\delta} \bar{\dm} + \bar{\dm}\bar{\delta}: \Omega^p(\R^n) \to \Omega^p(\R^n).
$$
Note that $\Delta = -\sum_{k=1}^{n-1}\partial_k^2$ and  $\bar{\Delta} =
-\sum_{k=1}^{n}\partial_k^2$.

We shall consider $\dm$ and $\delta$ also as operators on forms on $\R^n$. In
particular, $\Delta$ will also be viewed as an operator acting on
$\Omega^*(\R^n)$.

The insertion operator, given by contracting the first form index by the vector
field $X\in \X(\R^n)$, is denoted by $i_X:\Omega^p(\R^n) \to
\Omega^{p-1}(\R^n)$.

The following lemma collects identities which will be used later on.

\begin{lem}\label{DiffCoDiff} The following relations hold true as identities on forms on $\R^n$.
\begin{itemize}
   \item [(1)] $\iota^* \bar{\dm} = \dm\iota^*$ and $i_{\partial_n} \bar{\delta} = - \delta i_{\partial_n}$,
   \item [(2)] $i_{\partial_n}\partial_n = \delta - \bar{\delta}$ and
   $\iota^* \partial_n = \iota^* i_{\partial_n} \bar{\dm} + \dm \iota^* i_{\partial_n}$,
   \item [(3)] $\partial_n^2= \Delta - \bar{\Delta}$.
\end{itemize}
Moreover, $\partial_n$ commutes with $\dm$, $\delta$ and $i_{\partial_n}$.
\end{lem}

We emphasize that the second relation in Lemma \ref{DiffCoDiff}/(1), the first
relation in (2) and the relation (3) hold true as identities on $\R^n$, i.e.,
off the subspace $\R^{n-1}$.

\begin{proof} (1) is trivial. The first identity in (2) is obvious.
Cartan's formula $\Lie_X = i_X \dm+ \dm i_X$ applied to the normal vector field
$X = \partial_n$ yields $\partial_n = i_{\partial_n} \bar{\dm} + \bar{\dm}
i_{\partial_n}$. Hence $\iota^* \partial_n = \iota^* i_{\partial_n} \bar{\dm} +
\iota^* \bar{\dm} i_{\partial_n}$. This implies the second relation in (2)
using the first relation in (1). (3) is obvious by the respective formulas for
the Laplace operators.
\end{proof}

Now, we use the dual pairing of generalized Verma modules and induced
representations \cite{Kobayashi-Pevzner}, \cite{koss} to translate the
homomorphisms constructed in Section \ref{SingularVectors} into differential
operators:
$$
   \Omega^p(\gon_-(\R)) \to \Omega^q(\gon_-^\prime(\R)).
$$
By combining these operators with the identifications $\gon_-(\R) \simeq \R^n$
and $\gon_-^\prime(\R) \simeq \R^{n-1}$ given by the basis vectors $E_j$, we
obtain conformal symmetry breaking operators $\Omega^p(\R^n) \to
\Omega^q(\R^{n-1})$.

The following lemma collects important information on the translation of basic
operations on polynomial differential forms into operators on differential
forms. We recall the notation $\alpha = \sum_{j=1}^{n-1} \xi_j \otimes E_j$ and
$i_E = \sum_{j=1}^{n-1} \xi_j \otimes i_{E_j^*}$ (Section \ref{first-test}).

\begin{lem}\label{trans} The dualization sends
$$
   E_n \wedge \mapsto i_{\partial_n}, \quad \xi_j \mapsto {\rm{i}} \partial_j \quad
   \mbox{and} \quad i_E \mapsto {\rm{i}} d, \quad \alpha \wedge \mapsto -{\rm{i}} \delta.
$$
\end{lem}

\begin{proof} The first two claims are obvious. In order to prove the third rule, it suffices
to verify that the adjoint map of the principal symbol of ${\rm{i}} d$
corresponds to the insertion $i_E$. But the principal symbol of $\mathrm{i} d$
(regarded as an operator on $\Omega^p(\gon_-^\prime(\R)$) is given by
$$
   \sum_{j=1}^{n-1} \xi_j \otimes E_j^* \wedge: \Lambda^p (\gon_-^\prime(\R))^*
   \to \Lambda^{p+1} (\gon_-^\prime(\R))^*.
$$
Its adjoint equals
$$
   i_E = \sum_{j=1}^{n-1} \xi_j \otimes i_{E^*_j}: \Lambda^{p+1} (\gon_-^\prime(\R))
   \to \Lambda^{p} (\gon_-^\prime(\R)).
$$
Similarly, the principal symbol of $-\mathrm{i} \delta$ (regarded as an
operator on $\Omega^p(\gon_-^\prime(\R)$) is given by $\sum_{j=1}^{n-1} \xi_j
\otimes i_{E_j}: \Lambda^p (\gon_-^\prime(\R))^* \to \Lambda^{p-1}
(\gon_-^\prime(\R))^*$. Its adjoint equals
$$
   \sum_{j=1}^{n-1} \xi_j \otimes E_j \wedge = \alpha \wedge: \Lambda^{p-1} (\gon_-^\prime(\R))
   \to \Lambda^{p}(\gon_-^\prime(\R)).
$$
The proof is complete.
\end{proof}

\begin{example}\label{translation} As an illustration, we show how the Fourier
transform maps first type second-order homogeneous homomorphisms (encapsulated
by singular vectors) into second-order differential operators $\Omega^*(\R^n)
\to \Omega^*(\R^{n-1})$. By Example \ref{ExamplesPtoP}, the
${\fam2U}(\gog^\prime)$-homomorphism of generalized Verma modules
\begin{equation*}
   {\fam2M}^{\gog^\prime}_{\gop^\prime}\left(\Lambda^p(\gon_-^\prime(\R)) \otimes \C_{\lambda-2}\right) \to
   {\fam2 M}^{\gog}_{\gop} \left(\Lambda^p(\gon_-(\R)) \otimes \C_\lambda\right)
  \end{equation*}
is induced by
\begin{align*}
   \omega & \mapsto
   (\lambda\!+\!p\!-\!2) |\xi'|^2 \otimes \omega \\
   & -(2\lambda\!+\!n\!-\!3)(\lambda\!+\!p\!-\!2) \xi_n^2 \otimes \omega \\
   & -2(2\lambda\!+\!n\!-\!3) \xi_n E_n \wedge i_E(\omega) + 2 \alpha \wedge i_E(\omega).
\end{align*}
Now, using Lemma \ref{trans}, this homomorphism translates into the
differential operator
\begin{align*}
   \Omega^p(\R^n) \ni \omega \mapsto D^{(p\to p)}_2(\lambda)(\omega)
   & = (\lambda\!+\!p\!-\!2) \Delta \iota^*\omega \\
   & +(2\lambda\!+\!n\!-\!3) (\lambda\!+\!p\!-\!2)\iota^* \partial_n^2 (\omega) \\
   & +2(2\lambda\!+\!n\!-\!3) \dm\iota^* i_{\partial_n} \partial_n\omega + 2 \dm\delta \iota^*\omega\in\Omega^p(\R^{n-1}).
\end{align*}
\end{example}

\subsection{Families of the first type}\label{type1}

We recall that $a_j^{(N)}(\lambda)$ and $b_j^{(N)}(\lambda)$ denote even and
odd Gegenbauer coefficients, respectively (see Appendix).

\begin{theorem}\label{EvenDiffOp-type1} Assume that $N\in\N$ and $p=0,\dots,n-1$. Then the family
\begin{equation*}
   D_{2N}^{(p \to p)}(\lambda): \Omega^p(\R^n) \to \Omega^{p}(\R^{n-1}), \; \lambda \in \C
\end{equation*}
of differential operators of order $2N$ which is defined by the formula
\begin{align}
   D^{(p \to p)}_{2N}(\lambda)
   & = \sum_{j=0}^N(-1)^{N-j} p_j^{(N)}(\lambda;p)\Delta^j \iota^*\partial_n^{2N-2j} \nonumber\\
   & + \sum_{j=0}^{N-1} (-1)^{N-j} q_j^{(N-1)}(\lambda\!-\!1) \Delta^j \dm \iota^*i_{\partial_n}
   \partial_n^{2N-2j-1} \nonumber\\
   & + \sum_{j=0}^{N-1} (-1)^{N-j-1} r_j^{(N-1)}(\lambda\!-\!1) \Delta^j \dm
   \delta \iota^*\partial_n^{2N-2j-2},
\end{align}
where
\begin{align*}
   p_j^{(N)}(\lambda;p) & = (\lambda\!+\!p\!-\!2N) a_j^{(N)}(\lambda), \\
   q_j^{(N-1)}(\lambda) & = -2N(2\lambda\!+\!n\!-\!2N\!+\!1) b_j^{(N-1)}(\lambda), \\
   r_j^{(N-1)}(\lambda) & = 2N  a_j^{(N-1)}(\lambda)
\end{align*}
is infinitesimally equivariant in the sense that
\begin{equation}\label{equiv-even-2}
   D^{(p \to p)}_{2N}(\lambda) \dm\pi^\ch_{\lambda,p}(X)
   = \dm\pi^{\prime \ch}_{\lambda-2N,p}(X) D^{(p\to p)}_{2N}(\lambda), \;X \in
   \gog^{\prime}(\R).
\end{equation}
\end{theorem}

\begin{proof} The singular vector $v^{(p\to p)}_{2N}(\lambda)$ in Theorem \ref{EvenFromPtoP}
corresponds to an operator with the equivariance as in \eqref{equiv-even-2}.
The explicit formula for the operator follows by combining the explicit formula
for the singular vector with Lemma \ref{trans}.
\end{proof}

The intertwining relations \eqref{equiv-even} and \eqref{equiv-even-2} can also
be stated in terms of the geometrically defined representations
\begin{equation}\label{rep-geom-nc}
   \pi_\lambda^{(p)}(\gamma) \st e^{\lambda \Phi_\gamma} \gamma_*:
   \Omega^p (\R^n) \to \Omega^p(\R^n)
\end{equation}
and their analogs $\pi_\lambda^{\prime (p)}$ on $\R^{n-1}$. Here $\gamma$
denotes conformal diffeomorphisms of the Euclidean metric $g_0$ on $\R^n$,
i.e., $\gamma_*(g_0) = e^{2\Phi_\gamma} g_0$ for some $\Phi_\gamma \in
C^\infty(\R^n)$. In order to obtain such a formulation, one has to use the
relation
\begin{equation}\label{rep-rel-nc}
   \pi_{-\lambda-p}^{(p)} = \pi_{\lambda,p}^\ch.
\end{equation}
The identity \eqref{rep-rel-nc} is the non-compact analog of
\eqref{rep-relation}. In particular, we find that
\begin{equation*}
   D^{(p \to p)}_{2N}(\lambda) \dm\pi^{(p)}_{-\lambda-p}(X)
   = \dm\pi^{\prime (p)}_{-\lambda+2N-p}(X) D^{(p\to p)}_{2N}(\lambda), \; X \in \gog^{\prime}(\R).
\end{equation*}
This corresponds to the formulation in Section \ref{intro} (see
\eqref{inter-even}).

We continue with the formulation of the analogous result for {\em odd} order
conformal symmetry breaking operators of the first type.

\begin{theorem}\label{OddDiffOp-type1} Let $N\in\N_0$ and $p=0,\dots,n-1$. The family
\begin{equation*}
   D_{2N+1}^{(p\to p)}(\lambda): \Omega^p(\R^n) \to \Omega^{p}(\R^{n-1}), \; \lambda \in \C
\end{equation*}
of differential operators of order $2N+1$ which is defined by the formula
\begin{align}\label{D-first-odd}
   D^{(p\to p)}_{2N+1}(\lambda) & = \sum_{j=0}^N(-1)^{N-j} p_j^{(N)}(\lambda;p)
   \Delta^j \iota^* \partial_n^{2N+1-2j} \nonumber\\
   & +\sum_{j=0}^N(-1)^{N-j} q_j^{(N)}(\lambda\!-\!1) \Delta^j \dm
   \iota^*i_{\partial_n}\partial_n^{2N-2j} \nonumber\\
   & +\sum_{j=0}^{N-1}(-1)^{N-j-1} r_j^{(N-1)}(\lambda\!-\!1) \Delta^j \dm
   \delta \iota^*\partial_n^{2N-1-2j},
\end{align}
where
\begin{align*}
   p_j^{(N)}(\lambda;p) & = (\lambda\!+\!p\!-\!2N\!-\!1) b_j^{(N)}(\lambda), \\
   q_j^{(N)}(\lambda) & = a_j^{(N)}(\lambda), \\
   r_j^{(N-1)}(\lambda) & = 2N b_j^{(N-1)}(\lambda)
\end{align*}
is infinitesimally equivariant in the sense that
\begin{equation}\label{equiv-odd-2}
   D^{(p\to p)}_{2N+1}(\lambda) \dm\pi^\ch_{\lambda,p} (X)
   = \dm\pi^{\prime \ch}_{\lambda-2N-1,p}(X) D^{(p\to p)}_{2N+1}(\lambda), \; X \in
   \gog^{\prime}(\R).
\end{equation}
\end{theorem}

\begin{proof} The proof is parallel to that of Theorem \ref{OddDiffOp-type1}.
The operator $D^{(p\to p)}_{2N+1}(\lambda)$ is induced by the singular vector
in Theorem \ref{OddFromPtoP}. \footnote{In the final formulas we omit a factor
of $-i$.}
\end{proof}

Similarly as for even-order families, the intertwining relations
\eqref{equiv-odd} and \eqref{equiv-odd-2} can also be stated in terms of the
geometrically defined representations \eqref{rep-geom-nc} and their analogs on
$\R^{n-1}$. In particular, we find that
\begin{equation*}
   D^{(p\to p)}_{2N+1}(\lambda) \dm\pi^{(p)}_{-\lambda-p}(X)
   = \dm\pi^{\prime (p)}_{-\lambda+2N+1-p}(X) D^{(p\to p)}_{2N+1}(\lambda), \; X \in \gog^{\prime}(\R).
\end{equation*}
This corresponds to the formulation in Section \ref{intro} (see
\eqref{inter-odd}).

\subsection{Families of the second type}\label{type2}

The families $D_N^{(p \to p)}(\lambda)$ of the first type have natural
counterparts which map $\Omega^p(\R^n) \to \Omega^{p-1}(\R^{n-1})$. These are
the conformal symmetry breaking operators of the second type. We start with the
description of the even-order case.

\begin{theorem}\label{EvenDiffOp-type2} Assume that $N\in\N$ and $p=1,\dots,n$. Then the family
\begin{equation*}
   D_{2N}^{(p \to p-1)}(\lambda): \Omega^p(\R^n) \to \Omega^{p-1}(\R^{n-1}), \; \lambda\in\C
\end{equation*}
of differential operators of order $2N$ which is defined by the formula
\begin{align}\label{even-type2}
   D^{(p \to p-1)}_{2N}(\lambda) & =
   \sum_{j=0}^N(-1)^{N-j} p_j(\lambda;N,p) \Delta^j \iota^* i_{\partial_n} \partial_n^{2N-2j} \nonumber \\
   & + \sum_{j=0}^{N-1}(-1)^{N-j-1} q_j^{(N-1)}(\lambda\!-\!1) \Delta^j
   \delta \iota^*\partial_n^{2N-1-2j} \nonumber\\
   & + \sum_{j=0}^{N-1}(-1)^{N-j-1} r_j^{(N-1)}(\lambda\!-\!1) \Delta^j \dm
   \delta \iota^*i_{\partial_n}\partial_n^{2N-2-2j},
\end{align}
where
\begin{align*}
   p_j(\lambda;N,p) & = -(\lambda\!+\!n\!-\!p\!-\!2N\!+\!2j) a_j^{(N)}(\lambda), \\
   q_j^{(N-1)}(\lambda) & = -2N(2\lambda\!+\!n\!-\!2N\!+\!1) b_j^{(N-1)}(\lambda), \\
   r_j^{(N-1)}(\lambda) & = 2N  a_j^{(N-1)}(\lambda)
\end{align*}
is infinitesimally equivariant in the sense that
\begin{equation}\label{equiv-even}
   D^{(p \to p-1)}_{2N}(\lambda) \dm\pi^\ch_{\lambda,p}(X) = \dm\pi^{\prime \ch}_{\lambda-2N,p-1}(X)
   D^{(p \to p-1)}_{2N}(\lambda), \; X \in \gog^{\prime}(\R).
\end{equation}
\end{theorem}

The odd-order analog of Theorem \ref{EvenDiffOp-type2} reads as follows.

\begin{theorem}\label{OddDiffOp-type2} Let $N\in\N_0$ and $p=1,\dots,n$. The family
\begin{equation*}
   D_{2N+1}^{(p\to p-1)}(\lambda): \Omega^p(\R^n) \to \Omega^{p-1}(\R^{n-1}), \;
   \lambda \in \C
\end{equation*}
of differential operators of order $2N+1$ which is defined by the formula
\begin{align}
   D^{(p\to p-1)}_{2N+1}(\lambda) & = \sum_{j=0}^N(-1)^{N-j}
   p_j(\lambda;N,p) \Delta^j \iota^* i_{\partial_n} \partial_n^{2N+1-2j} \nonumber\\
   & + \sum_{j=0}^{N}(-1)^{N-j-1} q_j^{(N)}(\lambda\!-\!1) \Delta^j \delta \iota^* \partial_n^{2N-2j} \nonumber\\
   & + \sum_{j=0}^{N-1}(-1)^{N-j-1} r_j^{(N-1)}(\lambda\!-\!1) \Delta^{j} \dm \delta \iota^* i_{\partial_n}
   \partial_n^{2N-1-2j},
\end{align}
where
\begin{align*}
   p_j(\lambda;N,p) & = -(\lambda\!+\!n\!-\!p\!-\!2N\!+\!2j\!-\!1) b_j^{(N)}(\lambda), \\
   q_j^{(N)}(\lambda) & = a_j^{(N)}(\lambda), \\
   r_j^{(N-1)}(\lambda) & = 2N b_j^{(N-1)}(\lambda)
\end{align*}
is infinitesimally equivariant in the sense that
\begin{equation}\label{equiv-odd}
   D^{(p \to p-1)}_{2N+1}(\lambda) \dm \pi_{\lambda,p}^\ch(X)
   = \dm \pi^{\prime \ch}_{\lambda-2N-1,p-1} (X) D^{(p \to
   p-1)}_{2N+1}(\lambda), \; X \in \gog^{\prime}(\R).
\end{equation}
\end{theorem}

The proofs of Theorem \ref{EvenDiffOp-type2} and Theorem \ref{OddDiffOp-type2}
for families of the second type follow from the corresponding results in
Section \ref{type1} for families of the first type by conjugation with Hodge
star operators. The details of this argument will be given in Section
\ref{Hodge}.

\subsection{Hodge conjugation}\label{Hodge}

We first recall some well-known general facts.

\begin{lem}\label{Hodge-gen} On any Riemannian manifold $(M^n,g)$, the differential
operators $\dm$, $\delta$, $\Delta$ and the Hodge star operator $\star$ satisfy
the relations
\begin{itemize}
\item [(1)] $\star \, \star = (-1)^{p(n-p)}$,
\item [(2)] $\star \, \dm \, \star = (-1)^{n(p+1)+1} \delta$ and $\star \, \dm = \delta \, \star (-1)^{p+1}$,
\item [(3)] $\star \, \delta \, \star = (-1)^{n(p+1)} \dm$ and $\star \, \delta = \dm \, \star (-1)^p $,
\item [(4)] $\star \, \dm \delta = \delta \dm \, \star$, $\star \, \delta \dm = \dm
\delta \, \star$ and $\star \, \Delta = \Delta \,\star$
\end{itemize}
on $\Omega^p(M)$.
\end{lem}

\begin{proof} We omit the proof of (1) and of the first relation in (2). (1) and the first
claim in (2) imply the second claims in (2) and (3). Again, this yields the
first claim in (3) by using (1). The identities in (4) are immediate
consequences.
\end{proof}

Now let $\star$ and $\bar{\star}$ be the respective Hodge star operators on
$\Omega^*(\R^{n-1})$ and $\Omega^*(\R^{n})$.

\begin{lem}\label{HodgeLemma} We have the identities
\begin{align*}
   \star \circ \iota^* i_{\partial_n} \circ\bar{\star} & = \iota^* (-1)^{(p+1)(n-1)}, \\
   \star \circ \iota^* \circ\bar{\star} & = \iota^* i_{\partial_n} (-1)^{pn+1}
\end{align*}
on $\Omega^p(\R^n)$.
\end{lem}

\begin{proof} The second claim is easy to see. We omit the details. The first
claim follows by applying Lemma \ref{Hodge-gen}/(1) to the second claim.
\end{proof}

The following theorem is the main result of the present section.

\begin{theorem}\label{Hodge-c} Let $N \in \N$ and $p=1,\dots,n$. Then the even-order families
$D^{(p\to p)}_{2N}(\lambda)$
of the first type are Hodge conjugate to the even-order families $D^{(p\to
p-1)}_{2N}(\lambda)$ of the second type. More precisely, we have
\begin{align}\label{Hodge1}
   D^{(p\to p-1)}_{2N}(\lambda) = (-1)^{pn} \star D^{(n-p \to n-p)}_{2N}(\lambda) \, \bar{\star}.
\end{align}
Similarly, for $N \in \N_0$, the odd-order families $D^{(p\to
p)}_{2N+1}(\lambda)$ of the first type are Hodge conjugate to the odd-order
families $D^{(p\to p-1)}_{2N+1}(\lambda)$ of the second type. More precisely,
we have
\begin{align}\label{Hodge2}
   D^{(p\to p-1)}_{2N+1}(\lambda) = (-1)^{pn}\star D^{(n-p \to n-p)}_{2N+1}(\lambda) \, \bar{\star}.
\end{align}
\end{theorem}

\begin{proof} On the one hand, Theorem \ref{EvenDiffOp-type1} implies
\begin{align*}
   D^{(p \to p)}_{2N}(\lambda) & = \sum_{j=1}^N (-1)^{N-j}
   \left[p_j^{(N)}(\lambda;p) + r_{j-1}^{(N-1)}(\lambda\!-\!1)\right] (\dm\delta)^j \iota^*\partial_n^{2N-2j} \\
   & + \sum_{j=0}^N(-1)^{N-j} p_j^{(N)}(\lambda;p) (\delta\dm)^j \iota^*\partial_n^{2N-2j}\\
   & + \sum_{j=0}^{N-1}(-1)^{N-j} q_j^{(N-1)}(\lambda\!-\!1) (\dm\delta)^j \dm
   \iota^*i_{\partial_n} \partial_n^{2N-1-2j}
\end{align*}
with the coefficients
\begin{itemize}
   \item $p_j^{(N)}(\lambda;p) = (\lambda\!+\!p\!-\!2N) a_j^{(N)}(\lambda)$,
   \item $q_j^{(N-1)}(\lambda\!-\!1) = -2N(2\lambda\!+\!n\!-\!2N\!-\!1) b_j^{(N-1)}(\lambda\!-\!1)$,
   \item $r_j^{(N-1)}(\lambda\!-\!1) = 2N  a_j^{(N-1)}(\lambda\!-\!1)$.
\end{itemize}
On the other hand, Theorem \ref{EvenDiffOp-type2} shows that
\begin{align*}
   D^{(p\to p-1)}_{2N}(\lambda) & =
   \sum_{j=0}^N(-1)^{N-j} \left[p_j(\lambda;N,p) + r_{j-1}^{(N-1)}(\lambda\!-\!1)\right]
   (\dm\delta)^j \iota^*i_{\partial_n} \partial_n^{2N-2j} \\
   & + \sum_{j=1}^N(-1)^{N-j} p_j(\lambda;N,p) (\delta\dm)^j \iota^* i_{\partial_n} \partial_n^{2N-2j} \\
   & + \sum_{j=0}^{N-1}(-1)^{N-j-1} q_j^{(N-1)}(\lambda\!-\!1) (\delta\dm)^j \delta
   \iota^*\partial_n^{2N-1-2j},
\end{align*}
with the coefficients
\begin{itemize}
   \item  $p_j(\lambda;N,p) = -(\lambda\!+\!n\!-\!p\!-\!2N\!+\!2j) a_j^{(N)}(\lambda)$,
   \item  $q_j^{(N-1)}(\lambda\!-\!1) = -2N(2\lambda\!+\!n\!-\!2N\!-\!1) b_j^{(N-1)}(\lambda\!-\!1)$,
   \item  $r_j^{(N-1)}(\lambda\!-\!1) = 2N  a_j^{(N-1)}(\lambda\!-\!1)$.
\end{itemize}
Now, since $\partial_n$ commutes with $\bar{*}$, Lemma \ref{Hodge-gen} and
Lemma \ref{HodgeLemma} imply
\begin{align*}
   & \star (\dm \delta)^j \iota^* \, \bar{\star}
   = (\delta \dm )^j \star \iota^* \, \bar{\star} = (\delta \dm)^j \iota^* i_{\partial_n} (-1)^{pn+1}, \\
   & \star (\delta \dm)^j \iota^* \bar{\star} = (\dm \delta)^j \star \iota^* \, \bar{\star}
   = (\dm \delta)^j \iota^* i_{\partial_n}(-1)^{pn+1}
\end{align*}
and
\begin{align*}
   \star (\dm \delta)^j \dm \iota^* i_{\partial_n} \, \bar{\star} = (\delta \dm )^j \star \dm \iota^* i_{\partial_n}
   \, \bar{\star} = (\delta \dm)^j \delta \star \iota^* i_{\partial_n} \bar{\star} (-1)^{n-p} =
   (\delta \dm)^j \delta \iota^* (-1)^{pn+1}
\end{align*}
for $j \in \N_0$. Hence we find
\begin{align*}
   (-1)^{pn+1} & \star D^{(n-p\to n-p)}_{2N}(\lambda) \, \bar{\star} \\
   & = \sum_{j=1}^N(-1)^{N-j} \left[p_j^{(N)}(\lambda;n\!-\!p) + r_{j-1}^{(N-1)}(\lambda\!-\!1)\right]
   (\delta\dm)^j \iota^* i_{\partial_n} \partial_n^{2N-2j}\\
   & + \sum_{j=0}^N(-1)^{N-j}p_j^{(N)}(\lambda;n\!-\!p)(\dm\delta)^j \iota^*i_{\partial_n}\partial_n^{2N-2j}\\
   & + \sum_{j=0}^{N-1}(-1)^{N-j}q_j^{(N-1)}(\lambda\!-\!1) (\delta\dm)^j \delta  \iota^* \partial_n^{2N-1-2j}.
\end{align*}
But the relation
$$
   2N a_{j-1}^{(N-1)}(\lambda\!-\!1) = 2j a_j^{(N)}(\lambda)
$$
shows that
\begin{align*}
   p_j^{(N)}(\lambda;n\!-\!p) + r_{j-1}^{(N-1)}(\lambda\!-\!1) & = (\lambda\!+\!n\!-\!p\!-\!2N) a_j^{(N)}(\lambda)
   + 2N a_{j-1}^{(N-1)}(\lambda\!-\!1)\\
   & = (\lambda\!+\!n\!-\!p\!-\!2N\!+\!2j) a_j^{(N)}(\lambda) \\
   & = - p_j(\lambda;N,p).
\end{align*}
The latter identity implies \eqref{Hodge1}.

The analogous proof of \eqref{Hodge2} runs as follows. On the one hand, Theorem
\ref{OddDiffOp-type1} implies
\begin{align*}
   D^{(p\to p)}_{2N+1}(\lambda) & = \sum_{j=1}^N(-1)^{N-j}
   \left[p_j^{(N)}(\lambda;p) + r_{j-1}^{(N-1)}(\lambda\!-\!1)\right]
   (\dm\delta)^j \iota^* \partial_n^{2N+1-2j} \\
   & + \sum_{j=0}^N(-1)^{N-j} p_j^{(N)}(\lambda;p)(\delta\dm)^j \iota^* \partial_n^{2N+1-2j} \\
   & +\sum_{j=0}^N(-1)^{N-j} q_j^{(N)}(\lambda\!-\!1) (\dm\delta)^j \dm \iota^* i_{\partial_n} \partial_n^{2N-2j},
\end{align*}
with the coefficients
\begin{itemize}
   \item $p_j^{(N)}(\lambda;p) = (\lambda\!+\!p\!-\!2N\!-\!1) b_j^{(N)}(\lambda)$,
   \item $q_j^{(N)}(\lambda\!-\!1) = a_j^{(N)}(\lambda\!-\!1)$,
   \item $r_j^{(N-1)}(\lambda\!-\!1) = 2N b_j^{(N-1)}(\lambda\!-\!1)$.
\end{itemize}
On the other hand, Theorem \ref{OddDiffOp-type2} shows that
\begin{align*}
   D^{(p\to p-1)}_{2N+1}(\lambda) & = \sum_{j=1}^N(-1)^{N-j}
   \left[p_j(\lambda;N,p) + r_{j-1}^{(N-1)}(\lambda\!-\!1)\right] (\dm\delta)^j \iota^* i_{\partial_n}
   \partial_n^{2N+1-2j} \\
   & + \sum_{j=0}^N(-1)^{N-j} p_j(\lambda;N,p) (\delta\dm)^j \iota^*i_{\partial_n}\partial_n^{2N+1-2j} \\
   & + \sum_{j=0}^{N}(-1)^{N-j-1} q_j^{(N)}(\lambda\!-\!1) (\delta\dm)^j \delta \iota^*\partial_n^{2N-2j},
\end{align*}
with the coefficients
\begin{itemize}
   \item $p_j(\lambda;N,p) = -(\lambda\!+\!n\!-\!p\!-\!2N\!+\!2j\!-\!1) b_j^{(N)}(\lambda)$,
   \item $q_j^{(N)}(\lambda\!-\!1) = a_j^{(N)}(\lambda\!-\!1)$,
   \item $r_j^{(N-1)}(\lambda\!-\!1) = 2N b_j^{(N-1)}(\lambda\!-\!1)$.
\end{itemize}
Similarly as above, Lemma \ref{Hodge-gen} and Lemma \ref{HodgeLemma} yield
\begin{align*}
   (-1)^{pn+1} & \star \, D^{(n-p\to n-p)}_{2N+1}(\lambda) \, \bar{\star} \\
   & = \sum_{j=1}^N(-1)^{N-j} \left[p_j^{(N)}(\lambda;n\!-\!p) + r_j^{(N-1)}(\lambda\!-\!1)\right]
   (\delta\dm)^j \iota^* i_{\partial_n} \partial_n^{2N+1-2j} \\
   & + \sum_{j=0}^N(-1)^{N-j}p_j^{(N)}(\lambda;n\!-\!p) (\dm\delta)^j \iota^* i_{\partial_n} \partial_n^{2N+1-2j} \\
   & +\sum_{j=0}^N(-1)^{N-j} q_j^{(N)}(\lambda\!-\!1) (\delta\dm)^j \delta \iota^*\partial_n^{2N-2j}.
\end{align*}
But the relation
$$
   2N b_{j-1}^{(N-1)}(\lambda\!-\!1)=2j b_j^{(N)}(\lambda)
$$
shows that
\begin{align*}\label{help-sum}
   p_j^{(N)}(\lambda;n\!-\!p) + r_{j-1}^{(N-1)}(\lambda\!-\!1) & =
   (\lambda\!+\!n\!-\!p\!-\!2N\!-\!1) b_j^{(N)}(\lambda) + 2N b_{j-1}^{(N-1)}(\lambda\!-\!1) \notag \\
   & = (\lambda\!+\!n\!-\!p\!-\!2N\!+\!2j\!-\!1) b_j^{(N)}(\lambda) \notag \\
   & = -p_j(\lambda;N,p).
\end{align*}
The latter identity implies \eqref{Hodge2}. The proof is complete.
\end{proof}

Now we recall that the Hodge star operators $\star$ and $\bar{\star}$ are
conformally equivariant. More precisely, we have the following result.

\begin{lem}\label{Hodge-covariant} We have
$$
   \dm \pi_\lambda^{(n-p)}(X) (\bar{\star} \, \omega) = \bar{\star} \,
   \dm \pi^{(p)}_{\lambda+n-2p}(X) (\omega) \qquad \mbox{for $\omega \in \Omega^p(\R^n)$, $X \in \gog$}.
$$
Similarly, we have
$$
   \dm \pi_\lambda^{\prime (n-1-p)}(X) (\star \, \omega) = \star \,
   \dm \pi^{\prime (p)}_{\lambda+n-1-2p}(X) (\omega) \qquad \mbox{for $\omega \in \Omega^p(\R^{n-1})$,
   $X \in \gog^\prime$}.
$$
\end{lem}

\begin{proof} The Hodge star operator of $(M,g)$ satisfies the identity
$$
   \star_{\hat{g}} = e^{(n-2p)\varphi} \star_g, \quad \hat{g} = e^{2\varphi} g
$$
on $\Omega^p(M)$. In fact, with obvious notation we find
\begin{align*}
    \pi_\lambda^{(n-p)}(\gamma) (\star_g \, \omega) & = e^{\lambda \Phi_\gamma}
    \gamma_* (\star_g \, \omega) \\
    & = e^{\lambda \Phi_\gamma} \star_{\gamma_*(g)} \gamma_*(\omega) \\
    & = e^{\lambda \Phi_\gamma} \star_{e^{2\Phi_\gamma}g} \gamma_*(\omega) \\
    & = e^{\lambda \Phi_\gamma} e^{(n-2p)\Phi_\gamma} \star_g \gamma_*(\omega) \\
    & = \star_g \, e^{(\lambda+n-2p)\Phi_\gamma} \gamma_*(\omega) \\
    & = \star_g \, \pi_{\lambda+n-2p}^{(p)}(\gamma)(\omega)
\end{align*}
for $\omega \in \Omega^p(M)$. The assertions follow by differentiation.
\end{proof}

Theorem \ref{Hodge-c} and Lemma \ref{Hodge-covariant} show that the conformal
equivariance of one type of the families of conformal symmetry breaking
operators is equivalent to the conformal covariance of the other type. In
particular, Theorem \ref{EvenDiffOp-type2} and Theorem \ref{OddDiffOp-type2}
follow from Theorem \ref{EvenDiffOp-type1} and Theorem \ref{OddDiffOp-type1}.

\subsection{Operators of the third type}\label{CSBO-type3}

We start with the description of third type operators of even-orders.

\begin{theorem}\label{DO3-Even} Let $N\in\N$. Then the differential operator
\begin{align*}
   D_{2N}^{(0\to 1)}: \Omega^0(\R^n)\to \Omega^1(\R^{n-1})
\end{align*}
of order $2N$ which is defined by
\begin{align*}
   D_{2N}^{(0\to 1)} \st \sum_{j=0}^{N-1} (-1)^{N-j-1} b_j^{(N-1)}(2N\!-\!1)
   \dm (\delta\dm)^j \iota^* \partial_n^{2N-2j-1}
\end{align*}
is infinitesimally equivariant in the sense that
\begin{align*}
   D_{2N}^{(0\to 1)} \dm\pi^\ch_{2N-1,0}(X)
   = \dm\pi^{\prime\ch}_{-1,1}(X) D^{(0\to 1)}_{2N}, \quad  X\in\gog^\prime(\R).
\end{align*}
\end{theorem}

\begin{proof} The proof is parallel to that of Theorem \ref{EvenDiffOp-type1}. The
operator $D_{2N}^{(0\to 1)}$ is induced by the singular vector in Theorem
\ref{sv-third}/(1).
\end{proof}

We continue with the formulation of the analogous result for odd-order
operators of the third type.

\begin{theorem}\label{DO3-Odd} Let $N\in\N$. Then the differential operator
\begin{align*}
   D_{2N+1}^{(0\to 1)}: \Omega^0(\R^n)\to \Omega^1(\R^{n-1})
\end{align*}
of order $2N+1$ which is defined by
\begin{align*}
   D_{2N+1}^{(0\to 1)} \st \sum_{j=0}^{N}(-1)^{N-j} a_j^{(N)}(2N)
   \dm (\delta\dm)^{j} \iota^* \partial_n^{2N-2j}
\end{align*}
is infinitesimally equivariant in the sense that
\begin{align*}
   D_{2N+1}^{(0\to 1)} \dm\pi^\ch_{2N,0}(X)
   = \dm\pi^{\prime\ch}_{-1,1}(X) D^{(0\to 1)}_{2N+1}, \quad X \in \gog^\prime(\R).
\end{align*}
\end{theorem}

\begin{proof} The operator $D_{2N+1}^{(0\to 1)}$ is induced by the singular vector in Theorem
\ref{sv-third}/(2).
\end{proof}

Remark \ref{sv3-der} yields

\begin{bem} For any $N \in \N$, we have
$$
   D_N^{(0 \to 1)} = d \dot{D}_{N-1}^{(0 \to 0)}(N\!-\!1).
$$
\end{bem}

In addition, there is a first-order operator of the third type.

\begin{theorem}\label{DO3-FO} The differential operator
\begin{align*}
   D_1^{(p \to p+1)} \st \dm \iota^*: \Omega^p(\R^n) \to \Omega^{p+1}(\R^{n-1})
\end{align*}
of first-order is infinitesimally equivariant in the sense that
\begin{align*}
   D_1^{(p\to p+1)}\dm\pi^\ch_{-p,p}(X)
   = \dm\pi^{\prime\ch}_{-(p+1),p+1}(X) D^{(p\to p+1)}_{1},\quad X\in\gog^\prime(\R).
\end{align*}
\end{theorem}

\begin{proof} The operator $D_{1}^{(p\to p+1)}$ is induced by the singular vector in Theorem
\ref{sv-third}/(3).\footnote{We omit the factor $i$}.
\end{proof}

\subsection{Operators of the fourth type}\label{CSBO-type4}

The operators $D_{N}^{(0\to  1)}$ and $D_{1}^{(p\to p+1)}$ of the third type
have natural counter parts which map $\Omega^n(\R^n)\to \Omega^{n-2}(\R^{n-1})$
and $\Omega^p(\R^n)\to \Omega^{p-2}(\R^{n-1})$, respectively.

We start with the description of the even-order operators of the fourth type.

\begin{theorem}\label{DO4-Even} Let $N\in\N$. Then the differential operator
\begin{align*}
   D_{2N}^{(n\to n-2)}: \Omega^n(\R^n)\to \Omega^{n-2}(\R^{n-1})
\end{align*}
of order $2N$ which is defined by
\begin{align*}
   D_{2N}^{(n\to n-2)} \st \sum_{j=0}^{N-1}(-1)^{N-j} b_j^{(N-1)}(2N\!-\!1)
   \delta (\dm\delta)^j \iota^*i_{\partial_n}\partial_n^{2N-2j-1}
\end{align*}
is infinitesimally equivariant in the sense that
\begin{align*}
   D_{2N}^{(n\to n-2)} \dm\pi^\ch_{2N-1,n}(X)
   =\dm\pi^{\prime\ch}_{-1,n-2}(X) D^{(n\to n-2)}_{2N},\quad X\in\gog^\prime(\R).
\end{align*}
\end{theorem}

\begin{proof} The operator $D_{2N}^{(n\to n-2)}$ is induced by the singular vector in Theorem
\ref{sv-fourth}/(1).
\end{proof}

We continue with the formulation of the analogous result for odd-order
operators of the fourth type.

\begin{theorem}\label{DO4-Odd} Let $N\in\N_0$. Then the differential operator
\begin{align*}
   D_{2N+1}^{(n\to n-2)}:\Omega^n(\R^n)\to \Omega^{n-2}(\R^{n-1})
\end{align*}
of order $2N+1$ which is defined by
\begin{align*}
   D_{2N+1}^{(n\to n-2)}=\sum_{j=0}^{N}(-1)^{N-j+1} a_j^{(N)}(2N)
   \delta (\dm\delta)^{j}\iota^*i_{\partial_n}\partial_n^{2N-2j}
\end{align*}
is infinitesimally equivariant in the sense that
\begin{align*}
   D_{2N+1}^{(n\to n-2)}\dm\pi^\ch_{2N,n}(X)
   =\dm\pi^{\prime\ch}_{-1,n-2}(X) D^{(0\to 1)}_{2N+1},\quad X\in\gog^\prime(\R).
\end{align*}
\end{theorem}

\begin{proof} The operator $D_{2N+1}^{(n\to n-2)}$ is induced by the singular vector in Theorem
\ref{sv-fourth}/(2).
\end{proof}

Remark \ref{sv4-der} yields

\begin{bem}\label{D4-dot} For any $N \in \N$, we have
$$
   D_N^{(n \to n-2)} = \delta \dot{D}_{N-1}^{(n \to n-1)}(N\!-\!1).
$$
\end{bem}

In addition, there is a first-order operator of the fourth type.

\begin{theorem}\label{DO4-FO} Assume that $p=2,\dots,n$. The differential operator
\begin{align*}
   D_1^{(p\to p-2)} \st \delta \iota^*i_{\partial_n}: \Omega^p(\R^n)\to \Omega^{p-2}(\R^{n-1})
\end{align*}
of first-order is infinitesimally equivariant in the sense that
\begin{align*}
   D_{1}^{(p\to p-2)}\dm\pi^\ch_{p-n,p}(X)
   = \dm\pi^{\prime\ch}_{p-n-1,p-2}(X) D^{(p\to p-2)}_{1}, \quad X\in\gog^\prime(\R).
\end{align*}
\end{theorem}

\begin{proof} The operator $D_{1}^{(p\to p-2)}$ is induced by the singular vector in Theorem
\ref{sv-fourth}/(3).\footnote{We omit the factor $-i$}.
\end{proof}

\subsection{Operators on middle degree forms}\label{CSBO-middle}

Here we briefly describe some special issues related to conformal symmetry
breaking operators on middle degree forms. As in Section \ref{Hodge}, we let
$\star$ and $\bar{\star}$ denote the Hodge star operators of the Euclidean
metrics on $\R^{n-1}$ and $\R^n$, respectively. We divide the discussion into
three cases.

{\bf Case 1a: Let $n$ be odd and $p=\frac{n-1}{2}$.}

Let $\pr_{\pm}: \Omega^{\frac{n-1}{2}}(\R^{n-1}) \to
\Omega^{\frac{n-1}{2}}_{\pm}(\R^{n-1})$ denote the projections onto the eigenspaces
of the operator $\star$. Then the compositions
\begin{equation*}
   D^{(\frac{n-1}{2}\to \frac{n-1}{2}),{\pm}}_N(\lambda) \st
   \pr_\pm \circ D^{(\frac{n-1}{2}\to \frac{n-1}{2})}_N(\lambda)
\end{equation*}
are non-trivial and we have

\begin{theorem}\label{CSBO-cas1} For any $N\in\N$, the families
\begin{equation*}
   D^{(\frac{n-1}{2}\to \frac{n-1}{2}),{\pm}}_N(\lambda):
   \Omega^{\frac{n-1}{2}}(\R^{n}) \to \Omega^{\frac{n-1}{2}}_{\pm}(\R^{n-1})
\end{equation*}
of differential operators of order $N$ are infinitesimally equivariant in the
sense that
\begin{equation*}
   D^{(\frac{n-1}{2} \to \frac{n-1}{2}),{\pm}}_N(\lambda) \dm\pi^\ch_{\lambda,\frac{n-1}{2}}(X)
   = \dm \pi^{\prime\ch}_{\lambda-N,\frac{n-1}{2}}(X)  D^{(\frac{n-1}{2}\to \frac{n-1}{2}),{\pm}}_N(\lambda),\\
\end{equation*}
for all $X\in\gog^\prime(\R)$.
\end{theorem}

\begin{proof} The families are induced by the singular vectors
$v_N^{(\frac{n-1}{2} \to \frac{n-1}{2}),\pm}(\lambda)$ (defined in Section
\ref{fam3}).
\end{proof}

{\bf Case 1b: Let $n$ be odd and $p=\frac{n+1}{2}$.}

The compositions
\begin{align*}
   D^{(\frac{n+1}{2}\to \frac{n-1}{2}),{\pm}}_N(\lambda)
   & \st \pr_\pm\circ D^{(\frac{n-1}{2}\to \frac{n-1}{2})}_N(\lambda)\circ \bar{\star}, \quad \lambda \in \C
\end{align*}
are non-trivial and we have

\begin{theorem}\label{CSBO-case2} For any $N\in\N$, the families
\begin{equation*}
   D^{(\frac{n+1}{2}\to \frac{n-1}{2}),\pm}_N(\lambda):
   \Omega^{\frac{n+1}{2}}(\R^{n}) \to \Omega^{\frac{n-1}{2}}_\pm(\R^{n-1}), \quad \lambda \in \C
\end{equation*}
of differential operators of order $N$ are infinitesimally equivariant in the
sense that
\begin{equation*}
   D^{(\frac{n+1}{2}\to \frac{n-1}{2}),\pm}_N(\lambda) \dm \pi^\ch_{\lambda,\frac{n+1}{2}}(X)
   = \dm \pi^{\prime\ch}_{\lambda-N,\frac{n-1}{2}}(X)  D^{(\frac{n+1}{2}\to \frac{n-1}{2}),\pm}_N(\lambda)
\end{equation*}
for all $X\in\gog^\prime(\R)$.
\end{theorem}

\begin{proof} The families are induced by the singular vectors
$\bar{\star} \, v_N^{(\frac{n-1}{2}\to \frac{n-1}{2}),\pm}(\lambda)$ (defined
in Section \ref{fam3}).
\end{proof}

{\bf Case 2: Let $n$ be even and $p=\frac{n}{2}$.}

Let $\Omega^{\frac n2}_\pm(\R^n) \subset \Omega^{\frac n2}(\R^n)$ be the
eigenspaces of the operator $\bar{\star}$. We define
\begin{align*}
   D^{(\frac{n}{2}\to \frac{n}{2}),{\pm}}_N(\lambda)
   \st D^{(\frac{n}{2}\to \frac{n}{2})}_N(\lambda) |_{\Omega^{\frac{n}{2}}_\pm(\R^n)}.
\end{align*}
These operators are non-trivial and we have

\begin{theorem}\label{fam-mid-even} For any $N\in\N$, the families
\begin{align*}
   D^{(\frac{n}{2}\to \frac{n}{2}),\pm}_N(\lambda):
   \Omega^{\frac{n}{2}}_\pm(\R^{n})\to \Omega^{\frac{n}{2}}(\R^{n-1}), \quad \lambda \in \C
\end{align*}
of differential operators of order $N$ are infinitesimally equivariant in the
sense that
\begin{align*}
   D^{(\frac{n}{2}\to \frac{n}{2}),\pm}_N(\lambda) \dm\pi^\ch_{\lambda,\frac{n}{2}}(X)
   = \dm \pi^{\prime\ch}_{\lambda-N,\frac{n}{2}}(X)  D^{(\frac{n}{2}\to
   \frac{n}{2}),\pm}_N(\lambda), \; X \in \gog^\prime(\R).
\end{align*}
\end{theorem}

\begin{proof} The families are induced by the singular vectors
$v_N^{(\frac{n}{2}\to \frac{n}{2}),\pm}(\lambda)$ (defined in Section
\ref{fam3}).
\end{proof}

We refer to Remark \ref{history} and Remark \ref{special-2} for examples of
both types of families in dimension $n=2$.

\subsection{Proof of Theorem \ref{classification}}\label{proof-classify}

The following proof of the classification consists in a combination of
previously proved facts.

We recall that conformally covariant differential operators
$$
   D: \Omega^p(\R^n) \to \Omega^q(\R^{n-1})
$$
of order $N \in \N_0$ which are infinitesimally equivariant in the sense that
$$
   d\pi^{\prime (q)}_{\eta}(X) D = D d\pi^{(p)}_{\mu}(X), \; X \in \gog^\prime(\R)
$$
or, equivalently,
$$
   d \pi^{\prime \ch}_{-\eta-q,q}(X) D = D d\pi^\ch_{-\mu-p,p}(X), \; X \in \gog^\prime(\R)
$$
(see Remark \ref{rep-geo}) correspond to elements of the space
$$
   \Hom_{\gop^\prime}(\Lambda^q(\gon_-^\prime(\R)) \otimes \C_{-\eta-q},
   \Pol_N(\gon_-^*(\R)) \otimes \Lambda^p(\gon_-(\R)) \otimes \C_{-\mu-p})
$$
of singular vectors. Since $A$ acts on $\Pol_N(\gon_-^*(\R))$ by push-forward,
it follows that
$$
   \eta + q = N + \mu +p.
$$
Hence it remains to describe the spaces
$$
   \Hom_{\gop^\prime}(\Lambda^q(\gon_-^\prime(\R)) \otimes \C_{\lambda-N},
   \Pol_N(\gon_-^*(\R)) \otimes \Lambda^p(\gon_-(\R)) \otimes \C_{\lambda}), \;
   \lambda \in \C.
$$
Proposition \ref{hom-structure} evaluates the $\gol^\prime$-invariance of such
homomorphisms. In particular, there are four basic groups of such
homomorphisms. The additional $\gon_+^\prime(\R)$-invariance of singular
vectors in these groups has been discussed in Sections
\ref{sv-type1}--\ref{sv-type-4}. In fact, we proved that all singular vectors
are given by
\begin{itemize}
\item first type homomorphisms
\begin{equation*}
   \xi_n^{N} P(t) \otimes \id + \xi_n^{N-1} Q(t) E_n \wedge i_E + \xi_n^{N-2} R(t) \alpha \wedge i_E
\end{equation*}
in the case $q=p$ and general $\lambda$ (Theorems \ref{OddFromPtoP},
\ref{EvenFromPtoP}),
\item second type homomorphisms
\begin{equation*}
   \xi_n^{N} P(t) \otimes E_n + \xi_n^{N-1} Q(t) \alpha + \xi_n^{N-2} R(t) E_n \wedge \alpha \wedge i_E
\end{equation*}
in the case $q=p-1$ and general $\lambda$ (Theorems \ref{OddFromPtoP-1},
\ref{EvenFromPtoP-1}),
\item third type homomorphisms
\begin{equation*}
   \xi_n^{N-1} P(t) i_E
\end{equation*}
for $N \ge 1$ in the case $q=1$, $p=0$, $\lambda=N-1$, and
$$
   i_E
$$
in the case $N=1$, $q=p+1$, $\lambda=-p$ (Theorem \ref{sv-third}),
\item fourth type homomorphisms
\begin{equation*}
   \xi_n^{N-1} P(t) E_n \wedge \alpha
\end{equation*}
in the case $q=n-2$, $p=n$, $\lambda=N-1$, and
$$
   E_n \wedge \alpha
$$
in the case $N=1$, $q=p-2$, $\lambda=-(n-p)$ (Theorem \ref{sv-fourth}),
\item the compositions of these homomorphisms with the operator $\star$.
\end{itemize}
The quoted theorems characterize the appropriate polynomials $P$, $Q$ and $R$
in the variable $t = |\xi^\prime|^2/\xi_n^2$ with the property that the
corresponding homomorphisms are $\gon_+^\prime(\R)$-invariant. In all cases,
the singular vectors are unique up a constant multiple. The descriptions of the
corresponding differential operators can be found in the following respective
theorems:
\begin{itemize}
\item Theorem \ref{EvenDiffOp-type1}, Theorem \ref{OddDiffOp-type1},
\item Theorem \ref{EvenDiffOp-type2}, Theorem \ref{OddDiffOp-type2},
\item Theorem \ref{DO3-Even}, Theorem \ref{DO3-Odd}, Theorem \ref{DO3-FO},
\item Theorem \ref{DO4-Even}, Theorem \ref{DO4-Odd}, Theorem \ref{DO4-FO}.
\end{itemize}
The first two sets of results correspond to the first and second type operators
in Theorem \ref{classification}/(1),(2). The operators in the third set of
results correspond to the third type operators in Theorem
\ref{classification}/(3),(4). Similarly, the operators in the fourth set of
results correspond to the fourth type operators in Theorem
\ref{classification}/(5),(6). This completes the proof.

\subsection{Examples}\label{examples}

In the present section, we discuss some special cases in more detail. We first
display explicit formulas for conformal symmetry breaking operators in
low-order cases. Furthermore, we demonstrate the equivariance of the
first-order families both by direct calculations (for families acting on
one-forms) and by recognizing them as special cases of general conformally
covariant families on differential forms. Finally, we describe the relation of
the present results to previous discussions of the special case $n=2$ in
\cite{Juhl0} and \cite{KKP}.

\begin{example}\label{LowOrderExampleDiffOp} Here we display formulas for conformal
symmetry breaking operators of the first and second typ up to order $3$. In
each case, we give formulas in the style of the previous theorems as well as
formulas in terms of the operators $d$, $\delta$, $\bar{d}$ and $\bar{\delta}$.
First of all, the zeroth order families are
$$
   D_0^{(p \to p)}(\lambda) = (\lambda\!+\!p) \iota^* \quad \mbox{and} \quad
   D_0^{(p \to p-1)}(\lambda) = -(\lambda\!+\!n\!-\!p) \iota^* i_{\partial_n}.
$$
The equivariance of these operators follows from the relations
$$
   \gamma_* (\iota^*(\omega)) = \iota^*(\gamma_*(\omega)) \quad \mbox{and} \quad
   \gamma_* (\iota^* i_{\partial_n} (\omega) = e^{-\Phi_\gamma} \iota^*
   i_{\partial_n} (\gamma_*(\omega)), \; \gamma \in G'.
$$
By Theorem \ref{OddDiffOp-type1} and Theorem \ref{OddDiffOp-type2} for $N=0$
and Lemma \ref{DiffCoDiff}, the first-order families are
\begin{equation*}
   D^{(p\to p)}_1(\lambda) = (\lambda\!+\!p\!-\!1) \iota^* \partial_n + \dm \iota^*i_{\partial_n}
   =(\lambda\!+\!p)\dm\iota^* i_{\partial_n} + (\lambda\!+\!p\!-\!1)\iota^* i_{\partial_n} \bar{\dm}
\end{equation*}
and
\begin{equation*}
   D^{(p\to p-1)}_{1}(\lambda) = -(\lambda\!+\!n\!-\!p\!-\!1) \iota^*i_{\partial_n} \partial_n - \delta \iota^*
   = -(\lambda\!+\!n\!-\!p) \delta \iota^* + (\lambda\!+\!n\!-\!p\!-\!1) \iota^* \bar{\delta}
\end{equation*}
Next, using Theorem \ref{EvenDiffOp-type1} and Theorem \ref{EvenDiffOp-type2}
for $N=1$ and Lemma \ref{DiffCoDiff}, the second-order families read
\begin{align*}
   D^{(p \to p)}_2(\lambda) & = (\lambda\!+\!p\!-\!2) \Delta \iota^*
   +(2\lambda\!+\!n\!-\!3)(\lambda\!+\!p\!-\!2)\iota^* \partial_n^2 \\
   & + 2(2\lambda\!+\!n\!-\!3)\dm\iota^* i_{\partial_n}\partial_n+2\dm\delta \iota^*
\end{align*}
or, equivalently,
\begin{align*}
   D^{(p \to p)}_2(\lambda) & =
   (2\lambda\!+\!n\!-\!2) \left[(\lambda\!+\!p) \dm \delta + (\lambda\!+\!p\!-\!2)\delta \dm \right]\iota^*\\
   & -(2\lambda\!+\!n\!-\!3) \iota^*
   \left[(\lambda\!+\!p) \bar{\dm} \bar{\delta} + (\lambda\!+p\!-\!2)\bar{\delta} \bar{\dm} \right]
\end{align*}
and
\begin{align*}
   D^{(p \to p-1)}_2(\lambda) & = -(\lambda\!+\!n\!-\!p) \Delta \iota^*i_{\partial_n}
   - (2\lambda\!+\!n\!-\!3)(\lambda\!+\!n\!-\!p\!-\!2) \iota^* i_{\partial_n}
   \partial_n^2 \\ & -2(2\lambda\!+\!n\!-\!3)\delta\iota^*\partial_n+2 \dm \delta\iota^* i_{\partial_n}
\end{align*}
or, equivalently,
\begin{align*}
   D^{(p \to p-1)}_2(\lambda) & = -(2\lambda\!+\!n\!-\!2)
   \left[(\lambda\!+\!n\!-\!p)\delta \dm + (\lambda\!+\!n\!-\!p\!-\!2) \dm \delta\right] \iota^* i_{\partial_n}\\
   & + (2\lambda\!+\!n\!-\!3) \iota^* i_{\partial_n}
   \left[(\lambda\!+\!n\!-\!p)\bar{\delta} \bar{\dm} + (\lambda\!+\!n\!-\!p\!-\!2)\bar{\dm}\bar{\delta}
   \right].
\end{align*}
Finally, by Theorem \ref{OddDiffOp-type1} and Theorem \ref{OddDiffOp-type2} for
$N=1$ and Lemma \ref{DiffCoDiff}, the third-order families are given by
\begin{align*}
   D^{(p \to p)}_3(\lambda) & =(\lambda\!+\!p\!-\!3) \Delta \iota^*\partial_n
   + \Delta \dm \iota^ *i_{\partial_n} + 2\dm \delta \iota^*\partial_n \\
   & + \tfrac 13(2\lambda\!+\!n\!-\!5)(\lambda\!+\!p\!-\!3) \iota^*\partial_n^3
   + (2\lambda\!+\!n\!-\!5)\dm\iota^*i_{\partial_n}\partial_n^2
\end{align*}
or, equivalently,
\begin{align*}
   D^{(p \to p)}_3(\lambda)
   & = \tfrac 13(2\lambda\!+\!n\!-\!2)(\lambda\!+\!p) \dm\delta \dm \iota^* i_{\partial_n}
   - \tfrac 13(2\lambda\!+\!n\!-\!5)(\lambda\!+\!p\!-\!3)\iota^* i_{\partial_n}\bar{\dm} \bar{\delta} \bar{\dm}\\
   & + \tfrac 13(2\lambda\!+\!n\!-\!2)(\lambda\!+\!p\!-\!3) \delta \dm \iota^* i_{\partial_n} \bar{\dm}
   -\tfrac 13(2\lambda\!+\!n\!-\!5)(\lambda\!+\!p) \dm \iota^* i_{\partial_n} \bar{\dm} \bar{\delta}\\
   & + \tfrac 13\left[ 2(\lambda\!+\!p\!-\!3)(2\lambda\!+\!n\!-\!2)
   + 3(\lambda\!+\!n\!-\!p) \right]\dm\delta \iota^* i_{\partial_n}\bar{\dm}
\end{align*}
and
\begin{align*}
   D^{(p \to p-1)}_3(\lambda)& = -(\lambda\!+\!n\!-\!p\!-\!1) \Delta \iota^*i_{\partial_n}\partial_n
   - \Delta \delta \iota^* + 2\dm\delta \iota^* i_{\partial_n} \partial_n \\
   & -\tfrac 13 (2\lambda\!+\!n\!-\!5)(\lambda\!+\!n\!-\!p\!-\!3)\iota^* i_{\partial_n}\partial_n^3
   - (2\lambda\!+\!n\!-\!5) \delta \iota^* \partial_n^2
\end{align*}
or, equivalently,
\begin{align*}
   D^{(p \to p-1)}_3(\lambda)
   & = -\tfrac 13 (2\lambda\!+\!n\!-\!2)(\lambda\!+\!n\!-\!p) \delta \dm\delta \iota^*
   - \tfrac 13 (2\lambda\!+\!n\!-\!5)(\lambda\!+\!n\!-\!p\!-\!3)\iota^*\bar{\delta}\bar{\dm}\bar{\delta}\\
   & + \tfrac 13(2\lambda\!+\!n\!-\!2)(\lambda\!+\!n\!-\!p\!-\!3) \dm\delta\iota^*\bar{\delta}
   + \tfrac 13 (2\lambda\!+\!n\!-\!5)(\lambda\!+\!n\!-\!p) \delta\iota^*\bar{\delta}\bar{\dm}\\
   & + \tfrac 13 \left[ 2(\lambda\!+\!n\!-\!p\!-\!3)(2\lambda\!+\!n\!-\!2)
   + 3(\lambda\!+\!p) \right]\delta \dm\iota^*\bar{\delta}.
\end{align*}
The respective second forms of the conformal symmetry breaking operators of the
first and second type are special cases of the geometric formulas in Section
\ref{geometric}. Note that the displayed formulas easily confirm the Hodge
conjugation of both types of families.
\end{example}

\begin{example}\label{low-order-3-4-type} For the conformal symmetry breaking operators of the third
and fourth type of order $N \le 4$, we find
\begin{align*}
   D_1^{(0 \to 1)} & = d \iota^*  = d \dot{D}_0^{(0 \to 0)}(0), \\
   D_2^{(0 \to 1)} & = d \iota^* \partial_n = d \dot{D}_1^{(0 \to 0)}(1), \\
   D_3^{(0 \to 1)} & = d \delta d \iota^* + (n\!+\!1) d \iota^* \partial_n^2 = d \dot{D}_{2}^{(0 \to 0)}(2),  \\
   D_4^{(0 \to 1)} & = \frac{n\!+\!1}{3} d \iota^* \partial_n^3 + d \delta d \iota^*
   \partial_n = d \dot{D}_{3}^{(0 \to 0)}(3)
\end{align*}
and
\begin{align*}
   D_1^{(n \to n-2)} & = - \delta \iota^* i_{\partial_n} = \delta \dot{D}_0^{(n \to n-1)}(0), \\
   D_2^{(n \to n-2)} & = - \delta \iota^* i_{\partial_n} \partial_n = \delta \dot{D}_1^{(n \to n-1)}(1),  \\
   D_3^{(n \to n-2)} & = -(n\!+\!1) \delta \iota^* i_{\partial_n} \partial_n^2
   - \delta d \delta \iota^* i_{\partial_n} = \delta \dot{D}_2^{(n \to n-1)}(2), \\
   D_4^{(n \to n-2)} & = -\frac{n\!+\!1}{3} \delta i_{\partial_n} \partial_n^3
   - \delta d \delta i_{\partial_n} \partial_n = \delta \dot{D}_{3}^{(n \to n-1)}(3).
\end{align*}
The displayed formulas easily confirm the Hodge conjugation of both types of
operators.
\end{example}

Next, we confirm the equivariance of the first-order families of both types by
only using {\em direct} arguments.

\begin{bem}\label{equiv-fam-1} We directly demonstrate the infinitesimal equivariance of
the first-order family
$$
   D^{(1\to 1)}_1(\lambda) = \lambda \iota^*\partial_n + \dm \iota^* i_{\partial_n}:
   \Omega^1(\R^n) \to \Omega^1(\R^{n-1})
$$
of the first type (see Remark \ref{LowOrderExampleDiffOp}). Let $1 \le j \le
n-1$. We use $\Omega^1(\R^n) \simeq C^\infty(\R^n) \otimes \Lambda^1(\R^n)^*$
and determine the action of $E_j^+$ using the formula \eqref{adjoint}, where we
identify $(E_j^\pm)^* \simeq dx^j$. The sum in the formula \eqref{adjoint} acts
on the one-form $\omega = f_l dx^l$ by
$$
   \sum_{k=1}^n x_k (dx^j \delta_{kl} - dx^k \delta_{jl})(f_l).
$$
The latter sum reduces on tangential differential forms to a summation over
$k=1,\dots,n-1$, while on normal differential forms the summation is taken over
$k=1,\dots,n$. Thus, for tangential one-forms $\omega = f_l dx^l$,
$l=1,\dots,n-1$, we obtain
\begin{align*}
    & D^{(1\to 1)}_1(\lambda) \dm\pi^\ch_{\lambda,1}(E_j^+)(\omega) \\
    & = D^{(1\to 1)}_1(\lambda)
    \left(-\tfrac 12 \sum_{k=1}^n x_k^2 \partial_{j} dx^l + x_j(-\lambda + \sum_{k=1}^nx_k\partial_{k}) dx^l
    + \sum_{k=1}^{n-1} x_k (dx^j \delta_{kl} - dx^k \delta_{jl} )\right) (f_l) \\
    & = -\tfrac {\lambda}{2} \sum_{k=1}^{n-1} x_k^2 \partial_n
    \partial_j f_l dx^l - \lambda^2 x_j \partial_n f_l dx^l
    + \lambda x_j \sum_{k=1}^{n-1} x_k \partial_n\partial_k f_l dx^l + \lambda x_j \partial_n f_l dx^l \\
    & - \lambda \sum_{k=1}^{n-1} x_k
    \partial_n f_l dx^k \delta_{jl} + \lambda x_l \partial_n f_l dx^j.
\end{align*}
In order to simplify the formulas, we wrote here $f_l$ instead of
$\iota^*(f_l)$. But an analogous calculation shows that the latter sum
coincides with
\begin{equation*}
    \dm\pi^{\prime \ch}_{\lambda-1,1}(E_j^+) D^{(1\to 1)}_1(\lambda)(\omega) =
    \lambda \dm\pi^{\prime \ch}_{\lambda-1,1}(E_j^+) (\partial_n f_l dx^l).
\end{equation*}
Similarly, for normal forms $\omega = f_n dx^n$, we obtain
\begin{align*}
    & D^{(1\to 1)}_1(\lambda) \dm\pi^\ch_{\lambda,1}(E_j^+)(\omega)\\
    & = D^{(1\to 1)}_1(\lambda)
    \left(-\tfrac{1}{2} \sum_{k=1}^n x_k^2 \partial_{j} dx^n + x_j(-\lambda+\sum_{k=1}^n
    x_k\partial_{k}) dx^n + x_n dx^j \right) (f_n) \\
    & = -\tfrac{1}{2} \sum_{k,l=1}^{n-1} x_k^2 \partial_j \partial_l f_n dx^l
    + x_j \sum_{k,l=1}^{n-1} x_k \partial_k \partial_l f_n dx^l \\
    & -(\lambda-1) x_j \sum_{k=1}^{n-1} \partial_k f_n dx^k
    - \sum_{k=1}^{n-1} x_k \partial_j f_n dx^k
    + \sum_{k=1}^{n-1} x_k \partial_k f_n dx^j.
\end{align*}
But an analogous calculation shows that the latter sum coincides with
\begin{equation*}
    \dm\pi^{\prime \ch}_{\lambda-1,1}(E_j^+) D^{(1\to 1)}_1(\lambda)(\omega) =
    \dm\pi^{\prime \ch}_{\lambda-1,1}(E_j^+) (d \iota^* f_n).
\end{equation*}
Thus we have proved that
$$
    D^{(1\to1)}_1(\lambda) \dm\pi^\ch_{\lambda,1}(E_j^+)
    = \dm\pi^{\prime \ch}_{\lambda-1,1}(E_j^+) D^{(1\to 1)}_1(\lambda), \quad 1 \le j \le n-1.
$$
Now $(\dm\pi_{\lambda,1})^\ch(E_j^-) = \partial/\partial x_j$ and $SO(n-1,\R)$
both commute with $D^{(1\to1)}_1(\lambda)$. The assertion is obvious for $X=E$.
This completes the proof of the equivariance for $\gog^\prime(\R)$.
\end{bem}

\begin{bem}\label{equiv-fam-2} Here we sketch an analogous direct proof of the infinitesimal
equivariance of the first-order family
$$
   D_1^{(1 \to 0)}(\lambda) = -(\lambda\!+\!n\!-\!2) \iota^* i_{\partial_n} \partial_n - \delta
   \iota^*: \Omega^1(\R^n) \to C^\infty(\R^{n-1})
$$
of the second type. We shall use the same conventions as in Remark
\ref{equiv-fam-1}. Again, the non-trivial part of the assertion concerns the
actions of $X=E_j^+$. On tangential forms $\omega = f_lx^l$, $l=1,\dots,n-1$,
we obtain
\begin{align*}
   & D^{(1 \to 0)}_1(\lambda) \dm\pi^\ch_{\lambda,1}(E_j^+)(\omega) \\
   & = D^{(1 \to 0)}_1(\lambda)
   \left(-\tfrac 12 \sum_{k=1}^n x_k^2 \partial_{j} dx^l + x_j(-\lambda+\sum_{k=1}^nx_k\partial_{k}) dx^l
   + \sum_{k=1}^{n} x_k (dx^j \delta_{kl} - dx^k \delta_{jl})\right) (f_l) \\
   & = -\tfrac{1}{2} \sum_{k=1}^{n-1} x_k^2 \partial_l \partial_j f_l
   - (\lambda-1) x_j \partial_l f_l + x_j \sum_{k=1}^{n-1} x_k \partial_l \partial_k f_l.
\end{align*}
An analogous calculation shows that the latter sum coincides with
$$
   \dm\pi^{\prime \ch}_{\lambda-1,0}(E_j^+) D^{(1 \to 0)}_1(\lambda)(\omega) =
   \dm\pi^{\prime \ch}_{\lambda-1,0}(E_j^+)(\partial_l f_l).
$$
Similarly, for normal forms $\omega = f_n dx^n$, we obtain
\begin{align*}
    & D^{(1\to 0)}_1(\lambda) \dm\pi^\ch_{\lambda,1}(E_j^+)(\omega) \\
    & = D^{(1\to 0)}_1(\lambda)
    \left(-\tfrac 12 \sum_{k=1}^n x_k^2 \partial_{j} dx^n -\lambda x_j dx^n + x_j \sum_{k=1}^n x_k\partial_{k}
    dx^n + x_n dx^j \right) (f_n) \\
    & = (\lambda\!+\!n\!-\!2) \tfrac{1}{2} \sum_{k=1}^{n-1} x_k^2 \partial_n
    \partial_j f_n - (\lambda\!+\!n\!-\!2) x_j \sum_{k=1}^{n-1} x_k \partial_n \partial_k f_n \\
    & + \lambda(\lambda\!+\!n\!-\!2) x_j \partial_n f_n - (\lambda\!+\!n\!-\!2) x_j \partial_n f_n.
\end{align*}
An analogous calculation shows that the latter sum coincides with
$$
   \dm\pi^{\prime \ch}_{\lambda-1,0}(E_j^+) D^{(1\to 0)}_1(\lambda) (\omega) =
   -(\lambda\!+\!n\!-\!2) \dm\pi^{\prime \ch}_{\lambda-1,0}(E_j^+) (\partial_n f_n).
$$
Thus we have proved that
$$
   D^{(1 \to 0)}_1(\lambda) \dm\pi^\ch_{\lambda,1}(E_j^+)
   = \dm\pi^{\prime \ch}_{\lambda-1,0}(E_j^+) D^{(1\to 0)}_1(\lambda), \quad 1 \le j \le n-1.
$$
This completes the proof of the equivariance for $\gog^\prime(\R)$.
\end{bem}

The first-order families $D_1^{(1 \to 1)}(\lambda)$ and $D_1^{(1 \to
0)}(\lambda)$ studied in Remark \ref{equiv-fam-1} and Remark \ref{equiv-fam-2}
are special cases of conformally covariant families $\Omega^p(X) \to
\Omega^p(M)$ and $\Omega^p(X) \to \Omega^{p-1}(M)$ which are naturally
associated to any codimension one embedding $\iota: M \hookrightarrow X$ of
Riemannian manifolds. More precisely, we have the following results.

\begin{lem}\label{curved-first-type} For any hypersurface $\iota: M \hookrightarrow X$
and any metric $g$ on $X$, the family
$$
   D_1^{(p \to p)}(g;\lambda) \st \lambda \iota^* i_N d + (\lambda\!+\!1) d \iota^*
   i_N - \lambda(\lambda\!+\!1) H \iota^*: \Omega^p(X) \to \Omega^{p}(M)
$$
is conformally covariant in the sense that
$$
   e^{-\lambda \iota^*(\varphi)} D_1^{(p \to p)}(e^{2\varphi}g;\lambda) =
   D_1^{(p \to p)}(g;\lambda) e^{-(\lambda+1)\varphi}
$$
for all $\varphi \in C^\infty(X)$. Here $N$ and $H$ denote the unit normal
vector field of $M$ and the mean curvature, respectively.
\end{lem}

\begin{proof} Using $\hat{N} = e^{-\varphi} N$ and the transformation property
\begin{equation}\label{mean}
    e^{\varphi} \hat{H} = H + \langle d\varphi, N \rangle,
\end{equation}
we find
\begin{align*}
   e^{-\lambda \iota^*(\varphi)} & D_1^{(p \to p)}(e^{2\varphi}g;\lambda)
   \left(e^{(\lambda+1)\varphi} \omega \right) \\
   & = \lambda \iota^* (i_N d \omega + (\lambda\!+\!1) i_N (d\varphi \wedge \omega))
   + (\lambda\!+\!1) e^{-\lambda \iota^*(\varphi)} d \iota^* e^{-\varphi} i_N
   \left(e^{(\lambda+1)\varphi} \omega \right) \\
   & -\lambda(\lambda\!+\!1) (H + \langle d\varphi,N \rangle) \iota^* \omega
\end{align*}
for any $\omega \in \Omega^p(M)$. A simple computation shows that the latter
sum equals
$$
\lambda \iota^* i_N d \omega + (\lambda\!+\!1) d \iota^* i_N \omega -
\lambda(\lambda\!+\!1) H \iota^* \omega.
$$
The proof is complete.
\end{proof}

The conformally covariant family $D_1^{(p\to p)}(g;\lambda)$ (in Lemma
\ref{curved-first-type}) clearly generalizes the family $D_1^{(p \to
p)}(\lambda)$ (displayed in Example \ref{LowOrderExampleDiffOp}) (up to the
shift $\lambda \mapsto \lambda-p+1$). Therefore, Lemma \ref{curved-first-type}
yields an alternative proof of the equivariance of $D_1^{(p \to p)}(\lambda)$.

\begin{lem}\label{curved-second-type} For any hypersurface $\iota: M^{n-1} \hookrightarrow X^n$
and any metric $g$ on $X$, the family
\begin{align*}
   D_1^{(p \to p-1)}(g;\lambda) & \st (n\!-\!2p\!+\!\lambda\!+\!1) \iota^* \delta_g
   - (n\!-\!2p\!+\!\lambda\!+\!2) \delta_{\iota^* g} \iota^* \\
   & + (n\!-\!2p\!+\!\lambda\!+\!1)(n\!-\!2p\!+\!\lambda\!+\!2) i_{\H}: \; \Omega^p(X) \to \Omega^{p-1}(M)
\end{align*}
is conformally covariant in the sense that
$$
   e^{-\lambda \iota^*(\varphi)} D_1^{(p \to p-1)}(e^{2\varphi} g;\lambda) =
   D_1^{(p \to p-1)}(g;\lambda) e^{-(\lambda+2)\varphi}
$$
for all $\varphi \in C^\infty(X)$. Here $\H = H N$ denotes the mean-curvature
vector.
\end{lem}

\begin{proof} The conformal transformation law
$$
   e^{-(a-2)\varphi} \circ \hat{\delta} \circ e^{a\varphi} = \delta -
   (n\!-\!2p\!+\!a) i_{\grad(\varphi)}
$$
for $p$-forms on a manifold of dimension $n$ implies
$$
   e^{-\lambda \varphi} \circ \delta_{\hat{g}} = \delta_g \circ
   e^{-(\lambda+2)\varphi} - (n\!-\!2p\!+\!\lambda\!+\!2) i_{\grad_g(\varphi)} \circ e^{-(\lambda+2)\varphi}
$$
for $\varphi \in C^\infty(X)$ on $p$-forms on $X$, and
$$
   e^{-\lambda \psi} \circ \delta_{\hat{g}} = \delta_{g} \circ e^{-(\lambda+2) \psi}
   - (n\!-\!2p\!+\!\lambda\!+\!1) i_{\grad_{g}(\psi)} \circ e^{-(\lambda+2) \psi}
$$
for $\psi \in C^\infty(M)$ on $p$-forms on $M$. It follows that
$$
   e^{-\lambda \iota^*(\varphi)} \circ \left((n\!-\!2p\!+\!\lambda\!+\!1) \iota^*
   \delta_{\hat{g}} - (n\!-\!2p\!+\!\lambda\!+\!2) \delta_{\iota^* \hat{g}} \iota^*\right)
$$
differs from
$$
   \left((n\!-\!2p\!+\!\lambda\!+\!1) \iota^* \delta_{\hat{g}} -
   (n\!-\!2p\!+\!\lambda\!+\!2) \delta_{\iota^* \hat{g}} \iota^*\right) \circ e^{(\lambda-2) \varphi}
$$
by
$$
   -(n\!-\!2p\!+\!\lambda\!+\!1)(n\!-\!2p\!+\!\lambda\!+\!2) \iota^* i_{N(\varphi)} \circ e^{(\lambda-2)\varphi},
$$
where $N(\varphi) = \langle d\varphi,N \rangle N$. But the conformal
transformation law \eqref{mean} implies that
$$
   e^{-\lambda \iota^* (\varphi)} \circ i_{\hat{\H}}
   = (i_\H + i_{N(\varphi)}) \circ e^{-(\lambda+2) \varphi}.
$$
This completes the proof.
\end{proof}

The conformally covariant family $D_1^{(p\to p-1)}(g;\lambda)$ (in Lemma
\ref{curved-second-type}) clearly generalizes the family $D_1^{(p \to
p-1)}(\lambda)$ (displayed in Example \ref{LowOrderExampleDiffOp}) (up to the
shift $\lambda \mapsto \lambda-p+2$). Therefore, Lemma \ref{curved-second-type}
yields an alternative proof of the equivariance of $D_1^{(p \to
p-1)}(\lambda)$.

The following remark illustrates Theorem \ref{fam-mid-even}.


\begin{bem}\label{history} In dimension $n=2$, the conformal symmetry breaking operators
$D_1^{(1 \to 1)}(\lambda)$ and $D_2^{(1 \to 1)}(\lambda)$ of the first type
appeared already in \cite[Equations (8.200)--(8.202)]{Juhl0} as low-order
special cases of so-called relative differential intertwining operators which
are induced by homomorphisms of generalized Verma modules. In more details, the
relation is the following. In terms of the coordinates $x,y$ on $\R^2$ with the
normal variable $y$ of the subspace $\R^1$, we have the formulas
$$
   D_1^{(1 \to 1)}(\lambda) = -\lambda \iota^* \partial_y + \dm \iota^* i_{\partial_y}
$$
and
$$
   D_2^{(1 \to 1)}(\lambda) = -(\lambda\!+\!1) \partial_x^2 \iota^* +
   (2\lambda\!-\!1)(\lambda\!-\!1) \iota^* \partial_y^2 + 2(2\lambda\!-\!1) \dm \iota^* i_{\partial_y} \partial_y
$$
(see Example \ref{LowOrderExampleDiffOp}). By restriction to $\omega = f (dx
\pm i dy) \in \Omega^1(\R^2)$, we find
$$
   D_1^{(1 \to 1)}(\lambda)(\omega) = \mp i \lambda \iota^* (\partial_y f) dx + \dm \iota^* (f)
$$
and
$$
   D_2^{(1 \to 1)}(\lambda) (\omega) = -(\lambda\!+\!1) \iota^* (\partial_x^2 f) dx +
   (2\lambda\!-\!1)(\lambda\!-\!1) \iota^* (\partial_y^2 f) dx \pm 2 i(2\lambda\!-\!1) \dm \iota^* (\partial_y f).
$$
By Lemma \ref{form-rep}, the subspaces $\{f (dx \pm dy)\}$ are the eigenspaces
of the Hodge star operator on $\R^2$. Using respective shifts of $\lambda$ by
$1$ and $2$, we obtain the operators
$$
   f (dx \pm i dy) \mapsto \mp i (\lambda+1) \iota^* (\partial_y f) dx + \dm \iota^* (f)
$$
and
$$
   f (dx \pm i dy) \mapsto -(\lambda\!+\!3) \iota^* (\partial_x^2 f) dx +
   (2\lambda\!+\!3)(\lambda\!+\!1) \iota^* (\partial_y^2 f) dx \pm 2 i(2\lambda+3) \dm \iota^* (\partial_y f).
$$
These are the operators displayed in \cite{Juhl0}. This shows that the family
$$
   D_2^{(1 \to 1)}(\lambda): \Omega^1 (\R^2) \to \Omega^1(\R)
$$
of the first type restricts to two different families $\Omega^1_\pm(\R^2) \to
\Omega^1(\R)$.
\end{bem}

\begin{bem}\label{special-2} The special case $\Omega^1(\R^2) \to C^\infty(\R)$
of the second type families of order $N \in \N$ was analyzed in \cite{KKP}. We
show that their result is a special case of ours. We use the expansions
\begin{align*}
   C_m^{\alpha}(t) =
   \sum_{k=0}^{[m/2]}(-1)^k \frac{\Gamma(m-k+\alpha)}{\Gamma(\alpha) (m-2k)!k!}(2t)^{m-2k}, \; m \in \N_0
\end{align*}
(see \eqref{Gegenbauer-Taylor}) to define homogeneous differential operators
\begin{align*}
    C^\alpha_m \st ({\rm i} \partial_x)^m  C_m^{\alpha}\left(\frac{\partial_y}{{\rm i} \partial_x}\right).
\end{align*}
In particular, we find
\begin{align*}
   C^\alpha_0 & = \id,\\
   C^\alpha_1 & = 2\lambda\partial_y,\\
   C^\alpha_2 & = \lambda(\partial_x^2+ 2(\lambda+1) \partial_y^2),\\
   C^\alpha_3 & = \frac 23 \lambda(\lambda+1)(3 \partial_x^2\partial_y + 2(\lambda+2)\partial_y^3  ).
\end{align*}
Furthermore, we define the operators
\begin{equation}\label{eq:KKPOp1}
    D^1_m(\lambda) \st m(2\lambda\!+\!m\!-\!1)\partial_x ({\rm i} \partial_x)^{m-1}
    C^{\lambda + \frac 12}_{m-1}\left(\frac{\partial_y}{{\rm i}\partial_x}\right)
\end{equation}
and
\begin{align}\label{eq:KKPOp2}
    D^2_m(\lambda) & \st (2\lambda^2\!+\!2(m\!-\!1)\lambda\!+\!m(m\!-\!1))\partial_y ({\rm i} \partial_x)^{m-1}
    C^{\lambda+\frac 12}_{m-1}\left(\frac{\partial_y}{{\rm i}\partial_x}\right) \notag \\
    & +(\lambda\!-\!1)(2\lambda\!+\!1)(\partial_x^2 + \partial_y^2)
    ({\rm i} \partial_x)^{m-2}C^{\lambda+\frac 32}_{m-2}\left(\frac{\partial_y}{{\rm i}\partial_x}\right).
\end{align}
Let $\iota: \R \hookrightarrow \R^2$ be defined by $x \mapsto (x,0)$.

\begin{prop}\label{KKP-rel} Let $N\in\N$. The operators $D^1_N(\lambda)$ and $D^2_N(\lambda)$ are related
to the conformal symmetry breaking operators $D^{(1\to 0)}_{N}(\lambda)$ of the
second type through the identities
\begin{align}
   \iota^* (D^1_{2N}(\lambda)(f) + D^2_{2N}(\lambda)(g)) & = (-1)^N \frac{2 (\lambda+\frac 12)_N}{(N-1)!}
   D^{(1\to 0)}_{2N}(-\lambda)(f\dm x + g\dm y) \label{eq:KKPId1}
\end{align}
and
\begin{equation}
   \iota^*(D^1_{2N+1}(\lambda)(f) + D^2_{2N+1}(\lambda)(g))
   = (-1)^N \frac{2(2N\!+\!1) (\lambda\!+\!\frac 12)_N (\lambda\!+\!N)}{N!}
   D^{(1\to 0)}_{2N+1}(-\lambda)(f\dm x + g\dm y) \label{eq:KKPId2}
\end{equation}
for $f, g \in C^\infty(\R^2)$.
\end{prop}

\begin{proof} The proof rests on the identities
\begin{align}
   C_{2N}^{\lambda+\frac 12}(z)
   &=\frac{(\lambda+\frac 12)_N}{N!}\sum_{j=0}^N (-1)^j a_j^{(N)}(-\lambda-1) z^{2N-2j},\notag\\
   C_{2N+1}^{\lambda+\frac 12}(z)
   &=\frac{2(\lambda+\frac 12)_{N+1}}{N!}\sum_{j=0}^N (-1)^j b_j^{(N)}(-\lambda-1) z^{2N-2j+1} \label{eq:GegPoly}
\end{align}
(see the Appendix). By formula \eqref{even-type2} in Theorem
\ref{EvenDiffOp-type2} the even-order family $D_{2N}^{(1\to 0)}(\lambda)$ acts
on $\omega = f dx + g dy \in \Omega^1(\R^2)$ by
\begin{equation*}
   D^{(1\to 0)}_{2N}(\lambda)(\omega) = \iota^*(D^{(1\to 0),1}_{2N}(\lambda)(f)+D^{(1\to 0),2}_{2N}(\lambda)(g)),
\end{equation*}
where
\begin{align*}
   D^{(1\to 0),1}_{2N}(\lambda) & \st (-1)^{N} ({\rm i}\partial_x)^{2N-1}\partial_x
   \sum_{j=0}^{N-1} (-1)^j q^{(N-1)}_j(\lambda-1) \left(\frac{\partial_y}{{\rm i}\partial_x}\right)^{2N-2j-1}, \\
   D^{(1\to 0),2}_{2N}(\lambda) & \st (-1)^{N}({\rm i}\partial_x)^{2N}
   \sum_{j=0}^N (-1)^j p_j(\lambda;N,1) \left(\frac{\partial_y}{{\rm i}\partial_x}\right)^{2N-2j}
\end{align*}
and
\begin{align*}
   q^{(N-1)}_j(\lambda\!-\!1) & = -2N(2\lambda\!-\!2N\!+\!1) b^{(N-1)}_j(\lambda-1),\\
   p_j(\lambda;N,1) & =-(\lambda\!-\!2N\!+\!2j\!+\!1) a_j^{(N)}(\lambda).
\end{align*}
It follows that the claim \eqref{eq:KKPId1} is equivalent to the relations
\begin{align}
   D^1_{2N}(\lambda) & =(-1)^N \frac{2 (\lambda+\frac 12)_N}{(N-1)!}
   D^{(1\to 0),1}_{2N}(-\lambda), \label{eq:H1} \\
   D^2_{2N}(\lambda) & =(-1)^N \frac{2 (\lambda+\frac 12)_N}{(N-1)!}
   D^{(1\to 0),2}_{2N}(-\lambda). \label{eq:H2}
\end{align}
Now \eqref{eq:H1} directly follows from the definition \eqref{eq:KKPOp1}. We
proceed with the proof of \eqref{eq:H2}. By definition and \eqref{eq:GegPoly},
the left-hand side equals
\begin{align*}
   D^2_{2N}(\lambda) & = \frac{2(\lambda\!+\!\frac 12)_N}{(N\!-\!1)!} \Bigg[
   \sum_{j=0}^{N}(-1)^j \Big[\big(2N(2\lambda\!+\!2N\!-\!1)+2\lambda(\lambda\!-\!1)\big)
   b_j^{(N-1)}(-\lambda\!-\!1)\\
   & + (\lambda\!-\!1) a_{j-1}^{(N-1)}(-\lambda\!-\!2) + (\lambda\!-\!1)a_j^{(N-1)}(-\lambda\!-\!2)\Big]
   ({\rm i}\partial_x)^{2j} \partial_y^{2N-2j}\Bigg].
\end{align*}
Here we have set $a_{-1}^{(N-1)}(\lambda) \st 0$, $a_{N}^{(N-1)}(\lambda) \st
0$ and $b_{N}^{(N-1)}(\lambda) \st 0$. Now an elementary computation shows that
\begin{equation*}
   2\lambda b_j^{(N-1)}(-\lambda\!-\!1) + a_{j-1}^{(N-1)}(-\lambda\!-\!2) +
   a_j^{(N-1)}(-\lambda\!-\!2) = a^{(N)}_j(-\lambda)
\end{equation*}
for $j=0,\ldots,N$. Hence the identity
\begin{equation*}
   p_j(-\lambda;N,1) = (\lambda\!-\!1) a_j^{(N)}(-\lambda) + 2N(2\lambda\!+\!2N\!-\!1) b^{(N-1)}_j(-\lambda\!-\!1)
\end{equation*}
proves \eqref{eq:H2}. Similar arguments can be used to prove \eqref{eq:KKPId2}.
The proof is complete.
\end{proof}

The main result of \cite{KKP} states that, for any $m \in \N_0$, the family
$$
   \Omega^1(\R^2) \ni \omega = f dx + g dy
   \mapsto \iota^* (D^1_{m}(\lambda)(f) + D^2_{m}(\lambda)(g)) \in C^\infty(\R)
$$
satisfies the same intertwining relation as $D_m^{(1,0)}(-\lambda)$. Moreover,
the authors observed that the compositions of these families with the Hodge
star operator on $\Omega^1(\R^2)$ define additional families with the same
equivariance. These results follow from Proposition \ref{KKP-rel}.
\end{bem}

\section{Geometric formulas for conformal symmetry breaking operators}\label{geometric}

In the present section, we apply the results in Section \ref{DiffOp} to derive
formulas for all types of conformal symmetry breaking operators in terms of the
four geometric operators $\dm$, $\delta$, $\bar{\dm}$, $\bar{\delta}$, the
pull-back $\iota^*$ and the insertion $i_{\partial_n}$ of the normal vector
field. We shall refer to these formulas as to geometric formulas for the
families. These results generalize the low-order examples ($N \le 3$) displayed
in Example \ref{LowOrderExampleDiffOp}.

\subsection{Preparations}

We note that Lemma \ref{DiffCoDiff} implies the identities
\begin{align}
   \iota^*\partial_n^{2k} & = \sum_{i=0}^k (-1)^i \binom{k}{i}
   \Delta^{k-i}\iota^* \bar{\Delta}^i, \label{help-3} \\
   \iota^* i_{\partial_n} \partial_n^{2k+1} & = \sum_{i=0}^k (-1)^i
   \binom{k}{i} \Delta^{k-i}(\delta\iota^*-\iota^*\bar{\delta}) \bar{\Delta}^i \label{help-4}
\end{align}
and
\begin{align}
   \iota^*i_{\partial_n} \partial_n^{2k} & =
   \sum_{i=0}^{k}(-1)^i \binom{k}{i}\Delta^{k-i}\iota^* i_{\partial_n} \bar{\Delta}^i, \label{help-10} \\
   \iota^*\partial_n^{2k-1} & =
   \sum_{i=0}^{k-1}(-1)^i \binom{k-1}{i} \Delta^{k-i-1} (\dm\iota^* i_{\partial_n}
   + \iota^* i_{\partial_n} \bar{\dm}) \bar{\Delta}^i. \label{help-11}
\end{align}
Indeed, Lemma \ref{DiffCoDiff}/(3) yields $\partial_n^2 = \Delta -
\bar{\Delta}$. Hence $\partial_n^{2k} = \sum_{i=0}^k (-1)^i {k \choose i}
\Delta^{k-i} \bar{\Delta}^i$. This proves \eqref{help-3}. Moreover,
$i_{\partial_n} \partial_n = \delta - \bar{\delta}$ (Lemma
\ref{DiffCoDiff}/(2)) gives
\begin{align*}
   \iota^* i_{\partial_n} \partial_n^{2k+1} = \iota^* i_{\partial_n} \partial_n
   \partial_n^{2k} & = (\delta \iota^* - \iota^* \bar{\delta})
   \sum_{i=0}^k (-1)^i {k \choose i} \Delta^{k-i} \bar{\Delta}^i \\
   & = \sum_{i=0}^k (-1)^k {k \choose i} \Delta^{k-i} (\delta \iota^* - \iota^* \bar{\delta}) \bar{\Delta}^i.
\end{align*}
This proves \eqref{help-4}. Similar arguments prove \eqref{help-10} and
\eqref{help-11}.

Furthermore, we introduce the coefficients
\begin{equation}\label{a-even}
   \alpha_i^{(N)}(\lambda) \st (-1)^{i} 2^N \frac{N!}{(2N)!} \binom{N}{i}
   \prod_{k=i+1}^N (2\lambda\!+\!n\!-\!2k) \prod_{k=1}^i (2\lambda\!+\!n\!-\!2k\!-\!2N\!+\!1)
\end{equation}
and
\begin{equation}\label{b-even}
   \beta_i^{(N)}(\lambda) \st (-1)^i 2^N \frac{N!}{(2N\!+\!1)!} \binom{N}{i}
   \prod_{k=i+1}^N (2\lambda\!+\!n\!-\!2k) \prod_{k=1}^i (2\lambda\!+\!n\!-\!2k\!-\!2N\!-\!1)
\end{equation}
for $N \in \N_0$ and $i=0,\dots,N$. By convention, empty products are defined
as $1$.

The following observation will be useful to identify certain series as
hypergeometric functions. A series $\sum_{n \ge 0} c_n$ is a hypergeometric
function
$$
   c_0 \, {}_2F_1(a,b;c;x) = c_0  \sum_{n \ge 0} \frac{(a)_n (b)_n}{(c)_n} \frac{x^n}{n!}
$$
iff
$$
   \frac{c_{n+1}}{c_n} = \frac{(n+a)(n+b)}{n+c} \frac{x}{n+1}.
$$
We also recall the Zhu-Vandermonde formula
\begin{equation}\label{ZV}
   {}_2F_1(-n,b;c;1) = \sum_{j = 0}^n \frac{(-n)_j (b)_j}{(c)_j} \frac{1}{j!} =
   \frac{(c-b)_n}{(c)_n}, \; n \in \N_0.
\end{equation}

\begin{bem}\label{coeff-Jacobi} We identify the generating polynomials for the coefficients
$\alpha_i^{(N)}(\lambda)$ and $\beta_i^{(N)}(\lambda)$ as Jacobi polynomials.
Indeed, the relations
\begin{align*}
   \binom{N}{i} & =(-1)^{i} \frac{(-N)_i}{i!},\\
   \prod_{k=i+1}^N (2\lambda\!+\!n\!-\!2k) & = 2^{N-i}(\lambda\!+\!\tfrac n2\!-\!N)_{N-i}
   = (-2)^{N-i} \frac{(-\lambda\!-\!\tfrac n2\!+\!1)_N}{(-\lambda\!-\!\tfrac n2\!+\!1)_i}, \\
   \prod_{k=1}^i (2\lambda\!+\!n\!-\!2k\!-\!2N\!+\!1) & =(-2)^i (N\!+\!\tfrac 12\!-\!\lambda\!-\!\tfrac n2)_i
\end{align*}
imply that
\begin{equation*}
   \sum_{i=0}^N \alpha_i^{(N)}(\lambda) t^i = 4^N  \frac{N!}{(2N)!} (\lambda \!+\! \tfrac
   n2\! - \!N)_N \; {}_2F_1\left[\begin{matrix} -N, N\!+\!\tfrac 12\!-\!\lambda \!-\!\ tfrac n2 \\
   1\!-\!\lambda\!-\!\tfrac n2 \end{matrix};t\right].
\end{equation*}
The right-hand side is proportional to
$P_N^{(-\lambda-\frac{n}{2},-\frac{1}{2})}(1\!-\!2t)$. Similarly, we find
\begin{equation*}
   \sum_{i=0}^N \beta_i^{(N)}(\lambda) t^i = 4^N  \frac{N!}{(2N\!+\!1)!} (\lambda\!+\!\tfrac
   n2\!-\!N)_N \; {}_2F_1\left[\begin{matrix} -N,  N\!+\!\tfrac 32\!-\!\lambda\!-\!\tfrac n2 \\
   1\!-\!\lambda\!-\!\tfrac n2 \end{matrix};t\right].
\end{equation*}
The right-hand side is proportional to
$P_N^{(-\lambda-\frac{n}{2},\frac{1}{2})}(1\!-\!2t)$. For the definition of
Jacobi polynomials $P_N^{(\alpha,\beta)}(t)$ we refer to the Appendix.
\end{bem}

The following relations will be useful later on.

\begin{lem}\label{RelationJacobiGegenbauer} We have
\begin{align*}
   \sum_{j=0}^{N-i}(-1)^{N-j-i} \binom{N-j}{i}a_j^{(N)}(\lambda) & = \alpha^{(N)}_i(\lambda), \\
   \sum_{j=0}^{N-i}(-1)^{N-j-i} \binom{N-j}{i}b_j^{(N)}(\lambda) & = \beta^{(N)}_i(\lambda)
\end{align*}
for all $i=0,\dots,N$.
\end{lem}

\begin{proof} First, we note that Gegenbauer coefficients can be written in the
form
\begin{equation}\label{A-coeff}
   a_j^{(N)}(\lambda) = (-4)^{N-j} \frac{N!}{j!(2N\!-\!2j)!}
   \frac{(\lambda\!+\!\tfrac{n}{2}\!-\!2N\!+\!\tfrac 12)_N}{(\lambda\!+\!\tfrac{n}{2}\!-\!2N\!+\!\tfrac 12)_j}
\end{equation}
and
\begin{equation}\label{B-coeff}
   b_j^{(N)}(\lambda) = (-4)^{N-j} \frac{N!}{j!(2N\!-\!2j\!+\!1)!}
   \frac{(\lambda\!+\!\tfrac n2\!-\!2N\!-\!\tfrac 12)_N}{(\lambda\!+\!\tfrac n2\!-\!2N\!-\!\tfrac 12)_j}
\end{equation}
for $0\leq j\leq N-1$ and $a_N^{(N)}(\lambda) = b_N^{(N)}(\lambda)=1$. Now
\eqref{A-coeff} implies
\begin{equation*}
   \sum_{j=0}^{N-i}(-1)^{N-j-i} \binom{N-j}{i} a_j^{(N)}(\lambda) = \sum_{j=0}^{N-i} c_j
\end{equation*}
with
$$
   \frac{c_{j+1}}{c_j} = \frac{(-N\!+\!i\!+\!j)(-N\!+\tfrac{1}{2}\!+\!j)}
   {(\lambda\!+\!\tfrac{n}{2}\!-\!2N\!+\!\tfrac{1}{2}\!+\!j)} \frac{1}{j\!+\!1}.
$$
Hence
$$
   \sum_{j=0}^{N-i}(-1)^{N-j-i} \binom{N-j}{i} a_j^{(N)}(\lambda) = (-1)^{N-i}
   \binom{N}{i} a_0^{(N)}(\lambda) \; {}_2F_1\left[ \begin{matrix} \tfrac{1}{2}-N,i-N\\
   \lambda+\tfrac{n}{2}-2N+\tfrac{1}{2} \end{matrix} ;1\right].
$$
By the Zhu-Vandermonde formula \eqref{ZV}, we obtain
\begin{align*}
   \sum_{j=0}^{N-i}(-1)^{N-j-i} & \binom{N-j}{i} a_j^{(N)}(\lambda) \\
   & = (-1)^{N-i}
   \binom{N}{i} (-4)^N \frac{N!}{(2N)!} (\lambda\!+\!\tfrac{n}{2}\!-\!2N\!+\!\tfrac{1}{2})_N
   \frac{(\lambda\!+\!\tfrac{n}{2}\!-\!N)_{N-i}}{(\lambda\!+\!\tfrac{n}{2}\!-\!2N+\!\tfrac{1}{2})_{N-i}} \\
   & = \alpha^{(N)}_i(\lambda).
\end{align*}
Similar arguments using \eqref{B-coeff} prove the second relation.
\end{proof}

\subsection{Even-order families of the first and second type}\label{case-even}

The following result for even-order families of the first type basically
restates Theorem \ref{main-even} in Section \ref{intro}.

\begin{theorem}\label{coeffeven} Assume that $N \in \N$ and $p=0,\dots,n-1$.
The even-order families $D_{2N}^{(p \to p)}(\lambda)$ of the first type can be
written in the form
\begin{align}\label{SBO-even}
   D_{2N}^{(p \to p)}(\lambda)
   & = (\lambda\!+\!p) \sum_{i=0}^N \alpha_i^{(N)}(\lambda)
   (\dm\delta)^{N-i} \iota^* (\bar{\dm} \bar{\delta})^i \notag \\
   & + \sum_{i=1}^{N-1} (\lambda\!+\!p\!-\!2i) \alpha_i^{(N)}(\lambda)
   (\dm \delta)^{N-i} \iota^* (\bar{\delta} \bar{\dm})^i \notag \\
   & + (\lambda\!+\!p\!-\!2N) \sum_{i=0}^N \alpha_i^{(N)}(\lambda)
   (\delta \dm)^{N-i} \iota^*(\bar{\delta} \bar{\dm})^i
\end{align}
with the coefficients $\alpha_i^{(N)}(\lambda)$ defined by \eqref{a-even}.
\end{theorem}

For $p=0$, the families in Theorem \ref{coeffeven} reduce to the product of
$(\lambda-2N)$ and the equivariant even-order families studied in \cite{Juhl}.

\begin{proof} In the following, we shall use the conventions $r_{-1}^{(N-1)}(\lambda) = 0$ and
$q_{-1}^{(N-1)}(\lambda)=0$. We start by proving the formula
\begin{align}\label{even-prep}
   D_{2N}^{(p \to p)}(\lambda) & = \sum_{i=0}^{N} \sum_{j=0}^{N-i} (-1)^{N-j-i} \binom{N-j}{i}
   p_j^{(N)}(\lambda;p) (\delta\dm)^{N-i} \iota^* (\bar{\delta}\bar{\dm})^i \notag \\
   & + \sum_{i=0}^{N} \sum_{j=0}^{N-i} (-1)^{N-j-i} \binom{N-j}{i} S_j(\lambda;N,p)
   (\dm\delta)^{N-i} \iota^* (\bar{\dm}\bar{\delta})^{i} \notag \\
   & + \sum_{i=1}^{N-1} \sum_{j=0}^{N-i} (-1)^{N-j-i} \binom{N-j}{i} T_j(\lambda;N,p)
   (\dm\delta)^{N-i} \iota^* (\bar{\delta}\bar{\dm})^i
\end{align}
with
$$
   S_j(\lambda;N,p) \st \left[p_j^{(N)}(\lambda;p) + r_{j-1}^{(N-1)}(\lambda\!-\!1)
   + q_j^{(N-1)}(\lambda\!-\!1)\right], \; j=0,\dots,N
$$
and
$$
   T_j(\lambda;N,p) \st \left[p_j^{(N)}(\lambda;p) + r_{j-1}^{(N-1)}(\lambda\!-\!1)
   - q_{j-1}^{(N-1)}(\lambda\!-\!1)\right], \; j=0,\dots,N-1.
$$
By combining Theorem \ref{EvenDiffOp-type1} with the relations \eqref{help-3}
and \eqref{help-4}, we obtain
\begin{align*}
   D_{2N}^{(p \to p)}(\lambda) & = \sum_{j=1}^{N} \sum_{i=0}^{N-j}(-1)^{N-j-i} \binom{N-j}{i}
   \left[ p_j^{(N)}(\lambda;p) + r_{j-1}^{(N-1)}(\lambda\!-\!1) \right]
   (\dm \delta)^{N-i}\iota^*\bar{\Delta}^{i} \\
   & + \sum_{i=0}^N (-1)^{N-i} \binom{N}{i} p_0^{(N)}(\lambda;p)
   \left((\dm \delta)^{N-i} \!+\! (\delta \dm)^{N-i}\right) \iota^* \bar{\Delta}^i \\
   & - p_0^{(N)}(\lambda;p) \iota^* \bar{\Delta}^N  \\
   & + \sum_{j=1}^{N} \sum_{i=0}^{N-j} (-1)^{N-j-i} \binom{N-j}{i}
   p_j^{(N)}(\lambda;p) (\delta \dm)^{N-i} \iota^*\bar{\Delta}^{i} \\
   & + \sum_{j=1}^{N} \sum_{i=0}^{N-j} (-1)^{N-j-i-1} \binom{N-j}{i}
   q_{j-1}^{(N-1)}(\lambda\!-\!1) (\dm \delta)^{N-i}\iota^*\bar{\Delta}^{i}\\
   & + \sum_{j=1}^{N} \sum_{i=0}^{N-j} (-1)^{N-j-i} \binom{N-j}{i}
   q_{j-1}^{(N-1)}(\lambda\!-\!1) (\dm \delta)^{N-i-1}\iota^* (\bar{\dm} \bar{\delta})^{i+1}
\end{align*}
using $(\dm\delta)^j \Delta^{N-j-i} = (\dm\delta)^{N-i}$ if $j \ge 1$. Next, we
expand the powers of Laplacians $\bar{\Delta}^i$. We obtain
\begin{align*}
   D_{2N}^{(p \to p)}(\lambda) & = \sum_{j=0}^{N} \sum_{i=0}^{N-j} (-1)^{N-j-i} \binom{N-j}{i}
   \left[ p_j^{(N)}(\lambda;p) + r_{j-1}^{(N-1)}(\lambda\!-\!1) \right]
   (\dm \delta)^{N-i}\iota^*(\bar{\dm} \bar{\delta})^{i} \\
   & + \sum_{j=0}^N \sum_{i=1}^{N-j} (-1)^{N-j-i} \binom{N-j}{i}
   \left[p_j^{(N)}(\lambda;p) + r_{j-1}^{(N-1)}(\lambda\!-\!1)\right]
   (\dm \delta)^{N-i} \iota^*(\bar{\delta} \bar{\dm})^{i} \\
   & + \sum_{j=0}^N \sum_{i=0}^{N-j} (-1)^{N-j-i} \binom{N-j}{i}
   p_j^{(N)}(\lambda;p) (\delta \dm)^{N-i} \iota^*(\bar{\dm}\bar{\delta})^i \\
   & + \sum_{j=0}^N \sum_{i=1}^{N-j} (-1)^{N-j-i} \binom{N-j}{i}
   p_j^{(N)}(\lambda;p) (\delta \dm)^{N-i} \iota^*(\bar{\delta}\bar{\dm})^{i} \\
   & + \sum_{j=0}^{N-1} \sum_{i=1}^{N-j} (-1)^{N-j-i} \binom{N-j}{i}
   q_{j}^{(N-1)}(\lambda\!-\!1) (\dm \delta)^{N-i} \iota^*(\bar{\dm}\bar{\delta})^{i}\\
   & + \sum_{j=1}^{N} \sum_{i=1}^{N-j}(-1)^{N-j-i-1} \binom{N-j}{i}
   q_{j-1}^{(N-1)}(\lambda\!-\!1) (\dm \delta)^{N-i} \iota^*(\bar{\delta}
   \bar{\dm})^i \\
   & + \sum_{j=1}^N (-1)^{N-j-1} q_{j-1}^{(N-1)}(\lambda\!-\!1) (\dm \delta)^N
   \iota^* \\
   & - p_0^{(N)}(\lambda;p) \iota^* ((\bar{\dm} \bar{\delta})^N + (\bar{\delta} \bar{\dm})^N).
\end{align*}
By $\dm \iota^* \bar{\dm} = 0$, the third sum equals
$$
   \sum_{j=0}^N (-1)^{N-j} p_j^{(N)}(\lambda;p) (\delta \dm)^{N} + p_0^{(N)}(\lambda;p) (\bar{\dm} \bar{\delta})^N.
$$
Further simplifications and interchanges of summations yields
\eqref{even-prep}.

Now the identities
\begin{align}
   \sum_{j=0}^{N-i} (-1)^{N-j-i} \binom{N-j}{i} p_j^{(N)}(\lambda;p)
   & = (\lambda\!+\!p\!-\!2N) \alpha_i^{(N)}(\lambda), \label{id-1} \\
   \sum_{j=0}^{N-i} (-1)^{N-j-i} \binom{N-j}{i} S_j(\lambda;N,p)
   & = (\lambda\!+\!p) \alpha_i^{(N)}(\lambda), \label{id-2} \\
   \sum_{j=0}^{N-i} (-1)^{N-j-i}
   \binom{N-j}{i}  T_j(\lambda;N,p) & = (\lambda\!+\!p\!-\!2i) \alpha_i^{(N)}(\lambda) \label{id-3}
\end{align}
for $i=0,\dots,N$ imply that formula \eqref{even-prep} is equivalent to
\begin{align*}
   D_{2N}^{(p \to p)}(\lambda)
   & = (\lambda\!+\!p\!-\!2N) \sum_{i=0}^N \alpha_i^{(N)}(\lambda) (\delta\dm)^{N-i} \iota^* (\bar{\delta}\bar{\dm})^i \\
   & + (\lambda\!+\!p) \sum_{i=0}^N
   \alpha_i^{(N)}(\lambda) (\dm\delta)^{N-i} \iota^* (\bar{\dm}\bar{\delta})^{i} \\
   & + \sum_{i=1}^{N-1} (\lambda\!+\!p\!-\!2i) \alpha_i^{(N)}(\lambda) (\dm\delta)^{N-i} \iota^*
   (\bar{\delta}\bar{\dm})^{i}.
\end{align*}
This proves the theorem.

It remains to prove the identities \eqref{id-1}--\eqref{id-3}. \eqref{id-1} is
a direct consequence of the first identity in Lemma
\ref{RelationJacobiGegenbauer} using
$$
   p^{(N)}_j(\lambda;p) = (\lambda\!+\!p\!-\!2N) a_j^{(N)}(\lambda).
$$
Next, we observe that
\begin{align*}
   r_{j-1}^{(N-1)}(\lambda\!-\!1) & = 2N a_{j-1}^{(N-1)}(\lambda\!-\!1) \qquad
   \mbox{(by definition of $r_j^{(N-1)}(\lambda)$)} \\
   & = 2j a_j^{(N)}(\lambda) \quad \mbox{(from the definition of even Gegenbauer coefficients)}
\end{align*}
and
$$
   q_j^{(N-1)}(\lambda\!-\!1) = (2N\!-\!2j) a_j^{(N)}(\lambda)
   \quad \mbox{(from the definition of even Gegenbauer coefficients)}.
$$
Hence
$$
   S_j(\lambda;N,p) = (\lambda\!+\!p) a_j^{(N)}(\lambda).
$$
Thus the first identity in Lemma \ref{RelationJacobiGegenbauer} implies
\eqref{id-3}. Finally, a calculation shows that
\begin{align*}
   & q_{j}^{(N-1)}(\lambda\!-\!1) + q_{j-1}^{(N-1)}(\lambda\!-\!1) \\
   & = \frac{N!}{j!(2N\!+\!1\!-\!2j)!} (-2)^{N+1-j} \left[j(2\lambda\!+\!n) -
   N(2N\!+\!1)\right] \prod_{k=j}^{N-1} (2\lambda\!-\!4N\!+\!2k\!+\!n\!+\!1).
\end{align*}
Hence
\begin{align*}
   & \sum_{j=0}^{N-i} (-1)^{N-j} \binom{N-j}{i} \left[q_{j}^{(N-1)}(\lambda\!-\!1)
   + q_{j-1}^{(N-1)}(\lambda\!-\!1)\right] \\
   & = -(2\lambda\!+\!n) \sum_{j=0}^{N-i-1} 2^{N-j} \binom{N-j-1}{i}
   \frac{N!}{j!(2N\!-\!1\!-\!2j)!} \prod_{k=j+1}^{N-1} (2\lambda\!-\!4N\!+\!2k\!+\!n\!+\!1) \\
   & + N(2N\!+\!1) \sum_{j=0}^{N-i} 2^{N+1-j} \binom{N-j}{i} \frac{N!}{j!(2N\!+\!1\!-\!2j)!}
   \prod_{k=j}^{N-1} (2\lambda\!-\!4N\!+\!2k\!+\!n\!+\!1) \\
   & = -(2\lambda\!+\!n) 2^{2N-1} \binom{N-1}{i} \frac{N!}{(2N\!-\!1)!} \\
   & \qquad \qquad \qquad \qquad \times
   (\lambda\!+\!\tfrac{n}{2}\!-\!2N\!+\!\tfrac{3}{2})_{N-1} \;
   {}_2F_1(-N\!+\!\tfrac{1}{2},-N\!+\!i\!+\!1;\lambda\!+\!\tfrac{n}{2}\!-\!2N\!+\!\tfrac{3}{2};1) \\
   & + N(2N\!+\!1) 2^{2N+1} \binom{N}{i} \frac{N!}{(2N\!+\!1)!} \\
   & \qquad \qquad \qquad \qquad \times (\lambda\!+\!\tfrac{n}{2}\!-\!2N\!+\!\tfrac{1}{2})_N
   \; {}_2F_1(-N\!-\!\tfrac{1}{2},-N\!+\!i;\lambda\!+\!\tfrac{n}{2}\!-\!2N\!+\!\tfrac{1}{2};1).
\end{align*}
By the Zhu-Vandermonde formula \eqref{ZV}, the latter sum equals
\begin{align*}
   & -(2\lambda\!+\!n) 2^{2N-1} \binom{N-1}{i} \frac{N!}{(2N\!-\!1)!}
   \frac{(\lambda\!+\!\frac{n}{2}\!-\!2N\!+\!\frac{3}{2})_{N-1}}
   {(\lambda\!+\!\frac{n}{2}\!-\!2N\!+\frac{3}{2})_{N-i-1}}(\lambda\!+\!\tfrac{n}{2}\!-\!N\!+\!1)_{N-i-1} \\
   & + 2^{2N} \binom{N}{i} \frac{N!}{(2N\!-\!1)!}
   \frac{(\lambda\!+\!\frac{n}{2}\!-\!2N\!+\!\frac{1}{2})_{N}}{(\lambda\!+\!\frac{n}{2}\!-\!2N\!+\frac{1}{2})_{N-i}}
   (\lambda\!+\!\tfrac{n}{2}\!-\!N\!+\!1)_{N-i}.
\end{align*}
Now simplification gives
\begin{align*}
   & \frac{\left[-(\lambda\!+\!\tfrac{n}{2})(N\!-\!i) + N (\lambda\!+\!\frac{n}{2}\!-\!i)\right]}
   {(\lambda\!+\!\tfrac{n}{2}\!-\!N)}
   2^{2N} \frac{(N\!-\!1)!}{(2N\!-\!1)!} \binom{N}{i} (\lambda\!+\!\tfrac{n}{2}\!-\!N)_{N-i}
   (\lambda\!+\!\tfrac{n}{2}\!-\!i\!-\!N\!+\!\tfrac{1}{2})_i \\
   & = (-1)^i 2i \alpha_i^{(N)}(\lambda).
\end{align*}
Thus, we have proved that
$$
\sum_{j=0}^{N-i} (-1)^{N-j-i} \binom{N-j}{i} \left[q_{j}^{(N-1)}(\lambda\!-\!1)
   + q_{j-1}^{(N-1)}(\lambda\!-\!1)\right] = 2i \alpha_i^{(N)}(\lambda).
$$
Subtracting this identity from \eqref{id-2} proves \eqref{id-3}. The proof is
complete.
\end{proof}

Now Hodge conjugation relates the even-order families $\Omega^p(\R^n) \to
\Omega^{p-1}(\R^{n-1})$ of the second type to the even-order families
$\Omega^p(\R^n) \to \Omega^{p}(\R^{n-1})$ of the first type. More precisely,
Theorem \ref{Hodge-c} implies the following result.

\begin{theorem}\label{coeffeven2} For $p=1,\dots,n$ and $N \in \N$, the even-order
families $D_{2N}^{(p \to p-1)}(\lambda)$ of the second type can be written in
the form
\begin{align}\label{fam-even2}
   D_{2N}^{(p \to p-1)}(\lambda) & = -(\lambda\!+\!n\!-\!p\!-\!2N) \sum_{i=0}^N \alpha_i^{(N)}(\lambda)
   (\dm \delta)^{N-i} \iota^* i_{\partial_n} (\bar{\dm} \bar{\delta})^i \notag \\
   & - \sum_{i=1}^{N-1} (\lambda\!+\!n\!-\!p\!-\!2i) \alpha_i^{(N)}(\lambda)
   (\delta \dm)^{N-i} \iota^* i_{\partial_n} (\bar{\dm} \bar{\delta})^i \notag \\
   & -(\lambda\!+\!n\!-\!p) \sum_{i=0}^N  \alpha_i^{(N)}(\lambda)
   (\delta \dm)^{N-i} \iota^* i_{\partial_n} (\bar{\delta} \bar{\dm})^i
\end{align}
with the coefficients $\alpha_i^{(N)}(\lambda)$ defined by \eqref{a-even}.
\end{theorem}

\begin{proof} Theorem \ref{coeffeven} and Theorem \ref{Hodge-c} imply
\begin{align*}
   (-1)^{np} D_{2N}^{(p \to p-1)}(\lambda) & = \star \, D_{2N}^{(n-p \to n-p)}(\lambda) \, \bar{\star} \\
   & = (\lambda\!+\!n\!-\!p) \sum_{i=0}^N \alpha_i^{(N)}(\lambda)
   (\delta \dm)^{N-i} \star \iota^* \bar{\star} \, (\bar{\delta} \bar{\dm} )^i \\
   & + \sum_{i=1}^{N-1} (\lambda\!+\!n\!-\!p\!-\!2i) \alpha_i^{(N)}(\lambda)
   (\delta\dm )^{N-i} \star \iota^* \bar{\star} \, (\bar{\dm}\bar{\delta})^i \\
   & + (\lambda\!+\!n\!-\!p\!-\!2N) \sum_{i=0}^N \alpha_i^{(N)}(\lambda)
   (\dm \delta)^{N-i} \star \iota^* \bar{\star} \, (\bar{\dm}\bar{\delta})^i
\end{align*}
using Lemma \ref{Hodge-gen}/(4). But
$$
   \star \, \iota^* \bar{\star} = \iota^* i_{\partial_n} (-1)^{pn+1}
$$
on $\Omega^p(\R^n)$ by Lemma \ref{HodgeLemma}. The assertion follows by
combining both results.
\end{proof}

\subsection{Odd-order families of the first and second type}\label{case-odd}

We continue with the discussion of odd-order families. We start with odd-order
families of the first type. The following result basically restates Theorem
\ref{main-odd} in Section \ref{intro}.

\begin{theorem}\label{coeffodd2} Assume that $N \in \N_0$ and $p=0,\dots,n-1$. The odd-order family
$D^{(p \to p)}_{2N+1}(\lambda)$ of the first type can be written in the form
\begin{align}
   D^{(p\to p)}_{2N+1}(\lambda)
   & = \sum_{i=1}^N  \gamma_i^{(N)}(\lambda;p) (\dm \delta)^{N-i} \dm \iota^* i_{\partial_n}
   (\bar{\delta}\bar{\dm})^i \notag \\
   & + (\lambda\!+\!p) \sum_{i=0}^N  \beta_i^{(N)}(\lambda)
   (\dm\delta)^{N-i}\dm \iota^*i_{\partial_n}(\bar{\dm}\bar{\delta})^i \notag \\
   & + (\lambda\!+\!p\!-\!2N\!-\!1) \sum_{i=0}^N \beta_i^{(N)}(\lambda)
   (\delta\dm)^{N-i} \iota^*i_{\partial_n} \bar{\dm}(\bar{\delta}\bar{\dm})^i
\end{align}
with the coefficients
\begin{align}\label{a-coeff-odd}
   \gamma_i^{(N)}(\lambda;p) & = (-1)^i 2^N \frac{N!}{(N\!+\!1)(2N\!+\!1)!} \binom{N+1}{i} \notag \\
   & \times \left[(\lambda\!+\!p\!-\!2N\!-\!1)(N\!+\!1)(2\lambda\!+\!n\!-\!2i) +
   (\lambda\!+\!n\!-\!p)(2N\!+\!1)(N\!-\!i\!+\!1) \right] \notag \\
   & \times \prod_{k=i+1}^N (2\lambda\!+\!n\!-\!2k) \prod_{k=1}^{i-1}
   (2\lambda\!+\!n\!-\!2k\!-\!2N\!-\!1), \quad i=1,\dots,N,
\end{align}
and $\beta_i^{(N)}(\lambda)$ as in \eqref{b-even}.
\end{theorem}

The coefficients $\gamma_i^{(N)}(\lambda;p)$ can be written in the form
\begin{equation}\label{a-deco}
   \gamma_i^{(N)}(\lambda;p) = \gamma_i^{(N),+}(\lambda;p) + \gamma_i^{(N),-}(\lambda;p)
\end{equation}
with
\begin{multline}\label{a+-coeff}
   \gamma_i^{(N),+}(\lambda;p) = (-1)^i 2^N \frac{N!}{(2N\!+\!1)!} \binom{N+1}{i} \\
   \times (\lambda\!+\!p\!-\!2N\!-\!1) \prod_{k=i}^N (2\lambda\!+\!n\!-\!2k)
   \prod_{k=1}^{i-1} (2\lambda\!+\!n\!-\!2k\!-\!2N\!-\!1)
\end{multline}
and
\begin{multline}\label{a--coeff}
   \gamma^{(N),-}_i(\lambda;p) = (-1)^i 2^N \frac{N!}{(2N)!} \binom{N}{i} \\
   \times (\lambda\!+\!n\!-\!p) \prod_{k=i+1}^N (2\lambda\!+\!n\!-\!2k)
   \prod_{k=1}^{i-1}(2\lambda\!+\!n\!-\!2k\!-\!2N\!-\!1).
\end{multline}
The decomposition \eqref{a-deco} will be important in the proof below.

\begin{bem}\label{gamma-alternative} Alternatively, the coefficients
$\gamma_i^{(N)}(\lambda;p)$, $i=1,\dots,N$, can be written in the form
\begin{equation}\label{gamma-alt}
   \gamma_i^{(N)}(\lambda;p) = (\lambda\!+\!p\!-\!2i) \beta_i^{(N)}(\lambda) -
   (\lambda\!+\!p\!-\!2i\!+\!1) \beta_{i-1}^{(N)}(\lambda).
\end{equation}
\end{bem}

\begin{proof} The assertion follows from the identity
\begin{multline*}
   (\lambda\!+\!p\!-\!2N\!-\!1)(N\!+\!1)(2\lambda\!+\!n\!-\!2i) +
   (\lambda\!+\!n\!-\!p)(2N\!+\!1)(N\!-\!i\!+\!1) \\
   = (\lambda\!+\!p\!-\!2i\!+\!1)i(2\lambda\!+\!n\!-\!2i)
   + (\lambda\!+\!p\!-\!2i)(N\!-\!i\!+\!1)(2\lambda\!+\!n\!-\!2i\!-\!2N\!-\!1).
\end{multline*}
We omit the details.
\end{proof}

\begin{proof} Throughout the following proof we shall use the
conventions $q_{-1}^{(N)}(\lambda)=0$ and $r_{-1}^{(N-1)}(\lambda)=0$. We start
by proving the formula
\begin{align}
   D^{(p \to p)}_{2N+1}(\lambda)
   & = \sum_{i=0}^N \sum_{j=0}^{N-i}(-1)^{N-j-i} \binom{N-j}{i} S_j(\lambda;N,p)
   (\dm\delta)^{N-i} \dm \iota^* i_{\partial_n} (\bar{\dm}\bar{\delta})^i \notag \\
   & + \sum_{i=0}^N \sum_{j=0}^{N-i}(-1)^{N-j-i} \binom{N-j}{i} p_j^{(N)}(\lambda;p)
   (\delta\dm)^{N-i} \iota^* i_{\partial_n} \bar{\dm}(\bar{\delta}\bar{\dm})^i \notag \\
   & + \sum_{i=1}^N \sum_{j=-1}^{N-i}(-1)^{N-j-i-1} \binom{N-j}{i} T_j(\lambda;N,p)
   (\dm\delta)^{N-i} \dm \iota^* i_{\partial_n} (\bar{\delta}\bar{\dm})^i \label{odd-2}
\end{align}
with
$$
   S_j(\lambda;N,p) \st \left[p_j^{(N)}(\lambda;p) + r_{j-1}^{(N-1)}(\lambda\!-\!1) +
   q_j^{(N)}(\lambda\!-\!1)\right], \; j=0,\dots,N
$$
and
$$
   T_j(\lambda;N,p) \st
   \left[p_{j+1}^{(N)}(\lambda;p) + r_{j}^{(N-1)}(\lambda\!-\!1) - q_j^{(N)}(\lambda\!-\!1)\right],
   \; j=-1,\dots,N-1.
$$
First, we use the identities \eqref{help-11} to rewrite the family $D^{(p \to
p)}_{2N+1}(\lambda)$ in Theorem \ref{OddDiffOp-type1} in the form
\begin{align}\label{help-13}
   D^{(p \to p)}_{2N+1}(\lambda) & =
   \sum_{j=1}^N\sum_{i=0}^{N-j}(-1)^{N-j-i} \binom{N-j}{i}
   \left[p_j^{(N)}(\lambda;p)+r_{j-1}^{(N-1)}(\lambda\!-\!1)\right]
   (\dm\delta)^{N-i} \dm \iota^* i_{\partial_n} \bar{\Delta}^i \notag\\
   & + \sum_{j=1}^N \sum_{i=0}^{N-j}(-1)^{N-j-i} \binom{N-j}{i}
   \left[p_j^{(N)}(\lambda;p)+r_{j-1}^{(N-1)}(\lambda\!-\!1)\right]
   (\dm\delta)^{N-i} \iota^* i_{\partial_n} \bar{\dm}(\bar{\delta}\bar{\dm})^i \notag \\
   & + \sum_{j=1}^N\sum_{i=0}^{N-j}(-1)^{N-j-i} \binom{N-j}{i} p_j^{(N)}(\lambda;p)
   (\delta\dm)^{N-i} \iota^* i_{\partial_n} \bar{\dm}(\bar{\delta}\bar{\dm})^i \notag \\
   & + \sum_{i=0}^{N}(-1)^{N-i} \binom{N}{i} p_0^{(N)}(\lambda;p)
   \Delta^{N-i}(\dm \iota^* i_{\partial_n} + \iota^* i_{\partial_n} \bar{\dm}) \bar{\Delta}^i \notag\\
   & + \sum_{j=0}^N \sum_{i=0}^{N-j}(-1)^{N-j-i} \binom{N-j}{i} q_j^{(N)}(\lambda\!-\!1)
   (\dm\delta)^{N-i} \dm \iota^* i_{\partial_n} \bar{\Delta}^i.
\end{align}
Now we simplify this formula. The first sum in \eqref{help-13} equals
\begin{align}
   & \sum_{j=1}^N \sum_{i=0}^{N-j} (-1)^{N-j-i} \binom{N-j}{i}
   \left[p_j^{(N)}(\lambda;p)+r_{j-1}^{(N-1)}(\lambda\!-\!1)\right]
   (\dm\delta)^{N-i} \dm \iota^* i_{\partial_n} (\bar{\dm}\bar{\delta})^i \notag \\
   & + \sum_{j=1}^{N-1}\sum_{i=1}^{N-j}(-1)^{N-j-i} \binom{N-j}{i}
   \left[p_j^{(N)}(\lambda;p)+r_{j-1}^{(N-1)}(\lambda\!-\!1)\right]
   (\dm\delta)^{N-i} \dm \iota^* i_{\partial_n} (\bar{\delta}\bar{\dm})^i. \label{odd-help-1}
\end{align}
Next, the fourth sum in \eqref{help-13} coincides with the sum
\begin{align}
    & \sum_{i=0}^{N} (-1)^{N-i} \binom{N}{i} p_0^{(N)}(\lambda;p)(\dm\delta)^{N-i}\dm\iota^* i_{\partial_n}
    (\bar{\dm}\bar{\delta})^i \notag \\
    & + \sum_{i=1}^{N}(-1)^{N-i} \binom{N}{i} p_0^{(N)}(\lambda;p))(\dm\delta)^{N-i}\dm\iota^* i_{\partial_n}
    (\bar{\delta}\bar{\dm})^i \notag \\
    & + \sum_{i=0}^{N}(-1)^{N-i} \binom{N}{i} p_0^{(N)}(\lambda;p))(\dm\delta)^{N-i}\iota^* i_{\partial_n}
    \bar{\dm} (\bar{\delta}\bar{\dm})^i \notag \\
    & +\sum_{i=0}^{N-1}(-1)^{N-i} \binom{N}{i} p_0^{(N)}(\lambda;p))(\delta\dm)^{N-i}\iota^* i_{\partial_n}
    \bar{\dm} (\bar{\delta}\bar{\dm})^i. \label{odd-help-2}
\end{align}
By \eqref{odd-help-1} and \eqref{odd-help-2}, the first two sums and the fourth
sum in \eqref{help-13} combine to
\begin{align*}
   & \sum_{j=0}^N \sum_{i=0}^{N-j}(-1)^{N-j-i} \binom{N-j}{i}
   \left[p_j^{(N)}(\lambda;p) + r_{j-1}^{(N-1)}(\lambda\!-\!1)\right]
   (\dm\delta)^{N-i} \dm \iota^* i_{\partial_n} (\bar{\dm}\bar{\delta})^i \\
   & + \sum_{i=0}^{N-1}(-1)^{N-i} \binom{N}{i} p_0^{(N)}(\lambda;p))(\delta\dm)^{N-i}\iota^* i_{\partial_n}
   \bar{\dm}(\bar{\delta}\bar{\dm})^i \\
   & + \sum_{j=0}^{N-1}\sum_{i=1}^{N-j} (-1)^{N-j-i} \binom{N-j}{i}
   \left[p_j^{(N)}(\lambda;p) + r_{j-1}^{(N-1)}(\lambda\!-\!1)\right]
   (\dm\delta)^{N-i} \dm \iota^* i_{\partial_n} (\bar{\delta}\bar{\dm})^i\\
   & + \sum_{j=0}^N\sum_{i=0}^{N-j}(-1)^{N-j-i} \binom{N-j}{i}
   \left[p_j^{(N)}(\lambda;p) + r_{j-1}^{(N-1)}(\lambda\!-\!1)\right]
   (\dm\delta)^{N-i} \iota^* i_{\partial_n} \bar{\dm} (\bar{\delta}\bar{\dm})^i;
\end{align*}
we stress that $j$ runs from $j=0$. We combine the last two sums in the last
display by moving in the last sum one factor $\delta$ to the right of $\iota^*
i_{\partial_n}$ using the second rule in Lemma \ref{DiffCoDiff}/(1). The
calculation yields
\begin{align*}
   & \sum_{j=0}^{N-1} \sum_{i=1}^{N-j}(-1)^{N-j-i} \binom{N-j+1}{i}
   \left[p_j^{(N)}(\lambda;p) + r_{j-1}^{(N-1)}(\lambda\!-\!1)\right]
   (\dm\delta)^{N-i} \dm \iota^* i_{\partial_n} (\bar{\delta}\bar{\dm})^i \notag \\
   & - \sum_{j=0}^{N-1} \left[p_{j+1}^{(N)}(\lambda)+r_{j}^{(N-1)}(\lambda\!-\!1)\right]
   (\dm \delta)^{j} \dm \iota^* i_{\partial_n} (\bar{\delta}\bar{\dm})^{N-j} \\
   & + p_0^{(N)}(\lambda;p)) \iota^* i_{\partial_n} \bar{\dm} (\bar{\delta}\bar{\dm})^{N}.
\end{align*}
Summarizing the above results, we find that the first four sums in
\eqref{help-13} combine to
\begin{align*}
   & \sum_{j=0}^N \sum_{i=0}^{N-j}(-1)^{N-j-i} \binom{N-j}{i}
   \left[p_j^{(N)}(\lambda;p) + r_{j-1}^{(N-1)}(\lambda\!-\!1)\right]
   (\dm\delta)^{N-i} \dm \iota^* i_{\partial_n} (\bar{\dm}\bar{\delta})^i \\
   & + \sum_{i=0}^{N-1}(-1)^{N-i} \binom{N}{i} p_0^{(N)}(\lambda;p))(\delta\dm)^{N-i}\iota^* i_{\partial_n}
   \bar{\dm}(\bar{\delta}\bar{\dm})^i \\
   & + \sum_{j=0}^{N-1} \sum_{i=1}^{N-j}(-1)^{N-j-i} \binom{N-j+1}{i}
   \left[p_j^{(N)}(\lambda;p) + r_{j-1}^{(N-1)}(\lambda\!-\!1)\right]
   (\dm\delta)^{N-i} \dm \iota^* i_{\partial_n} (\bar{\delta}\bar{\dm})^i \notag \\
   & - \sum_{j=0}^{N-1} \left[p_{j+1}^{(N)}(\lambda) + r_{j}^{(N-1)}(\lambda\!-\!1)\right]
   (\dm \delta)^{j} \dm \iota^* i_{\partial_n} (\bar{\delta}\bar{\dm})^{N-j} \\
   & + p_0^{(N)}(\lambda;p)) \iota^* i_{\partial_n} \bar{\dm} (\bar{\delta}\bar{\dm})^{N} \\
   & + \sum_{j=1}^N \sum_{i=0}^{N-j}(-1)^{N-j-i} \binom{N-j}{i} p_j^{(N)}(\lambda;p)
   (\delta\dm)^{N-i} \iota^* i_{\partial_n} \bar{\dm}(\bar{\delta}\bar{\dm})^i.
\end{align*}
Now we perform an index shift in the third sum and summarize the second sum,
the fifth term and the sixth sum. We obtain
\begin{align*}
   & \sum_{j=0}^N \sum_{i=0}^{N-j}(-1)^{N-j-i} \binom{N-j}{i}
   \left[p_j^{(N)}(\lambda;p) + r_{j-1}^{(N-1)}(\lambda\!-\!1)\right]
   (\dm\delta)^{N-i} \dm \iota^* i_{\partial_n} (\bar{\dm}\bar{\delta})^i \\
   & + \sum_{j=-1}^{N-2} \sum_{i=1}^{N-j-1} (-1)^{N-j-i-1}  \binom{N-j}{i}
   \left[p_{j+1}^{(N)}(\lambda) + r_{j}^{(N-1)}(\lambda\!-\!1)\right]
   (\dm\delta)^{N-i} \dm \iota^* i_{\partial_n} (\bar{\delta}\bar{\dm})^i \notag \\
   & + \sum_{j=0}^N \sum_{i=0}^{N-j}(-1)^{N-j-i} \binom{N-j}{i} p_j^{(N)}(\lambda;p)
   (\delta\dm)^{N-i} \iota^* i_{\partial_n} \bar{\dm}(\bar{\delta}\bar{\dm})^i \\
   & - \sum_{j=0}^{N-1} \left[p_{j+1}^{(N)}(\lambda) + r_{j}^{(N-1)}(\lambda\!-\!1)\right]
   (\dm \delta)^{j} \dm \iota^* i_{\partial_n} (\bar{\delta}\bar{\dm})^{N-j}.
\end{align*}
The last sum can be regarded as the contribution $i=N-j$ in the second sum.
Hence these terms combine to
\begin{align*}
   & \sum_{j=0}^{N-1} \sum_{i=1}^{N-j} (-1)^{N-j-i-1}  \binom{N-j}{i}
   \left[p_{j+1}^{(N)}(\lambda) + r_{j}^{(N-1)}(\lambda\!-\!1)\right]
   (\dm\delta)^{N-i} \dm \iota^* i_{\partial_n} (\bar{\delta}\bar{\dm})^i \notag \\
   & + \sum_{i=1}^N(-1)^{N-i} \binom{N+1}{i} p_0^{(N)}(\lambda;p))
   (\dm\delta)^{N-i} \dm \iota^* i_{\partial_n} (\bar{\delta}\bar{\dm})^i
\end{align*}
After interchanges of the summations we find
\begin{align}\label{odd-inter1}
   & \sum_{i=0}^N \sum_{j=0}^{N-i}(-1)^{N-j-i} \binom{N-j}{i}
   \left[p_j^{(N)}(\lambda;p) + r_{j-1}^{(N-1)}(\lambda\!-\!1)\right]
   (\dm\delta)^{N-i} \dm \iota^* i_{\partial_n} (\bar{\dm}\bar{\delta})^i \notag \\
   & + \sum_{i=1}^N \sum_{j=0}^{N-i} (-1)^{N-j-i-1} \binom{N-j}{i}
   \left[p_{j+1}^{(N)}(\lambda)+r_{j}^{(N-1)}(\lambda\!-\!1)\right]
   (\dm\delta)^{N-i} \dm \iota^* i_{\partial_n} (\bar{\delta}\bar{\dm})^i \notag \\
   & + \sum_{i=1}^N(-1)^{N-i} \binom{N+1}{i} p_0^{(N)}(\lambda;p))
   (\dm\delta)^{N-i} \dm \iota^* i_{\partial_n} (\bar{\delta}\bar{\dm})^i \notag \\
   & + \sum_{i=0}^N\sum_{j=0}^{N-i}(-1)^{N-j-i} \binom{N-j}{i} p_j^{(N)}(\lambda;p)
   (\delta\dm)^{N-i} \iota^* i_{\partial_n} \bar{\dm}(\bar{\delta}\bar{\dm})^i.
\end{align}
Finally, the fifth sum in \eqref{help-13} expands as
\begin{align}\label{odd-inter2}
   & \sum_{i=0}^N \sum_{j=0}^{N-i}(-1)^{N-j-i} \binom{N-j}{i}q_j^{(N)}(\lambda\!-\!1)
   (\dm\delta)^{N-i} \dm \iota^* i_{\partial_n} (\bar{\dm}\bar{\delta})^i \notag \\
   & + \sum_{i=1}^N\sum_{j=0}^{N-i}(-1)^{N-j-i} \binom{N-j}{i}q_{j}^{(N)}(\lambda\!-\!1)
   (\dm\delta)^{N-i} \dm \iota^* i_{\partial_n} (\bar{\delta}\bar{\dm})^i.
\end{align}
Combining \eqref{odd-inter1} and \eqref{odd-inter2}, we obtain the formula
\eqref{odd-2}.

In order to complete the proof of the theorem, it remains to verify the
identities
\begin{equation}\label{help-14}
   \sum_{j=0}^{N-i}(-1)^{N-j-i} \binom{N\!-\!j}{i}
   \left[p_j^{(N)}(\lambda;p) \!+\! r_{j-1}^{(N-1)}(\lambda\!-\!1) \!+\! q_j^{(N)}(\lambda\!-\!1)\right]
   = (\lambda\!+\!p) \beta_i^{(N)}(\lambda)
\end{equation}
and
\begin{equation}\label{help-24}
   \sum_{j=0}^{N-i}(-1)^{N-j-i} \binom{N-j}{i} p_j^{(N)}(\lambda;p)
   = (\lambda\!+\!p\!-\!2N\!-\!1) \beta_i^{(N)}(\lambda)
\end{equation}
for $i=0,\dots,N$, and
\begin{equation}\label{help-34}
   \sum_{j=-1}^{N-i}(-1)^{N-j-i-1} \binom{N-j}{i}
   \left[p_{j+1}^{(N)}(\lambda;p)+r_{j}^{(N-1)}(\lambda\!-\!1)-q_j^{(N)}(\lambda\!-\!1)\right]
   = \gamma_i^{(N)}(\lambda;p)
\end{equation}
for $i=1,\dots,N$. But since by definition
$$
   p_j^{(N)}(\lambda;p) = (\lambda\!+\!p\!-\!2N\!-\!1) b_j^{(N)}(\lambda),
$$
the identity \eqref{help-24} is a direct consequence of the second part of
Lemma \ref{RelationJacobiGegenbauer}. Next, we observe that
\begin{align*}
   r_{j-1}^{(N-1)}(\lambda\!-\!1) & = 2N b_{j-1}^{(N-1)}(\lambda\!-\!1) \qquad
   \mbox{(by definition of $r_j^{(N-1)}(\lambda)$)} \\
   & = 2j b_j^{(N)}(\lambda) \quad \mbox{(from the definition of odd Gegenbauer coefficients)}
\end{align*}
and
$$
   q_j^{(N)}(\lambda\!-\!1) = (2N\!-\!2j\!+\!1) b_j^{(N)}(\lambda)
   \quad \mbox{(from the definition of odd Gegenbauer coefficients)}.
$$
Hence
$$
   p_j^{(N)}(\lambda;p) + r_{j-1}^{(N-1)}(\lambda\!-\!1) + q_j^{(N)}(\lambda\!-\!1)
   = (\lambda\!+\!p) b_j^{(N)}(\lambda)
$$
and \eqref{help-14} follows from the second identity in Lemma
\ref{RelationJacobiGegenbauer}. In order to prove \eqref{help-34}, we first
calculate
\begin{align}\label{sum-term}
   p_{j+1}^{(N)}(\lambda;p) & + r_j^{(N-1)}(\lambda\!-\!1) - q_j^{(N)}(\lambda\!-\!1) \notag \\
   & = (\lambda\!+\!p\!-\!2N\!+\!1\!+\!2j) b_{j+1}^{(N)}(\lambda) - (2N\!-\!2j\!+\!1)
   b_j^{(N)}(\lambda) \notag \\
   & = -\left[(\lambda\!+\!p\!-\!2N\!-\!1)(N\!+\!1) + (\lambda\!+\!n\!-\!p)(j\!+\!1)\right] \notag \\
   & \times (-2)^{N-j} \frac{N!}{(j\!+\!1)! (2N\!-\!2j)!}
   \prod_{k=j+1}^{N-1} (2\lambda\!+\!n\!-\!4N\!+\!2k\!-\!1)
\end{align}
for $j=0,\dots,N-1$. This formula extends to $j=-1$. Now we split the left-hand
side of \eqref{help-34} into two parts:
\begin{multline*}
   \sum_{j=-1}^{N-i} (-1)^{N-i-j} (-2)^{N-j} \binom{N-j}{i}
   \frac{(N\!+\!1)!}{(j\!+\!1)! (2N\!-\!2j)!} \\
   \times (\lambda\!+\!p\!-\!2N\!-\!1) \prod_{k=j+1}^{N-1} (2\lambda\!+\!n\!-\!4N\!+\!2k\!-\!1)
\end{multline*}
and
\begin{equation*}
   \sum_{j=0}^{N-i} (-1)^{N-i-j} (-2)^{N-j} \binom{N-j}{i} \frac{N!}{j!(2N\!-\!2j)!}
   (\lambda\!+\!n\!-\!p) \prod_{k=j+1}^{N-1} (2\lambda\!+\!n\!-\!4N\!+\!2k\!-\!1);
\end{equation*}
note that the second sum runs from $j=0$. For $1 \le i \le N$ both sums are
hypergeometric and we find the respective formulas
\begin{align*}
   & (-1)^{i} 4^{N} \binom{N+1}{i} \frac{N!}{(2N\!+\!1)!}
   (\lambda\!+\!p\!-\!2N\!-\!1) \left(\lambda\!+\!\tfrac{n}{2}\!-\!2N-\!\tfrac{1}{2}\right)_N \\
   & \times
   {}_2F_1(-N\!-\!\tfrac{1}{2},i\!-\!N\!-\!1;\lambda\!+\!\tfrac{n}{2}\!-\!2N-\!\tfrac{1}{2};1)
\end{align*}
and
\begin{align*}
   & (-1)^{i} 2^{2N-1} \binom{N}{i} \frac{N!}{(2N)!} (\lambda\!+\!n\!-\!p)
   \left(\lambda\!+\!\tfrac{n}{2}\!-\!2N\!+\!\tfrac{1}{2}\right)_{N-1} \\
   & \times {}_2F_1(-N\!+\!\tfrac{1}{2},i\!-\!N;\lambda\!+\!\tfrac{n}{2}\!-\!2N+\!\tfrac{1}{2};1).
\end{align*}
Application of the Zhu-Vandermonde formula \eqref{ZV} yields
\begin{align*}
   & (-1)^{i} 4^{N} \binom{N+1}{i} \frac{N!}{(2N\!+\!1)!}
   (\lambda\!+\!p\!-\!2N\!-\!1)  \left(\lambda\!+\tfrac{n}{2}\!-\!N \right)_{N+1-i}
   \frac{\left(\lambda\!+\tfrac{n}{2}\!-\!2N\!-\!\tfrac{1}{2}\right)_N}
   {\left(\lambda\!+\tfrac{n}{2}\!-\!2N\!-\!\tfrac{1}{2}\right)_{N+1-i}} \\
   & = \gamma_i^{(N),+}(\lambda;p)
\end{align*}
and
\begin{align*}
   (-1)^{i} 2^{2N-1} \binom{N}{i} \frac{N!}{(2N)!} (\lambda\!+\!n\!-\!p)
   \left(\lambda\!+\!\tfrac{n}{2}\!-\!N\right)_{N-i}
   \frac{\left(\lambda\!+\tfrac{n}{2}\!-\!2N\!+\!\tfrac{1}{2}\right)_{N-1}}
   {\left(\lambda\!+\tfrac{n}{2}\!-\!2N\!+\!\tfrac{1}{2}\right)_{N-i}} = \gamma_i^{(N),-}(\lambda;p)
\end{align*}
(see \eqref{a+-coeff}, \eqref{a--coeff}). This proves \eqref{help-34} for $1
\le i \le N$.
\end{proof}

Finally, Hodge conjugation relates the odd-order families $\Omega^p(\R^n) \to
\Omega^{p-1}(\R^{n-1})$ of the second type to the odd-order families
$\Omega^p(\R^n) \to \Omega^{p}(\R^{n-1})$ of the first type. More precisely,
Theorem \ref{Hodge-c} implies the following result.

\begin{theorem}\label{coeffodd} Assume that $N \in \N_0$ and $p=1,\dots,n$.
The odd-order families $D^{(p \to p-1)}_{2N+1}(\lambda)$ of the second type can
be written in the form
\begin{align}\label{SBO-odd}
   D^{(p \to p-1)}_{2N+1}(\lambda)
   & = - \sum_{i=1}^N \gamma_i^{(N)}(\lambda;n\!-\!p)
   (\delta \dm)^{N-i} \delta \iota^* (\bar{\dm} \bar{\delta})^i \notag \\
   & + (\lambda\!+\!n\!-\!p\!-\!2N\!-\!1) \sum_{i=0}^N \beta_i^{(N)} (\lambda)
   (\dm \delta)^{N-i} \iota^* \bar{\delta} (\bar{\dm} \bar{\delta})^i \notag \\
   & - (\lambda\!+\!n\!-\!p) \sum_{i=0}^N \beta_i^{(N)}(\lambda)
   (\delta \dm)^{N-i} \delta \iota^* (\bar{\delta} \bar{\dm})^i
\end{align}
with the coefficients $\beta_i^{(N)}(\lambda)$ and $\gamma_i^{(N)}(\lambda;p)$
as defined in \eqref{b-even} and \eqref{a-coeff-odd}, respectively.
\end{theorem}

\begin{proof} Theorem \ref{coeffodd2} and Theorem \ref{Hodge-c} imply
\begin{align*}
   (-1)^{np} D_{2N+1}^{(p \to p-1)}(\lambda) & = \star \, D_{2N+1}^{(n-p \to n-p)}(\lambda) \, \bar{\star} \\
   & = \sum_{i=1}^N \gamma_i^{(N)}(\lambda;n\!-\!p)
   (\delta \dm)^{N-i} \star \dm \iota^* i_{\partial_n} \bar{\star} \, (\bar{\dm} \bar{\delta})^i \notag \\
   & + (\lambda\!+\!n\!-\!p\!-\!2N\!-\!1) \sum_{i=0}^N \beta_i^{(N)} (\lambda)
   (\dm \delta)^{N-i} \star \iota^* i_{\partial_n} \bar{\dm} \, \bar{\star} \,(\bar{\dm} \bar{\delta})^i \notag \\
   & + (\lambda\!+\!n\!-\!p) \sum_{i=0}^N \beta_i^{(N)}(\lambda)
   (\delta \dm)^{N-i} \star \dm \iota^* i_{\partial_n} \bar{\star} \, (\bar{\delta} \bar{\dm})^i
\end{align*}
using Lemma \ref{Hodge-gen}/(4). But Lemma \ref{Hodge-gen}/(2) and Lemma
\ref{HodgeLemma} give
\begin{equation*}
   \star \, \dm \iota^* i_{\partial_n} \bar{\star} = \delta \star \iota^* i_{\partial_n} \bar{\star} (-1)^{n-p}
   = \delta \iota^* (-1)^{np+1}
\end{equation*}
and
\begin{equation*}
   \star \, \iota^* i_{\partial_n} \bar{\dm} \, \bar{\star} = \star \,
   \iota^* i_{\partial_n} \bar{\star} \, \bar{\delta} (-1)^p = \iota^* \bar{\delta} (-1)^{pn}
\end{equation*}
on $\Omega^p(\R^n)$. Combining these results proves the assertion.
\end{proof}

\subsection{Operators of the third and fourth type}\label{third-fourth}

In the present section, we derive geometrical formulas for the conformal
symmetry breaking operators of the third and the fourth type.

\begin{theorem}\label{CSBO-3-even} The even-order operators $D_{2N}^{(0\to 1)}$, $N \in \N$,
of the third type can be written in the form
\begin{align*}
   D_{2N}^{(0\to  1)} = \sum_{i=0}^{N-1} \beta_j^{(N-1)}(2N\!-\!1)
   (\dm\delta)^{N-i-1} \dm \iota^* i_{\partial_n}\bar{\dm} (\bar{\delta}\bar{\dm})^i
\end{align*}
with the coefficients $\beta_j^{(N)}(\lambda)$ defined by \eqref{b-even}.
\end{theorem}

\begin{proof} By the definition of $D_{2N}^{(0\to 1)}$ (Theorem \ref{DO3-Even}) and the expansion
\eqref{help-11} of $\iota^*\partial_n^{2N-2j-1}$, we obtain
\begin{align*}
   D_{2N}^{(0\to 1)} = \sum_{j=0}^{N-1} \sum_{i=0}^{N-j-1}(-1)^{N-j-i-1} \binom{N-j-1}{i}
   b_j^{(N-1)}(2N\!-\!1)\dm (\delta\dm)^{N-i-1} \iota^* i_{\partial_n}\bar{\dm}(\bar{\delta}\bar{\dm})^i.
\end{align*}
Interchanging summations and applying Lemma \ref{RelationJacobiGegenbauer}
completes the proof.
\end{proof}

\begin{theorem}\label{CSBO-3-odd} The odd-order operators $D_{2N+1}^{(0\to 1)}$, $N\in\N_0$, of
the third type can be written in the form
\begin{align*}
   D_{2N+1}^{(0\to 1)} = \sum_{i=0}^N \alpha_j^{(N)}(2N) (\dm\delta)^{N-i} \dm\iota^*(\bar{\delta}\bar{\dm})^i
\end{align*}
with the coefficients $\alpha_j^{(N)}(\lambda)$ defined by \eqref{a-even}.
\end{theorem}

\begin{proof} By the definition of $D_{2N+1}^{(0\to 1)}$ (Theorem \ref{DO3-Odd}) and the
expansion \eqref{help-10} of $\iota^*\partial_n^{2N-2j}$, we obtain
\begin{align*}
   D_{2N+1}^{(0\to 1)}=\sum_{j=0}^{N} \sum_{i=0}^{N-j} (-1)^{N-j-i} \binom{N-j}{i}
   a_j^{(N)}(2N) \dm (\delta\dm)^{N-i}\iota^* (\bar{\delta}\bar{\dm})^i.
\end{align*}
Interchanging summations and applying Lemma \ref{RelationJacobiGegenbauer}
completes the proof.
\end{proof}

\begin{bem} In view of $D_{N}^{(0\to 1)} = \dm \dot{D}_{N-1}^{(0\to 0)}(N-1)$, the same
results also follow from the geometrical formula for $D_N^{(0 \to
0)}(\lambda)$.
\end{bem}

We continue with the derivation of the analogous formulas for the conformal
symmetry breaking operators of the fourth type.

\begin{theorem}\label{CSBO-4-even} The even-order operators $D_{2N}^{(n\to n-2)}$, $N\in\N$,
of the fourth type can be written in the form
\begin{align*}
   D_{2N}^{(n\to n-2)} = \sum_{i=0}^{N-1} \beta_j^{(N-1)}(2N\!-\!1)
   (\delta\dm)^{N-i-1}\delta\iota^*\bar{\delta} (\bar{\dm}\bar{\delta})^i,
\end{align*}
with the coefficients $\beta_j^{(N)}(\lambda)$ defined by \eqref{b-even}.
\end{theorem}

\begin{proof} By the definition of $D_{2N}^{(n\to n-2)}$ (Theorem \ref{DO4-Even}) and the expansion
\eqref{help-11} of $\iota^* i_{\partial_n}\partial_n^{2N-2j-1}$, we obtain
\begin{align*}
   D_{2N}^{(n\to n-2)}=\sum_{j=0}^{N-1}\sum_{i=0}^{N-j-1}(-1)^{N-j-i-1} \binom{N-j-1}{i}
   b_j^{(N-1)}(2N\!-\!1) \delta (\dm\delta)^{N-i-1}\iota^*\bar{\delta}(\bar{\dm}\bar{\delta})^i.
\end{align*}
Interchanging summations and applying Lemma \ref{RelationJacobiGegenbauer}
completes the proof.
\end{proof}

\begin{theorem}\label{CSBO-4-odd} The odd-order operators $D_{2N+1}^{(n\to n-2)}$, $N\in\N_0$,
of the fourth type can be written in the form
\begin{align*}
   D_{2N+1}^{(n\to n-2)} = -\sum_{i=0}^{N} \alpha_i^{(N)}(2N)
   (\delta\dm)^{N-i} \delta\iota^*i_{\partial_n}(\bar{\dm}\bar{\delta})^i,
\end{align*}
with the coefficients $\beta_j^{(N)}(\lambda)$ defined by \eqref{b-even}.
\end{theorem}

\begin{proof} By the definition of $D_{2N+1}^{(n\to n-2)}$ (Theorem \ref{DO4-Odd}) and
the expansion \eqref{help-10} of $\iota^* i_{\partial_n}\partial_n^{2N-2j}$, we
obtain
\begin{align*}
   D_{2N+1}^{(n\to n-2)}=\sum_{j=0}^{N}\sum_{i=0}^{N-j}(-1)^{N-j-i+1} \binom{N-j}{i}
   a_j^{(N)}(2N)\delta (\dm\delta)^{N-i}\iota^*i_{\partial_n}(\bar{\dm}\bar{\delta})^i.
\end{align*}
Interchanging summations and applying Lemma \ref{RelationJacobiGegenbauer}
completes the proof.
\end{proof}

\begin{bem} In view of $D_{N}^{(n\to n-2)}=\delta \dot{D}_{N-1}^{(n\to n-1)}(N-1)$, the same
results also follow from the geometrical formula for $D_N^{(n\to
n-1)}(\lambda)$.
\end{bem}

\section{Factorization identities for conformal symmetry breaking operators}\label{Properties}

In the present section, we discuss natural identities which describe
factorizations of conformal symmetry breaking operators into products of
conformally covariant operators. There are two types of such factorizations.
The main factorizations contain Branson-Gover operators as factors and the
supplementary factorizations contain exterior differentials and
co-differentials as factors. The main factorizations generalize corresponding
results in \cite{Juhl}.

As consequences, we shall see that the conformal symmetry breaking operators
serve as a natural organizing package for several important differential
operators acting on differential forms. In particular, we shall prove that for
the Euclidean metric $g_0$ they naturally capture the Branson-Gover, gauge
companion and $Q$-curvature operators \cite{BransonGover}.

\subsection{Branson-Gover, gauge companion and $Q$-curvature operators}\label{Branson-Gover}

We first recall some basic facts on these constructions. For more details we
refer to \cite{BransonGover}, \cite{Gover-Srni} and \cite{AG}. The
Branson-Gover operators $L_{2N}^{(p)}(g)$ on a Riemannian manifold $(M^n,g)$ of
dimension $n \ge 3$ are conformally covariant differential operators on
$\Omega^p(M)$ of order $2N \ge 2$. They satisfy the intertwining relation
\begin{equation}\label{conf-cov}
   e^{(\frac{n}{2}-p+N)\sigma} \hat{L}_{2N}^{(p)} (\omega)
   = L_{2N}^{(p)} (e^{(\frac{n}{2}-p-N)\sigma} \omega), \quad \omega \in \Omega^p(M),
\end{equation}
where $L_{2N}^{(p)} = L_{2N}^{(p)}(g)$ and $\hat{L}_{2N}^{(p)}=
L_{2N}^{(p)}(\hat{g})$ are the respective Branson-Gover operators for the
metrics $g$ and $\hat{g} = e^{2\sigma}g$.\footnote{We suppress the conditions
on the dimension $n$ and the order $2N$ which guarantee the existence of the
operators.} For even $n$ and $p \le \tfrac{n}{2}-1$, the operators
$L_{n-2p}^{(p)}$ will be called the {\em critical} Branson-Gover operators. The
operator $L_{2N}^{(0)}$ on $C^\infty(M)$ reduces to a constant multiple of the
GJMS-operator of order $2N$. In particular, in even dimension $n$, the critical
Branson-Gover operator $L_n^{(0)}$ is a constant multiple of the critical
GJMS-operator $P_n$.

For general metrics, even $n$ and $p \le \tfrac{n}{2}-1$, we also consider the
gauge companion operator
$$
   G_{n-2p+1}^{(p)}: \Omega^p(M) \to \Omega^{p-1}(M)
$$
and the (critical) $Q$-curvature operator
$$
   Q_{n-2p}^{(p)}: \Omega^p(M)|_{\ker (\dm)} \to \Omega^p(M).
$$
These operators appear in the factorization identities
$$
   L_{n-2p}^{(p)} = (n\!-\!2p) G_{n-2p-1}^{(p+1)} \dm \quad \mbox{and} \quad G_{n-2p-1}^{(p+1)} = \delta
   Q^{(p+1)}_{n-2p-2}.
$$
The operators $Q_{n-2p}^{(p)}$ generalize Branson's Q-curvature $Q_n$ on
functions. The property
\begin{equation}\label{double}
   L_{n-2p}^{(p)} \sim \delta Q^{(p+1)}_{n-2p-2} \dm
\end{equation}
is known as the double factorization property of the critical Branson-Gover
operators.

Under conformal changes $g \mapsto e^{2\sigma} g$, these operators transform
according to
\begin{equation}\label{G-conf}
  e^{(n-2p) \sigma} \widehat{G}_{n-2p-1}^{(p+1)}(\omega) =
  G_{n-2p-1}^{(p+1)}(\omega) + i_{\grad(\sigma)} L^{(p+1)}_{n-2p-2}(\omega)
\end{equation}
and
\begin{equation}\label{Q-conf}
  e^{(n-2p)\sigma} \widehat{Q}_{n-2p}^{(p)}(\omega) =
  Q_{n-2p}^{(p)}(\omega) + L^{(p)}_{n-2p}(\sigma \omega).
\end{equation}
Since the critical Branson-Gover operator $L_{n-2p}^{(p)}$ annihilates closed
forms, the transformation property \eqref{G-conf} implies that the restriction
of $G_{n-2p-1}^{(p+1)}$ to closed forms is conformally covariant.

On the Euclidean space $(\R^{n-1},g_0)$, these operators take the following
form.\footnote{In contrast to the general theory, in the present case the gauge
companion operators and the $Q$-curvature operators are defined on all forms.}
First, the Branson-Gover operators are given by the explicit formula
\begin{equation}\label{BG-flat}
   L_{2N}^{(p)} = (\tfrac{n-1}{2}\!-\!p\!+\!N) (\delta \dm)^N + (\tfrac{n-1}{2}\!-\!p\!-\!N) (\dm\delta)^N.
\end{equation}
Now assume that $n-1$ is even and $n\!-\!1\!-\!2p \ge 2$. Then the critical
Branson-Gover operator
\begin{equation*}
   L^{(p)}_{n-1-2p} = (n\!-\!1\!-\!2p) \delta(\dm\delta)^{\frac{n-3}{2}-p}\dm
\end{equation*}
factors through
\begin{equation*}
   G^{(p+1)}_{n-2-2p} \st \delta(\dm\delta)^{\frac{n-3}{2}-p}: \Omega^{p+1}(\R^{n-1}) \to \Omega^p(\R^{n-1})
\end{equation*}
and the critical $Q$-curvature operator
\begin{equation*}
   Q_{n-3-2p}^{(p+1)} \st (\dm\delta)^{\frac{n-3}{2}-p}: \Omega^{p+1}(\R^{n-1})\to\Omega^{p+1}(\R^{n-1}).
\end{equation*}

The conformal covariance \eqref{conf-cov} of the Branson-Gover operators
$L_{2N}^{(p)}$ implies the equivariance of $L_{2N}^{(p)}(g_0)$ under the
conformal group of $\R^{n}$. In fact, assume that $\gamma$ is a conformal
diffeomorphism of the Euclidean metric $g_0$ on $\R^n$, i.e., $\gamma_*(g_0) =
e^{2\Phi_\gamma} g_0$ for some $\Phi_\gamma \in C^\infty(\R^n)$. Then
\eqref{conf-cov} implies
$$
   e^{(\frac{n}{2}-p+N) \Phi_\gamma} L_{2N}^{(p)} (\gamma_*(g_0)) =
   L_{2N}^{(p)}(g_0) e^{(\frac{n}{2}-p-N)\Phi_\gamma}.
$$
But, by the naturality of the Branson-Gover operators, we have
$$
L_{2N}^{(p)} (\gamma_*(g_0)) = \gamma_* L_{2N}^{(p)}(g_0) \gamma^*.
$$
Hence
$$
   e^{(\frac{n}{2}-p+N) \Phi_\gamma} \gamma_* L_{2N}^{(p)}(g_0) =
   L_{2N}^{(p)}(g_0) e^{(\frac{n}{2}-p-N)\Phi_\gamma} \gamma_*.
$$
In other words, the operator $L_{2N}^{(p)}(g_0)$ satisfies the intertwining
property
\begin{equation}\label{flat-intertwining}
   \pi_{\frac{n}{2}-p+N}^{(p)}(\gamma) L_{2N}^{(p)}(g_0) = L_{2N}^{(p)}(g_0)
   \pi_{\frac{n}{2}-p-N}^{(p)}(\gamma),
\end{equation}
where
\begin{equation*}
   \pi_\lambda^{(p)}(\gamma) = e^{\lambda \Phi_\gamma} \gamma_*: \Omega^p (\R^n) \to \Omega^p(\R^n).
\end{equation*}
We also recall that $\pi_{-\lambda-p}^{(p)} = \pi_{\lambda,p}^\ch$ (see
\eqref{rep-rel-nc}). This is the non-compact analog of \eqref{rep-relation}.

\subsection{Main factorizations}\label{factor1}

Here we show that for special values of the parameter $\lambda$ the families of
the first and second type factorize as compositions of respective lower-order
families and Branson-Gover operators acting on forms on $\R^{n-1}$ and
$\R^{n}$. In particular, the Branson-Gover operators on forms on $\R^{n-1}$ and
$\R^{n}$ are naturally captured by both types of even-order conformal symmetry
breaking operators.

We start with the discussion of even-order families of the first type.

\begin{theorem}\label{MainFactEven1} Let $N \in \N$ and $p=0,\dots,n-1$. Then
the even-order families of the first type satisfy the factorization identities
\begin{equation}\label{fact-1b-1}
   (\tfrac{n}{2}\!-\!p\!+\!k) D_{2N}^{(p\to p)} (k\!-\!\tfrac{n}{2})
   = D_{2N-2k}^{(p\to p)}(-k\!-\!\tfrac{n}{2}) \circ \bar{L}_{2k}^{(p)}
\end{equation}
and
\begin{equation}\label{fact-1b-2}
   (\tfrac{n-1}{2}\!-\!p-\!k\!) D_{2N}^{(p\to p)}(2N\!-\!k\!-\!\tfrac{n-1}{2}) =
   L_{2k}^{(p)} \circ D_{2N-2k}^{(p\to p)}(2N\!-\!k\!-\!\tfrac{n-1}{2})
\end{equation}
for $k=1,\dots,N-1$, $N \ge 2$. Moreover, in the extremal case $k=N$, we have
\begin{equation}\label{fact-1}
   D^{(p\to p)}_{2N}(N\!-\!\tfrac{n-1}{2}) = -L^{(p)}_{2N} \iota^* \quad \mbox{and}
   \quad D^{(p\to p)}_{2N}(N\!-\!\tfrac{n}{2}) = -\iota^* \bar{L}^{(p)}_{2N}.
\end{equation}
\end{theorem}

By Hodge conjugation, Theorem \ref{MainFactEven1} implies the following result
for families of the second type.

\begin{theorem}\label{MainFactEven2} Let $N \in \N$ and $p=1,\dots,n$. Then the
even-order families of the second type satisfy the factorization identities
\begin{equation}\label{fact-2b-1}
   (\tfrac{n}2\!-\!p\!-\!k) D_{2N}^{(p\to p-1)} (k\!-\!\tfrac n2)
   = D_{2N-2k}^{(p\to p-1)}(-k\!-\!\tfrac n2) \circ \bar{L}_{2k}^{(p)}
\end{equation}
and
\begin{equation}\label{fact-2b-2}
   (\tfrac{n+1}{2}\!-\!p+\!k\!) D_{2N}^{(p\to p-1)}(2N\!-\!k\!-\!\tfrac{n-1}{2})
   = L_{2k}^{(p-1)} \circ D_{2N-2k}^{(p\to p-1)}(2N\!-\!k\!-\!\tfrac{n-1}{2})
\end{equation}
for $k=1,\dots,N-1$, $N \ge 2$. Moreover, in the extremal case $k=N$, we have
\begin{equation}\label{fact-2}
   D^{(p \to p-1)}_{2N}(N\!-\!\tfrac{n-1}{2}) = -L^{(p-1)}_{2N} \iota^* i_{\partial_n} \quad
   \mbox{and} \quad
   D^{(p \to p-1)}_{2N}(N\!-\!\tfrac{n}{2}) = -\iota^* i_{\partial_n} \bar{L}^{(p)}_{2N}.
\end{equation}
\end{theorem}

Some comments concerning these results are in order.

The relation \eqref{rep-relation} shows that Theorem \ref{MainFactEven1} and
Theorem \ref{MainFactEven2} are compatible with the respective equivariance
properties of the Branson-Gover operators and the conformal symmetry breaking
operators. Indeed, the relation
$$
   \dm\pi^{\prime \ch}_{\lambda-2N,p}(X) D_{2N}^{(p \to p)}(\lambda)
   = D_{2N}^{(p \to p)}(\lambda) \dm\pi^\ch_{\lambda,p}(X), \quad  X \in \gog^{\prime}(\R)
$$
(see Theorem \ref{equiv-even}) implies
$$
   \dm\pi^{\prime \ch}_{-\frac{n-1}{2}-N,p}(X) D_{2N}^{(p \to p)} (N-\tfrac{n-1}{2})
   = D_{2N}^{(p \to p)}(N-\tfrac{n-1}{2}) \dm\pi^\ch_{-\frac{n-1}{2}+N,p}(X).
$$
Hence, by \eqref{fact-1} and \eqref{rep-relation}, we find
$$
   \dm\pi^{\prime(p)}_{\frac{n-1}{2}+N-p}(X) L_{2N}^{(p)} = L_{2N}^{(p)}
   \dm\pi^{\prime(p)}_{\frac{n-1}{2}-N-p}(X)
$$
which fits with \eqref{flat-intertwining}.

The relations in Theorem \ref{MainFactEven1} generalize results in \cite{Juhl}
($p=0$). An important difference to the results in \cite{Juhl} is that the
above identities contain non-trivial numerical factors on the left-hand sides.

Theorem \ref{MainFactEven1} and Theorem \ref{MainFactEven2} are also suggested
by the multiplicity free character identities (see Proposition \ref{charind})
and the equivariance of the Branson-Gover operators. From this perspective,
their validity can be regarded as a cross-check of the explicit formulas for
the families displayed in Theorem \ref{coeffeven} and Theorem \ref{coeffeven2}.
The latter argument, however, does not yield the constant factors which relate
both sides of the identities. These factors can also be determined by the
following formal argument. We note that the right-hand sides of
\eqref{fact-1b-1} and \eqref{fact-1b-2} are {\em quadratic} in $p$ (each factor
is linear in $p$). Therefore, the left-hand side must contain an additional
linear factor in $p$ which easily can be read off. We stress that these
constant factors actually may vanish. For instance, the left-hand side of
\eqref{fact-2b-1} vanishes if $n$ is even and $k=p-\frac{n}{2} \ge 1$. For
these parameters, the right-hand side contains the operator
$\bar{L}^{(p)}_{2p-n} = (\bar{d}\bar{\delta})^{p-\frac{n}{2}}$.

The following proof derives the assertions from the geometric formulas in
Section \ref{geometric}.

\begin{proof} [Proof of Theorem \ref{MainFactEven1}] We first prove \eqref{fact-1b-1}.
By \eqref{BG-flat}, we have
\begin{equation}\label{Lbar-flat}
   \bar{L}^{(p)}_{2k} = (\tfrac n2\!-\!p\!+\!k) (\bar{\delta} \bar{\dm})^k
   + (\tfrac n2\!-\!p\!-\!k) (\bar{\dm}\bar{\delta})^k.
\end{equation}
Hence Theorem \ref{coeffeven} implies
\begin{align*}
   D_{2N}^{(p\to p)}(k\!-\!\tfrac n2)
   & = (k\!-\!\tfrac n2\!+\!p)\sum_{i=0}^N \alpha_i^{(N)}(k\!-\!\tfrac n2)
   (\dm\delta)^{N-i} \iota^* (\bar{\dm} \bar{\delta})^i  \notag\\
   & + \sum_{i=1}^{N-1} (k\!-\!\tfrac n2\!+\!p\!-\!2i) \alpha_i^{(N)}(k\!-\!\tfrac n2)
   (\dm \delta)^{N-i} \iota^* (\bar{\delta} \bar{\dm})^i \notag\\
   & + (k\!-\!\tfrac n2\!+\!p\!-\!2N) \sum_{i=0}^N \alpha_i^{(N)}(k\!-\!\tfrac n2)
   (\delta \dm)^{N-i} \iota^*(\bar{\delta} \bar{\dm})^i
\end{align*}
and
\begin{align*}
   D_{2N-2k}^{(p\to p)}(-k\!-\!\tfrac n2) \circ \bar{L}^{(p)}_{2k}
   & = (\tfrac n2\!-\!p\!-\!k) (-k\!-\!\tfrac n2\!+\!p) \sum_{i=k}^N \alpha_{i-k}^{(N-k)}(-k\!-\!\tfrac n2)
   (\dm \delta)^{N-i} \iota^* (\bar{\dm} \bar{\delta})^i \notag \\
   & +(\tfrac n2\!-\!p\!+\!k) (-k\!-\!\tfrac n2\!+\!p) \alpha_{0}^{(N-k)}(-k\!-\!\tfrac n2)
   (\dm \delta)^{N-k} \iota^* (\bar{\delta} \bar{\dm})^k \notag \\
   & +(\tfrac n2\!-\!p\!+\!k) \sum_{i=k+1}^{N-1} (k\!-\!\tfrac n2\!+\!p\!-\!2i)
   \alpha_{i-k}^{(N-k)}(-k\!-\!\tfrac n2) (\dm \delta)^{N-i} \iota^* (\bar{\delta} \bar{\dm})^i \notag \\
   & + (\tfrac n2\!-\!p\!+\!k) (k\!-\!\tfrac n2\!+\!p\!-\!2N) \sum_{i=k}^N  \alpha_{i-k}^{(N-k)}(-k\!-\!\tfrac n2)
   (\delta \dm)^{N-i} \iota^*(\bar{\delta} \bar{\dm})^i.
\end{align*}
We note that, in the last sum, the term in the second line originates from the
composition of the corresponding first sum in \eqref{SBO-even} with
$\bar{L}^{(p)}_{2k}$. It can be merged with the third line by extending the sum
to run from $i=k$. Thus, it remains to prove
\begin{align*}
   \alpha_i^{(N)}(k-\tfrac n2) & = 0, \quad i=0,\dots,k-1,\\
   \alpha_i^{(N)}(k-\tfrac n2) & = \alpha_{i-k}^{(N-k)}(-k-\tfrac n2),\quad i=k,\dots,N.
\end{align*}
Both equalities are direct consequences of the definition \eqref{a-even} of the
coefficients $\alpha_i^{(N)}(\lambda)$. We omit the details of the calculation.
This proves \eqref{fact-1b-1}.

Next, we prove \eqref{fact-1b-2}. By \eqref{BG-flat}, we have
\begin{equation}\label{L-flat}
    L^{(p)}_{2k} = (\tfrac{n-1}{2}\!-\!p\!+\!k) (\delta\dm)^k + (\tfrac{n-1}{2}\!-\!p\!-\!k)(\dm\delta)^k.
\end{equation}
Hence Theorem \ref{coeffeven} implies
\begin{align*}
    D_{2N}^{(p\to p)} (2N\!-\!k\!-\!\tfrac{n-1}{2})
    & = (2N\!-\!k-\!\tfrac{n-1}{2}\!+\!p) \sum_{i=0}^N \alpha_i^{(N)}(2N\!-\!k\!-\!\tfrac{n-1}{2})
    (\dm\delta)^{N-i} \iota^* (\bar{\dm} \bar{\delta})^i  \notag \\
    & + \sum_{i=1}^{N-1} (2N\!-\!k\!-\!\tfrac{n-1}{2}\!+\!p\!-\!2i) \alpha_i^{(N)}(2N\!-\!k-\!\tfrac{n-1}{2})
    (\dm \delta)^{N-i} \iota^* (\bar{\delta} \bar{\dm})^i  \notag \\
    & + (-k\!-\!\tfrac{n-1}{2}\!+\!p) \sum_{i=0}^N  \alpha_i^{(N)}(2N\!-\!k-\!\tfrac{n-1}{2})
    (\delta \dm)^{N-i} \iota^*(\bar{\delta} \bar{\dm})^i
\end{align*}
and
\begin{align*}
   L^{(p)}_{2k} & \circ D_{2N-2k}^{(p\to p)}(2N\!-\!k-\!\tfrac{n-1}{2}) \\
   & = (\tfrac{n-1}{2}\!-\!p\!-\!k)(2N\!-\!k-\!\tfrac{n-1}{2}\!+\!p)
   \sum_{i=0}^{N-k} \alpha_{i}^{(N-k)}(2N\!-\!k-\!\tfrac{n-1}{2})
   (\dm\delta)^{N-i} \iota^* (\bar{\dm} \bar{\delta})^i \notag \\
   & + (\tfrac{n-1}{2}\!-\!p\!-\!k) \sum_{i=1}^{N-k-1}
   (2N\!-\!k-\!\tfrac{n-1}{2}\!+\!p\!-\!2i)
   \alpha_{i}^{(N-k)}(2N\!-\!k\!-\!\tfrac{n-1}{2}) (\dm \delta)^{N-i} \iota^* (\bar{\delta} \bar{\dm})^i \notag\\
   & + (\tfrac{n-1}{2}\!-\!p\!+\!k) (k\!-\!\tfrac{n-1}{2}\!+\!p)
   \sum_{i=0}^{N-k} \alpha_{i}^{(N-k)}(2N\!-\!k\!-\!\tfrac{n-1}{2}) (\delta \dm)^{N-i}
   \iota^*(\bar{\delta} \bar{\dm})^i \\
   & + (\tfrac{n-1}{2}\!-\!p\!-\!k) (k\!-\!\tfrac{n-1}{2}\!+\!p)
   \alpha_{N-k}^{(N-k)}(2N\!-\!k\!-\!\tfrac{n-1}{2}) (\delta \dm)^{k} \iota^*(\bar{\delta}
   \bar{\dm})^{N-k}.
\end{align*}
We note that, in the last sum, the term in the fourth line originates from the
composition of the corresponding third sum in \eqref{SBO-even} with
$L^{(p)}_{2k}$. It can be merged with the second line by extending the sum to
run up to $i=N-k$. Thus, it remains to prove
\begin{align*}
    \alpha_i^{(N)}(2N\!-\!k\!-\!\tfrac{n-1}{2}) & = 0, \quad i=N-k+1,\dots,N, \\
    \alpha_i^{(N)}(2N\!-\!k\!-\!\tfrac{n-1}{2}) & = \alpha_i^{(N-k)}(2N\!-\!k-\!\tfrac{n-1}{2}), \quad i=0,\dots,N-k.
\end{align*}
But both equalities are direct consequences of the definition of the
coefficients $\alpha_i^{(N)}(\lambda)$. We omit the details of the
calculations. This proves \eqref{fact-1b-2}.

It only remains to prove \eqref{fact-1}. Theorem \ref{coeffeven} and
\eqref{L-flat} yield
$$
   D_{2N}^{(p \to p)}(N\!-\!\tfrac{n-1}{2}) = (N\!-\!\tfrac{n-1}{2}\!+\!p) (\dm
   \delta)^N \iota^* + (-N\!-\!\tfrac{n-1}{2}\!+\!p) (\delta \dm)^N \iota^* = -L_{2N}^{(p)} \iota^*
$$
using
\begin{align*}
   \alpha_i^{(N)}(N\!-\!\tfrac{n-1}{2}) & = 0 \quad \mbox{for $i \ge 1$}, \\
   \alpha_0^{(N)}(N\!-\!\tfrac{n-1}{2}) & = 1.
\end{align*}
This proves the first identity. Similarly, Theorem \ref{coeffeven} and
\eqref{Lbar-flat} imply
$$
   D_{2N}^{(p \to p)}(N\!-\!\tfrac{n}{2}) = (N\!-\!\tfrac{n}{2}\!+\!p) \iota^* (\bar{\dm} \bar{\delta})^N
   + (-N\!-\!\tfrac{n}{2}\!+\!p) \iota^* (\bar{\delta} \bar{\dm})^N = - \iota^* \bar{L}_{2N}^{(p)}
$$
using
\begin{align*}
   \alpha_i^{(N)}(N\!-\!\tfrac{n}{2}) & = 0 \quad \mbox{for $i \le N-1$}, \\
   \alpha_N^{(N)}(N\!-\!\tfrac{n}{2}) & = 1.
\end{align*}
This proves the second identity.
\end{proof}

We continue with the {\bf proof of Theorem \ref{MainFactEven2}}.

The relation \eqref{fact-1b-1} implies
$$
   (-\tfrac{n}{2}+p+k) \star D_{2N}^{(n-p \to n-p)}(k-\tfrac n2) \, \bar{\star}
   = \left( \star \, D_{2N-2k}^{(n-p \to n-p)}(-k-\tfrac n2) \, \bar{\star}\right)
   \bar{\star}^2 \left(\bar{\star} \, \bar{L}_{2k}^{(n-p)} \, \bar{\star}\right).
$$
By Theorem \ref{Hodge-c}, the latter identity is equivalent to
$$
   (-\tfrac n2+p+k) D_{2N}^{(p \to p-1)}(k-\tfrac n2)
   = D_{2N-2k}^{(p \to p-1)}(-k-\tfrac n2) \, \bar{\star}^2
   \left(\bar{\star} \, \bar{L}_{2k}^{(n-p)} \, \bar{\star}\right).
$$
Combining this result with $\bar{\star} \, \bar{L}_{2k}^{(n-p)} \, \bar{\star}
= - \bar{\star}^2 \bar{L}_{2k}^{(p)}$ proves \eqref{fact-2b-1}. Similarly, the
first identity in \eqref{fact-1} yields
$$
   \star \, D^{(n-p \to n-p)}_{2N}(N-\tfrac{n-1}{2}) \, \bar{\star} = - \star L^{(n-p)}_{2N} \iota^* \bar{\star}.
$$
By Theorem \ref{Hodge-c}, this identity is equivalent to
$$
   (-1)^{np} D^{(p \to p-1)}_{2N}(N-\tfrac{n-1}{2})
   = - \left(\star \, L^{(n-p)}_{2N} \, \star \right) \star^2 \left(\star \, \iota^* \bar{\star} \right).
$$
Now, using $\star \, L_{2N}^{(n-p)} \, \star = - \star^2 L_{2N}^{(p-1)}$ and
Lemma \ref{HodgeLemma}, we find
$$
   D^{(p \to p-1)}_{2N}(N-\tfrac{n-1}{2}) = - L_{2N}^{(p-1)} \iota^* i_{\partial_n}.
$$
This proves the first identity in \eqref{fact-2}.

We omit the analogous proofs of \eqref{fact-2b-2} and of the second identity in
\eqref{fact-2}. \hfill $\square$

The main factorizations for even-order families of the first and the second
type have analogs for odd-oder families.

\begin{theorem}\label{MainFactOdd1} Let $N \in \N$ and $p=0,\dots,n-1$. Then the
odd-order families of the first type satisfy the factorization identities
\begin{equation}\label{main-fact-odd-1}
   (\tfrac{n}{2}\!-\!p\!+\!k) D_{2N+1}^{(p\to p)} (k\!-\!\tfrac{n}{2})
   = D_{2N+1-2k}^{(p\to p)}(-k\!-\!\tfrac{n}{2}) \circ \bar{L}_{2k}^{(p)}
\end{equation}
and
\begin{equation}\label{main-fact-odd-2}
   (\tfrac{n-1}{2}\!-\!p-\!k\!) D_{2N+1}^{(p\to p)}(2N\!+\!1\!-\!k\!-\!\tfrac{n-1}{2}) =
   L_{2k}^{(p)} \circ D_{2N+1-2k}^{(p\to p)}(2N\!+\!1\!-\!k\!-\!\tfrac{n-1}{2})
\end{equation}
for $k=1,\dots,N$.
\end{theorem}

\begin{proof} We apply Theorem \ref{coeffodd2}. Similar arguments as in the proof of
Theorem \ref{MainFactEven1} show that for the proof of \eqref{main-fact-odd-1}
it suffices to prove
\begin{align*}
   & \beta_i^{(N)}(k\!-\!\tfrac{n}{2}) = 0, \quad i=0,\dots,k-1, \\
   & \beta_i^{(N)}(k\!-\!\tfrac{n}{2}) = \beta_{i-k}^{(N-k)}(-k\!-\!\tfrac{n}{2}), \quad i=k,\dots,N
\end{align*}
and
\begin{align*}
   & \gamma_i^{(N)}(k\!-\!\tfrac{n}{2}) = 0, \quad i=0,\dots,k-1, \\
   & \gamma_i^{(N)}(k\!-\!\tfrac{n}{2}) = \gamma_{i-k}^{(N-k)}(-k\!-\!\tfrac{n}{2}), \quad i=k+1,\dots,N.
\end{align*}
But the first set of assertions follows directly from the definition
\eqref{b-even}. In turn, the second set of assertions follows by combining the
first set with Remark \ref{gamma-alternative}. Similarly, for the proof of
\eqref{main-fact-odd-2} it suffices to prove
\begin{align*}
   & \beta_i^{(N)}(2N\!+\!1\!-\!k\!-\!\tfrac{n-1}{2}) = 0, \quad i=N-k+1,\dots,N, \\
   & \beta_i^{(N)}(2N\!+\!1\!-\!k\!-\!\tfrac{n-1}{2}) = \beta_{i}^{(N-k)}(2N\!+\!1\!-\!k\!-\!\tfrac{n-1}{2}),
   \quad i=0,\dots,N-k
\end{align*}
and
\begin{align*}
   & \gamma_i^{(N)}(2N\!+\!1\!-\!k\!-\!\tfrac{n-1}{2}) = 0, \quad i=N-k+1,\dots,N, \\
   & \gamma_i^{(N)}(2N\!+\!1\!-\!k\!-\!\tfrac{n-1}{2}) = \gamma_{i}^{(N-k)}(2N\!+\!1\!-\!k\!-\!\tfrac{n-1}{2}),
   \quad i=0,\dots,N-k.
\end{align*}
Again, the first set of relations directly follows from the definition
\eqref{b-even} and the second set follows by combining the first set with
Remark \ref{gamma-alternative}.
\end{proof}

Now Hodge conjugation yields

\begin{theorem}\label{MainFactOdd2} Let $N \in \N_0$ and $p=1,\dots,n$. Then the
odd-order families of the second type satisfy the factorization identities
\begin{equation}\label{main-fact-odd-3}
   (\tfrac{n}{2}\!-\!p\!-\!k) D_{2N+1}^{(p\to p-1)} (k\!-\!\tfrac{n}{2})
   = D_{2N+1-2k}^{(p\to p-1)}(-k\!-\!\tfrac{n}{2}) \circ \bar{L}_{2k}^{(p)}
\end{equation}
and
\begin{equation}\label{main-fact-odd-4}
   (\tfrac{n+1}{2}\!-\!p+\!k\!) D_{2N+1}^{(p\to p-1)}(2N\!+\!1\!-\!k\!-\!\tfrac{n-1}{2}) =
   L_{2k}^{(p-1)} \circ D_{2N+1-2k}^{(p\to p-1)}(2N\!+\!1\!-\!k\!-\!\tfrac{n-1}{2})
\end{equation}
for $k=1,\dots,N$.
\end{theorem}

Finally, we reformulate Theorem \ref{MainFactEven1} and Theorem
\ref{MainFactEven2} so that the numerical coefficients on the left-hand sides
of the factorization identities disappear. For the families of the first type
we find the following result.

\begin{theorem}\label{MainFactEven1-re} Assume that $n$ is even and $p < \frac{n}{2}$. Then the identities
in Theorem \ref{MainFactEven1} are equivalent to the factorizations
\begin{align}
   \tilde{D}_{2N}^{(p \to p)}(k\!-\!\tfrac{n}{2}) & =
   \tilde{D}_{2N-2k}^{(p \to p)}(-k\!-\!\tfrac{n}{2}) \circ \tilde{\bar{L}}_{2k}^{(p)}, \label{reno-1} \\
   \tilde{D}_{2N}^{(p \to p)}(2N\!-\!k\!-\!\tfrac{n-1}{2}) & =
   \tilde{L}_{2k}^{(p)} \circ
   \tilde{D}_{2N-2k}^{(p \to p)}(2N\!-\!k-\!\tfrac{n-1}{2}) \label{reno-2}
\end{align}
for $k=1,\dots,N$ with the renormalized families
\begin{align*}
   \tilde{D}_{2N}^{(p \to p)}(\lambda) \st \frac{D_{2N}^{(p \to p)}(\lambda)}{\lambda\!+\!p\!-\!2N}
\end{align*}
and the renormalized Branson-Gover operators
\begin{equation}\label{BG-reno1}
   \tilde{L}^{(p)}_{2N} \st \frac{L_{2N}^{(p)}}{\frac{n}{2}\!-\!p\!+\!N} =
   (\delta d)^N + \cdots
\end{equation}
on a manifold of dimension $n$.
\end{theorem}

We stress that \eqref{reno-1} and \eqref{reno-2} also include the case $k=N$
with $\tilde{D}_0^{(p \to p)}(\lambda) = \iota^*$. The assumptions in Theorem
\ref{MainFactEven1-re} guarantee that both sides of the factorization
identities are well-defined. Theorem \ref{MainFactEven1-re} resembles the
factorization identities for residue families on functions \cite{Juhl},
\cite{Juhl1}.

For the second type families we have the following analogous result.

\begin{theorem}\label{MainFactEven2-re} Assume that $n$ is even and $p < \frac{n}{2}$.
Then the identities in Theorem \ref{MainFactEven2} are equivalent to the
factorizations
\begin{align}
   \tilde{D}_{2N}^{(p \to p-1)}(k\!-\!\tfrac{n}{2}) & =
   \tilde{D}_{2N-2k}^{(p \to p-1)}(-k\!-\!\tfrac{n}{2}) \circ \tilde{\bar{L}}_{2k}^{(p)}, \label{reno-3} \\
   \tilde{D}_{2N}^{(p \to p-1)}(2N\!-\!k\!-\!\tfrac{n-1}{2}) & =
   \tilde{L}_{2k}^{(p-1)} \circ
   \tilde{D}_{2N-2k}^{(p \to p-1)}(2N\!-\!k-\!\tfrac{n-1}{2}) \label{reno-4}
\end{align}
for $k=1,\dots,N$ with the renormalized families
\begin{align*}
   \tilde{D}_{2N}^{(p \to p-1)}(\lambda) \st \frac{D_{2N}^{(p \to p-1)}(\lambda)}{\lambda\!+\!n\!-\!p}.
\end{align*}
\end{theorem}

Similarly as above, the identities \eqref{reno-3} and \eqref{reno-4} include
the case $k=N$ with $\tilde{D}_0^{(p \to p-1)}(\lambda) =
-\iota^*i_{\partial_n}$. The assumptions in Theorem \ref{MainFactEven2-re}
guarantee that both sides of the factorization identities are well-defined.

We omit the formulation of the odd-order analogs of these results.

\subsection{Supplementary factorizations}\label{factor2}

In the present section, we establish additional factorization identities for
the conformal symmetry breaking families. These involve the four geometric
operators $\dm$, $\delta$, $\bar{\dm}$ and $\bar{\delta}$ as factors.

We first formulate the identities which involve even-order families of the
first type, i.e., families of the form $D_{2N}^{(p \to p)}(\lambda)$.

\begin{theorem}\label{SuppFact} For $N \in \N$, we have the factorization identities
\begin{align}
   D^{(p \to p)}_{2N}(-p\!+\!2N) & = -(2N) \dm D^{(p \to p-1)}_{2N-1}(-p\!+\!2N), \quad 1 \le p \le n-1,
   \label{eq:supp2b} \\
   D^{(p \to p)}_{2N}(-p) & = (2N)  D^{(p+1 \to p)}_{2N-1}(-p\!-\!1) \bar{\dm}, \quad 0 \le p \le
   n-1. \label{eq:supp2}
\end{align}
Moreover, for $N \in \N_0$, we have
\begin{equation}\label{eq:supp1}
   (n\!-\!2p\!-\!2N\!-\!1) D^{(p \to p-1)}_{2N+1}(p\!-\!n\!+\!2N\!+\!1) =
   \delta D^{(p\to p)}_{2N}(p\!-\!n\!+\!2N\!+\!1)
\end{equation}
for $1 \le p \le n-1$ and
\begin{equation} \label{eq:supp1b}
   (n\!-\!2p\!+\!2N) D_{2N+1}^{(p+1 \to p)}(-n\!+\!p\!+\!1) =
   D_{2N}^{(p \to p)}(-n\!+\!p) \bar{\delta}
\end{equation}
for $0 \le p \le n-1$.
\end{theorem}

We stress that all factors on the right-hand sides of these identities are
conformally equivariant. For instance, the individual factors on the right-hand
side of \eqref{eq:supp2b} intertwine
$$
   D_{2N-1}^{(p \to p-1)}(-p\!+\!2N): \dm \pi^{(p)}_{-2N} \to \dm \pi_0^{\prime(p-1)}
   \quad \mbox{and} \quad
   d: \dm\pi_0^{\prime (p-1)} \to \dm\pi_0^{\prime(p)}.
$$
Similarly, the factors on the right-hand side of \eqref{eq:supp1b} intertwine
$$
   D_{2N}^{(p \to p)}(-n\!+\!p): \dm \pi^{(p)}_{n-2p} \to \dm \pi_{n-2p+2N}^{\prime (p)}
  \quad \mbox{and} \quad \bar{\delta}: \dm \pi^{(p+1)}_{n-2p-2} \to \dm \pi_{n-2p}^{(p)}.
$$
Note that the latter intertwining relation is a consequence of general
properties of the co-differential. In fact, on $(M^n,g)$ we have
$\delta_{\gamma_*g} = \gamma_* \delta_g \gamma^*$ for any diffeomorphisms
$\gamma$ and
$$
   \delta_{e^{2\varphi} g} \circ e^{-(n-2p)\varphi} = e^{-(n-2p+2)\varphi} \circ
   \delta_g \quad \mbox{on $\Omega^p(M)$}
$$
for any conformal change of $g$.

Note that the relations \eqref{eq:supp2} are equivalent to the commutative
triangles mentioned in \cite[page 593]{Juhl0}.

The following proof of Theorem \ref{SuppFact} shows that all results are direct
consequences of the geometric formulas in Section \ref{geometric}.

\begin{proof} The assertions follow from Theorem \ref{coeffeven} and Theorem \ref{coeffodd}.
We start with the proof of \eqref{eq:supp2b}. On the one hand, Theorem
\ref{coeffodd} implies
\begin{align}\label{a-1}
   \dm D_{2N-1}^{(p \to p-1)}(-p\!+\!2N) & = -\sum_{i=1}^{N-1}
   \gamma_i^{(N-1)}(-p\!+\!2N;n\!-\!p) (\dm\delta)^{N-i} \iota^* (\bar{\dm}\bar{\delta})^i \notag \\
   & - (n\!-\!2p\!+\!2N) \sum_{i=1}^{N-1} \beta_i^{(N-1)}(-p\!+\!2N) (\dm\delta)^{N-i} \iota^*
   (\bar{\delta} \bar{\dm})^i \notag \\
   & + (n\!-\!2p\!+\!1) \beta_{N-1}^{(N-1)}(-p\!+\!2N) \iota^* (\bar{\dm}\bar{\delta})^N \notag \\
   & - (n\!-\!2p\!+\!2N) \beta_0^{(N-1)}(-p\!+\!2N) (\dm\delta)^N \iota^*.
\end{align}
On the other hand, Theorem \ref{coeffeven} gives
\begin{align}\label{a-2}
   D_{2N}^{(p \to p)}(-p\!+\!2N) & = \sum_{i=1}^{N-1} (2N\!-\!2i) \alpha_i^{(N)}(-p\!+\!2N)
   (\dm\delta)^{N-i} \iota^* (\bar{\delta} \bar{\dm})^i \notag \\
   & + 2N \sum_{i=1}^{N-1} \alpha_i^{(N)}(-p\!+\!2N) (\dm\delta)^{N-i} \iota^*
   (\bar{\dm}\bar{\delta})^i \notag \\
   & + 2N \alpha_0^{(N)}(-p\!+\!2N) (\dm\delta)^{N} \iota^* \notag \\
   & + 2N \alpha_N^{(N)}(-p\!+\!2N) \iota^* (\bar{\dm}\bar{\delta})^N.
\end{align}
Now we have the relations
\begin{equation}\label{rel-1b}
   \gamma_i^{(N-1)}(-p\!+\!2N;n\!-\!p) = \alpha_i^{(N)}(-p\!+\!2N), \quad i=1,\dots,N-1
\end{equation}
and
\begin{equation}\label{rel-1a}
   2N (n\!-\!2p\!+\!2N) \beta_i^{(N-1)}(-p\!+\!2N) = (2N\!-\!2i) \alpha_i^{(N)}(-p\!+\!2N), \quad
i=0,\dots,N-1.
\end{equation}
Indeed, \eqref{a-coeff-odd} implies
\begin{multline*}
   \gamma_i^{(N-1)}(-p\!+\!2N;n\!-\!p) = (-1)^i 2^{N-1} \frac{(N\!-\!1)! }{(2N\!-\!1)!} \binom{N}{i} \\
   \times (n\!+\!2N\!-\!2p)(n\!-\!2p\!+\!2N\!-\!2i\!+\!1)
   \prod_{k=i+1}^{N-1} (-2p\!+\!4N\!+\!n\!-\!2k) \prod_{k=1}^{i-1} (n\!-\!2p\!+\!2N\!-\!2k\!+\!1).
\end{multline*}
But the last line can be simplified by extending the products up to $k=N$ and
$k=i$, respectively. The result coincides with $\alpha_i^{(N)}(-p\!+\!2N)$.
This proves \eqref{rel-1b}. Next, the left-hand side of \eqref{rel-1a} equals
\begin{align*}
   & 2N (-1)^i 2^{N-1} \frac{(N\!-\!1)!}{(2N\!-\!1)!} \binom{N-1}{i} \\ &
   \times (n\!+\!2N\!-\!2p) \prod_{k=i+1}^{N-1} (-2p\!+\!n\!+\!4N\!-\!2k)
   \prod_{k=1}^i (-2p\!+\!2N\!+n\!-\!2k\!+\!1).
\end{align*}
The last line can be simplified by extending the first product up to $k=N$. By
simplification, the result coincides with the right-hand side of
\eqref{rel-1a}.

In addition, we have the relation
\begin{equation}\label{supp-1b}
   \alpha_N^{(N)}(-p\!+\!2N) = -(n\!-\!2p\!+\!1) \beta_{N-1}^{(N-1)}(-p\!+\!2N).
\end{equation}
In fact, by definition we have
$$
   -(n\!-\!2p\!+\!1) \beta_{N-1}^{(N-1)}(-p\!+\!2N) = (-1)^{N} 2^{N-1}
   \frac{(N\!-\!1)!}{(2N\!-\!1)!} (n\!-\!2p\!+\!1) \prod_{k=1}^{N-1} (n\!-\!2p\!+\!2N\!-\!2k\!+\!1).
$$
The latter formula simplifies by extending the product up to $k=N$. The result
coincides with $\alpha_N^{(N)}(-p+N)$.

Now \eqref{eq:supp2b} follows by combining \eqref{a-1} and \eqref{a-2} with the
relations \eqref{rel-1b}--\eqref{supp-1b}.

Next, we prove \eqref{eq:supp2}. Theorem \ref{coeffodd} implies
\begin{align}\label{b-1}
   D_{2N-1}^{(p+1 \to p)}(-p\!-\!1) \bar{\dm} & = -\sum_{i=1}^{N-1}
   \gamma_i^{(N-1)}(-p\!-\!1;n\!-\!p\!-\!1) (\delta \dm)^{N-i} \iota^* (\bar{\delta} \bar{\dm})^i \notag \\
   & + (n\!-\!2p\!-\!2N\!-\!1) \sum_{i=0}^{N-2} \beta_i^{(N-1)}(-p\!-\!1) (\dm \delta)^{N-1-i} \iota^* (\bar{\delta}
   \bar{\dm})^{i+1} \notag \\
   & + (n\!-\!2p\!-\!2N\!-\!1) \beta_{N-1}^{(N-1)}(-p\!-\!1) \iota^* (\bar{\delta} \bar{\dm})^N \notag \\
   & - (n\!-\!2p\!-\!2) \beta_0^{(N-1)}(-p\!-\!1) (\delta \dm)^N \iota^*.
\end{align}
But Theorem \ref{coeffeven} yields
\begin{align}\label{b-2}
   D_{2N}^{(p \to p)}(-p)
   & = -2N \sum_{i=1}^{N-1} \alpha_i^{(N)}(-p) (\delta \dm)^{N-i} \iota^* (\bar{\delta} \bar{\dm})^i \notag  \\
   & - \sum_{i=1}^{N-1} 2i \alpha_i^{(N)}(-p) (\dm \delta)^{N-i} \iota^* (\bar{\delta}
   \bar{\dm})^i \notag \\
   & - 2N \alpha_N^{(N)}(-p) \iota^* (\bar{\delta} \bar{\dm})^N \notag \\
   & - 2N \alpha_0^{(N)}(-p) (\delta \dm)^{N} \iota^*.
\end{align}
By combining \eqref{b-1} and \eqref{b-2} with the relations
\begin{equation}\label{rel-2a}
   \gamma_i^{(N-1)}(-p\!-\!1;n\!-\!p\!-\!1) = \alpha_i^{(N)}(-p), \quad
   i=1,\dots,N-1,
\end{equation}
\begin{equation}\label{rel-2b}
   2N (n\!-\!2p\!-\!2N\!-\!1) \beta_{i-1}^{(N-1)}(-p\!-\!1) = -2i \alpha_i^{(N)}(-p), \quad i=1,\dots,N
\end{equation}
and
\begin{equation}\label{rel-2c}
   (n\!-\!2p\!-\!2) \beta_0^{(N-1)}(-p\!-\!1) = \alpha_0^{(N)}(-p),
\end{equation}
we find
$$
   D_{2N}^{(p \to p)}(-p) = 2N D_{2N-1}^{(p+1 \to p)}(-p\!-\!1) \bar{\dm}.
$$
This proves \eqref{eq:supp2}.

Next, we prove the relations \eqref{rel-2a}--\eqref{rel-2c}. By
\eqref{a-coeff-odd}, we have
\begin{align*}
& \gamma_i^{(N-1)}(-p\!-\!1;n\!-\!p\!-\!1) = (-1)^{i} 2^{N-1} \frac{(N\!-\!1)!}{(2N\!-\!1)!} \binom{N}{i} \\
& \times (n\!-\!2p\!-\!2N\!-\!1)(n\!-\!2p\!-\!2i\!-\!2) \prod_{k=i+1}^{N-1}
(n\!-\!2p\!-\!2k\!-\!2) \prod_{k=1}^{i-1}(n\!-\!2p\!-\!2k\!-\!2N\!-\!1).
\end{align*}
The last line simplifies by letting the products run from $k=i$ and $k=0$,
respectively. The result coincides with $-\alpha_i^{(N)}(-p)$. This proves
\eqref{rel-2a}. Next, we have
\begin{multline*}
   2N (n\!-\!2p\!-\!2N\!-\!1) \beta_{i-1}^{(N-1)}(-p\!-\!1) = (-1)^{i-1} 2^N \frac{N!}{(2N\!-\!1)!}
   \binom{N-1}{i-1} \\ \times (n\!-\!2p\!-\!2N\!-\!1) \prod_{k=i}^{N-1}
   (n\!-\!2p\!-\!2k\!-\!2) \prod_{k=1}^{i-1} (n\!-\!2p\!-\!2N\!-\!2k\!-\!1).
\end{multline*}
Again, the last line simplifies by letting the last product run from $k=0$.
Further simplification confirms \eqref{rel-2b}. Finally, \eqref{rel-2c} is easy
to verify.

We continue with the proof of \eqref{eq:supp1}. On the one hand, Theorem
\ref{coeffodd} implies
\begin{align}\label{c-1}
   D_{2N+1}^{(p \to p-1)}(p\!-\!n\!+\!2N\!+\!1) & = -\sum_{i=1}^N \gamma_i^{(N)}(p\!-\!n\!+\!2N\!+\!1;n\!-\!p)
   (\delta \dm)^{N-i} \delta \iota^* (\bar{\dm} \bar{\delta})^i \notag \\
   & - (2N\!+\!1) \sum_{i=1}^{N-1} \beta_i^{(N)}(p\!-\!n\!+\!2N\!+\!1)
   (\delta \dm)^{N-i} \delta \iota^* (\bar{\delta} \bar{\dm})^i \notag \\
   & - (2N\!+\!1) \beta_0^{(N)}(p\!-\!n\!+\!2N\!+\!1) (\delta \dm)^N \delta \iota^* \notag\\
   & - (2N\!+\!1) \beta_N^{(N)}(p\!-\!n\!+\!2N\!+\!1) \iota^* (\bar{\delta}
   \bar{\dm})^N.
\end{align}
On the other hand, Theorem \ref{coeffeven} yields
\begin{align}\label{c-2}
   \delta D_{2N}^{(p \to p)}(p\!-\!n\!+\!2N\!+\!1) & = (2p\!-\!n\!+\!2N\!+\!1) \sum_{i=1}^N
   \alpha_i^{(N)}(p\!-\!n\!+\!2N\!+\!1) \delta (\dm \delta)^{N-i} \iota^* (\bar{\dm} \bar{\delta})^i \notag \\
   & +  (2p\!-\!n\!+\!2N\!+\!1)  \alpha_0^{(N)}(p\!-\!n\!+\!2N\!+\!1) \delta (\dm \delta)^{N} \iota^* \notag \\
   & + (2p\!-\!n\!+\!1) \alpha_N^{(N)}(p\!-\!n\!+\!2N\!+\!1) \delta \iota^* (\bar{\delta} \bar{\dm})^N \\
   & + \sum_{i=1}^{N-1} (2p\!-\!n\!+\!2N\!+\!1\!-\!2i) \alpha_i^{(N)}(p\!-\!n\!+\!2N\!+\!1)
   (\delta \dm)^{N-i} \delta \iota^*(\bar{\delta} \bar{\dm})^{i}. \notag
\end{align}
But we have the relations
\begin{equation}\label{rel-3a}
   \gamma_i^{(N)}(p\!-\!n\!+\!2N\!+\!1;n\!-\!p) = \alpha_i^{(N)}(p\!-\!n\!+\!2N\!+\!1)
\end{equation}
for $i=1,\dots,N$ and
\begin{equation}\label{rel-3b}
   -(n\!-\!2p\!-\!2N\!-\!1) (2N\!+\!1) \beta_i^{(N)}(p\!-\!n\!+\!2N\!+\!1)
   = (2p\!-\!n\!+\!2N\!-\!2i\!+\!1) \alpha_i^{(N)}(p\!-\!n\!+\!2N\!+\!1)
\end{equation}
for $i=0,\dots,N$.

The assertion \eqref{eq:supp1} follows by combining \eqref{c-1} and \eqref{c-2}
with \eqref{rel-3a} and \eqref{rel-3b}.

Now, we prove the relations \eqref{rel-3a} and \eqref{rel-3b}. The definition
\eqref{a-coeff-odd} yields
\begin{multline*}
   \gamma_i^{(N)}(p\!-\!n\!+\!2N\!+\!1;n\!-\!p)
   = (-1)^i 2^N \frac{N!}{(N\!+\!1)(2N)!} \binom{N+1}{i} (N\!-\!i\!+\!1) \\
   \times (2p\!-\!n\!+\!2N\!+\!1) \prod_{k=i+1}^N (2p\!-\!n\!+\!4N\!-\!2k\!+\!2)
   \prod_{k=1}^{i-1} (2p\!-\!n\!+\!2N\!-\!2k\!+\!1).
\end{multline*}
The second line simplifies by letting the last product run from $k=0$. Further
simplification confirms \eqref{rel-3a}. Next, we find
\begin{multline*}
   (2p\!-\!n\!+\!2N\!-\!2i\!+\!1) \alpha_i^{(N)}(p\!-\!n\!+\!2N\!+\!1) = (-1)^i
   2^N \frac{N!}{(2N)!} \binom{N}{i} \\ \times (2p\!-\!n\!+\!2N\!-\!2i\!+\!1)
   \prod_{k=i+1}^N (2p\!-\!n\!+\!4N\!-\!2k\!+\!2) \prod_{k=1}^i (2p\!-\!n\!+\!2N\!-\!2k\!+\!3)
\end{multline*}
and
\begin{multline*}
   (n\!-\!2p\!-\!2N\!-\!1) (2N\!+\!1) \beta_i^{(N)}(p\!-\!n\!+\!2N\!+\!1)
   = (-1)^{i} 2^N \frac{N!}{(2N)!} \binom{N}{i} \\
   \times (n\!-\!2p\!-\!2N\!-\!1) \prod_{k=i+1}^N (2p\!-\!n\!+\!4N\!-\!2k\!+\!2)
   \prod_{k=1}^i (2p\!-\!n\!+\!2N\!-\!2k\!+\!1).
\end{multline*}
These formulas easily imply \eqref{rel-3b}.

Finally, we prove \eqref{eq:supp1b}. On the one hand, Theorem \ref{coeffodd}
implies
\begin{align}\label{help4-1}
   D_{2N+1}^{(p+1 \to p)}(-n\!+\!p\!+\!1)
   & = -\sum_{i=1}^N \gamma_i^{(N)}(-n\!+\!p\!+\!1;n\!-\!p\!-\!1)
   (\delta \dm)^{N-i} \delta \iota^*(\bar{\dm}\bar{\delta})^i \notag \\
   & - (2N\!+\!1) \sum_{i=1}^{N-1} \beta_i^{(N)}(-n\!+\!p\!+\!1) (\dm \delta)^{N-i} \iota^*
   \bar{\delta} (\bar{\dm} \bar{\delta})^i \notag \\
   & - (2N\!+\!1) \beta_0^{(N)}(-n\!+\!p+\!1) (\dm \delta)^N \iota^* \bar{\delta} \notag \\
   & - (2N\!+\!1) \beta_N^{(N)}(-n\!+\!p\!+\!1) \iota^* \bar{\delta} (\bar{\dm} \bar{\delta})^N.
\end{align}
On the other hand, Theorem \ref{coeffeven} gives
\begin{align}\label{help4-2}
   D_{2N}^{(p \to p)}(-n\!+\!p) \bar{\delta}
   & = \sum_{i=1}^{N-1} (-n\!+\!2p\!-\!2i) \alpha_i^{(N)}(-n\!+\!p)
   (\dm\delta)^{N-i} \iota^* (\bar{\delta} \bar{\dm})^i \bar{\delta} \notag \\
   & + (-n\!+\!2p\!-\!2N) \sum_{i=1}^N \alpha_{i-1}^{(N)}(-n\!+\!p)
   (\delta \dm)^{N-i} \delta \iota^* (\bar{\dm} \bar{\delta})^i \notag \\
   & + (-n\!+\!2p) \alpha_0^{(N)}(-n\!+\!p) (\dm \delta)^{N} \iota^* \bar{\delta} \notag \\
   & + (-n\!+\!2p\!-\!2N) \alpha_N^{(N)}(-n\!+\!p) \iota^* (\bar{\delta} \bar{\dm})^N \bar{\delta}
\end{align}
Combining \eqref{help4-1} and \eqref{help4-2} with the relations
\begin{equation}\label{rel-4b}
   \gamma_i^{(N)}(-n\!+\!p\!+\!1;n\!-\!p\!-\!1) = \alpha_{i-1}^{(N)}(-n\!+\!p), \quad i=1,\dots,N
\end{equation}
and
\begin{equation}\label{rel-4a}
   (n\!-\!2p\!+\!2N) (2N\!+\!1) \beta_i^{(N)}(-n\!+\!p\!+\!1) = (n\!-\!2p\!+\!2i) \alpha_i^{(N)}(-n\!+\!p),
   \quad i=0,\dots,N
\end{equation}
completes the proof of \eqref{eq:supp1b}.

Now we prove \eqref{rel-4b} and \eqref{rel-4a}. Equation \eqref{a-coeff-odd}
implies that
\begin{multline*}
   \gamma_i^{(N)}(-n\!+\!p\!+\!1;n\!-\!p\!-\!1) = (-1)^i 2^N \frac{N!}{(2N\!+\!1)!}
   \frac{1}{N\!+\!1} \binom{N+1}{i} i (2N\!+\!1) \\ \times (n\!-\!2p\!+\!2N)
   \prod_{k=i+1}^N (-n\!+\!2p\!-\!2k\!+\!2) \prod_{k=1}^{i-1}
   (-n\!+\!2p\!-\!2k\!-\!2N\!+\!1).
\end{multline*}
The second line simplifies by letting the first product run up to $k=N+1$.
Further simplification shows that the result coincides with
$\alpha_{i-1}^{(N)}(-n\!+\!p)$. This proves \eqref{rel-4b}. Finally, the
left-hand side of \eqref{rel-4a} equals
\begin{equation*}
   (-1)^i 2^N \frac{N!}{(2N)!} \binom{N}{i} (n\!-\!2p\!+\!2N) \prod_{k=i+1}^N
   (-n\!+\!2p\!-\!2k\!+\!2) \prod_{k=1}^i (-n\!+\!2p\!-\!2k\!-\!2N\!+\!1).
\end{equation*}
The last expression simplifies by letting the first product run up to $k=N+1$.
On the other hand, the right-hand side of \eqref{rel-4a} equals
\begin{equation*}
   (-1)^i 2^N \frac{N!}{(2N)!} \binom{N}{i} (n\!-\!2p\!+\!2i) \prod_{k=i+1}^N
   (-n\!+\!2p\!-\!2k) \prod_{k=1}^i (-n\!+\!2p\!-\!2N\!-\!2k\!+\!1).
\end{equation*}
The latter expression simplifies by letting the first product run from $k=i$.
This proves \eqref{rel-4a}. The proof is complete.
\end{proof}

By Hodge conjugation, Theorem \ref{SuppFact} implies the following
factorization identities which involve even-order families of the second type.

\begin{theorem}\label{SuppFact-2} For $N \in \N$, we have the factorization identities
\begin{align*}
   D^{(p\to p-1)}_{2N}(p\!-\!n\!+\!2N) & = -(2N) \delta D^{(p \to p)}_{2N-1}(p\!-\!n\!+\!2N), \quad 1 \le p \le n-1, \\
   D^{(p\to p-1)}_{2N}(p\!-\!n) & = (2N)  D^{(p-1 \to p-1)}_{2N-1}(p\!-\!n) \bar{\delta}, \quad 1 \le p \le n.
\end{align*}
Moreover, for $N \in \N_0$, we have
\begin{align*}
   (-n\!+\!2p\!-\!2N\!-\!1) D^{(p \to p)}_{2N+1}(-p\!+\!2N\!+\!1) & =
   \dm D^{(p \to p-1)}_{2N}(-p\!+\!2N\!+\!1), \quad 1 \le p \le n-1,\\
   (n\!-\!2p\!-\!2N\!-\!2) D_{2N+1}^{(p \to p)}(-p) & = D_{2N}^{(p+1 \to p)}(-p\!-\!1) \bar{\dm},
   \quad 0 \le p \le n-1. \label{eq:supp1b}
\end{align*}
\end{theorem}

\begin{proof} The results follow by combining Theorem \ref{Hodge-c} with
Theorem \ref{SuppFact}. We omit the details.
\end{proof}

\begin{bem}\label{SF-CBGO} Theorem \ref{SuppFact} and Theorem \ref{SuppFact-2} have extensions
to general metrics in terms of residue families \cite{fjs}. From that point of
view, the following arguments are of interest. The left-hand side of
\eqref{eq:supp1} vanishes for odd $n$ and $p=\frac{n-1}{2}-N$. Hence
$$
   \delta D_{2N}^{(p \to p)}(-\tfrac{n-1}{2}+N) = 0.
$$
By Theorem \ref{MainFactEven1}, it follows that
$$
   \delta L_{2N}^{(\frac{n-1}{2}-N)} \iota^* = 0.
$$
But this vanishing follows from the double factorization \eqref{double} of the
critical Branson-Gover operators on a manifold of even dimension $n$.
Similarly, for even $n$ and $p=\frac{n}{2}-N-1$, the last identity in Theorem
\ref{SuppFact-2} implies
$$
   D_{2N}^{(\frac{n}{2}-N \to \frac{n}{2}-N-1)}(-\tfrac{n}{2}+N) \bar{d} = 0.
$$
By Theorem \ref{MainFactEven2}, it follows that
$$
   \iota^* i_{\partial_n} \bar{L}_{2N}^{(\frac{n}{2}-N)} \bar{d} = 0.
$$
Again this vanishing result should be regarded as a consequence of the double
factorization of the critical Branson-Gover operators.
\end{bem}

\subsection{Applications}\label{applic}

In the present section, we describe two applications to gauge companion and
$Q$-curvature operators.

Equation \eqref{fact-1} in Theorem \ref{MainFactEven1} and Equation
\eqref{fact-2} in Theorem \ref{MainFactEven2} show that even-order symmetry
breaking operators of both types specialize at certain arguments to
Branson-Gover operators. The following result provides an analogous description
of the gauge companion operators in terms of odd-order conformal symmetry
breaking operators of the second type.

\begin{theorem}\label{TheGOperator} Assume that $n-1$ is even and that $n-2p \ge 1$. Then
the restriction of $D^{(p \to p-1)}_{n-2p}(-p)$ to the space of closed
$p$-forms is given by the gauge companion operator $G_{n-2p}^{(p)}$. More
precisely, we have
\begin{equation}\label{DG}
   D^{(p\to p-1)}_{n-2p}(-p)|_{\ker(\bar{\dm})} = -G^{(p)}_{n-2p} \iota^*.
\end{equation}
\end{theorem}

\begin{proof} Let $n-2p=2N+1$. Theorem \ref{coeffodd} implies
\begin{equation*}
   D^{(p\to p-1)}_{n-2p}(-p)|_{\ker(\bar{\dm})} = - \sum_{i=1}^N
   \gamma_i^{(N)}(-p;n\!-\!p) (\delta \dm)^{N-i} \delta \iota^* (\bar{\dm} \bar{\delta})^i
   - (n\!-\!2p) \beta_0^{(N)}(-p) (\delta \dm)^{N} \delta \iota^*.
\end{equation*}
But
$$
   \gamma_i^{(N)}(-p;n\!-\!p) = 0 \quad \mbox{for $i \ge 1$}
   \quad \mbox{and} \quad (n\!-\!2p) \beta^{(N)}_0(-p) = 1
$$
yield the result
\begin{equation*}
   D^{(p\to p-1)}_{n-2p}(-p)|_{\ker(\bar{\dm})} = -\delta (\dm\delta)^{\frac{n-1}{2}-p}\iota^*.
\end{equation*}
This proves the assertion.
\end{proof}

\begin{example} We illustrate Theorem \ref{TheGOperator} in low-order special cases. Using Example
\ref{LowOrderExampleDiffOp}, we get
\begin{align*}
   & D_1^{(p\to p-1)}(-p) = -\delta\iota^* = -G_1^{(p)} \iota^* \quad \text{for}\quad n-2p=1,\\
   & D_3^{(p\to p-1)}(-p)|_{\ker(\bar{d})} = -\delta \dm\delta\iota^* = -G_3^{(p)}
   \iota^* \quad \text{for }\quad n-2p=3.
\end{align*}
Note that for the first-order operator the  restriction to the kernel of
$\bar{\dm}$ is unnecessary.
\end{example}

We continue with a discussion of {\em $Q$-curvature} operators. We observe that
the supplementary factorization \eqref{eq:supp2} (or Theorem \ref{coeffeven})
shows that the restriction of $D_{2N}^{(p\to p)}(\lambda)$ to $\ker(\bar{\dm})$
vanishes at $\lambda=-p$. This motivates the following definition.

\begin{defn}\label{QCurvPoly} The $Q$-curvature polynomial
$$
   Q^{(p)}_{2N}(\lambda): \Omega^p(\R^n)|_{\ker(\bar{\dm})} \to \Omega^p(\R^{n-1})
$$
is defined by
\begin{equation}\label{Q-poly}
   (\lambda+p) Q^{(p)}_{2N}(\lambda) \st D_{2N}^{(p \to p)}(\lambda)|_{\ker(\bar{\dm})}.
\end{equation}
\end{defn}

Definition \ref{QCurvPoly} extends the notion of $Q$-curvature polynomials
introduced in \cite{Juhl}. We recall that, for general metrics, we define
$Q^{res}_{2N}(\lambda) \st D^{res}_{2N}(\lambda)(1)$. The polynomial
$Q^{res}_{2N}(\lambda)$ vanishes at $\lambda=0$. For the Euclidean metric
$g_0$, however, all $Q$-curvature polynomials $Q^{res}_{2N}(\lambda)$ vanish
identically. The fact that this is no longer the case for $p>0$ follows from
the following consequence of Theorem \ref{coeffeven}.

\begin{cor}\label{Q-pol} The polynomials $Q_{2N}^{(p)}(\lambda):
\Omega^p(\R^n) \to \Omega^p(\R^{n-1})$ are given by the explicit formula
$$
   Q_{2N}^{(p)}(\lambda) = \sum_{i=0}^N \alpha_i^{(N)}(\lambda) (\dm\delta)^{N-i} \iota^* (\bar{\dm}
   \bar{\delta})^i
$$
with the coefficients $\alpha_i^{(N)}$ as defined by \eqref{a-even}.
\end{cor}

\begin{example} We have
\begin{equation*}
   Q^{(p)}_{2}(\lambda) = (2\lambda\!+\!n\!-\!2) \dm \delta\iota^* -(2\lambda\!+\!n\!-\!3)\iota^*\bar{\dm}\bar{\delta}
\end{equation*}
and
\begin{align*}
   Q^{(p)}_{4}(\lambda) & = \tfrac 13 \big[(2\lambda\!+\!n\!-\!2)(2\lambda\!+\!n\!-\!4)
   (\dm\delta)^2 \iota^*
   -2(2\lambda\!+\!n\!-\!4)(2\lambda\!+\!n\!-\!5)(\dm\delta)\iota^*(\bar{\dm}\bar{\delta})\\
   & + (2\lambda\!+\!n\!-\!5)(2\lambda\!+\!n\!-\!7)(\bar{\dm}\bar{\delta})^2\big].
\end{align*}
\end{example}

Corollary \ref{Q-pol} shows that, for general $\lambda$, the $Q$-curvature
polynomial $Q^{(p)}_{2N}(\lambda)$ is {\em not} a tangential operator with
respect to the subspace $\R^{n-1}$. However, for certain values of $\lambda$ it
reduces to a tangential operator. In fact, we have the following result.

\begin{cor}\label{TangQPoly} At the argument $\lambda=N-\frac{n-1}{2}$,
the $Q$-curvature polynomial $Q^{(p)}_{2N}(\lambda)$ reduces to a tangential
operator. More precisely, we have the formula
$$
   Q_{2N}^{(p)}(N-\tfrac{n-1}{2}) = (\dm\delta)^N \iota^*.
$$
\end{cor}

\begin{proof} The result follows from Corollary \ref{Q-pol} using
$$
\alpha_i^{(N)}(N-\tfrac{n-1}{2})=0 \quad \mbox{for $i \ge 1$}
$$
and
$$
\alpha_0^{(N)}(N-\tfrac{n-1}{2})=1.
$$
The proof is complete.
\end{proof}

Corollary \ref{TangQPoly} shows that, for the Euclidean metric, the following
definition of the $Q$-curvature operator using conformal symmetry breaking
operators coincides with the definition by
\begin{equation}\label{Q-classical}
L_{2N}^{(p)}|_{\ker \dm} = (\tfrac{n}{2}\!-\!N\!-\!p) Q_{2N}^{(p)},
\end{equation}
which generalizes the classical definition
$$
P_{2N}(1) = (\tfrac{n}{2}\!-\!N) Q_{2N}
$$
of the $Q$-curvature $Q_{2N}$. Note that \eqref{Q-classical} defines
$Q_{2N}^{(p)}$ only if $(\tfrac{n}{2}\!-\!N\!-\!p) \ne 0$, i.e., only in the
non-critical case.

\begin{defn}\label{Q-curvature}
\begin{equation}
   Q_{2N}^{(p)} \st Q_{2N}^{(p)}(N-\tfrac{n-1}{2}):
   \Omega^p(\R^{n-1})|_{\ker(\dm)} \to \Omega^p(\R^{n-1}).
\end{equation}
\end{defn}

The above observations also lead to the following result which should be
considered as an analog of the holographic description of the critical Branson
curvature $Q_n$ of $(M^n,g)$ in terms of $\dot{D}_n^{res}(0)(1)$ (\cite{Juhl},
\cite{JG-holo}); here dot denotes the derivative with respect to $\lambda$.

\begin{theorem}\label{holo-formula} Assume that $n-1$ is even and that $n-2p \ge 3$.
Then
\begin{equation}\label{holo}
   \dot{D}_{n-1-2p}^{(p \to p)}(-p)|_{\ker (\bar{\dm})} = Q_{n-1-2p}^{(p)} \iota^*
\end{equation}
as an identity of operators on closed forms on $\R^n$.
\end{theorem}

\begin{proof} On the one hand, \eqref{Q-poly} implies
$$
   \dot{D}_{2N}^{(p \to p)}(-p)|_{\ker (\bar{d})} = Q_{2N}^{(p)}(-p).
$$
On the other hand, we have
$$
   Q_{2N}^{(p)}(N-\tfrac{n-1}{2}) = (\dm\delta)^N \iota^*
$$
by Corollary \ref{TangQPoly}. In the critical case $2N=n-1-2p$, a comparison of
both facts proves the assertion.
\end{proof}

Finally, we give a proof of the double factorization property of the critical
Branson-Gover operators for the Euclidean metric on $\R^{n-1}$ from the
perspective of conformal symmetry breaking operators.

\begin{cor}\label{OneSideFact} Let $n-1$ be even and $p < \frac{n-1}{2}$. Then the
critical Branson-Gover operators $L_{n-2p-1}^{(p)}$ of the Euclidean metric on
$\R^{n-1}$ satisfy the double factorization identity
\begin{equation*}
   L_{n-2p-1}^{(p)} = (n\!-\!2p\!-\!1) \, \delta Q^{(p+1)}_{n-2p-3} \dm,
\end{equation*}
where $Q_{n-2p-3}^{(p+1)}=(\dm\delta)^{\frac{n-3}{2}-p}$.
\end{cor}

\begin{proof} The supplementary factorizations \eqref{eq:supp2} and \eqref{eq:supp1}
read
\begin{equation}\label{eq:FirstFact}
   D_{2N}^{(p\to p)}(-p) = 2N D_{2N-1}^{(p+1\to p)}(-p-1) \bar{\dm},
\end{equation}
and
\begin{equation}\label{eq:SecondFact}
   (2N\!+\!2p\!-\!n\!+1)D_{2N-1}^{(p+1\to p)}(p\!-\!n\!+\!2N)
   = -\delta D_{2N-2}^{(p+1\to p+1)}(p\!-\!n\!+\!2N),
\end{equation}
respectively. By Theorem \ref{TangQPoly}, we have
\begin{equation*}
   D_{2N-2}^{(p+1\to p+1)}(p\!-\!n\!+\!2N)|_{\ker(\bar{\dm})}
   = (2N\!+\!2p\!-\!n\!+\!1) \, Q_{2N-2}^{(p+1)}(p\!-\!n\!+\!2N).
\end{equation*}
Now assume that $(2N+2p-n+1) \ne 0$. Then the restriction of
\eqref{eq:SecondFact} to the subspace of closed forms gives, after division by
the common factor,
\begin{equation}\label{eq:QApperas}
   D_{2N-1}^{(p+1\to p)}(p\!-\!n\!+\!2N)|_{\ker(\bar{\dm})}
   = -\delta Q_{2N-2}^{(p+1\to p+1)}(p\!-\!n\!+\!2N).
\end{equation}
Next, we apply analytic continuation in the dimension $n$ and conclude that if
$2N=n-2p-1$ then \eqref{eq:FirstFact}\ and \eqref{eq:QApperas} combine into
\begin{equation*}
   D_{n-2p-1}^{(p\to p)}(-p) = -(n\!-\!2p\!-\!1) \delta Q_{n-2p-3}^{(p+1)}(-p-1) \bar{\dm}.
\end{equation*}
An application of Theorem \ref{MainFactEven1} to the left-hand side and an
application of Theorem \ref{TangQPoly} to the right-hand side turn the last
equation into
\begin{equation*}
   L_{n-2p-1}^{(p)} = (n\!-\!1\!-\!2p)\, \delta Q_{n-2p-3}^{(p+1)} \dm
\end{equation*}
using $\dm \iota^* = \iota^* \bar{\dm}$. This completes the proof.
\end{proof}

Analogous arguments using the supplementary factorizations \eqref{eq:supp2b}
and \eqref{eq:supp1b} prove the double factorization identity
\begin{equation}\label{double-fact-2}
   \iota^* \bar{L}^{(p)}_{2p-n} = (n\!-\!2p) d \dot{D}_{2p-n-2}^{(p-1 \to
   p-1)}(-n\!+\!p\!-\!1) \bar{\delta}
\end{equation}
for even $n$ and $\frac{n}{2} < p \le n-1$. Alternatively,
\eqref{double-fact-2} is a direct consequence of Theorem \ref{coeffeven}.

\section*{Appendix: Gegenbauer and Jacobi polynomials}\label{app}

We summarize basic conventions and properties concerning Gegenbauer and Jacobi
polynomials.

First, we recall that the Pochhammer symbol of $a\in\C$ is defined by
$$
   (a)_l \st a(a+1) \cdots (a+l-1)
$$
for $l\in\N$, and $(a)_0 \st 1$. Then the generalized hypergeometric function
${}_pF_q$ of type $(p,q)$, $p,q\in\N$, is defined by the Taylor series
\begin{equation}
    {}_pF_q\left[\begin{matrix}a_1\;,\;\dots\;,\;a_p\\b_1\;,\;\dots\;,\;b_q\end{matrix};z\right]
      \st \sum_{l=0}^\infty\frac{(a_1)_l\dots (a_p)_l}{(b_1)_l\dots (b_q)_l}\frac{z^l}{l!},
\end{equation}
for $a_i\in\C$ ($1\leq i\leq q$), $b_j\in\C\setminus\{-\N_0\}$ ($1\leq j\leq
q$), and $z\in\C$.

Let $m\in\N$ and $\alpha, \beta \in \C$ such that $(\alpha+1)_m \neq 0$. The
corresponding Jacobi polynomial of degree $m$ is defined by
\begin{equation}\label{jac}
   P_m^{(\alpha,\beta)}(z) \st \frac{(\alpha+1)_m}{m!} {}_2F_1
   \left[\begin{matrix}-m\;,\;1+\alpha+\beta+m\\\alpha+1\end{matrix};\frac{1-z}{2}\right].
\end{equation}
A specialization of Jacobi polynomials leads to Gegenbauer polynomials:
\begin{align}
   C_m^{\alpha}(z) & = \frac{(2\alpha)_m}{(\alpha+\frac 12)_m} P_m^{(\alpha-\frac 12,\alpha-\frac 12)}(z)\notag\\
   & = \frac{(2\alpha)_{m}}{m!}
   \,_2F_1\left[\begin{matrix}-m\;,\;2\alpha+m\\ \alpha+\tfrac{1}{2}\end{matrix};\frac{1-z}{2}\right].
\end{align}
An explicit formula for Gegenbauer polynomials (see \cite[Section $3.15$,
formula $(9)$]{batemanerdelyi}) reads
\begin{equation}\label{Gegenbauer-Taylor}
   C_m^{\alpha}(z)=\sum_{k=0}^{\lfloor m/2\rfloor} (-1)^k \frac{\Gamma(m-k+\alpha)}
   {\Gamma(\alpha)k!(m\!-\!2k)!} (2z)^{m-2k}.
\end{equation}
In particular, we have the even polynomials
\begin{equation*}
   \frac{N!}{(-\lambda\!-\!\frac{n-1}{2})_N} C_{2N}^{-\lambda-\frac{n-1}{2}}(z)
   = \sum_{j=0}^N a_j^{(N)}(\lambda) (-1)^j z^{2N-2j}
\end{equation*}
and the odd polynomials
\begin{equation*}
   \frac{N!}{2(-\lambda\!-\!\frac{n-1}{2})_{N+1}} C_{2N+1}^{-\lambda-\frac{n-1}{2}}(z)
   = \sum_{j=0}^N b_j^{(N)}(\lambda)(-1)^j z^{2N+1-2j}
\end{equation*}
with the coefficients
\begin{equation}\label{eq:DefA}
   a_j^{(N)}(\lambda) \st (-2)^{N-j} \frac{N!}{j!(2N\!-\!2j)!}
   \prod_{k=j}^{N-1}(2\lambda\!-\!4N\!+\!2k\!+\!n\!+\!1) a_N^{(N)}(\lambda)
\end{equation}
and
\begin{equation}\label{eq:DefB}
   b_j^{(N)}(\lambda)\st (-2)^{N-j} \frac{N!}{j!(2N\!-\!2j\!+\!1)!}
   \prod_{k=j}^{N-1}(2\lambda\!-\!4N\!+\!2k\!+\!n\!-\!1) b_N^{(N)}(\lambda)
\end{equation}
for $0\leq j\leq N-1$ and $a_N^{(N)}(\lambda) = b_N^{(N)}(\lambda)=1$ (see
\cite[Theorems 5.1.2, 5.1.4]{Juhl}).

In the present paper, we also refer to the coefficients \eqref{eq:DefA} and
\eqref{eq:DefB} with {\em non-trivial normalizations} $a_N^{(N)}(\lambda)$ and
$b_N^{(N)}(\lambda)$ as {\em even} and {\em odd} {\em Gegenbauer coefficients},
respectively. However, we usually reserve the notation $a_j^{(N)}(\lambda)$ and
$b_j^{(N)}(\lambda)$ for the coefficients \eqref{eq:DefA} and \eqref{eq:DefB}
with the trivial normalizations $a_N^{(N)}(\lambda)=1$ and
$b_N^{(N)}(\lambda)=1$. Gegenbauer coefficients with non-trivial normalizations
will be denoted by different letters, e.g., $p_j^{(N)}(\lambda)$.

Accordingly (and by abuse of language) we shall also talk about Gegenbauer
polynomials with non-trivial normalizations. Additional parameters in
Gegenbauer coefficients are always separated by a semicolon $``;"$. The
Gegenbauer coefficients satisfy the recurrence relations
\begin{align}
   (N\!-\!j\!+\!1)(2N\!-\!2j\!+\!1) a^{(N)}_{j-1}(\lambda) +
   j(2\lambda\!+\!n\!-\!4N\!+\!2j\!-\!1) a^{(N)}_j(\lambda) & = 0, \label{Gegen-rec1} \\
   (N\!-\!j\!+\!1)(2N\!-\!2j\!+\!3) b^{(N)}_{j-1}(\lambda) +
   j(2\lambda\!+\!n\!-\!4N\!+\!2j\!-\!3) b^{(N)}_j(\lambda) & = 0 \label{Gegen-rec2}
\end{align}
for all $1 \le j \le N$.


\end{document}